\documentclass[11pt]{article}
\usepackage[T1]{fontenc}
\usepackage[utf8]{inputenc}

\usepackage{graphicx}
\usepackage{amsmath}
\usepackage{amssymb}
\usepackage{amsfonts}
\usepackage{amsrefs}
\usepackage{amsthm}
\usepackage{mathtools}
\usepackage{multirow}
\usepackage{enumitem}
\usepackage{tikz-cd}
\usepackage{soul}
\usepackage{abstract}
\usepackage{mathrsfs}
\usepackage{bbm}
\usepackage{float}
\usepackage{geometry}
\usepackage{xcolor}
\usepackage{adjustbox}
\usepackage{tabularray}
\usepackage{subfig}
\usepackage{setspace}
\usepackage{titlesec}
\usepackage{dsfont}
\usepackage{stmaryrd}
\usepackage{breqn}
\usepackage{authblk}
\usepackage[labelfont=bf, labelsep=none]{caption}
\usetikzlibrary{braids}
\usepackage{calc}

\colorlet{underlinecolour}{white!10!gray}
\setulcolor{underlinecolour}

\newdimen\mywidth
\setbox0=\hbox{<text to measure>}
\mywidth=\wd0

\newcounter{dummy}
\makeatletter
\newcommand\myitem[1][]{\item[#1]\refstepcounter{dummy}\def\@currentlabel{#1}}
\makeatother

\titleformat{\section}
{\bfseries}
{\thesection}
{1em}
{}

\titleformat*{\subsection}{\bfseries}

\newcommand{\kk}{\mathbbm{k}}
\newcommand{\Imin}{I_{\mathrm{min}}}
\newcommand{\Ii}{I_{i}}
\newcommand{\g}{\mathfrak{g}}
\newcommand{\h}{\mathfrak{h}}
\newcommand{\gaff}{\mathfrak{\hat{g}}}
\newcommand{\haff}{\mathfrak{\hat{h}}}
\newcommand{\glone}{\mathfrak{gl}_{1}}
\newcommand{\Qov}{\mathring{Q}^{\vee}}
\newcommand{\Pov}{\mathring{P}^{\vee}}
\newcommand{\Pbar}{\overline{P}}
\newcommand{\Hom}{\mathrm{Hom}}
\newcommand{\ivector}{\underline{i}}
\newcommand{\kvector}{\underline{k}}
\newcommand{\ch}{\mathrm{ch}}
\newcommand{\ev}{\mathrm{ev}}
\newcommand{\res}{\mathrm{res}}

\newcommand{\Uq}{U_{q}(\mathfrak{g})}
\newcommand{\Uqs}{U_{q}(\mathfrak{s})}
\newcommand{\Uaff}{U_{q}(\mathfrak{\hat{g}})}
\newcommand{\Udash}{U'_{q}(\mathfrak{\hat{g}})}
\newcommand{\Uaffsltwo}{U_{q}(\widehat{\mathfrak{sl}}_{2})}
\newcommand{\Utor}{U_{q}(\mathfrak{g}_{\mathrm{tor}})}
\newcommand{\Uh}{\mathcal{U}_{h}}
\newcommand{\Uv}{\mathcal{U}_{v}}
\newcommand{\Udiag}{\mathcal{U}_{d}}
\newcommand{\UtorA}{U_{q}(\mathfrak{sl}_{n+1,\mathrm{tor}})}
\newcommand{\Utorsltwo}{U_{q}(\mathfrak{sl}_{2,\mathrm{tor}})}
\newcommand{\Utorglone}{U_{q_{1},q_{2},q_{3}}(\ddot{\mathfrak{gl}}_{1})}
\newcommand{\Uqaffs}{\widehat{\Uqs}}
\newcommand{\Uqaffg}{\widehat{\Uq}}

\newcommand{\UdashA}{U'_{q}(A_{1}^{(1)})}
\newcommand{\UdashAA}{U'_{q}(A_{2}^{(1)})}
\newcommand{\UdashC}{U'_{q}(C_{2}^{(1)})}
\newcommand{\UdashG}{U'_{q}(G_{2}^{(1)})}

\newcommand{\Rh}{\mathcal{R}_{h}}
\newcommand{\Rv}{\mathcal{R}_{v}}
\newcommand{\Rdiag}{\mathcal{R}_{d}}

\newcommand{\Dpsi}{\Delta^{\psi}_{u}}
\newcommand{\Dbarplus}{\overline{\Delta}_{+}}

\newcommand{\degv}{\deg_{v}}
\newcommand{\degj}{\deg_{j}}
\newcommand{\degZ}{\deg_{\Zbb}}

\newcommand{\Tb}{\textbf{T}}
\newcommand{\Xb}{\textbf{X}}

\newcommand{\B}{\mathcal{B}}
\newcommand{\Bd}{\mathcal{\dot{B}}}
\newcommand{\Bdd}{\mathcal{\ddot{B}}}
\newcommand{\Bh}{\mathcal{B}_{h}}
\newcommand{\Bv}{\mathcal{B}_{v}}
\newcommand{\Bdiag}{\mathcal{B}_{d}}
\newcommand{\Bb}{\mathbf{B}}

\newcommand{\gh}{\gamma_{h}}
\newcommand{\gv}{\gamma_{v}}

\newcommand{\xip}{x_{i}^{+}}
\newcommand{\xim}{x_{i}^{-}}
\newcommand{\xipm}{x_{i}^{\pm}}
\newcommand{\xjp}{x_{j}^{+}}
\newcommand{\xjm}{x_{j}^{-}}
\newcommand{\xjpm}{x_{j}^{\pm}}
\newcommand{\xp}{x^{+}}
\newcommand{\xm}{x^{-}}
\newcommand{\xpm}{x^{\pm}}

\newcommand{\xbpm}{\mathbf{x}^{\pm}}
\newcommand{\xbp}{\mathbf{x}^{+}}
\newcommand{\xbm}{\mathbf{x}^{-}}
\newcommand{\xb}{\mathbf{x}}
\newcommand{\kb}{\mathbf{k}}
\newcommand{\hb}{\mathbf{h}}
\newcommand{\Cb}{\mathbf{C}}

\newcommand{\e}{\mathfrak{e}}
\newcommand{\te}{\mathfrak{t}}
\newcommand{\s}{\mathfrak{s}}

\newcommand{\sfrak}{\mathfrak{s}}
\newcommand{\sfrakv}{\mathfrak{s}^{v}}
\newcommand{\sfrakj}{\mathfrak{s}^{(j)}}
\newcommand{\sfrakZ}{\mathfrak{s}^{\Zbb}}

\newcommand{\Scal}{\mathcal{S}}
\newcommand{\Pcal}{\mathcal{P}}
\newcommand{\T}{\mathcal{T}}
\newcommand{\U}{\mathcal{U}}
\newcommand{\Hcal}{\mathcal{H}}

\newcommand{\Wcal}{\mathcal{W}}
\newcommand{\Acal}{\mathcal{A}}
\newcommand{\X}{\mathcal{X}}
\newcommand{\Ycal}{\mathcal{Y}}
\newcommand{\Zcal}{\mathcal{Z}}
\newcommand{\Rcal}{\mathcal{R}}
\newcommand{\Tcal}{\mathcal{T}}
\newcommand{\Ecal}{\mathcal{E}}
\newcommand{\Mcal}{\mathcal{M}}

\newcommand{\Ocal}{\mathcal{O}}
\newcommand{\Oint}{\mathcal{O}_{\mathrm{int}}}
\newcommand{\Oaff}{\widehat{\mathcal{O}}_{\mathrm{int}}}

\newcommand{\Zbb}{\mathbb{Z}}
\newcommand{\Nbb}{\mathbb{N}}
\newcommand{\Cbb}{\mathbb{C}}

\newcommand{\Qbb}{\mathbb{Q}}

\newcommand{\Xun}{X_{n}^{(1)}}
\newcommand{\Th}[1]{\Theta_{0{#1}}}
\newcommand{\Xtrip}{\overset{\ldots}{X}_{n}}

\newcommand{\one}{^{(1)}}
\newcommand{\two}{^{(2)}}
\newcommand{\three}{^{(3)}}
\newcommand{\onetwo}{^{(1,2)}}
\newcommand{\onethree}{^{(1,3)}}
\newcommand{\twothree}{^{(2,3)}}
\newcommand{\alphapower}{^{(\alpha)}}
\newcommand{\betapower}{^{(\beta)}}
\newcommand{\alphabetapower}{^{(\alpha,\beta)}}
\newcommand{\betaalphapower}{^{(\beta,\alpha)}}

\newcommand{\SL}{SL_{2}(\mathbb{Z})}
\newcommand{\GL}{GL_{2}(\mathbb{Z})}

\newcommand{\Gam}[1]{\Gamma_{1}({#1})}
\newcommand{\TGam}[1]{\Tilde{\Gamma}_{1}({#1})}

\usepackage{hyperref}
\hypersetup{
    pdftitle={Duncan Laurie - Tensor products, q-characters and R-matrices for quantum toroidal algebras},
    pdfauthor={Duncan Laurie},
    hidelinks,
    pdfhighlight={/N},
    colorlinks=true,
    linkcolor=violet!75!black,
    citecolor=violet!75!black,
    urlcolor=black,
    }
\newcommand\blfootnote[1]{%
  \begingroup
  \renewcommand\thefootnote{}\footnote{#1}%
  \addtocounter{footnote}{-1}%
  \endgroup
}

\newenvironment{talign*}
 {\let\displaystyle\textstyle\csname align*\endcsname}
 {\endalign}

\usepackage{stackrel}
\usepackage{amssymb}
\usepackage{amsmath}
\newcommand{\leftrarrows}{\mathrel{\raise.75ex\hbox{\oalign{%
  $\scriptstyle\leftarrow$\cr
  \vrule width0pt height.5ex$\hfil\scriptstyle\relbar$\cr}}}}
\newcommand{\lrightarrows}{\mathrel{\raise.75ex\hbox{\oalign{%
  $\scriptstyle\relbar$\hfil\cr
  $\scriptstyle\vrule width0pt height.5ex\smash\rightarrow$\cr}}}}
\newcommand{\Rrelbar}{\mathrel{\raise.75ex\hbox{\oalign{%
  $\scriptstyle\relbar$\cr
  \vrule width0pt height.5ex$\scriptstyle\relbar$}}}}

\makeatletter
\def\leftrightarrowsfill@{\arrowfill@\leftrarrows\Rrelbar\lrightarrows}
\newcommand{\xleftrightarrows}[2][]{\ext@arrow 3399\leftrightarrowsfill@{#1}{#2}}
\makeatother

\newtheoremstyle{thmstyleone}
  {12pt}
  {12pt}
  {\itshape}
  {0pt}
  {\bfseries}
  {}
  {.5em}
  {\thmname{#1}\hspace{.3em}\thmnumber{#2.}}

\newtheoremstyle{thmstyletwo}
  {12pt}
  {12pt}
  {\normalfont}
  {0pt}
  {\itshape}
  {}
  {.5em}
  {\thmname{#1}\hspace{.3em}\thmnumber{#2.}}

\newtheoremstyle{thmstylethree}
  {12pt}
  {12pt}
  {\normalfont}
  {0pt}
  {\bfseries}
  {}
  {.5em}
  {\thmname{#1}\hspace{.3em}\thmnumber{#2.}}

\newtheoremstyle{thmstylefour}
  {12pt}
  {12pt}
  {\itshape}
  {0pt}
  {\bfseries}
  {}
  {.5em}
  {\thmname{#1.}}

\theoremstyle{thmstyleone}
\newtheorem{thm}{Theorem}[section]
\newtheorem{prop}[thm]{Proposition}
\newtheorem{lem}[thm]{Lemma}
\newtheorem{sublem}[thm]{Sublemma}
\newtheorem{cor}[thm]{Corollary}

\theoremstyle{thmstyletwo}

\theoremstyle{thmstylethree}
\newtheorem{defn}[thm]{Definition}
\newtheorem{eg}[thm]{Example}
\newtheorem{rmk}[thm]{Remark}

\theoremstyle{thmstylefour}
\newtheorem{Q}{Question}
\newtheorem{notation}{Notation}
\newtheorem{introthm}{Theorem}
\newtheorem{introcor}{Corollary}

\numberwithin{equation}{section}

\begin{document}

\title{Tensor products, $q$-characters and $R$-matrices \\ for quantum toroidal algebras}
\author{Duncan Laurie}
\affil{\normalsize{\textit{School of Mathematics and Maxwell Institute for Mathematical Sciences,}} \\ \normalsize{\textit{University of Edinburgh, Peter Guthrie Tait Road, Edinburgh, EH9 3FD.}}}
\date{}

\maketitle\blfootnote{E-mail: \url{duncan.laurie@ed.ac.uk}}\blfootnote{ORCID: 0009-0006-9331-4835.}

\vspace{-30pt}

\begin{abstract}
    We introduce a new topological coproduct $\Dpsi$ for quantum toroidal algebras $\Utor$ in all untwisted types, leading to a well-defined tensor product on the category $\Oaff$ of integrable representations.
    This is defined by twisting the Drinfeld coproduct $\Delta_{u}$ with an anti-involution $\psi$ of $\Utor$ that swaps its horizontal and vertical quantum affine subalgebras.
    Other applications of $\psi$ include generalising the celebrated Miki automorphism from type $A$, and an action of the universal cover of $\SL$.
    \\
    
    Next, we investigate the ensuing tensor representations of $\Utor$, and prove quantum toroidal analogues for a series of influential results by Chari-Pressley on the affine level.
    In particular, there is a compatibility with Drinfeld polynomials, and the product of irreducibles is generically irreducible.
    We moreover show that the $q$-character of a tensor product is equal to the product of $q$-characters for its factors.
    Furthermore, we obtain $R$-matrices with spectral parameter which provide solutions to the (trigonometric, quantum) Yang-Baxter equation, and endow $\Oaff$ with a meromorphic braiding.
    These give rise to a commuting family of transfer matrices for each module.
\end{abstract}

\vspace{10pt}

\begingroup
\titleformat{\section}[block]{\normalsize\bfseries\filcenter}{}{}{}
\tableofcontents
\endgroup

\section{Introduction}

Quantum toroidal algebras $\Utor$ are the double affine objects within the quantum setting, formed by applying Drinfeld's \emph{quantum affinization} procedure to the affine quantum groups.
They therefore contain, and are generated by, horizontal and vertical quantum affine subalgebras $\Uh$ and $\Uv$.
Since their introduction by Ginzburg-Kapranov-Vasserot \cite{GKV95}, these algebras have become a highly active area of research.
Even in the simplest cases, quantum toroidal algebras have found remarkable connections and applications across mathematics and physics, providing a powerful algebraic framework that links representation theory, geometry, quantum integrable systems, and combinatorics.
\\

Nevertheless, quantum toroidal algebras remain rather mysterious, with far less understood than for their finite and affine type counterparts.
For example, they are not known to possess any coproduct or Hopf algebra structures, and their module categories were not previously equipped with either a tensor product or a braiding.
One of the major obstacles is a lack of (anti-)automorphisms that swap $\Uh$ and $\Uv$ -- we shall call these horizontal--vertical symmetries.
The only existing example was the celebrated Miki automorphism in type $A$, which has been instrumental for studying $\UtorA$, quantum toroidal $\glone$, and their connections.
In this paper we will address each of these difficulties.
\\

In particular, $\Utor$ has an important category $\Oaff$ of integrable representations \cites{Hernandez05,GTL16} which exists as the toroidal analogue of the finite dimensional modules for quantum affine algebras -- indeed, its irreducible objects are classified by Drinfeld polynomials.
It is closed under finite direct sums, and contains all integrable modules that are highest weight with respect to the \emph{loop triangular decomposition} for $\Utor$, but fails to be semisimple.
\\

The following natural and fundamental question then arises: does $\Oaff$ possess a tensor product and therefore a monoidal structure?
On the finite and affine levels, such constructions come automatically as quantum groups are Hopf algebras, and provide the basis for seemingly endless directions -- see Section \ref{section:R matrices} for further discussion.
But in the case of quantum toroidal algebras, we need to work harder.
\\

As mentioned above, $\Utor$ is not known to carry a coproduct, except in types $A_{1}^{(1)}$ and $A_{2}^{(1)}$ \cite{JZ22}.
The only existing alternative is a Drinfeld \emph{topological coproduct} $\Delta_{u}$ depending on a spectral parameter $u$, which maps to a completion of the tensor square \cites{Hernandez05,Damiani24}.
However, $\mathrm{im}(\Delta_{u})$ contains infinite sums whose actions on a tensor product of modules in $\Oaff$ may not converge even after specialising $u$.
In particular, while we can pick some $u$ such that $\Delta_{u}$ endows a \emph{fixed} tensor product with a $\Utor$-module structure, it is not possible to produce in this way a well-defined tensor product on the category as a whole -- see Section \ref{section:tensor products} for more details.
\\

Various attempts have been made to overcome these issues.
Notably, Hernandez \cites{Hernandez05,Hernandez07} constructed a \emph{fusion product} by enlarging the category to one in which $\Delta_{u}$ does define a tensor product, and then specializing back to $\Oaff$.
Furthermore, a series of papers by Miki \cites{Miki00,Miki01,Miki07} explore these directions in type $A$.
Addressing this problem in all untwisted types is one of our major goals in this work.
\\

In order to do this, we require horizontal--vertical symmetries for $\Utor$.
Until recent results by the author in the simply laced case \cite{Laurie24a}, it was not entirely clear whether such (anti-)automorphisms should exist outside type $A$.
Here, we further extend our constructions to all untwisted types, which are essential for approaching the representation theory in later sections.
\\

So how do we obtain these symmetries?
The philosophy is to first consider our action $\Bdd\curvearrowright\Utor$ of the extended double affine braid group from \cite{Laurie24a}.
Similar to quantum toroidal algebras, $\Bdd$ contains horizontal and vertical affine braid subgroups $\Bh$ and $\Bv$.
These preserve $\Uh$ and $\Uv$ respectively, with each restricted action coinciding with Lusztig's braid group action on the affine level \cite{Lusztig93}.
We then take an involution $\te$ of $\Bdd$ that swaps $\Bh$ and $\Bv$, and \emph{pass it across the action} to obtain an anti-involution $\psi = (b\cdot z \mapsto \te(b)\cdot z)$ with the desired properties.

\begin{introthm}
    There exists an anti-involution $\psi$ of $\Utor$ which exchanges $\Uh$ and $\Uv$ in all untwisted types.
\end{introthm}

In fact, $\te$ lifts the famous duality involution for double affine Hecke algebras used by Cherednik to realise the difference Fourier transform in his celebrated proof \cite{Cherednik95} of Macdonald's evaluation conjectures.
Our anti-involution $\psi$ may therefore be considered as the quantum analogue of this duality.
Beyond the simply laced case, our proof requires a finer understanding of the structure of $\Bdd$ coming from its Coxeter-style presentation due to Ion-Sahi \cite{IS20}.
\\

Direct consequences of the existence of $\psi$ include the following, which correspond to celebrated results for the quantum toroidal algebra $\Utorglone$ of type $\glone$ by Burban-Schiffmann \cites{BS12,Schiffmann12} and Miki \cite{Miki07}.

\begin{introcor}
    \begin{itemize}
        \item There is a congruence group action of the universal cover $\widetilde{\SL}$ on the quantum toroidal algebra $\Utor$.
        \item The action of
        $S = \begin{bsmallmatrix} 0 & 1 \\ -1 & 0 \end{bsmallmatrix}$
        provides a generalisation of the Miki automorphism.
    \end{itemize}
\end{introcor}

This is compatible with an existing action
$\widetilde{\SL}\curvearrowright\Bdd$
\cites{Cherednik05,IS20}, and can therefore be used to further enlarge our braid group action from \cite{Laurie24a}.
Moreover, the $\Utorglone$ analogues of these results already play a fundamental role in studying its representation theory, as well as the various applications to geometry and physics.
Our work should therefore lay the foundation for extensions of these directions.
\\

For example, quantum toroidal algebras admit `horizontal' and `vertical' representations.
Known instances of the former are written in terms of vertex operators and $q$-deformed free bosons \cites{Jing98(2),Saito98}, while in type $A$, sets of generalised Young diagrams (coloured partitions) often give concrete descriptions of the latter \cites{FJMM13,JM24}.
Twisting by our $\widetilde{\SL}$-action and horizontal--vertical symmetries allows us to pass between and relate these classes of modules together.
\\

Equipped with our anti-involution $\psi$, we are now able to successfully construct a tensor product on $\Oaff$.
The key is to conjugate $\Delta_{u}$ by $\psi$ in order to produce a new topological coproduct $\Dpsi$ for $\Utor$.
While its image still contains infinite sums, all but finitely many summands act by zero on any product of modules, and our convergence issues fall away.
In particular, $\Dpsi$ leads to a well-defined tensor product for $\Oaff$, endowing the category with a monoidal structure, and its Grothendieck group with the structure of a ring.

\begin{introthm}
    The topological coproduct
    $\Dpsi = (\psi\otimes\psi)\circ\Delta_{u}\circ\psi$
    for $\Utor$ gives rise to a well-defined tensor product on the module category $\Oaff$.
\end{introthm}

One may roughly think of this solution as follows.
Quantum toroidal algebras possess a $\Zbb^{2}$--grading such that
$\Uh\subset\Zbb\times \lbrace 0 \rbrace$
and
$\Uv\subset\lbrace 0 \rbrace \times \Zbb$,
where we label the lattice directions as horizontal and vertical accordingly.
Both $\Delta_{u}$ and modules in $\Oaff$ can then be considered \emph{vertically infinite} with respect to this grading.
Indeed, the tensor factors in summands of $\Delta_{u}(z)$ generally have unbounded vertical degree.
But it is not true that elements of $V\in\Oaff$ are annihilated by the $(m,n)$ graded piece of $\Utor$ for $\lvert n\rvert \gg 0$.
Hence every summand of $\Delta_{u}(z)$ might have non-zero action on some $v\in V\one\otimes V\two$, leading to the aforementioned convergence issues.
However, $\psi$ descends to $\Zbb^{2}$ as reflection in the line $x = y$ and so $\Dpsi$ is instead \emph{horizontally infinite}, giving some intuition for why all of our problems then disappear.
\\

Explicit expressions for $\psi(z)$ are usually very complicated, and so in order to better understand the monoidal structure on $\Oaff$ we proceed to prove a series of results involving $\Dpsi$ and our tensor product.
These include various toroidal analogues of influential works by Chari-Pressley \cite{CP94} for finite dimensional representations of quantum affine algebras.
For example, there is a compatibility with $\ell$-highest weight vectors and Drinfeld polynomials.

\begin{introthm}
    Suppose that $V\alphapower\in\Oaff$ contains an $\ell$-highest weight vector $v\alphapower$ with Drinfeld polynomials $\Pcal\alphapower(z)$ for $\alpha = 1,\dots,n$.
    Then $v\one\otimes\dots\otimes v^{(n)} \in \bigotimes_{\alpha=1}^{n} V\alphapower$ is $\ell$-highest weight with Drinfeld polynomials $\prod_{\alpha=1}^{n} \Pcal\alphapower(z)$.
\end{introthm}

It suffices to consider the $n=2$ case, where our proofs require a detailed analysis of the action of different generators of $\Utor$ on $v\one\otimes v\two$.
In particular, we can relate our toroidal tensor product to the affine one -- see Sections \ref{subsection:action of Uv} and \ref{subsection:action of U0} for more details.
\\

For any $a\in\Cbb^{\times}$, representations $V$ of $\Utor$ can be twisted by the algebra automorphism that scales the $(m,n)$ graded piece by $a^{n\hslash}$, where $\hslash$ is the Coxeter number of $\gaff$ -- denote the resulting module by $V_{a}$.
It turns out that a tensor product of irreducibles objects in $\Oaff$ is generically irreducible with respect to the spectral parameter $a$.
Moreover the category is in some sense generated from a set of \emph{fundamental modules} $V(\lambda_{i},a) = V(\lambda_{i},1)_{a^{1/\hslash}}$, where $i$ runs over the vertices of the affine Dynkin diagram and $a\in\Cbb^{\times}$.

\begin{introthm}
    \begin{itemize}
        \item Suppose that $V\alphapower\in\Oaff$ is irreducible for $\alpha = 1,\dots,n$.
        Then
        $V\one_{a_{1}} \otimes\dots\otimes V^{(n)}_{a_{n}}$
        is irreducible for all but countably many
        $(a_{1},\dots,a_{n}) \in (\Cbb^{\times})^{n}$.
        \item In this case,
        $V\one_{a_{1}} \otimes\dots\otimes V^{(n)}_{a_{n}}$
        is isomorphic to
        $V^{(\sigma(1))}_{a_{\sigma(1)}} \otimes\dots\otimes V^{(\sigma(n))}_{a_{\sigma(n)}}$
        for any permutation $\sigma\in S_{n}$.
        \item Every irreducible representation in $\Oaff$ is isomorphic to a subquotient of some tensor product
        $V(\lambda_{i_{1}},a_{1}) \otimes\dots\otimes V(\lambda_{i_{n}},a_{n})$
        of fundamental modules.
    \end{itemize}
\end{introthm}

The theory of $q$-characters provides a powerful combinatorial tool for studying the category $\Oaff$ \cites{FR99,Hernandez05,Hernandez07}.
In the particular case of quantum affine algebras, this is moreover connected to the cluster algebra structure on various Grothendieck rings of finite dimensional modules.
However, a fundamental property missing for other quantum affinizations was a compatibility between the $q$-character morphism
$\chi_{q} : K(\Oaff) \rightarrow \Ycal$
and a multiplication on $K(\Oaff)$ extended from some tensor product.
In Section \ref{section:q-characters} we are able to prove such a result for quantum toroidal algebras, as well as establish a relationship between our tensor product and Hernandez' fusion product on the level of Grothendieck rings.

\begin{introthm}
    \begin{itemize}
        \item Our tensor product on $\Oaff$ is compatible with the $q$-character morphism, in particular
        $\chi_{q}(V\one \otimes V\two) = \chi_{q}(V\one) \cdot \chi_{q}(V\two)$
        for all representations $V\one,V\two\in\Oaff$.
        \item The $q$-character morphism
        $\chi_{q} : K(\Oaff) \rightarrow \Ycal$
        is a ring homomorphism.
        \item Our tensor product $\otimes$ and Hernandez' fusion product $\ast_{f}$ give rise to the same product on $K(\Oaff)$.
    \end{itemize}
\end{introthm}

An essential feature of module categories for finite and affine quantum groups is the presence of (meromorphic) braidings.
Namely, there exist $R$-matrix intertwiners that exchange tensor factors in a product of modules, and satisfy the Yang-Baxter equation.
These are a fundamental ingredient in the various applications to low-dimensional topology, quantum integrable systems, cluster algebras, Schur-Weyl dualities, and so on.
On the toroidal level, we obtain $R$-matrices for all direct sums $V\alphapower$ of tensor products of irreducible objects in $\Oaff$.
For example, if each $V\alphapower$ is irreducible we have the following.

\begin{introthm}
    There exist unique
    $\Hom_{\Cbb}(V\alphapower \otimes V\betapower,V\betapower \otimes V\alphapower)$-valued rational functions $\Rcal\alphabetapower(x)$ such that
    \begin{itemize}
        \item $\Rcal\alphabetapower(b/a)$
        is a $\Utor$-module homomorphism
        $V\alphapower_{a} \otimes V\betapower_{b} \rightarrow
        V\betapower_{b} \otimes V\alphapower_{a}$
        sending
        $v\alphapower \otimes v\betapower \mapsto
        v\betapower \otimes v\alphapower$
        whenever $\Rcal\alphabetapower(x)$ does not have a pole at $b/a$,
        \item $\Rcal\alphabetapower(b/a)$ is moreover an isomorphism if
        $V\alphapower_{a} \otimes V\betapower_{b}$
        is irreducible,
        \item the (trigonometric, quantum) Yang-Baxter equation is satisfied.
    \end{itemize}
\end{introthm}

Just like for our tensor product, we can moreover relate these $\Rcal\alphabetapower(x)$ to the intertwiners coming from \cite{CP94}.
Indeed, they may be seen as \emph{glued together} from infinitely many quantum affine $R$-matrices.
\\

When considering the connections with quantum physics, as well as studying the module categories themselves, an important role is played by \emph{transfer matrices}.
These are certain commuting linear operators on representations, and are used to establish the integrability of the corresponding quantum systems via Bethe ansatz techniques.
Using our $R$-matrices we initiate such directions on the toroidal level.

\begin{introthm}
    For each $V\alphapower$ and $V\betapower$ there exists an associated transfer matrix
    $\Tcal\alphabetapower(x) \in \mathrm{End}_{\Cbb}(V\alphapower)(x)$
    such that all
    $[\Tcal\onetwo(b/a),\Tcal\onethree(c/a)] = 0$.
\end{introthm}

Let us briefly remark that each of our results carries over to quantum toroidal $\glone$ in an appropriate way.
This algebra is related to $\Utor$, but slightly separate and more symmetric.
Note, however, that the analogues for this particular case can be derived from existing works \cites{Miki07,FJMM15}.
We nevertheless reference $\Utorglone$ at various points in order to frame it within our more general setting.

\subsection{Future directions}

Our work in this paper opens up a range of different avenues for investigation going forwards.
For example, the author plans to explore toroidal versions of the generalised Schur-Weyl dualities, monoidal categorification of cluster algebras, and applications to the theory of $q$-characters already established for quantum affine algebras.
\\

Furthermore, quantum toroidal algebras are connected to geometry via Nakajima's morphism \cites{Nakajima01,Nakajima02} to the equivariant $K$-theory of (Steinberg-style fiber products of) quiver varieties on the affine Dynkin diagrams.
Here, representations in $\Oaff$ can be realized by taking the $K$-theory of certain fibers.
Moreover, these quiver varieties realize Quot schemes and resolutions of Hilbert schemes for Kleinian singularities \cites{CGGS21a,CGGS21b}.
\\

Relevant parts of \cite{VV02}*{Lem. 8.1} and its proof interpret $\Delta_{u}$ on the geometric side, using specialisation to torus fixed points.
However, it is not at all clear how to see our horizontal--vertical symmetries $\psi$, topological coproduct $\Dpsi$, or resulting tensor product within this setting.
This is an interesting problem deserving further investigation.
\\

In another direction, Fock space representations for $\UtorA$ are constructed combinatorially in \cite{FJMM13} as a semi-infinite limit of exterior powers of vector representations, written in terms of a basis of coloured partitions.
In turn, Macmahon modules are then obtained by taking semi-infinite wedges inside a tensor product of Fock modules, with a basis of $3D$ coloured partitions.
\\

It is natural to ask whether such directions might exist in more generality.
Indeed, Young wall models for Fock space representations of quantum affine algebras have now been realised in all affine types \cites{Premat04,KK08,FHKS24,HJKL24,Laurie25}.
Moreover, the author \cite{Laurie24b} has defined vector representations of $\Utor$ in types $A_{n}^{(1)}$, $D_{n}^{(1)}$, $E_{6}^{(1)}$ and $E_{7}^{(1)}$, with the actions given explicitly with respect to Young column bases.
\\

However, poles in the coproduct parameter provide an obstacle to deriving exterior power and Fock space representations using $\Delta_{u}$ in the same way as \cite{FJMM13}.
Furthermore, to the author's knowledge, vector representations for quantum toroidal algebras are not yet known in other types.
Nevertheless, since our topological coproduct $\Dpsi$ leads to a well-defined tensor product on $\Oaff$ and thus all Fock modules, one might hope to obtain Macmahon representations via a semi-infinite limit construction.
\\

After writing this paper, the author became aware of work by Guay-Nakajima-Wendlandt \cite{GNW18} for the affine Yangian $Y_{h}(\gaff)$, where they define a tensor product on the analogue $\Ocal$ of our category $\Oaff$.
This lifts to a coproduct on some completion of $Y_{h}(\gaff)$, which may alternatively be viewed as a topological coproduct for the affine Yangian.
They moreover conjecture that quantum toroidal algebras should possess similar structures -- our results confirm this expectation.
\\

It would be interesting to understand in a precise way how the work of \cite{GNW18} relates to ours.
Indeed, Gautam and Toledano Laredo \cite{GTL16} proved that the representation theory of quantum toroidal algebras is equivalent in some sense to that of affine Yangians.
In particular, they constructed an equivalence between $\Oaff$ and a certain subcategory of $\Ocal$.
One might hope to upgrade this to an equivalence of \emph{monoidal categories}, similar to the results of \cite{GTL17} for quantum affine algebras and Yangians.
\\

Furthermore, $\Ocal$ has been equipped with a meromorphic braiding by $R$-matrices in \cite{AGW23}.
The question therefore arises as to whether we can further upgrade the equivalence from \cite{GTL16} to one of \emph{meromorphic braided monoidal categories}.
It is worth noting that the construction of the topological coproduct in \cite{GNW18} is rather different to our definition of $\Dpsi$.
Appel-Gautam-Wendlandt \cite{AGW23} relate it to the Drinfeld coproduct by twisting with the negative part of the Gaussian decomposition for the $R$-matrix -- perhaps we can relate this to conjugation by $\psi$ in the quantum toroidal setting.
Once again, the author hopes to explore these directions in future work.
\\

Let us briefly remark that \cite{GNW18} -- and thus \cite{AGW23} -- does not cover $Y_{h}(\gaff)$ of types $A_{1}^{(1)}$ and $A_{2}^{(2)}$, with the latter instead treated in \cite{Ueda20}.
Their results and ours together indicate that the various constructions should exist for all quantum toroidal algebras and affine Yangians (both untwisted and twisted), and maybe even the quantum affinizations and Yangians associated to any symmetrizable Kac-Moody Lie algebra.

\subsection{Structure of the paper}

This paper is organised as follows.
In Section \ref{section:Preliminaries}, after setting up our basic notations, we recall the fundamental definitions regarding quantum groups.
We then introduce their quantum affinizations, and collect all of the necessary preliminaries such as topological coproducts, $\ell$-highest weight representation theory, and $q$-character morphisms.
Moreover we use the results of \cite{Miki01} to extend the finite presentation and braid group action from the author's previous work \cite{Laurie24a} to an even broader class of affinizations.
Section \ref{section:Quantum toroidal algebras} focuses on the structure of quantum toroidal algebras in particular, including our action of the extended double affine braid groups.
We also outline their Coxeter-style presentation due to Ion-Sahi \cite{IS20}, which plays an important role in our proofs later on.
\\

In Section \ref{section:horizontal-vertical symmetries} we obtain horizontal--vertical symmetries of quantum toroidal algebras in all untwisted types.
We describe our anti-involution $\psi$, and discuss a range of immediate consequences such as a modular action of the universal cover of $\SL$ and generalisations of the Miki automorphism.
Section \ref{section:tensor products} introduces the topological coproduct $\Dpsi$, establishes a monoidal structure on $\Oaff$, and proves a series of results for our tensor product.
Sections \ref{subsection:action of Uv} and \ref{subsection:action of U0} in particular explore the action of $\Utor$ on a tensor product of modules in detail.
\\

The main goal of Section \ref{section:q-characters} is to establish a compatibility between our tensor product on $\Oaff$ and the $q$-character morphism.
Our proof requires a precise understanding of how weight spaces for tensor representations decompose into $\ell$-weight spaces, which we address in Section \ref{subsection:q-characters proof}.
As a consequence, we are able to relate our tensor product to Hernandez' fusion product on the level of Grothendieck rings.
We conclude in Section \ref{section:R matrices} by obtaining $R$-matrices which satisfy the Yang-Baxter equation, as well as their associated commuting transfer matrices.

\subsection{Acknowledgements}

I would like to thank David Hernandez, Andrea Appel and Sachin Gautam for their interest in this work, and Arun Soor for his help with Lemma \ref{lem:kernels and images generically}.
This research was financially supported by the European Research Council (ERC) under the European Union's Horizon 2020 research and innovation programme [grant number 948885], and the Engineering and Physical Sciences Research Council (EPSRC) [grant number EP/T517811/1].

\section{Preliminaries} \label{section:Preliminaries}

\subsection{Basic notations} \label{subsection:basic notations}

Consider a Kac-Moody Lie algebra $\s$ with generalized Cartan matrix $A = (a_{ij})_{i,j\in I}$ and finite index set $I$.
We shall assume that $A$ is symmetrizable, which is to say that there exists a diagonal matrix $D = \mathrm{diag}(d_{i} ~|~ i\in I)$ with relatively prime entries in $\Zbb_{>0}$ such that the product $DA$ is symmetric.
Its Cartan subalgebra $\h$ contains simple coroots $\alpha^{\vee}_{i}$ and fundamental coweights $\Lambda^{\vee}_{i}$ for each $i\in I$, as well as $\mathrm{corank}(A)$ scaling elements.
The coweight lattice $P^{\vee}$ is the $\Zbb$-span of the simple coroots and scaling elements, and moreover contains the coroot lattice
$Q^{\vee} = \bigoplus_{i\in I}\Zbb\alpha^{\vee}_{i}$.
\\

With the natural pairing $\langle ~,~ \rangle$ between $\h$ and its dual space $\h^{*}$ we define the weight lattice
$P = \lbrace \lambda \in \h^{*} ~|~ \langle\lambda,P^{\vee}\rangle \subset \Zbb \rbrace$,
simple roots $\alpha_{i}$ and fundamental weights $\Lambda_{i}$ for each $i\in I$.
In particular, these must satisfy $\langle\alpha_{j},\alpha^{\vee}_{i}\rangle = a_{ij}$
and
$\langle\Lambda_{j},\alpha^{\vee}_{i}\rangle = \delta_{ij}$
for all $i,j \in I$.
We denote the root lattice
$\bigoplus_{i\in I}\Zbb\alpha_{i}$
by $Q$, and let
$P^{+} = \lbrace \lambda \in P ~|~ \mathrm{all~} \lambda(\alpha^{\vee}_{i}) \geq 0 \rbrace$
be the set of dominant integral weights.
The standard non-degenerate symmetric bilinear form $(~,~)$ on $\h^{*}$ satisfies
$(\alpha_{i},\alpha_{j}) = d_{i}a_{ij}$
for all $i,j\in I$, and induces an isomorphism
$\nu : \h \rightarrow \h^{*}$ which maps each $\alpha_{i}^{\vee}\mapsto d_{i}^{-1}\alpha_{i}$.
Throughout this paper we may occasionally identify the elements of $\h$ with their images under $\nu$ without mention.
\\

Let $D(A)$ be the Dynkin diagram associated to our generalized Cartan matrix $A$, with vertex set $I$ and $a_{ij}a_{ji}$ edges between any distinct $i,j\in I$ that point to $j$ whenever $a_{ij} \geq a_{ji}$.
The corresponding braid group $\B$ is defined as the group generated by $\lbrace T_{i} ~|~ i\in I\rbrace$ subject to the braid relations $T_{i}T_{j}T_{i}\ldots = T_{j}T_{i}T_{j}\ldots$ with $a_{ij}a_{ji} + 2$ factors on each side whenever $a_{ij}a_{ji} \leq 3$.
The Weyl group $W = \langle s_{i} ~|~ i\in I \rangle$ is the quotient obtained by specifying that each generator is self-inverse, and acts on $P^{\vee}$ via
$s_{i}(x) = x - \langle\alpha_{i},x\rangle\alpha_{i}^{\vee}$
for each $i\in I$.
Note that both $\B$ and $W$ are constructed independently of the orientation of arrows in $D(A)$, but that the action on $P^{\vee}$ is not.
\\

Throughout this paper, every algebra associated to a Cartan datum shall be considered with respect to the field $\kk = \mathbb{Q}(q)$ for an indeterminate $q$.
Setting $q_{i} = q^{d_{i}}$ for all $i\in I$, the $q_{i}$-integers, $q_{i}$-factorials and $q_{i}$-binomial coefficients are defined as
\begin{align*}
    [s]_{i} = \frac{q_{i}^{s}-q_{i}^{-s}}{q_{i}-q_{i}^{-1}},
    \qquad
    [s]_{i}! = \prod_{\ell=1}^{s} [\ell]_{i},
    \qquad
    \begin{bmatrix}{s}\\ {r}\end{bmatrix}_{i} = \frac{[s]_{i}!}{[s-r]_{i}!\,[r]_{i}!}
\end{align*}
respectively for all non-negative integers $s\geq r$.
When our generalized Cartan matrix is symmetric, since all $d_{i} = 1$ we may drop the $i$ subscripts above for simplicity.
\\

For certain elements $\xipm$ and $\xpm_{i,m}$ of the quantum algebras introduced in later sections, we introduce the divided powers $(\xipm)^{(s)} = (\xipm)^{s}/[s]_{i}!$ and $(\xpm_{i,m})^{(s)} = (\xpm_{i,m})^{s}/[s]_{i}!$ for each non-negative integer $s$.
Following Jing \cite{Jing98} we shall also define their twisted commutators inductively via
$[b_{1},b_{2}]_{u} = [b_{1},b_{2}]'_{u} = b_{1}b_{2} - u b_{2}b_{1}$
and
\begin{align*}
    [b_{1},\dots,b_{s}]_{u_{1}\cdots u_{s-1}}
    &= [b_{1},[b_{2},\dots,b_{s}]_{u_{1}\cdots u_{s-2}}]_{u_{s-1}},
    \\
    [b_{1},\dots,b_{s}]'_{u_{1}\cdots u_{s-1}}
    &= [[b_{1},\dots,b_{s-1}]'_{u_{1}\cdots u_{s-2}},b_{s}]_{u_{s-1}},
\end{align*}
noting that if $f$ is an anti-homomorphism then
$f([b_{1},\dots,b_{s}]_{u_{1}\cdots u_{s-1}})
= [f(b_{s}),\dots,f(b_{1})]'_{u_{s-1}\cdots u_{1}}$.
\\

Let us now restrict our focus to the affine case, where our conventions mostly follow \cite{Kac90}.
We shall consider an indecomposable affine Kac-Moody algebra $\gaff$ with Cartan matrix $A = (a_{ij})_{i,j\in I}$ and index set $I = \lbrace 0,\dots,n\rbrace$.
Since $\mathrm{corank}(A) = 1$ its Cartan subalgebra $\haff$ has a basis consisting of the simple coroots $\alpha^{\vee}_{0},\dots,\alpha^{\vee}_{n}$ together with a unique scaling element $d$ (alternatively, this can be replaced by $\Lambda^{\vee}_{0}$).
Furthermore, the centre of $\gaff$ is spanned by a canonical non-divisible element
$c \in \bigoplus_{i\in I}\Zbb_{>0}\alpha^{\vee}_{i}$.
\\

On the other hand, the dual space $\haff^{*}$ possesses a basis
$\lbrace\Lambda_{0},\alpha_{0},\dots,\alpha_{n}\rbrace$ and the root lattice $Q$ contains a unique standard non-divisible imaginary root $\delta$.
Since the natural pairing between $\haff$ and $\haff^{*}$ is given by
$\langle\Lambda_{i},\alpha^{\vee}_{j}\rangle = \delta_{ij}$,
$\langle\Lambda_{i},d\rangle = \langle\delta,\alpha^{\vee}_{j}\rangle = 0$
and $\langle\delta,d\rangle = 1$, the bilinear form $(~,~)$ is determined by
\begin{align*}
    (\alpha_{i},\alpha_{j}) = d_{i}a_{ij}, \qquad
    (\alpha_{i},\Lambda_{0}) = d_{0}\delta_{i0}, \qquad
    (\Lambda_{0},\Lambda_{0}) = 0,
\end{align*}
for all $i,j\in I$ and in particular satisfies $(\delta,\alpha_{i}) = 0$.
The corresponding isomorphism $\nu : \haff \rightarrow \haff^{*}$ sends $\Lambda_{0}^{\vee}\mapsto d_{0}^{-1}\Lambda_{0}$.
Moreover, we can now express explicitly
\begin{itemize}
    \item the affine weight lattice $P = \bigoplus_{i\in I}\Zbb\Lambda_{i}\oplus\Zbb\delta$,
    \item the affine coweight lattice $P^{\vee} = \bigoplus_{i\in I}\Zbb \alpha^{\vee}_{i}\oplus\Zbb d$,
    \item the set of dominant affine integral weights $P^{+} = \bigoplus_{i\in I}\Nbb\Lambda_{i}\oplus\Zbb\delta$.
\end{itemize}
Removing the null root $\delta$ produces the classical weight lattice $\Pbar = \bigoplus_{i\in I}\Zbb\Lambda_{i}$ which can be viewed as both a sublattice and a quotient of $P$, as well as its subset of dominant classical weights $\Pbar^{+} = \bigoplus_{i\in I}\Nbb\Lambda_{i}$.
Note that the action of the affine Weyl group $W$ on $P$ descends to an action on $\Pbar$.
\\

Each node $i\in I$ of the affine Dynkin diagram $D(A)$ has a numerical label $a_{i}$, and a dual label $a_{i}^{\vee}$ coming from the diagram with the same vertex numbering and all arrows reversed.
The affine Dynkin diagrams, together with their $a_{i}$ and $a_{i}^{\vee}$ labels, can be found for example in the author's thesis \cite{Laurie24b}*{App. A} -- there our choice of vertex numbering matches Bourbaki \cite{Bourbaki68}*{Plates I--IX} in all untwisted types, and the twisted types are obtained by reversing arrows.
The affine Cartan matrix of type $X_{n}^{(r)}$ is then symmetrized by a positive integer multiple of
$\mathrm{diag}(a_{0}^{\vee}/a_{0},\dots,a_{n}^{\vee}/a_{n})$.
Furthermore, the null root $\delta$ equals $\sum_{i\in I} a_{i}\alpha_{i}$ with $a_{0} = 1$ outside type $A_{2n}^{(2)}$, and the central element $c$ is $\sum_{i\in I} a_{i}^{\vee}\alpha^{\vee}_{i}$ with $a^{\vee}_{0} = 1$.
The level of an affine or classical weight $\lambda$ is given by the pairing $\langle\lambda,c\rangle$ and is invariant under the Weyl group action.
\\

A vertex $i\in I$ is minuscule if it is sent to $0$ by some automorphism of the affine Dynkin diagram, and we denote the set of minuscule nodes by $\Imin \subset \lbrace i\in I~|~a_{i} = a_{0}\rbrace$.
An automorphism is inner if it fixes the $0$ vertex, and thus restricts to an automorphism of the finite Dynkin diagram.
The outer automorphism group $\Omega$ is then the quotient of the entire automorphism group by the subgroup of inner automorphisms, and therefore has elements indexed by $\Imin$.
In particular, for each $i\in\Imin$ we let $\pi_{i}$ be the corresponding element of $\Omega$, which is uniquely determined by the condition $\pi_{i}(0) = i$.
\\

In all affine types except $A_{2n}^{(1)}$ we can fix a sign function $o:I\rightarrow\lbrace\pm 1\rbrace$ satisfying $o(i) = -o(j)$ whenever $a_{ij}<0$.
We shall write $o_{i,j}$ as shorthand for $o(i)/o(j)$.
However, in type $A_{2n}^{(1)}$ this is not possible since the affine Dynkin diagram contains an odd length cycle.
For our purposes, there are two approximations to a sign function to consider in this case: $o(i) = (-1)^{i}$ and $-o(i) = (-1)^{i+1}$.
Furthermore, we define $o_{i,j} = (-1)^{\overline{j-i}}$ for all $i,j\in I$, where $\overline{j-i}$ is the anti-clockwise distance $i\rightarrow j$ in the affine Dynkin diagram.
\\

Contained in each affine Lie algebra $\gaff$ is a corresponding finite dimensional simple Lie algebra $\g$ with Cartan matrix $(a_{ij})_{i,j\in I_{0}}$ where $I_{0} = \lbrace 1,\dots,n\rbrace$.
(More generally, we shall let $\Ii = I\setminus\lbrace i\rbrace$ for every $i\in I$.)
It has simple roots $\alpha_{i}$, simple coroots $\alpha_{i}^{\vee}$, fundamental weights $\omega_{i}$, and fundamental coweights $\omega_{i}^{\vee}$ for each $i\in I_{0}$ and we denote its root, coroot, weight and coweight lattices by $\mathring{Q}$, $\Qov$, $\mathring{P}$ and $\Pov$.
By mapping each $\omega_{i}^{\vee} \mapsto a_{0}\Lambda_{i}^{\vee} - a_{i}\Lambda_{0}^{\vee}$ we can embed $\Pov$ inside $P^{\vee}$ at level $0$, so that $\langle\delta,\omega_{i}^{\vee}\rangle = 0$ for all $i\in I_{0}$.
The image is invariant under the action of the finite Weyl group $\mathring{W} = \langle s_{i}~|~i\in I_{0}\rangle$.
Similarly, we can view $\mathring{P}$ inside the affine weight lattice $P$ by sending each
$\omega_{i} \mapsto a_{0}^{\vee}\Lambda_{i} - a_{i}^{\vee}\Lambda_{0}$.
In order to simplify our notation in later sections we shall moreover define $\omega_{0}^{\vee} = 0$ and $\omega_{0} = 0$.
\\

As explained in the general case above, the affine braid group $\B$ has a Coxeter presentation with generators $T_{0},\dots,T_{n}$ satisfying the braid relations for all distinct $i,j\in I$.
Since this construction is independent of the orientation of arrows, note that any affine braid group is isomorphic to one of untwisted type.
We remark that in types $A_{1}^{(1)}$ and $A_{2}^{(2)}$ this is simply the free group generated by $T_{0}$ and $T_{1}$ since $a_{01}a_{10} = 4$.
\\

However, for affine braid groups in particular there exists a second realization due to Bernstein as follows.
In all untwisted and $A_{2n}^{(2)}$ types, let $M = \Qov$ and $A_{i}^{\vee} = \alpha_{i}$ for each $i\in I$.
Conversely, in the remaining twisted types we define $M = \mathring{Q}$ and all $A_{i}^{\vee} = \alpha_{i}^{\vee}$.
Then in each case, the Bernstein presentation of $\B$ is generated by the finite braid group $\B_{0} = \langle T_{i}~|~i\in I_{0}\rangle$ and the lattice $\lbrace X_{\beta}~|~\beta \in M \rbrace$, with
\begin{itemize}
    \item $T_{i}X_{\beta} = X_{\beta}T_{i}$ if $(\beta,A_{i}^{\vee}) = 0, \hfill \refstepcounter{equation}(\theequation)\label{first extended affine Bernstein}$
    \item $T_{i}^{-1}X_{\beta}T_{i}^{-1} = X_{s_{i}(\beta)}$ if $(\beta,A_{i}^{\vee}) = 1. \hfill \refstepcounter{equation}(\theequation)\label{second extended affine Bernstein}$
\end{itemize}

When $M = \Qov$ the correspondence between the two presentations is given by $T_{0} = X_{\theta^{\vee}} \Theta^{-1}$ where $\Theta = T_{s_{\theta}}$ for $\theta$ the highest root $\sum_{i\in I_{0}} a_{i}\alpha_{i}$ of $\g$, and $\theta^{\vee} = \nu^{-1}(a_{0}^{-1}\theta)$.
Otherwise, $\theta$ is the short dominant root in $M = \mathring{Q}$ and we instead have $T_{0} = X_{\theta} \Theta^{-1}$.
See \cite{IS20}*{Ch. 3} for more details, noting that the Bernstein presentation there is obtained from ours by applying the automorphism of $\B$ which inverts $T_{1},\dots,T_{n}$ and fixes each $X_{\beta}$.
\\

The \emph{extended} affine braid group $\Bd$ may on the one hand be formed as the semidirect product $\Omega \ltimes \B$ with $\pi T_{i} \pi^{-1} = T_{\pi(i)}$ for all $i\in I$ and $\pi\in\Omega$.
However we can also obtain a Bernstein presentation for $\Bd$ by replacing $M$ in the above with a larger lattice $N$, defined to be $\Pov$ in all untwisted and $A_{2n}^{(2)}$ types and $\mathring{P}$ otherwise.
\\

When $N = \Pov$ set $\beta_{\theta} = \theta^{\vee}$ and $\beta_{i} = \omega_{i}^{\vee}$ for each $i\in I$, and when $N = \mathring{P}$ set $\beta_{\theta} = \theta$ and each $\beta_{i} = \omega_{i}$.
Let $v_{i} = w_{0}w_{0i}$ where $w_{0}$ is the longest element\footnote{For a nice explanation of how to find a reduced expression for any $w_{0}$ (and thus $w_{0i}$) by $2$-colouring the Dynkin diagram, see Allen Knutson's answer at
\url{https://mathoverflow.net/questions/54926/longest-element-of-weyl-groups}
(last accessed 31$^{\mathrm{st}}$ Jan 2025).
Alternatively, \cite{BKOP14}*{Table 1} contains such an expression in each finite type.}
of $\mathring{W}$ and $w_{0i}$ is the longest element of the isotropy subgroup $\langle s_{j}~|~j\not= i\rangle$ of $\beta_{i}$.
The correspondence between the Coxeter and Bernstein presentations of $\Bd$ is then given by $T_{0} = X_{\beta_{\theta}} \Theta^{-1}$ and $\pi_{i} = X_{\beta_{i}} T_{v_{i}}^{-1}$ for each $i\in \Imin$.

\begin{rmk} \label{rmk:alternative Bernstein presentation}
    There is an automorphism of $\Bd$ which inverts $T_{0},\dots,T_{n}$ and fixes each element of $\Omega$.
    Letting $Y_{\beta}$ be the image of $X_{\beta}$ for all $\beta \in N$, we obtain an \textit{alternative Bernstein presentation} for $\Bd$ matching that of \cite{IS20}*{Prop. 9.1}.
    In particular, for each $i\in I_{0}$ and $\beta \in N$ we have the relations
    \begin{itemize}
        \item $T_{i}Y_{\beta} = Y_{\beta}T_{i}$ if $(\beta,A_{i}^{\vee}) = 0, \hfill \refstepcounter{equation}(\theequation)\label{first alternative extended affine Bernstein}$
        \item $T_{i} Y_{\beta}T_{i} = Y_{s_{i}(\beta)}$ if $(\beta,A_{i}^{\vee}) = 1. \hfill \refstepcounter{equation}(\theequation)\label{second alternative extended affine Bernstein}$
    \end{itemize}
    It immediately follows that the Coxeter presentation relates to this alternative Bernstein realization via $T_{0} = \Theta^{-1}Y_{-\beta_{\theta}}$ and $\pi_{i} = Y_{\beta_{i}} T_{v_{i}^{-1}}$ for each $i\in \Imin$.
\end{rmk}

\subsection{Drinfeld-Jimbo quantum groups} \label{subsection:Drinfeld-Jimbo quantum groups}

For an arbitrary symmetrizable Kac-Moody algebra $\s$ with generalized Cartan matrix $(a_{ij})_{i,j\in I}$, the corresponding quantum group is given in terms of certain Chevalley-style generators as follows.

\begin{defn} \label{defn:quantum group}
The quantum group $\Uqs$ is the unital associative $\kk$-algebra generated by elements $q^{h}$ for each $h\in P^{\vee}$ and $x_{i}^{\pm}$ for all $i\in I$, subject to the following relations:
\begin{itemize}
    \item $\displaystyle q^{0} = 1$,
    \item $\displaystyle q^{h}q^{h'} = q^{h+h'}$,
    \item $\displaystyle q^{h} \xipm q^{-h} = q^{\pm\langle\alpha_{i},h\rangle} \xipm$,
    \item $\displaystyle [\xip,\xjm] = \frac{\delta_{ij}}{q_{i}-q_{i}^{-1}} (k_{i} - k_{i}^{-1})$,
    \item $\displaystyle \sum_{s=0}^{1-a_{ij}} (-1)^{s} (\xipm)^{(s)} \xjpm (\xipm)^{(1-a_{ij}-s)} = 0$ whenever $i\not= j$,
\end{itemize}
where $k_{i} = q^{d_{i}\alpha^{\vee}_{i}}$ for each $i\in I$.
\end{defn}

This is called the Drinfeld-Jimbo realization for $\Uqs$, and makes clear a natural $\kk$-algebra anti-involution
$\sigma = (q^{h} \mapsto q^{-h},~\xipm \mapsto \xipm)$
and $\Qbb$-algebra involution
$\omega = (q \mapsto q^{-1},~q^{h} \mapsto q^{h},~\xipm \mapsto x^{\mp}_{i})$.

\begin{eg}
    Associated to any affine Kac-Moody algebra $\gaff$ there exists a quantum affine algebra $\Uaff$ provided by the above definition.
    We define $\Udash$ to be the subalgebra generated by all $\xipm$ and $k_{i}^{\pm 1}$, which can alternatively be obtained by replacing the affine coweight lattice $P^{\vee}$ with the classical coweight lattice $\Pbar^{\vee} = \bigoplus_{i\in I}\Zbb \alpha^{\vee}_{i}$.
\end{eg}

\begin{defn}
    A triangular decomposition of an algebra $A$ consists of three subalgebras $A^{-}$, $A^{0}$ and $A^{+}$ such that multiplication
    $a_{-} \otimes a_{0} \otimes a_{+} \mapsto a_{-}a_{0}a_{+}$
    provides an isomorphism of vector spaces
    $A^{-} \otimes A^{0} \otimes A^{+} \cong A$.
\end{defn}

It is clear that for any Drinfeld-Jimbo quantum group there exists a natural triangular decomposition
$\Uqs \cong U^{-} \otimes U^{0} \otimes U^{+}$
into negative, zero and positive subalgebras
$\langle \xim ~|~ i\in I \rangle$,
$\langle q^{h} ~|~ h\in P^{\vee} \rangle$ and
$\langle \xip ~|~ i\in I \rangle$ respectively.

\subsubsection{Coproducts} \label{subsubsection:coproducts}

The quantum group $\Uqs$ possesses various Hopf algebra structures.
Throughout this paper we shall use the one with coproduct $\Delta$ given by
\begin{align*}
    \Delta(q^{h}) = q^{h}\otimes q^{h}, \qquad
    \Delta(\xip) = \xip\otimes 1 + k_{i}^{-1}\otimes \xip, \qquad
    \Delta(\xim) = \xim\otimes k_{i} + 1\otimes \xim,
\end{align*}
counit $\varepsilon$ satisfying $\varepsilon(q^{h}) = 1$ and $\varepsilon(\xipm) = 0$, and antipode $S$ with
\begin{align*}
    S(q^{h}) = q^{-h}, \qquad
    S(\xip) = -\xip k_{i}, \qquad
    S(\xim) = -k_{i}^{-1}\xim.
\end{align*}

Our choice is the same as for example \cite{Hernandez09} and is denoted by $\Dbarplus$ in \cite{KMPY96}, where the following alternative commonly-used coproducts are also presented:
\begin{alignat*}{3}
    &\Delta_{+}(q^{h}) = q^{h}\otimes q^{h}, \qquad
    &&\Delta_{+}(\xip) = \xip\otimes 1 + k_{i}\otimes \xip, \qquad
    &&\Delta_{+}(\xim) = \xim\otimes k_{i}^{-1} + 1\otimes \xim,
    \\
    &\Delta_{-}(q^{h}) = q^{h}\otimes q^{h}, \qquad
    &&\Delta_{-}(\xip) = \xip\otimes k_{i}^{-1} + 1\otimes \xip, \qquad
    &&\Delta_{-}(\xim) = \xim\otimes 1 + k_{i}\otimes \xim,
    \\
    &\overline{\Delta}_{-}(q^{h}) = q^{h}\otimes q^{h}, \qquad
    &&\overline{\Delta}_{-}(\xip) = \xip\otimes k_{i} + 1\otimes \xip, \qquad
    &&\overline{\Delta}_{-}(\xim) = \xim\otimes 1 + k_{i}^{-1}\otimes \xim.
\end{alignat*}
These are obtained by conjugating $\Delta = \Dbarplus$ with $\sigma$, $\omega\sigma$ and $\omega$ respectively.

\subsubsection{Highest weight theory} \label{subsubsection:highest weight theory}

Here we introduce some of the basic definitions regarding modules for quantum groups.

\begin{defn}
    \begin{itemize}
        \item A representation $V$ of $\Uqs$ is a weight module if it decomposes as a direct sum
        $\bigoplus_{\lambda\in P} V_{\lambda}$
        of its weight spaces
        $V_{\lambda} = \lbrace u\in V ~\vert~ q^{h}\cdot u = q^{\langle \lambda,h\rangle}u~\mathrm{for~all}~h\in P^{\vee}\rbrace$.
        \item It is moreover a highest weight module with highest weight $\lambda\in P$ if there exists some non-zero $v_{\lambda} \in V_{\lambda}$ such that
        $V = \Uqs\cdot v_{\lambda}$
        and all $\xip\cdot v_{\lambda} = 0$.
    \end{itemize}
\end{defn}

\begin{eg}
    \begin{itemize}
        \item The Verma module $M(\lambda)$ is the quotient of $\Uqs$ by the left ideal generated by
        $\lbrace q^{h} - q^{\langle \lambda,h\rangle}1 ~|~ h\in P^{\vee}\rbrace$
        and $U^{+} = \langle \xip ~|~ i\in I\rangle$.
        It has the universal property that every highest weight module with highest weight $\lambda$ is the image of $M(\lambda)$ under the unique homomorphism that sends $1 \mapsto v_{\lambda}$.
        \item $M(\lambda)$ possesses a unique maximal submodule, hence the corresponding quotient $V(\lambda)$ is the unique irreducible highest weight module of highest weight $\lambda$ up to isomorphism.
    \end{itemize}
\end{eg}

A weight module is integrable if all $\xipm$ act locally nilpotently, that is for each $v\in V$ we have $(\xipm)^{k}\cdot v = 0$ for some $k\geq 0$.
An element $v\in V$ is extremal if there exists a set of vectors $\lbrace v_{w}\rbrace_{w\in W}$ such that
\begin{itemize}
    \item $v_{e} = v$,
    \item if $\langle w\lambda,\alpha^{\vee}_{i}\rangle\geq 0$ then $\xip\cdot v_{w}=0$ and $(\xim)^{(\langle w\lambda,\alpha^{\vee}_{i}\rangle)}\cdot v_{w}=v_{s_{i}w}$,
    \item if $\langle w\lambda,\alpha^{\vee}_{i}\rangle\leq 0$ then $\xim\cdot v_{w}=0$ and $(\xip)^{(-\langle w\lambda,\alpha^{\vee}_{i}\rangle)}\cdot v_{w}=v_{s_{i}w}$.
\end{itemize}
Such a set must be unique, with each $v_{w}$ spanning $V_{w\lambda}$.
In this case, we say that $V$ is an extremal weight module \cite{Kashiwara94}.
For each $\lambda\in P$ define $V^{\mathrm{ext}}(\lambda)$ to be the representation of $\Uqs$ generated by a non-zero vector $v_{\lambda}$, subject only to the condition that it is an extremal vector of weight $\lambda$.
In particular, if $\lambda$ is dominant then $V^{\mathrm{ext}}(\lambda)$ is isomorphic to the irreducible highest weight module $V(\lambda)$.
\\

Let $\Oint$ be the category of integrable representations $V$ of $\Uqs$ with finite dimensional weight spaces, for which there exist $\mu_{1},\dots,\mu_{r} \in P$ such that
\begin{align*}
    \lbrace \lambda \in P ~|~ V_{\lambda} \not= 0 \rbrace
    \subset
    \bigcup_{j=1}^{r} (\mu_{j} - Q^{+})
\end{align*}
where $Q^{+} = \bigoplus_{i\in I} \Nbb \alpha_{i}$ is the positive root lattice.
Then $\Oint$ is closed under finite direct sums and tensor products, and moreover we have the following structural result from \cite{HK02}*{Ch. 3}.

\begin{thm} \label{thm:category O semisimple}
    The category $\Oint$ is semisimple, and the indecomposable objects are precisely the irreducible highest weight modules $V(\lambda)$ with $\lambda \in P^{+}$.
\end{thm}

Therefore, in many situations, in order to understand the entire category $\Oint$ it is enough to consider those $V(\lambda)$ for which $\lambda$ is a dominant integral weight.

\subsubsection{Braid group action} \label{subsubsection:braid group action}

We briefly recall the action of the braid group $\B$ on the quantum group $\Uqs$ due to Lusztig \cite{Lusztig93}.
For every $i\in I$ there exists an automorphism $\Tb_{i}$ of $\Uqs$ defined by $\Tb_{i}(q^{h}) = q^{s_{i}(h)}$ for each $h\in P^{\vee}$ and
\begin{align*}
    &\Tb_{i}(\xip) = -\xim k_{i}, \qquad \Tb_{i}(\xjp) = \sum_{s=0}^{-a_{ij}} (-1)^{s} q_{i}^{-s} (\xip)^{(-a_{ij}-s)} \xjp (\xip)^{(s)} \mathrm{~~if~} i\not= j, \\
    & \Tb_{i}(\xim) = -k_{i}^{-1}\xip, \qquad \Tb_{i}(\xjm) = \sum_{s=0}^{-a_{ij}} (-1)^{s} q_{i}^{s} (\xim)^{(s)} \xjm (\xim)^{(-a_{ij}-s)} \mathrm{~~if~} i\not= j.
\end{align*}
Its inverse $\Tb_{i}^{-1}$ is given by $\Tb^{-1}_{i}(q^{h}) = q^{s_{i}(h)}$ and
\begin{align*}
    &\Tb_{i}^{-1}(\xip) = -k_{i}^{-1}\xim, \qquad \Tb_{i}^{-1}(\xjp) = \sum_{s=0}^{-a_{ij}} (-1)^{s} q_{i}^{-s} (\xip)^{(s)} \xjp (\xip)^{(-a_{ij}-s)} \mathrm{~~if~} i\not= j, \\
    &\Tb_{i}^{-1}(\xim) = -\xip k_{i}, \qquad \Tb_{i}^{-1}(\xjm) = \sum_{s=0}^{-a_{ij}} (-1)^{s} q_{i}^{s} (\xim)^{(-a_{ij}-s)} \xjm (\xim)^{(s)} \mathrm{~~if~} i\not= j.
\end{align*}
In particular, we note that
$\Tb_{i}(k_{j}) = \Tb^{-1}_{i}(k_{j}) = k_{j} k_{i}^{-a_{ij}}$ for all $j\in I$.
A quick check verifies that each $\Tb_{i}^{-1} = \sigma \Tb_{i} \sigma$, where $\sigma$ is the anti-involution of $\Uqs$ introduced earlier.

\begin{thm} \label{thm:braid group action}
    The braid group $\B$ acts on the quantum group $\Uqs$ via $T_{i} \mapsto \Tb_{i}$ for each $i\in I$.
\end{thm}

Throughout this paper we shall use without comment that $\Tb_{i} \Tb_{j} (\xipm) = \xjpm$ and $\Tb_{i}^{-1} \Tb_{j}^{-1} (\xipm) = \xjpm$ whenever $a_{ij} = a_{ji} = -1$.
The short technical proof of this result can be found in \cite{Lusztig93}*{Ch. 37}.
\\

Every automorphism $\pi$ of the associated Dynkin diagram $D(A)$ gives rise to an automorphism $S_{\pi}$ of $\Uqs$ which permutes the generators accordingly:
\begin{align*}
    S_{\pi}(\xjpm) = \xpm_{\pi(j)}, \qquad S_{\pi}(q^{h}) = q^{\pi(h)},
\end{align*}
where $\pi(h)$ is given by the natural action on $P^{\vee}$, extended trivially from the permutation of the simple coroots.
We note in particular that each $S_{\pi}(k_{i}^{\pm 1}) = k_{\pi(i)}^{\pm 1}$.

\begin{cor} \label{cor:Bd action}
    The extended affine braid group $\Bd$ acts on the quantum affine algebras $\Uaff$ and $\Udash$ via $T_{i} \mapsto \Tb_{i}$ and $\pi \mapsto S_{\pi}$ for all $i\in I$ and $\pi\in\Omega$.
\end{cor}

\subsection{Quantum affinizations} \label{subsection:quantum affinizations}

Any Drinfeld-Jimbo quantum group can be affinized within the quantum setting as follows.

\begin{defn} \label{defn:quantum affinization}
    The quantum affinization of $\Uqs$ is the unital associative $\kk$-algebra $\Uqaffs$ with generators $\xpm_{i,m}$, $h_{i,r}$, $q^{h}$, $C^{\pm 1}$ ($i\in I$, $m\in\Zbb$, $r\in\Zbb^{*}$, $h\in P^{\vee}$) and relations
\begin{itemize}
    \item $C^{\pm 1}$ central, $\hfill \refstepcounter{equation}(\theequation)\label{eqn:quantum affinization relations 1}$
    \item $\displaystyle C^{\pm 1}C^{\mp 1} = q^{0} = 1,
    \hfill \refstepcounter{equation}(\theequation)\label{eqn:quantum affinization relations 2}$
    \item $\displaystyle q^{h}q^{h'} = q^{h+h'},
    \hfill \refstepcounter{equation}(\theequation)\label{eqn:quantum affinization relations 3}$
    \item $\displaystyle [q^{h},h_{i,r}] = 0,
    \hfill \refstepcounter{equation}(\theequation)\label{eqn:quantum affinization relations 4}$
    \item $\displaystyle [h_{i,r},h_{j,s}] = \delta_{r+s,0} \frac{[ra_{ij}]_{i}}{r} \frac{C^{r}-C^{-r}}{q_{j}-q_{j}^{-1}},
    \hfill \refstepcounter{equation}(\theequation)\label{eqn:quantum affinization relations 5}$
    \item $\displaystyle q^{h} \xpm_{i,m} q^{-h} = q^{\pm \langle \alpha_{i},h\rangle} \xpm_{i,m},
    \hfill \refstepcounter{equation}(\theequation)\label{eqn:quantum affinization relations 6}$
    \item $\displaystyle [h_{i,r},\xpm_{j,m}] = \pm \frac{[ra_{ij}]_{i}}{r} C^{\frac{r \mp \lvert r\rvert}{2}} \xpm_{j,r+m},
    \hfill \refstepcounter{equation}(\theequation)\label{eqn:quantum affinization relations 7}$
    \item $\displaystyle [\xp_{i,m},\xm_{j,l}] = \frac{\delta_{ij}}{q_{i}-q_{i}^{-1}} (C^{-l}\phi^{+}_{i,m+l} - C^{-m}\phi^{-}_{i,m+l}),
    \hfill \refstepcounter{equation}(\theequation)\label{eqn:quantum affinization relations 8}$
    \item $\displaystyle [\xpm_{i,m+1},\xpm_{j,l}]_{q_{i}^{\pm a_{ij}}} + [\xpm_{j,l+1},\xpm_{i,m}]_{q_{i}^{\pm a_{ij}}} = 0,
    \hfill \refstepcounter{equation}(\theequation)\label{eqn:quantum affinization relations 9}$
\end{itemize}
and whenever $i\not= j$, for any integers $m$ and $m_{1},\dots,m_{a'}$ where $a' = 1 - a_{ij}$,
\begin{itemize}
    \item $\displaystyle
    \sum_{\pi\in S_{a'}}
    \sum_{s=0}^{a'} (-1)^{s}
    {\begin{bmatrix}a'\\s\end{bmatrix}}_{i}
    \xpm_{i,m_{\pi(1)}}\dots\xpm_{i,m_{\pi(s)}}
    \xpm_{j,m}
    \xpm_{i,m_{\pi(s+1)}}\dots\xpm_{i,m_{\pi(a')}}
    = 0.
    \hfill \refstepcounter{equation}(\theequation)\label{eqn:quantum affinization relations 10}$
\end{itemize}
Here each $k_{i} = q^{d_{i}\alpha^{\vee}_{i}}$ and the $\phi^{\pm}_{i,\pm s}$ are given by the formula
$$ \sum_{s\geq 0} \phi^{\pm}_{i,\pm s} z^{\pm s} =
k_{i}^{\pm 1} \exp{\left( \pm (q_{i}-q_{i}^{-1})\sum_{s'>0}h_{i,\pm s'} z^{\pm s'} \right)}
$$
when $s\geq 0$, and are zero otherwise.
\end{defn}

One may alternatively view $\Uqaffs$ as a deformation quantization of the one-dimensional central extension of the loop Lie algebra $\s[t,t^{-1}]$ of smooth maps $S^{1} \rightarrow \s$.
In particular, when $\s = \g$ is finite type, this is the loop-style realization of the corresponding untwisted affine Kac-Moody algebra $\gaff$ without derivation.
Loosely speaking, the $\xp_{i,m}$, $\xm_{i,m}$, $h_{i,r}$, $q^{h}$ generators above correspond to the elements $e_{i}t^{m}$, $f_{i}t^{m}$, $h_{i}t^{r}$, $h$ respectively inside $\s[t,t^{-1}]$, and $C^{\pm 1}$ is identified with the central extension.

\begin{rmk}
    \begin{itemize}
        \item Relations (\ref{eqn:quantum affinization relations 10}) are called the affine $q$-Serre relations.
        \item The definition of $\Uqaffs$ varies slightly between sources.
        We use the one found for example in \cites{Damiani12,Damiani24,Miki99} since it is more precise regarding the isomorphism between the two presentations of the quantum affine algebra (see Section \ref{subsubsection:quantum affine algebras}).
        The definition found in other works such as \cites{Beck94,Jing98,Hernandez09} can then be obtained by adjoining $C^{\pm 1/2}$ and scaling each $\xpm_{i,m}$ generator by $C^{m/2}$.
    \end{itemize}
\end{rmk}

It is clear that any quantum affinization $\Uqaffs$ possesses the following natural automorphisms and anti-automorphisms.
\begin{itemize}
    \item Every automorphism $\pi$ of the underlying Dynkin diagram gives rise to an automorphism $\Scal_{\pi}$ of $\Uqaffs$ defined by
    \begin{align*}
        &\Scal_{\pi}(\xpm_{i,m}) = o_{i,\pi(i)}^{m}\xpm_{\pi(i),m}, \qquad
        \Scal_{\pi}(h_{i,r}) = o_{i,\pi(i)}^{r}h_{\pi(i),r}, \qquad
        \Scal_{\pi}(q^{h}) = q^{\pi(h)}, \qquad
        \Scal_{\pi}(C) = C.
    \end{align*}
    \item For each $i\in I$ there is an automorphism $\X_{i}$ given by
    \begin{align*}
        &\X_{i}(\xpm_{j,m}) = \upsilon(j)^{\delta_{ij}} \xpm_{j,m\mp\delta_{ij}}, \qquad \X_{i}(h_{j,r}) = h_{j,r}, \qquad
        \X_{i}(q^{h}) = C^{-\langle \Lambda_{i},h \rangle} q^{h}, \qquad \X_{i}(C) = C,
    \end{align*}
    where $\upsilon$ is any $\lbrace \pm 1 \rbrace$-valued function on $I$, for example a sign function.
    \item There is also an anti-involution $\eta$ with
    \begin{align*}
        &\eta(\xpm_{i,m}) = \xpm_{i,-m}, \qquad \eta(h_{i,r}) = -C^{r} h_{i,-r}, \qquad \eta(q^{h}) = q^{-h}, \qquad \eta(C) = C.
    \end{align*}
    \item There exists a $\Qbb$-algebra involution $\Wcal$ sending $q \mapsto q^{-1}$ such that
    \begin{align*}
        &\Wcal(\xpm_{i,m}) = C^{m} x^{\mp}_{i,m}, \qquad
        \Wcal(h_{i,r}) = -h_{i,r}, \qquad
        \Wcal(q^{h}) = q^{h}, \qquad
        \Wcal(C) = C^{-1}.
    \end{align*}
\end{itemize}

\begin{rmk} \label{rmk:affinizations of morphisms}
    We can roughly think of $\Scal_{\pi}$, $\eta$ and $\Wcal$ as `affinizations' of the corresponding (anti-)automorphisms $S_{\pi}$, $\sigma$ and $\omega$ from Section \ref{subsection:Drinfeld-Jimbo quantum groups}.
    Indeed, the former restrict to the latter on
    $\langle q^{h},\, \xpm_{i,0} ~|~ h\in P^{\vee},\, i\in I \rangle$.
\end{rmk}

For each subset $J = \lbrace j_{1},\dots,j_{p}\rbrace$ of $I$, let $\U(J) = \U(j_{1},\dots,j_{p})$ be the subalgebra of $\Uqaffs$ generated by
$\lbrace \xpm_{i,m},\, h_{i,r},\, k_{i}^{\pm 1},\, C^{\pm 1}~|~ i = j_{1},\dots,j_{p}, \, m\in\Zbb, \, r\in\Zbb^{*} \rbrace$.
Theorem 2 and Corollary 3 of \cite{Hernandez05} imply that this is in fact a copy of the quantum affinization associated to the full Dynkin subdiagram on $J$.
For later use, we record that the isomorphism
$h_{i} : \UdashA \xrightarrow{\sim} \U(i)$
is given \cite{Beck94} by
\begin{align} \label{eqn:hi isomorphism}
    & q \mapsto q_{i}, \qquad k_{1} \mapsto k_{i}, \qquad k_{0} \mapsto C k_{i}^{-1}, \qquad \xpm_{1} \mapsto \xpm_{i,0}, \\
    & \xp_{0} \mapsto - o(i) C k_{i}^{-1} \xm_{i,1}, \qquad
    \xm_{0} \mapsto - o(i) \xp_{i,-1} k_{i} C^{-1}.
\end{align}

Throughout this section, we shall freely use the Drinfeld new realization of the quantum affine algebra $\Uaff$ as the quantum affinization $\Uqaffg$ of corresponding the finite quantum group.
However, we postpone any further explanation of this result until Section \ref{subsubsection:quantum affine algebras}.

\subsubsection{Gradings and scaling automorphisms} \label{subsubsection:Gradings and scaling automorphisms}

Any quantum affinization $\Uqaffs$ possesses a fine grading $\deg$ taking values in $Q\oplus\Zbb\delta'$, given by
\begin{align*}
    \deg(\xpm_{i,m}) = (\pm \alpha_{i},m\delta'), \qquad
    \deg(h_{i,r}) = (0,r\delta'), \qquad
    \deg(C^{\pm 1}) = \deg(q^{h}) = (0,0).
\end{align*}
We shall write the resulting decomposition into graded pieces as
\begin{align} \label{eqn:decomposition of Uqaffs}
    \Uqaffs
    = \bigoplus_{\substack{\mu\in\mathring{Q} \\ k,\ell\in\Zbb}} \U_{\mu+k\delta,\ell\delta'}.
\end{align}
Projecting $\deg$ to $\mathring{Q}\oplus\Zbb\delta'$ and then taking the height defines a $\Zbb$--grading
\begin{align*}
    \degv(\xpm_{i,m}) = \pm \mathds{1}_{i\in I_{0}} + \hslash m, \qquad
    \degv(h_{i,r}) = \hslash r, \qquad
    \degv(C^{\pm 1}) = \degv(q^{h}) = 0,
\end{align*}
where $\hslash = \sum_{i\in I} a_{i}$ is the Coxeter number of $\gaff$.
The grading $\degv$ can be thought of as \emph{not seeing the horizontal $\delta$ direction}.
(Conversely, taking the height within $Q$ produces a $\Zbb$--grading $\deg_{h}$ which does not see the vertical $\delta'$ direction.)
By instead projecting $\deg$ to $\Zbb\alpha_{j}$ or $\Zbb\delta'$ we obtain coarse $\Zbb$--gradings
\begin{gather*}
    \degj(\xpm_{i,m}) = \pm \delta_{ij}, \qquad
    \degj(C^{\pm 1}) = \degj(q^{h}) = \degj(h_{i,r}) = 0,
    \\
    \degZ(\xpm_{i,m}) = m, \qquad
    \degZ(h_{i,r}) = r, \qquad
    \degZ(C^{\pm 1}) = \degZ(q^{h}) = 0,
\end{gather*}
for each $j\in I$.
To every $\Zbb$--grading we can associate scaling automorphisms
\begin{align*}
    \sfrakv_{a} : z \mapsto a^{\degv(z)} z, \qquad
    \sfrakj_{a} : z \mapsto a^{\degj(z)} z, \qquad
    \sfrakZ_{a} : z \mapsto a^{\degZ(z)} z,
\end{align*}
for any $a\in\Cbb^{\times}$, where $z$ is a homogeneous element of $\Uqaffs$.
Note that
$\degv = \hslash\degZ + \sum_{j\in I_{0}} \degj$
and thus
$\sfrakv_{a} = (\sfrakZ_{a})^{\hslash} \prod_{j\in I_{0}} \sfrakj_{a}$.

\begin{rmk} \label{rmk:degree operators for quantum affinizations}
    One may enlarge $\Uqaffs$ by adding generators $D^{\pm 1}$ such that conjugation by $D$ acts as some scaling automorphism.
    Various references include $D^{\pm 1}$ corresponding to $\sfrakZ_{q}$ in their definition of $\Uqaffs$, in which case $\Scal_{\pi}$, $\X_{i}$, $\eta$ and $\Wcal$ extend by mapping $D$ to $D$, $D q^{\Lambda_{0}^{\vee}}$, $D$ and $D^{-1}$ respectively.
\end{rmk}

\subsubsection{Topological coproducts} \label{subsubsection:topological coproducts}

Unlike quantum groups, quantum affinizations are not known to possess Hopf algebra or even coproduct structures.
Nevertheless, Drinfeld did define in an unpublished note -- see also \cites{DF93,DI97} -- a \textit{topological} coproduct for $U_{q}(\widehat{\mathfrak{sl}}_{n+1})$ with respect to the Drinfeld new presentation, taking values in a completion of its tensor square.
This was later generalised by Hernandez \cite{Hernandez05} to a topological coproduct for general quantum affinizations, depending on a spectral parameter.
However, compatibility with the affine $q$-Serre relations (\ref{eqn:quantum affinization relations 10}) was known only in finite \cites{Enriquez00,Grosse07} and simply laced \cite{DI97} types.
\\

Recent work of Damiani \cite{Damiani24} addresses this issue, proving that there exists a topological coproduct $\Delta_{u}$ of $\Uqaffs$ in the general case.
Her method relies upon careful consideration of the specific completion
$\Uqaffs \widehat{\otimes} \Uqaffs$
into which $\Delta_{u}$ maps.
For simplicity, we shall not dwell on these (important) subtleties here and instead refer the interested reader to \cite{Damiani24}.
For example, there \S3 defines the completions considered,
\S7 proves the coassociativity and counit properties,
and
Remark 7.7 discusses differences with \cite{Hernandez05}.

\begin{thm} \label{thm:Damiani topological coproduct}
    \cites{Damiani24}
    There is a unique algebra morphism
    $\Delta_{u} : \Uqaffs \rightarrow \Uqaffs \widehat{\otimes} \Uqaffs$
    sending
    \begin{align*}
        C^{\pm 1} &\mapsto C^{\pm 1} \otimes C^{\pm 1}, \\
        q^{h} &\mapsto q^{h} \otimes q^{h}, \\
        C^{s}\phi^{+}_{i,r} &\mapsto \sum_{k+\ell=r}
        (C^{s+\ell} \phi^{+}_{i,k} \otimes C^{s}\phi^{+}_{i,\ell})
        u^{-\ell}, \\
        C^{s}\phi^{-}_{i,r} &\mapsto \sum_{k+\ell=r}
        (C^{s}\phi^{-}_{i,k} \otimes C^{s+k} \phi^{-}_{i,\ell})
        u^{-\ell}, \\
        \xp_{i,m} &\mapsto \xp_{i,m} \otimes 1
        + \sum_{\ell\geq 0} (C^{m-\ell} \phi^{+}_{i,\ell} \otimes \xp_{i,m-\ell})
        u^{\ell-m}, \\
        \xm_{i,m} &\mapsto (1 \otimes \xm_{i,m}) v^{-m}
        + \sum_{\ell\leq 0} (\xm_{i,m-\ell} \otimes C^{m-\ell} \phi^{-}_{i,\ell})
        u^{-\ell},
    \end{align*}
    for all $h\in P^{\vee}$, $i\in I$ and $r,s,m\in\Zbb$.
    This map is injective, and satisfies the coassociativity property
    \begin{align*}
        (\Delta_{u} \widehat{\otimes} \mathrm{id}) \circ \Delta_{u}
        =
        (\mathrm{id} \widehat{\otimes} \Delta_{u}) \circ \Delta_{u}
        :
        \Uqaffs \rightarrow
        \Uqaffs^{\widehat{\otimes} 3}.
    \end{align*}
    Moreover $\Delta_{u}$ possesses a counit
    $\varepsilon : \Uqaffs \rightarrow \mathbb{Q}(q)$ given by
    \begin{align*}
        \varepsilon(C^{\pm 1})
        = \varepsilon(q^{h})
        = \varepsilon(\phi^{\pm}_{i,r})
        = 1,
        \qquad
        \varepsilon(\xpm_{i,m})
        = 0,
    \end{align*}
    such that
    $(\varepsilon \widehat{\otimes} \mathrm{id}) \circ \Delta_{u}
    =
    (\mathrm{id} \widehat{\otimes} \varepsilon) \circ \Delta_{u}
    =
    \mathrm{id}$.
\end{thm}

It is worth noting that the power of $u$ in each of the expressions above records \emph{minus} the degree $\deg_{\Zbb}$ of the second factor.
Crucially, when working with $\Delta_{u}$ it is therefore often enough to consider only $u = 1$ since
$\Delta_{u} = (\mathrm{id} \otimes \sfrakZ_{u^{-r}}) \circ \Delta_{1}$.
Furthermore, we have that
\begin{align*}
    \Delta_{u} : h_{i,r} \mapsto
    \begin{cases}
        h_{i,r} \otimes 1 + (C^{r} \otimes h_{i,r}) u^{-r}
        & \mathrm{if~} r>0, \\
        h_{i,r} \otimes C^{r} + (1 \otimes h_{i,r}) u^{-r}
        & \mathrm{if~} r<0,
    \end{cases}
\end{align*}
and hence (after specialising $u$) $\Delta_{u}$ sends $\Uqaffs^{0}$ into the usual non-completed tensor square
$\Uqaffs^{0} \otimes \Uqaffs^{0}$
since its generators are mapped to finite sums of elementary tensors.

\begin{rmk} \label{rmk:affinizations of coproducts}
    \begin{itemize}
        \item One can roughly think of $\Delta_{u}$ as an `affinization' of the coproduct $\Delta_{+}$ for $\Uqs$.
        Namely, $\Delta_{u}$ sends elements of
        $\langle q^{h},\, \xpm_{i,0} ~|~ h\in P^{\vee},\, i\in I \rangle$
        to their images under $\Delta_{+}$ plus series of terms which vanish as $u \rightarrow 0$.
        \item Using Remark \ref{rmk:affinizations of morphisms} we then see that conjugating $\Delta_{u}$ by $\Wcal$, $\eta$ and $\Wcal\eta$ produces such affinizations for $\Delta_{-}$, $\Delta = \overline{\Delta}_{+}$ and $\overline{\Delta}_{-}$ respectively.
        For example, Hernandez' topological coproduct in \cites{Hernandez05,Hernandez07} corresponds to $\Delta = \overline{\Delta}_{+}$ in this way.
    \end{itemize}
\end{rmk}

\begin{rmk} \label{rmk:coproduct not always tensor product}
    Although $\Delta_{u}$ does not give a well-defined morphism to $\Uqaffs \otimes \Uqaffs$, it can still be used to define tensor products of $\Uqaffs$-modules in certain specific cases by specialising $u$ to particular elements of $\Cbb^{\times}$.
    See for example the construction of Fock space and Macmahon representations for $\UtorA$ by Feigin-Jimbo-Miwa-Mukhin \cite{FJMM13}.
\end{rmk}

\subsubsection{\texorpdfstring{$\ell$}{l}-highest weight theory} \label{subsubsection:l-highest weight theory}

It is known \cite{Hernandez05} that for any quantum affinization there exists a so-called \emph{loop} triangular decomposition
$\Uqaffs \cong \Uqaffs^{-} \otimes \Uqaffs^{0} \otimes \Uqaffs^{+}$
into the subalgebras
\begin{align*}
    \langle \xm_{i,m} ~|~ i\in I,\, m\in\Zbb \rangle, \qquad
    \langle q^{h},\, h_{i,r},\, C^{\pm 1} ~|~ h\in P^{\vee},\, i\in I,\, r\in\Zbb^{*} \rangle, \qquad
    \langle \xp_{i,m} ~|~ i\in I,\, m\in\Zbb \rangle,
\end{align*}
respectively.
This allows us to define the notion of $\ell$-highest weight modules for quantum affinizations, analogously to the constructions of Section \ref{subsubsection:highest weight theory} for quantum groups.

\begin{defn} \label{defn:l-weights}
    \begin{itemize}
        \item An $\ell$-weight is a triple $(\lambda,\Psi,c)$ where $c\in\Cbb^{\times}$, $\lambda \in \h^{*}$ and
        $\Psi = (\Psi^{\pm}_{i,\pm s})_{i\in I,\, s \geq 0}$
        with all $\Psi^{\pm}_{i,\pm s} \in \Cbb$, satisfying the condition $\Psi^{\pm}_{i,0} = q_{i}^{\pm\langle\lambda,\alpha_{i}^{\vee}\rangle}$ for each $i\in I$.
        \item The set of $\ell$-weights is denoted by $P_{\ell}$.
    \end{itemize}
\end{defn}

\begin{defn} \label{defn:l-highest weight modules}
    \begin{itemize}
        \item A vector $v$ inside a $\Uqaffs$-module $V$ has $\ell$-weight $(\lambda,\Psi,c) \in P_{\ell}$ if
        \begin{align*}
            q^{h}\cdot v = q^{\langle \lambda,h \rangle}v, \qquad
            \phi^{\pm}_{i,\pm s}\cdot v = \Psi^{\pm}_{i,\pm s}v, \qquad
            C^{\pm 1}\cdot v = c^{\pm 1}v,
        \end{align*}
        for all $h\in P^{\vee}$, $i\in I$ and $s\in\Zbb_{\geq 0}$.
        \item Moreover, $v$ is $\ell$-highest weight if $\xp_{i,m}\cdot v = 0$ for all $i\in I$ and $m\in\Zbb$.
        \item If $V = \Uqaffs\cdot v$ for some $\ell$-highest weight vector $v$ of $\ell$-weight $(\lambda,\Psi,c) \in P_{\ell}$, then we call it an $\ell$-highest weight module of $\ell$-highest weight $(\lambda,\Psi,c)$.
    \end{itemize}
\end{defn}

The required compatibility between $\lambda$ and $\Psi$ is due to the fact that $k_{i}^{\pm 1} = \phi^{\pm}_{i,0}$.
Similarly to Section \ref{subsubsection:highest weight theory}, for each $(\lambda,\Psi,c) \in P_{\ell}$ we can define the associated Verma module $M(\lambda,\Psi,c)$ of $\ell$-highest weight $(\lambda,\Psi,c)$ as the quotient of $\Uqaffs$ by the left ideal generated by
\begin{align*}
    \lbrace \xp_{i,m},\,
    q^{h} - q^{\langle\lambda,h\rangle},\,
    \phi^{\pm}_{i,\pm s} - \Psi^{\pm}_{i,\pm s},\,
    C^{\pm 1} - c^{\pm 1}
    ~|~
    i\in I,\, m\in\Zbb,\, h\in P^{\vee},\, s\in\Zbb_{\geq 0} \rbrace.
\end{align*}
Again, this satisfies the universal property that $M(\lambda,\Psi,c)$ surjects onto any $\Uqaffs$-module of $\ell$-highest weight $(\lambda,\Psi,c)$, with $1$ sent to the $\ell$-highest weight vector $v$ in Definition \ref{defn:l-highest weight modules}.
Moreover, $M(\lambda,\Psi,c)$ contains a unique maximal submodule and the corresponding quotient $V(\lambda,\Psi,c)$ is the unique irreducible module of $\ell$-highest weight $(\lambda,\Psi,c)$ up to isomorphism.
\\

There also exists a notion of integrability for representations of quantum affinizations.

\begin{defn}
    \begin{itemize}
        \item A representation of $\Uqaffs$ is integrable if it is integrable as a $\Uqs$-module via restriction to
        $\langle q^{h},\, \xpm_{i,0} ~|~ h\in P^{\vee},\, i\in I \rangle$,
        with finite dimensional weight spaces.
        \item The category $\Oaff$ consists of representations of $\Uqaffs$ whose restrictions to
        $\langle q^{h},\, \xpm_{i,0} ~|~ h\in P^{\vee},\, i\in I \rangle$
        lie in $\Oint$.
    \end{itemize}
\end{defn}

In particular, $\Oaff$ contains all integrable $\ell$-highest weight representations, and the irreducible such modules are precisely the irreducible objects of $\Oaff$.
However, it is important to note that while $\Oaff$ is closed under taking submodules, quotients and finite direct sums, it is not a semisimple category (even when $\s$ is finite type).
\\

For any $\Uqaffs$-module $V$ we can define the weight spaces $V_{\lambda}$ exactly as in Section \ref{subsubsection:highest weight theory}, so that $\xpm_{i,m}\cdot V_{\lambda} \subset V_{\lambda\pm\alpha_{i}}$ and more generally
\begin{equation} \label{eqn:action on weight spaces}
    \U_{\beta + k\delta, \ell\delta'} \cdot V_{\lambda} \subset V_{\lambda + \beta + k\delta}
\end{equation}
for all $\beta\in\mathring{Q}$, $\lambda\in P$ and $k,\ell\in\Zbb$.
It follows that whenever $V$ is integrable, for any $v\in V$ there exists some $k\geq 0$ such that all $(\xpm_{i,m})^{k}\cdot v = 0$.
Furthermore, if $V$ is of $\ell$-highest weight $(\lambda,\Psi,c)$ then it must be diagonalisable as a representation of $\Uqaffs^{0}$ and moreover $V = \bigoplus_{\mu\leq\lambda} V_{\mu}$.

\begin{defn}
    Given a $\Uqaffs$-module $V\in\Oaff$ whose weights are contained in some
    $\bigcup_{\ell=1}^{N} (\lambda_{\ell} - Q^{+})$, for each $J\subset I$ we can define a subspace
    $V(J) = \bigoplus_{\ell=1}^{N} \bigoplus_{\mu \in Q(J)^{+}} V_{\lambda_{\ell} - \mu}$
    where
    $Q(J)^{+} = \bigoplus_{j\in J} \Nbb\alpha_{j}$.
\end{defn}

\begin{notation}
    When $J = \lbrace j\rbrace$ is a singleton we may write $\U(j)$, $V(j)$ and $Q(j)^{+}$ as shorthand for $\U(J)$, $V(J)$ and $Q(J)^{+}$.
\end{notation}

It is clear from (\ref{eqn:action on weight spaces}) that $V(J)$ becomes a $\U(J)$-module via restriction.

\begin{defn}
    A $\Uqaffs$-module $V$ is type $1$ if
    $C$ acts by the identity and it admits a weight space decomposition
    $V = \bigoplus_{\lambda\in P} V_{\lambda}$,
    so in particular the eigenvalues of each $k_{i}$ lie in $q^{\Zbb}$.
\end{defn}

\begin{prop}
    Any integrable $\ell$-highest weight $\Uqaffs$-module is the twist of a type $1$ representation.
\end{prop}
\begin{proof}
    It is clear that $\U(i)\cdot v$ is a finite dimensional $U_{q}(\widehat{\mathfrak{sl}}_{2})$-module for each $i\in I$, whereby \cite{CP91}*{§3.2} implies that
    $k_{i}\cdot v = \varepsilon_{i} q^{m_{i}} v$ and
    $C\cdot v = \varepsilon v$
    for some $m_{i}\in\Zbb_{\geq 0}$ and
    $\varepsilon_{i},\varepsilon\in\lbrace\pm 1\rbrace$.
    Twisting by the automorphism
    \begin{align*}
        \xp_{i,m} \mapsto \varepsilon_{i} \xp_{i,m}, \qquad
        \xm_{i,m} \mapsto \varepsilon^{m} \xm_{i,m}, \qquad
        h_{i,r} \mapsto \varepsilon^{(r - \lvert r\rvert)/2} h_{i,r}, \qquad
        \displaystyle q^{h} \mapsto q^{h} \prod_{i\in I} \varepsilon_{i}^{\langle\Lambda_{i},h\rangle/d_{i}}, \qquad
        C \mapsto \varepsilon C,
    \end{align*}
    then produces a type $1$ representation of $\Uqaffs$.
\end{proof}

\begin{notation}
    For the purposes of this paper we may therefore assume from now on that all such modules are type $1$, and write each element of $P_{\ell}$ as a pair $(\lambda,\Psi)$.
\end{notation}

\begin{defn}
    The set of $\ell$-dominant weights $P_{\ell}^{+}$ is the collection of $(\lambda,\Psi) \in P_{\ell}$ for which there exist (Drinfeld) polynomials $P_{i}(z) \in \Cbb[z]$ with all $P_{i}(0) = 1$ and
    \begin{align*}
        \sum_{s\geq 0} \Psi^{\pm}_{i,\pm s} z^{\pm s}
        =
        q_{i}^{\deg(P_{i})} \frac{P_{i}(zq_{i}^{-1})}{P_{i}(zq_{i})}
    \end{align*}
    in $\Cbb\llbracket z\rrbracket$ or $\Cbb\llbracket z^{-1}\rrbracket$ respectively.
\end{defn}

In this case, it follows that every $\langle \lambda,\alpha^{\vee}_{i} \rangle = \deg(P_{i}) \geq 0$ and so $\lambda$ must be dominant.

\begin{thm} \label{thm:l-highest weight classification}
    \cite{Hernandez05}
    An irreducible $\ell$-highest weight representation $V(\lambda,\Psi)$ is integrable if and only if $(\lambda,\Psi) \in P_{\ell}^{+}$.
\end{thm}

For finite types this is originally due to Chari-Pressley \cites{CP94,CP95}, where in fact these modules are precisely the irreducible finite dimensional representations of the quantum affine algebra.
In type $A_{n}^{(1)}$ the result was first proved by Miki \cite{Miki00} using their automorphism of $\UtorA$ from \cite{Miki99}.
Nakajima \cite{Nakajima01} later addressed all simply laced types, via geometric methods involving the equivariant K-theory of quiver varieties on the underlying Dynkin diagram.

\begin{notation}
    The irreducible, integrable $\ell$-highest weight module $V(\lambda,\Psi)$ corresponding to the Drinfeld polynomials $\Pcal(z) = (P_{i}(z))_{i\in I}$ may alternatively be denoted by $V(\Pcal(z))$.
\end{notation}

\begin{lem} \label{lem:twisting by scaling automorphisms}
    Twisting $V(\Pcal(z))$ by the scaling automorphisms $\sfrakv_{a}$ and $\sfrakZ_{a}$ from Section \ref{subsubsection:Gradings and scaling automorphisms} produces (up to isomorphism) those with polynomials $V(\Pcal(a^{\hslash}z))$ and $V(\Pcal(az))$.
\end{lem}
\begin{proof}
    Both $\sfrakv_{a}$ and $\sfrakZ_{a}$ preserve $\Utor^{+}$, and moreover scale any $\phi_{i,r}^{\pm}$ by $a^{\hslash r}$ and $a^{r}$ respectively.
\end{proof}

\begin{rmk}
    One can define and obtain analogous results for $\ell$-lowest weight modules simply by twisting every representation with $\Wcal$.
\end{rmk}

\subsubsection{\texorpdfstring{$q$}{q}-characters} \label{subsubsection:q-characters}

Here we recall the $q$-character morphism
$\chi_{q} : K(\Oaff) \rightarrow \Ycal$,
as introduced by Hernandez \cite{Hernandez05} -- see Section \ref{section:q-characters} for historical discussion and motivations.
Consider a representation $V\in\Oaff$ with weight space decomposition
$V = \bigoplus_{\lambda\in P} V_{\lambda}$.
Since $C^{\pm 1}$ act trivially, the actions of all $h_{i,r}$ commute and we can further decompose as a direct sum
$\bigoplus_{(\lambda,\Psi)\in P_{\ell}} V_{\lambda,\Psi}$
of $\ell$-weight spaces, where
\begin{align*}
    V_{\lambda,\Psi} = \lbrace v\in V_{\lambda} ~|~ \exists N\in\Nbb \text{ such that all } (\phi^{\pm}_{i,\pm s} - \Phi^{\pm}_{i,\pm s})^{N} \cdot v = 0\rbrace.
\end{align*}
Note that the finite-dimensionality of each $V_{\lambda,\Psi}$ follows from that of $V_{\lambda}$.
Let $\Ecal_{\ell}$ be the ring of maps $c : P_{\ell} \rightarrow \Zbb$ for which $c(\lambda,\Psi) = 0$ if $\lambda$ lies outside some finite union of cones $\bigcup_{j=1}^{r} (\mu_{j} - Q^{+})$.

\begin{defn}
    The formal character of $V\in\Oaff$ is
    $\ch_{q}(V) = \sum_{(\lambda,\Psi)\in P_{\ell}} \dim(V_{\lambda,\Psi}) e_{\lambda,\Psi}$
    where each
    $e_{\lambda,\Psi} : P_{\ell} \rightarrow \Zbb$
    is the indicator function of $(\lambda,\Psi)$.
\end{defn}

\begin{prop} \label{prop:formal characters}
    \cite{Hernandez05}
    For any representation $V\in\Oaff$ there exist $N_{\lambda,\Psi} \in \Nbb$ such that
    $\ch_{q}(V) = \sum_{(\lambda,\Psi)\in P_{\ell}^{+}} N_{\lambda,\Psi} \, \ch_{q}(V(\lambda,\Psi))$.
\end{prop}

In particular, the formal character of any module in $\Oaff$ is a sum of formal characters of irreducibles.

\begin{defn}
    Let $QP_{\ell}^{+}$ be the set of $\ell$-weights $(\lambda,\Psi)\in P_{\ell}$ for which there exist
    \begin{itemize}
        \item polynomials $Q_{i}(z),R_{i}(z)\in\Cbb[z]$ for each $i\in I$ such that $Q_{i}(0) = R_{i}(0) = 1$ and
        \begin{align*}
        \sum_{s\geq 0} \Psi^{\pm}_{i,\pm s} z^{\pm s}
        =
        q_{i}^{\deg(Q_{i})-\deg(R_{i})} \frac{Q_{i}(zq_{i}^{-1}) R_{i}(zq_{i})}{Q_{i}(zq_{i}) R_{i}(zq_{i}^{-1})},
        \end{align*}
    \item $\mu\in P^{+}$ and $\alpha\in Q^{+}$ such that $\lambda = \mu - \alpha$.
    \end{itemize}
\end{defn}

Taking $P_{i}(z) = Q_{i}(z)$ and $R_{i}(z) = 1$ for all $i\in I$, it is clear that $P_{\ell}^{+} \subset QP_{\ell}^{+}$.
The following serves as an extension of Theorem \ref{thm:l-highest weight classification} to all $\ell$-weights of modules in category $\Oaff$.

\begin{prop} \label{prop:l-weights in QPl+}
    For any representation $V\in\Oaff$, if $\dim(V_{\lambda,\Psi}) > 0$ then $(\lambda,\Psi) \in QP_{\ell}^{+}$.
\end{prop}
\begin{proof}
    This clearly reduces to the case $\s = \mathfrak{sl}_{2}$, with $U(i)\cdot v$ a finite dimensional $\Uaffsltwo$-module for any $v\in V$.
    Every such representation is isomorphic to a tensor product of so-called evaluation representations \cite{CP91}, which are pulled back along the morphism
    $\ev_{a} : \Uaffsltwo \rightarrow U_{q}(\mathfrak{sl}_{2})$
    due to Jimbo \cite{Jimbo86}, given by
    \begin{align*}
        \xpm_{1} \mapsto \xpm, \qquad
        \xpm_{0} \mapsto a^{\pm 1} q^{\mp 1} \xpm, \qquad
        k_{1}^{\pm 1} \mapsto k^{\pm 1}, \qquad
        k_{0}^{\pm 1} \mapsto k^{\mp 1},
    \end{align*}
    or alternatively
    $\xp_{1,m} \mapsto a^{m} q^{-m} k^{m} \xp$
    and
    $\xm_{1,m} \mapsto a^{m} q^{-m} \xm k^{m}$.
    Individual evaluation representations are checked directly via explicit computations, whereby simple identities such as
    \begin{align*}
        \Delta(\phi^{\pm}_{i,\pm s})
        = \sum_{r=0}^{s} \phi^{\pm}_{i,\pm r} \otimes \phi^{\pm}_{i,\pm (s-r)}
        \mod
        \widehat{U_{q}(\mathfrak{sl}_{2})}_{\mp} \otimes \widehat{U_{q}(\mathfrak{sl}_{2})}_{\pm}
    \end{align*}
    for the coproduct $\Delta$ on Drinfeld new generators complete the proof -- see \cite{FR99} for more details.
    The existence of $\mu \in P^{+}$ and $\alpha \in Q^{+}$ follows from Theorem \ref{thm:l-highest weight classification} and Proposition \ref{prop:formal characters}.
\end{proof}

Consider the group $\Mcal$ of monomials in commuting variables $k_{\varrho}$ and $Y_{i,a}^{\pm 1}$
($\varrho\in\h$, $i\in I$, $a\in\Cbb^{\times}$)
with
$k_{\varrho_{1}}k_{\varrho_{2}} = k_{\varrho_{1} + \varrho_{2}}$
and $k_{0} = 1$.
We define $\Ycal$ to be the subring of $\Mcal^{\Zbb}$ consisting of elements
\begin{align*}
    \sum_{\alpha\in\Lambda} n\alphapower k_{\varrho\alphapower} \prod_{\substack{i\in I \\ a\in\Cbb^{\times}}} Y_{i,a}^{u\alphapower_{i,a}}
\end{align*}
such that, assuming all $n\alphapower\not= 0$ without loss of generality,
\begin{itemize}
    \item $\lbrace (i,a) ~|~ u\alphapower_{i,a} \not= 0 \rbrace$ is finite for each $\alpha\in\Lambda$,
    \item $\langle \nu(\varrho\alphapower), \alpha_{i}^{\vee} \rangle = \sum_{a\in\Cbb^{\times}} u\alphapower_{i,a}$
    for all $i\in I$ and $\alpha\in\Lambda$,
    \item $\lbrace \nu(\varrho\alphapower) ~|~ \alpha\in\Lambda \rbrace$
    is contained in a finite union of cones
    $\bigcup_{j=1}^{r} (\mu_{j} - Q^{+})$
    with $\mu_{j} \in \h^{*}$.
\end{itemize}

Note that the second condition resembles the compatibility relations for an $\ell$-weight $(\lambda,\Psi)\in P_{\ell}$, and moreover $\Ycal$ is naturally equipped with an $\h$--grading.

\begin{eg}
    \begin{itemize}
        \item It is clear that $k_{\nu(\Lambda_{i})} Y_{i,a} \in \Ycal$ for all $i\in I$ and $a\in\Cbb^{\times}$.
        \item Every $(\lambda,\Psi)\in QP_{\ell}^{+}$ has an associated monomial
        \begin{align*}
            Y_{\lambda,\Psi} = k_{\nu(\lambda)} \prod_{\substack{i\in I \\ a\in\Cbb^{\times}}} Y_{i,a}^{\beta_{i,a} - \gamma_{i,a}} \in \Ycal
        \end{align*}
        where each
        $Q_{i}(z) = \prod_{a\in\Cbb^{\times}} (1-az)^{\beta_{i,a}}$
        and
        $R_{i}(z) = \prod_{a\in\Cbb^{\times}} (1-az)^{\gamma_{i,a}}$.
    \end{itemize}
\end{eg}

\begin{defn}
    The $q$-character of $V\in\Oaff$ is defined to be
    $\chi_{q}(V) = \sum_{(\lambda,\Psi)\in QP_{\ell}^{+}} \dim(V_{\lambda,\Psi}) Y_{\lambda,\Psi} \in \Ycal$.
\end{defn}

Let us briefly explain how these formal and $q$-characters are \textit{finer morphisms} than the classical character map.
Define $\Ecal \subset (\h^{*})^{\Zbb}$ to be the set of functions $\h^{*} \rightarrow \Zbb$ that are supported on a finite union of cones $\bigcup_{j=1}^{r} (\mu_{j} - Q^{+})$ for some $\mu_{j} \in \h^{*}$, and $e_{\lambda}$ to be the indicator function of each $\lambda \in \h^{*}$.
Then $\Ecal$ is equipped with a natural ring structure by setting
$e_{\lambda} e_{\mu} = e_{\lambda + \mu}$
for all $\lambda,\mu\in\h^{*}$.

\begin{defn}
    The (classical) character map $\ch : \Oint \rightarrow \Ecal$ is given by
    $\ch(V) = \sum_{\lambda\in\h^{*}} \dim(V_{\lambda}) e_{\lambda}$.
\end{defn}

Crucially, this morphism is injective on simple objects: $\ch(V(\lambda)) = \ch(V(\mu))$ implies that $\lambda = \mu$.
However, if we define $\res : \Oaff \rightarrow \Oint$ as restriction to $\langle q^{h},\, \xpm_{i,0} ~|~ h\in P^{\vee},\, i\in I \rangle$, then $\ch\circ\res$ fails to distinguish the irreducible modules in $\Oaff$.
Nevertheless, there exist natural maps
\begin{itemize}
    \item $\beta : \Ecal_{\ell} \rightarrow \Ecal$ defined using the projection $P_{\ell} \rightarrow P$ to the first factor,
    \item $\gamma : \Ycal \rightarrow \Ecal$ extended linearly from
    $k_{\nu(\omega)} \prod Y_{i,a}^{u_{i,a}} \mapsto e(\omega)$,
\end{itemize}
for which we have the following commutative diagrams:
\[\begin{tikzcd}[ampersand replacement=\&]
	\Oaff \&\& {\Ecal_{\ell}} \&\&\& \Oaff \&\& \Ycal \\
	\\
	\Oint \&\& \Ecal \&\&\& \Oint \&\& \Ecal
	\arrow["{\ch_{q}}", from=1-1, to=1-3]
	\arrow["\res"', from=1-1, to=3-1]
	\arrow["\beta", from=1-3, to=3-3]
	\arrow["{\chi_{q}}", from=1-6, to=1-8]
	\arrow["\res"', from=1-6, to=3-6]
	\arrow["\gamma", from=1-8, to=3-8]
	\arrow["\ch", from=3-1, to=3-3]
	\arrow["\ch", from=3-6, to=3-8]
\end{tikzcd}\]

The different character maps $\chi_{q}$, $\ch_{q}$ and $\ch$ depend only on the isomorphism class of a representation, and therefore linearly extend to homomorphisms from the Grothendieck groups $K(\Oaff)$ and $K(\Oint)$.

\begin{prop}
    \cite{Hernandez05}
    The $q$-character morphism
    $\chi_{q} : K(\Oaff) \rightarrow \Ycal$
    is injective.
\end{prop}

For quantum affine algebras, when $\s$ is of finite type, this can be upgraded to a ring homomorphism.
Indeed, $K(\Oaff)$ possesses a natural multiplication in this setting, coming from the coproduct for $\Udash$ and ensuing tensor product on $\Oaff$.
Frenkel and Reshetikhin then proved \cite{FR99} that the $q$-character morphism is compatible with these structures.
Obtaining quantum toroidal analogues of the tensor product on $\Oaff$, ring structure for $K(\Oaff)$, and their compatibility with $q$-characters are some of the major results in this paper -- see Sections \ref{section:tensor products} and \ref{section:q-characters}.

\begin{rmk}
    It is clear that $\chi_{q}$ is $\Cbb^{\times}$-equivariant with respect to the spectral actions on $K(\Oaff)$ and $\Ycal$ defined by $\sfrakv_{a}$ and $Y_{i,b} \mapsto Y_{i,ab}$ respectively.
    In the case of quantum affine algebras, this sets the stage for working \emph{modulo general position} with certain monoidal subcategories of $\Oaff$ which also play a crucial role in the relation to cluster algebras \cite{HL10}.
\end{rmk}

\subsubsection{Finite presentation} \label{subsubsection:finite presentation}

While the original definition of $\Uqaffs$ involves infinitely many generators and relations, the author obtained in \cite{Laurie24a}*{Prop. 4.8} a surprising \emph{finite} presentation whenever $a_{ij}a_{ji}\leq 3$ for all distinct $i,j\in I$, ie. the Dynkin diagram has at most triple arrows.
The condition on arrows was required since our proof uses the Drinfeld-Jimbo realization for each $\U(i,j)$ subalgebra.
This presentation played a crucial role in defining the braid group action on $\Uqaffs$ (see Section \ref{subsubsection:affinized braid group action}) and other subsequent results in \cite{Laurie24a}.
\\

We remark that in the specific case of $\s = \widehat{\mathfrak{sl}}_{n+1}$ ($n \geq 2$), a finite presentation and braid group action for $\Uqaffs = \UtorA$ were first shown by Miki \cite{Miki99}.
Furthermore, in a subsequent work \cite{Miki01} they obtained such results for $\s = \widehat{\mathfrak{sl}}_{2}$ but with the finite presentation involving extra generators and relations.
\\

Combining and extending our work in \cite{Laurie24a} with that of \cite{Miki01}, here we are able to upgrade these results to hold for all quantum affinizations where $a_{ij}a_{ji}\leq 3$ or $a_{ij} = a_{ji} = -2$ for all distinct $i,j\in I$.
From now on we shall call this condition (D).

\begin{rmk}
    For our purposes in later sections, it is essential that the finite presentation includes $\xpm_{0,\pm 1}$ as generators -- rather than $\xpm_{0,\mp 1}$ as in \cites{Laurie24a,Miki99} -- since our proof of Theorem \ref{thm:psi} requires the key observation that $\psi(\xpm_{0,\pm 1}) = \xpm_{0,\pm 1}$.
\end{rmk}

For notational convenience, we assume that the coweight lattice $P^{\vee}$ is spanned by the fundamental coweights.
However, this result can be extended to include scaling elements simply by adjoining the corresponding $q^{\pm h}$ generators and imposing any relations in Definition \ref{defn:quantum affinization} which involve them.

\begin{thm} \label{thm:finite Utor presentation}
    Let $\s$ be a symmetrizable Kac-Moody Lie algebra with generalised Cartan matrix $(a_{ij})_{i,j\in I}$ satisfying condition (D).
    Then the quantum affinization $\widehat{U_{q}(\s)}$ has a finite presentation with generators
\begin{itemize}
    \item $C^{\pm 1}$, $k_{i}^{\pm 1}$, $\xpm_{i,0}$, $\xpm_{i,\pm 1}$ for all $i\in I,
    \hfill \refstepcounter{equation}(\theequation)\label{eqn:simplified generators 1}$
    \item $\xpm_{i,\mp 1}$ whenever some $a_{ij} = a_{ji} = -2,
    \hfill \refstepcounter{equation}(\theequation)\label{eqn:simplified generators 2}$
\end{itemize}
    and the following relations:
\begin{itemize}
    \item $C^{\pm 1}$ central,
    $\hfill \refstepcounter{equation}(\theequation)\label{eqn:simplified relations 1}$
    \item $\displaystyle C^{\pm 1}C^{\mp 1} = k_{i}^{\pm 1} k_{i}^{\mp 1} = 1,
    \hfill \refstepcounter{equation}(\theequation)\label{eqn:simplified relations 2}$
    \item $\displaystyle [k_{i},k_{j}] = 0,
    \hfill \refstepcounter{equation}(\theequation)\label{eqn:simplified relations 3}$
    \item $\displaystyle k_{i} \xpm_{j,m} k_{i}^{-1} = q_{i}^{\pm a_{ij}} \xpm_{j,m},
    \hfill \refstepcounter{equation}(\theequation)\label{eqn:simplified relations 4}$
    \item $\displaystyle [\xp_{i,m},\xm_{i,-m}] = \frac{C^{m} k_{i} - C^{-m} k_{i}^{-1}}{q_{i}-q_{i}^{-1}},
    \hfill \refstepcounter{equation}(\theequation)\label{eqn:simplified relations 5}$
    \item $\displaystyle [\xp_{i,\pm 1},\xm_{i,0}]
    = C [\xp_{i,0},\xm_{i,\pm 1}],
    \hfill \refstepcounter{equation}(\theequation)\label{eqn:simplified relations 6}$
    \item $\displaystyle [\xp_{i,m},\xm_{j,\ell}] = 0$ if $i\not= j,
    \hfill \refstepcounter{equation}(\theequation)\label{eqn:simplified relations 7}$
    \item $[\xpm_{i,m+1},\xpm_{j,\ell}]_{q_{i}^{\pm a_{ij}}} + [\xpm_{j,\ell+1},\xpm_{i,m}]_{q_{i}^{\pm a_{ij}}} = 0,
    \hfill \refstepcounter{equation}(\theequation)\label{eqn:simplified relations 8}$
\end{itemize}
whenever all generators involved are present; when $a_{ij}a_{ji} \leq 3$,
\begin{itemize}
    \item $\displaystyle
    \sum_{s=0}^{1 - a_{ij}} (-1)^{s}
    {\begin{bmatrix}1 - a_{ij}\\s\end{bmatrix}}_{i}
    y_{i}^{s} y_{j} y_{i}^{1 - a_{ij}-s} = 0,
    \hfill \refstepcounter{equation}(\theequation)\label{eqn:simplified relations 9}$
\end{itemize}
for $(y_{i},y_{j}) = (x^{\pm}_{i,0},x^{\pm}_{j,0}),(x^{\pm}_{i,\pm 1},x^{\pm}_{j,0}),(x^{\pm}_{i,0},x^{\pm}_{j,\pm 1})$; and when $a_{ij} = a_{ji} = -2$,
\begin{itemize}
    \item $\displaystyle [\xp_{i,2},\xp_{i,1}]_{q_{i}^{2}}
    = [\xm_{i,-1},\xm_{i,-2}]_{q_{i}^{-2}}
    = 0,
    \hfill \refstepcounter{equation}(\theequation)\label{eqn:simplified relations 10}$
    \item $[\xp_{i,1},\xp_{j,1}]_{q_{i}^{-2}} +
    [\xp_{j,2},\xp_{i,0}]_{q_{i}^{-2}}$
    is central,
    $\hfill \refstepcounter{equation}(\theequation)\label{eqn:simplified relations 11}$
    \item $[\xm_{i,-1},\xm_{j,-1}]_{q_{i}^{2}} + [\xm_{j,0},\xm_{i,-2}]_{q_{i}^{2}}$
    is central,
    $\hfill \refstepcounter{equation}(\theequation)\label{eqn:simplified relations 12}$
    \item $\displaystyle
    \sum_{s=0}^{1 - a_{ij}} (-1)^{s}
    {\begin{bmatrix}1 - a_{ij}\\s\end{bmatrix}}_{i}
    (\xpm_{i,0})^{s} \xpm_{j,0} (\xpm_{i,0})^{1 - a_{ij}-s} = 0,
    \hfill \refstepcounter{equation}(\theequation)\label{eqn:simplified relations 13}$
\end{itemize}
where we define
$\xpm_{i,\pm 2} = \pm [2]_{i}^{-1} [h_{i,\pm 1},\xpm_{i,\pm 1}]$,
$h_{i,1} = k_{i}^{-1} [\xp_{i,1},\xm_{i,0}]$
and
$h_{i,-1} = k_{i} [\xp_{i,0},\xm_{i,-1}]$.
\end{thm}

We would like to use the results of \cite{Miki01} in our proof of Theorem \ref{thm:finite Utor presentation}, as well as later on in this paper.
However, there are minor differences between our definition of $\Utorsltwo$ -- as the quantum affinization of $U_{q}(\widehat{\mathfrak{sl}}_{2})$ -- and that of Miki \cite{Miki01}, which includes relations (\ref{eqn:quantum affinization relations 9}) only for $i = j$ and the affine $q$-Serre relations (\ref{eqn:quantum affinization relations 10}) only with $m = m_{1} = \dots = m_{a'} = 0$.
The following lemma allows us to circumvent this issue.

\begin{lem} \label{lem:equivalent definitions of Utorsl2}
    The definition of
    $\Utorsltwo$
    presented in \cite{Miki01} is equivalent to that of Definition \ref{defn:quantum affinization}.
\end{lem}
\begin{proof}
    Miki proved \cite{Miki01}*{Lem. 3} that all affine $q$-Serre relations (\ref{eqn:quantum affinization relations 10}) hold \emph{as a consequence} of the relations included in their definition of $\Utorsltwo$.
    Furthermore, Damiani mentions in \cite{Damiani24}*{Rmk. 2.11} that (\ref{eqn:quantum affinization relations 9}) is redundant outside the rank $1$ case $\widehat{U_{q}(\mathfrak{sl}_{2})}$, referencing her earlier work \cite{Damiani12} for the proof -- in particular, see Remarks §9.10 and §11.10 there.
\end{proof}

\begin{proof}[Proof of Theorem \ref{thm:finite Utor presentation}]
Define an algebra $\Acal$ with generators
(\ref{eqn:simplified generators 1})--(\ref{eqn:simplified generators 2})
and relations
(\ref{eqn:simplified relations 1})--(\ref{eqn:simplified relations 13}),
and pick some rank $2$ subalgebra
$\Acal(k,\ell) = \langle \text{(\ref{eqn:simplified generators 1})--(\ref{eqn:simplified generators 2})} \,|\, i = k,\ell \rangle$
where $k\not= \ell$.
We would like to check that sending
\begin{align} \label{eqn:rank 2 morphism}
    C^{\pm 1} \mapsto C^{\pm 1}, \qquad
    k_{i}^{\pm 1} \mapsto k_{i}^{\pm 1}, \qquad
    \xpm_{i,m} \mapsto \xpm_{i,m},
\end{align}
for all generators
(\ref{eqn:simplified generators 1})--(\ref{eqn:simplified generators 2})
with $i = k,\ell$ extends to a well-defined isomorphism $p_{k\ell} : \Acal(k,\ell) \xrightarrow{\sim} \U(k,\ell)$.
If all generators are of type (\ref{eqn:simplified generators 1}) then this follows by applying $\eta$ to \cite{Laurie24a}*{Prop. 4.8}, while the case $a_{k\ell} = a_{\ell k} = -2$ comes from \cite{Miki01}*{Prop. 5} and Lemma \ref{lem:equivalent definitions of Utorsl2}.
\\

If $a_{k\ell} a_{\ell k} \leq 3$ but generators of the form (\ref{eqn:simplified generators 2}) are present, we furthermore let $\Hcal(k,\ell)$ be the algebra with generators (\ref{eqn:simplified generators 1}) for $i = k,\ell$ and relations (\ref{eqn:simplified relations 1})--(\ref{eqn:simplified relations 9}).
Then (\ref{eqn:rank 2 morphism}) defines an isomorphism
$\Hcal(k,\ell) \xrightarrow{\sim} \U(k,\ell)$
by applying $\eta$ to \cite{Laurie24a}*{Prop. 4.8}, as well as well-defined morphisms
$\Hcal(k,\ell) \rightarrow \Acal(k,\ell)$ and
$p_{k\ell} : \Acal(k,\ell) \rightarrow \U(k,\ell)$
since we are only imposing more relations.
Any valid composition of all three maps is by definition the identity, and so $p_{k\ell}$ must be an isomorphism.
\\

Letting
$\Acal(k) = \langle \text{(\ref{eqn:simplified generators 1})--(\ref{eqn:simplified generators 2})} \,|\, i = k \rangle$
for each $k\in I$, it is clear that
$p_{k} = p_{k\ell}\vert_{\Acal(k)} = p_{\ell k}\vert_{\Acal(k)}$
is well-defined and independent of $\ell$, whereby (\ref{eqn:rank 2 morphism}) clearly extends to an isomorphism $\Acal \xrightarrow{\sim} \Uqaffs$.
\end{proof}

Of course, such finite presentations can be incredibly useful when defining morphisms to and from these algebras, as well as for verifying well-definedness, surjectivity, and so on.
Indeed, Theorem \ref{thm:finite Utor presentation} plays a key role in constructing our braid group action in Section \ref{subsubsection:affinized braid group action}, as well as our definition of $\psi$ and proof that it is an anti-involution in Section \ref{section:horizontal-vertical symmetries}.

\begin{rmk}
    \begin{itemize}
        \item This result gives a finite Drinfeld new style presentation for the quantum toroidal algebra $\Utor = \widehat{\Uaff}$ in all untwisted and twisted types except $A_{2}^{(2)}$, as well as for all untwisted quantum affine algebras $\Uaff \cong \Uqaffg$.
        \item Relations (\ref{eqn:simplified relations 1})--(\ref{eqn:simplified relations 13}) are a subset of those in the original definition for $\Uqaffs$ which only involve generators (\ref{eqn:simplified generators 1})--(\ref{eqn:simplified generators 2}).
        In particular, we do not see `shadows' of other relations appearing in our simplified presentation.
    \end{itemize}
\end{rmk}

Note that we do not propose that our presentation in Theorem \ref{thm:finite Utor presentation} is minimal -- indeed, it should be possible to remove certain relations and further strengthen this result.
However, it is enough for the purposes of this paper and so we leave such considerations for now.

\begin{Q}
    Does such a finite presentation exist for all $\Uqaffs$, without assuming condition (D)?
\end{Q}

The fact that existence holds in cases where not every $\U(i,j)$ is isomorphic to an (untwisted) quantum affine algebra indicates that the answer might be \emph{yes}.
The author hopes to return to this question in future work.

\subsubsection{Braid group action} \label{subsubsection:affinized braid group action}

Here we present an affinized version of the braid group action from Section \ref{subsubsection:braid group action}, which will play a fundamental role in our proof of Theorem \ref{thm:psi}.
In finite types, when $\Uqaffs$ is an untwisted quantum affine algebra (see Section \ref{subsubsection:quantum affine algebras}), this result originally appeared in work of Beck \cite{Beck94}.
However, this really comes as a consequence of Lusztig's braid group action on $\Uaff$ (Theorem \ref{thm:braid group action}) and the Bernstein presentation for $\Bd$, rather than being proven \emph{on the level of affinizations}.
\\

Moving beyond the finite case, Miki addressed the quantum toroidal algebras $\UtorA$ \cite{Miki99} and $\Utorsltwo$ \cite{Miki01}, ie. when $\s$ is of type $A_{n}^{(1)}$.
Subsequently, the author \cite{Laurie24a} treated all quantum affinizations with $a_{ij}a_{ji}\leq 3$ for every distinct $i,j\in I$.
The following combines and extends the work done there with that of \cite{Miki01}.
\\

For each $i\in I$, we wish to define an automorphism $\T_{i}$ of $\Uqaffs$ whose restriction to $\U(i) \cong \UdashA$ coincides with that of $\Tb_{1}$ from Section \ref{subsubsection:braid group action}, ie.
$\T_{i} \circ h_{i} = h_{i} \circ \Tb_{1}$.
To this end, note that
\begin{align*}
    -\Tb_{1}^{-1}(\xp_{0})
    &= -\frac{1}{[2]} [ [\xp_{0}, \xp_{1}]_{q^{-2}}, \xp_{1}]
    = -\frac{1}{[2]} C [ [k_{1}^{-1} \xm_{1,1}, \xp_{1,0}]_{q^{-2}}, \xp_{1,0}]
    \\
    &= -\frac{1}{[2]} C [ k_{1}^{-1} [\xm_{1,1}, \xp_{1,0}], \xp_{1,0}]
    = \frac{1}{[2]} [h_{1,1}, \xp_{1,0}]
    \\
    &= \xp_{1,1}
\end{align*}
and similarly $-\Tb_{1}^{-1}(\xm_{0}) = \xm_{1,-1}$, hence
$\Tb_{1}(\xp_{1,1}) = - C k_{1}^{-1} \xm_{1,1}$ and
$\Tb_{1}(\xm_{1,-1}) = - \xp_{1,-1} k_{1} C^{-1}$.
Furthermore, using the fact that
$\Tb_{i}^{-1} = \eta \Tb_{1}^{-1} \eta$ we can then prove that
\begin{align*}
    \Tb_{i}(\xp_{i,-1})
    &= k_{i}^{2} \sum_{s=0}^{2} (-1)^{s} q_{i}^{3s} (\xm_{i,0})^{(s)} \xp_{i,-1} (\xm_{i,0})^{(2-s)},
    \\
    \Tb_{i}(\xm_{i,1})
    &= \sum_{s=0}^{2} (-1)^{s} q_{i}^{-3s}(\xp_{i,0})^{(2-s)} \xm_{i,\pm 1} (\xp_{i,0})^{(s)} k_{i}^{-2}.
\end{align*}
We moreover want $\T_{i}$ to commute with $\X_{i}$ for all $j\not= i$, and its restriction to
$\langle q^{h},\, \xpm_{i,0} ~|~ h\in P^{\vee},\, i\in I \rangle$
to coincide with the action of $T_{i}$ on $\Uqs$ from Theorem \ref{thm:braid group action}.
Therefore, let $\T_{i}$ act on the generators (\ref{eqn:simplified generators 1})--(\ref{eqn:simplified generators 2}) of our finite presentation from Theorem \ref{thm:finite Utor presentation} as follows:
\begin{itemize}
    \item $\displaystyle
    \T_{i}(C^{\pm 1}) = C^{\pm 1}$,
    \item $\displaystyle
    \T_{i}(q^{h}) = q^{s_{i}(h)}$,
    \item $\displaystyle
    \T_{i}(\xp_{i,0}) = - \xm_{i,0} k_{i}$,
    \item $\displaystyle
    \T_{i}(\xm_{i,0}) = - k_{i}^{-1} \xp_{i,0}$,
    \item $\displaystyle
    \T_{i}(\xp_{i,1})
    = [2]_{i}^{-1} k_{i}^{-2} [ [\xp_{i,1}, \xm_{i,0}]_{q_{i}^{-2}}, \xm_{i,0}]
    = - C k_{i}^{-1} \xm_{i,1}$,
    \item $\displaystyle
    \T_{i}(\xm_{i,-1})
    = [2]_{i}^{-1} [\xp_{i,0}, [\xp_{i,0}, \xm_{i,-1}]_{q_{i}^{2}}] k_{i}^{2}
    = - \xp_{i,-1} k_{i} C^{-1}$,
    \item $\displaystyle
    \T_{i}(\xp_{i,-1})
    = [2]_{i}^{-1} k_{i}^{2} [ [\xp_{i,-1}, \xm_{i,0}]_{q_{i}^{4}}, \xm_{i,0}]_{q_{i}^{2}}
    = k_{i}^{2} \sum_{s=0}^{2} (-1)^{s} q_{i}^{3s} (\xm_{i,0})^{(s)} \xp_{i,-1} (\xm_{i,0})^{(2-s)}$,
    \item $\displaystyle
    \T_{i}(\xm_{i,1})
    = [2]_{i}^{-1} [\xp_{i,0}, [\xp_{i,0}, \xm_{i,1}]_{q_{i}^{-4}}]_{q_{i}^{-2}} k_{i}^{-2}
    = \sum_{s=0}^{2} (-1)^{s} q_{i}^{-3s}(\xp_{i,0})^{(2-s)} \xm_{i,\pm 1} (\xp_{i,0})^{(s)} k_{i}^{-2}$,
    \item $\displaystyle
    \T_{i}(\xp_{j,m}) = \sum_{s=0}^{-a_{ij}} (-1)^{s} q_{i}^{-s} (\xp_{i,0})^{(-a_{ij}-s)} \xp_{j,m} (\xp_{i,0})^{(s)}$
    if $i\not= j$,
    \item $\displaystyle
    \T_{i}(\xm_{j,m}) = \sum_{s=0}^{-a_{ij}} (-1)^{s} q_{i}^{s} (\xm_{i,0})^{(s)} \xm_{j,m} (\xm_{i,0})^{(-a_{ij}-s)}$
    if $i\not= j$.
\end{itemize}

\begin{prop} \label{prop:defining Ti}
    The above extends to a well-defined automorphism $\T_{i}$ of $\Uqaffs$ with inverse
    $\T_{i}^{-1} = \eta\T_{i}\eta$
    whenever condition (D) holds.
\end{prop}
\begin{proof}[Proof sketch.]
    Checking that $\T_{i}$ respects the relations of $\Uqaffs$ reduces to a check on each $\U(i,j,\ell)$.
    If $\#\lbrace i,j,\ell\rbrace < 3$ then $\U(i,j,\ell)$ is isomorphic to one of
    \begin{align*}
        \UdashA\times\UdashA, \qquad
        \UdashA, \qquad
        \UdashAA, \qquad
        \UdashC, \qquad
        \UdashG, \qquad
        \Utorsltwo,
    \end{align*}
    and this is covered by the affine case together with \cite{Miki01}*{Prop. 6}.
    Otherwise, all relations involving only
    $\lbrace k_{j}^{\pm 1}, k_{\ell}^{\pm 1}, \xpm_{j,0}, \xpm_{\ell,0} \rbrace$
    are preserved due to Theorem \ref{thm:braid group action}.
    The rest then follow by applying $\X_{j}$ and $\X_{\ell}$, which in particular commute with $\T_{i}$.
    Similarly, invertibility of $\T_{i}$ is verified on each $\U(i,j)$ and follows from the affine case and \cite{Miki01}.
    See the proof of \cite{Laurie24a}*{Prop. 4.10} for more details.
\end{proof}

\begin{rmk}
    There is a small error in the formulae for
    $\T_{i}(\xpm_{i,\mp 1}) = \Tb_{i}(\xpm_{i,\mp 1})$
    found on p.9 and p.18 of \cite{Laurie24a}, which should instead read as above.
    This does not impact any of the other work done there.
\end{rmk}

We now have all of the automorphisms required to define our `affinized braid group action' on $\Uqaffs$.

\begin{defn} \label{defn:affinized braid group}
    For any generalised Cartan matrix $(a_{ij})_{i,j\in I}$ we define the \emph{affinized braid group} $\widehat{\B}$ to be the group generated by $\lbrace T_{i},\, X_{i}~|~i\in I \rbrace$ and the automorphism group $\Omega$ of the associated Dynkin diagram, with relations
    \begin{itemize}
        \item $T_{i}T_{j}T_{i}\ldots = T_{j}T_{i}T_{j}\ldots$ whenever $a_{ij}a_{ji}\leq 3$, with $a_{ij}a_{ji} + 2$ factors on each side,
        \item $X_{i}X_{j} = X_{j}X_{i}$,
        \item $T_{i}X_{j} = X_{j}T_{i}$ whenever $i\not= j$,
        \item $T_{i}^{-1}X_{i}T_{i}^{-1} = X_{i} \prod_{j\in I}X_{j}^{-a_{ij}}$,
        \item $\pi T_{i} \pi^{-1} = T_{\pi(i)}$,
        \item $\pi X_{i} \pi^{-1} = X_{\pi(i)}$,
    \end{itemize}
    for all $i,j\in I$ and $\pi\in\Omega$.
\end{defn}

When the underlying Dynkin diagram satisfies condition (D) and moreover possesses a sign function $o$, we have the following.

\begin{thm} \label{thm:affinized braid group action}
    The group $\widehat{\B}$ acts on the quantum affinization $\Uqaffs$ via $T_{i} \mapsto \T_{i}$ and $X_{i} \mapsto \X_{i}$ for all $i\in I$, and $\pi \mapsto \Scal_{\pi}$ for all $\pi\in \Omega$.
\end{thm}
\begin{proof}[Proof sketch.]
    Commutativity of $\T_{i}$ and $\X_{j}$ for $i\not= j$ is clear from the definitions, while
    $\T_{i}^{-1}\X_{i}\T_{i}^{-1} = \X_{i}\prod_{j\in I}\X_{j}^{-a_{ij}}$
    is checked by restricting to each $\U(i,\ell)$.
    In particular, since $\U(i,\ell)$ is isomorphic to one of
    \begin{align*}
        \UdashA\times\UdashA, \qquad
        \UdashA, \qquad
        \UdashAA, \qquad
        \UdashC, \qquad
        \UdashG, \qquad
        \Utorsltwo,
    \end{align*}
    this is covered by the affine case and \cite{Miki01}*{Prop. 6}.
    The braid relation between $\T_{i}$ and $\T_{j}$ on elements of $\U(\ell)$ is checked on $\U(i,j,\ell)$.
    Similarly to our proof of Proposition \ref{prop:defining Ti}, if $\#\lbrace i,j,\ell\rbrace < 3$ then we are done by either the affine case or \cite{Miki01}.
    Otherwise, the braid relation clearly holds on $k_{\ell}^{\pm 1}$ and $\xpm_{\ell,0}$ by Theorem \ref{thm:braid group action}, and we reach the other generators of $\U(\ell)$ from Theorem \ref{thm:finite Utor presentation} by applying $\X_{\ell}^{-1}$.
    The remaining relations of Definition \ref{defn:affinized braid group} are checked without much difficulty.
    See \cite{Laurie24a}*{Thm. 4.11} for more details.
\end{proof}

If no such $o$ exists as the Dynkin diagram contains an odd length cycle, Theorem \ref{thm:affinized braid group action} should instead hold for a slightly modified version of $\widehat{\B}$.
Let us illustrate this in the case of the cyclic $A_{2n}^{(1)}$ quiver.
First, $\pi_{1}\in\Omega$ must have order $4n+2$ in $\widehat{\B}$ rather than $2n+1$.
This is because, as discussed in Section \ref{subsection:basic notations}, there is no sign function on the affine Dynkin diagram and so
$\Scal_{\pi_{1}}^{2n+1} = \sfrakZ_{-1}$
has order $2$, mapping
\begin{align*}
    \xpm_{i,m} \mapsto (-1)^{m}\xpm_{i,m}, \qquad
    h_{i,r} \mapsto (-1)^{r}h_{i,r}, \qquad
    k_{i} \mapsto k_{i}, \qquad
    C \mapsto C.
\end{align*}
The automorphism $\zeta_{j} = \sfrakj_{-1}$ maps each $\xpm_{j,m} \mapsto -\xpm_{j,m}$ and fixes the other generators, and we have that
\begin{align*}
    \Scal_{\pi_{1}} \zeta_{j} \Scal_{\pi_{1}}^{-1} = \zeta_{\pi_{1}(j)}, \qquad
    \Scal_{\pi_{1}} \X_{2n} \Scal_{\pi_{1}}^{-1} = \zeta_{0}\X_{0}, \qquad
    \T_{0}^{-1} \X_{0} \T_{0}^{-1} = \zeta_{0}\X_{2n}\X_{0}^{-1}\X_{1}.
\end{align*}
By adding $\zeta_{0}$ as a generator in $\widehat{\B}$ and adjusting the group relations with respect to the above discussion, we are able to extend Theorem \ref{thm:affinized braid group action} to include type $A_{2n}^{(1)}$ via essentially the same proof.
Similar methods allow us to further generalise to all $\s$ satisfying condition (D).

\begin{rmk}
    In the case of quantum toroidal algebras (when $\sfrak$ is an affine Lie algebra) we shall see in Section \ref{section:Quantum toroidal algebras} that this action restricts to the extended double affine braid group $\Bdd$.
    This is important as $\Bdd$ possesses an involution $\te$ which is essential for defining our horizontal--vertical symmetry $\psi$ of $\Utor$.
\end{rmk}

\begin{Q}
    Does such an affinized braid group action exist for all $\Uqaffs$, without assuming condition (D)?
\end{Q}

As in Section \ref{subsubsection:finite presentation} we expect our results to extend to all quantum affinizations, and leave such directions for future work.

\subsubsection{Quantum affine algebras} \label{subsubsection:quantum affine algebras}

In untwisted types, the quantum affine algebra has an alternative \textit{Drinfeld new presentation}, first stated by Drinfeld \cite{Drinfeld88}, as the quantum affinization of the corresponding finite quantum group.
The equivalence of the two realizations is precisely the commutativity of the following diagram, taken from \cite{Hernandez09}.
\[\begin{tikzcd}
	\g & & & & \gaff \\
	\Uq & & & & \Udash
	\arrow["\mathrm{Quantization}"', from=1-1, to=2-1]
	\arrow["\mathrm{Quantum~Affinization}"', from=2-1, to=2-5]
	\arrow["\mathrm{Affinization}", from=1-1, to=1-5]
	\arrow["\mathrm{Quantization}", from=1-5, to=2-5]
\end{tikzcd}\]

Furthermore, extending $\Uqaffg$ with the degree-style generators $D^{\pm 1}$ corresponding to $\sfrakZ_{q}$ (see Remark \ref{rmk:degree operators for quantum affinizations}) produces a similar presentation for $\Uaff$.
\\

The Drinfeld new realization quantizes the loop presentation for untwisted affine Lie algebras, and has been immensely useful for studying the representation theory of $\Uaff$ and $\Udash$.
In particular, it was implemented by Chari and Pressley in a systematic treatment of the finite dimensional modules and their $R$-matrices \cites{CP91,CP94,CP95,CP97}, as well as by Frenkel and Jing \cites{FJ88,Jing89} to construct vertex representations.
\\

The relationship between the two realizations was first studied by Beck \cite{Beck94}, who used the Bernstein presentation for $\Bd$ and its action on the quantum affine algebra to construct a morphism from $\Uqaffg$ to $\Udash$.
Jing \cite{Jing98} then defined an inverse morphism using $q$-commutators, while Damiani proved the surjectivity \cite{Damiani12} and injectivity \cite{Damiani15} of Beck's map.

\begin{rmk}
    A generalization of the Drinfeld new realization which includes all twisted types was also stated in \cites{Drinfeld88}.
    A morphism from the Drinfeld-Jimbo presentation was initially defined by Jing and Zhang \cites{JZ07,JZ10}, but the proof of an isomorphism between the two presentations was once again completed by Damiani in \cites{Damiani12,Damiani15}.
    (It is worth noting that the affine $q$-Serre relations in \cites{Damiani12,Damiani15} differ slightly from those in \cites{JZ07,JZ10}.)
    Furthermore, the construction of vertex representations was extended to twisted types in \cite{Jing90}.
    However, we omit the twisted case here as it is not required for our purposes.
\end{rmk}

Let us now present Jing's isomorphism.
For each $i_{1}\in I_{0}$ there exist sequences $\underline{i} = (i_{1},i_{2},\dots,i_{\hslash -1})$ in $I_{0}$ and $\underline{\epsilon} = (\epsilon_{1},\dots,\epsilon_{\hslash -2})$ in $\mathbb{Q}_{\leq 0}$ such that
\begin{align*}
    (\alpha_{i_{1}}+\dots +\alpha_{i_{s}}, \alpha_{i_{s+1}}) = \epsilon_{s} \mathrm{~for~} s=1,\dots,\hslash -2,
\end{align*}
where we recall that $\hslash = \sum_{i\in I}a_{i}$ is the Coxeter number of $\gaff$.
Then for any such sequences, the following extends to a $\kk$-algebra isomorphism from the Drinfeld-Jimbo realization of $\Udash$ to its Drinfeld new realization as the quantum affinization $\Uqaffg$:
\begin{itemize}
    \item $\xipm \mapsto \xpm_{i,0}$ and $k_{i} \mapsto k_{i}$ for each $i\in I_{0}$,
    \item $\xp_{0} \mapsto \left[\xm_{i_{\hslash -1},0},\dots,\xm_{i_{2},0},\xm_{i_{1},1}\right]_{q^{\epsilon_{1}}\dots q^{\epsilon_{\hslash -2}}} C k_{\theta}^{-1}$,
    \item $\xm_{0} \mapsto a(-q)^{-\epsilon} C^{-1} k_{\theta} \left[\xp_{i_{\hslash -1},0},\dots,\xp_{i_{2},0},\xp_{i_{1},-1}\right]_{q^{\epsilon_{1}}\dots q^{\epsilon_{\hslash -2}}}$,
    \item $k_{0} \mapsto C k_{\theta}^{-1}$,
\end{itemize}
where $k_{\theta} = k_{1}^{a_{1}}\dots k_{n}^{a_{n}}$, $\epsilon = \epsilon_{1}+\dots+\epsilon_{\hslash -2}$, and $a$ is a constant depending on type (in particular $a=1$ when $\gaff$ is simply laced). Example sequences in all types can be found in \cite{Jing98}*{Table 2.1}.
Furthermore, the above isomorphism extends to $\Uaff$ by sending $q^{d} \mapsto D$.
\\

It is clear in both presentations that $\Uaff$ and $\Udash$ contain a natural copy of the finite quantum group $\Uq$ -- it is the subalgebra generated by $\lbrace \xipm,\, k_{i}^{\pm 1}~|~i\in I_{0}\rbrace$ in the Drinfeld-Jimbo, and by $\lbrace \xpm_{i,0},\, k_{i}^{\pm 1}~|~i\in I_{0}\rbrace$ in the Drinfeld new.
\\

We shall now specialise some of the earlier results in this subsection to the particular case of untwisted quantum affine algebras.
First, it is important to note that none of the topological coproducts introduced in Section \ref{subsubsection:topological coproducts} coincide with any of the coproducts from Section \ref{subsubsection:coproducts}.
Instead, Damiani \cite{Damiani24} formulates in a precise way the notion of $\Delta_{u}$ as a ``$P$-equivariant deformation of $\Delta_{+}$'', where the actions of $\Delta_{u}$ and $\Delta_{+}$ on the Drinfeld new generators differ by some ``controllable terms''.
\\

As for the representation theory, it is clear that Sections \ref{subsubsection:highest weight theory} and \ref{subsubsection:l-highest weight theory} provide different definitions of integrability for representations of $\Uaff$ and $\Udash$.
Moreover, the notions of highest weight and $\ell$-highest weight modules are distinct.
In particular, $\Oaff$ is precisely the category of finite dimensional modules.
Since for any $\ell$-weight we have that $\Psi$ determines $\lambda = \sum_{i\in I} \langle \lambda,\alpha^{\vee}_{i} \rangle \Lambda_{i}$ uniquely, the irreducible finite dimensional representations are therefore parametrised by Drinfeld polynomials $\Pcal(z) = (P_{i}(z))_{i\in I_{0}}$.
See the works of Chari-Pressley \cites{CP91,CP94,CP95} for more details.

\begin{notation}
    To avoid confusion in later sections, we shall denote by $\Xb_{i}$ the automorphism $\X_{i}$ of $\Uaff$ or $\Udash$ for each $i\in I_{0}$, where $\upsilon$ is the restriction to $I_{0}$ of some affine sign function $o:I\rightarrow\lbrace\pm 1\rbrace$.
    Moreover, we shall write the anti-involution $\eta$ as $\eta'$ and note that $\Tb_{i}^{-1} = \eta' \Tb_{i} \eta'$ for all $i\in I_{0}$.
\end{notation}

The following then provides a loop-style analogue of Corollary \ref{cor:Bd action} with respect to the Bernstein and Drinfeld new presentations.

\begin{thm} \cite{Beck94} \label{thm:loop Bd action}
    The extended affine braid group $\Bd$ acts on the quantum affine algebras $\Uaff$ and $\Udash$ via $T_{i} \rightarrow \Tb_{i}$ and $X_{\omega_{i}^{\vee}} \rightarrow \Xb_{i}$ for each $i\in I_{0}$.
\end{thm}

\section{Quantum toroidal algebras} \label{section:Quantum toroidal algebras}

We have seen in Section \ref{subsubsection:quantum affine algebras} how untwisted quantum affine algebras arise as a special case of the quantum affinization procedure.
By taking the quantum affinization of their Drinfeld-Jimbo realizations, we obtain another important class of algebras: the \emph{quantum toroidal algebras} $\Utor$.
These can therefore be considered as the double affine objects within the quantum setting.
\\

Quantum toroidal algebras are the quantum deformations of universal central extensions $\g[s^{\pm 1},t^{\pm 1}] \oplus \mathbb{K}$ of the toroidal Lie algebras \cite{Enriquez03} of regular rational (polynomial) maps from a complex $2$-torus into the finite dimensional simple Lie algebra $\g$, as described in \cite{MRY90}.
\\

It should be noted that quantum toroidal algebras do not occur as the Drinfeld-Jimbo quantum groups associated to any Kac-Moody algebras, similar to how double affine braid groups are not the braid groups of any Coxeter diagram and toroidal Lie algebras are not Kac-Moody algebras.
It follows that they do not themselves possess quantum affinizations via Definition \ref{defn:quantum affinization}, and are thus in some sense extremal with respect to this process.
\\

The study of quantum toroidal algebras is an incredibly rich and fruitful area of research within mathematics and physics, with a diverse range of connections and applications including -- but far from limited to -- the following:
\begin{itemize}
    \item They were first introduced in the $ADE$ case by Ginzburg-Kapranov-Vasserot \cite{GKV95} in their study of Langlands reciprocity for algebraic surfaces.
    In particular, $\Utor$ is shown to act via Hecke operators on the $\Cbb$-valued functions of a certain moduli space of vector bundles on the surface.
    \item There is a toroidal Schur-Weyl duality between $\UtorA$ and the double affine Hecke algebra $\ddot{\Hcal}$ of type $\mathfrak{gl}_{\ell}$ due to Varagnolo-Vasserot \cite{VV96}, which establishes an equivalence between right $\ddot{\Hcal}$-modules and a particular category of integrable left $\UtorA$-modules.
    \item Nakajima \cites{Nakajima01,Nakajima02} realized simply laced $\Utor$ via a morphism to the equivariant K-theory of quiver varieties on the affine Dynkin diagram.
    This was recently extended to arbitrary types (and indeed to \emph{shifted} quantum loop groups) by Varagnolo-Vasserot \cites{VV23a,VV23b} using critical K-theory, and is a powerful geometric approach for their representation theory.
    \item The type $A$ quantum toroidal algebras and their Miki automorphisms provide a remarkable algebraic framework and set of tools for studying symmetric function theory, such as the (wreath) Macdonald polynomials -- see \cites{OS24,OSW22,Wen19} and references therein.
    \item Quantum toroidal algebras enjoy a wealth of applications into quantum integrable systems.
    Even just in the $\glone$ case, their representation theory and $R$-matrices are fundamental for solving $XXZ$ type models via Bethe ansatz techniques \cites{FJMM15,FJMM17,FJM19}.
\end{itemize}
However, despite these many varied directions, quantum toroidal algebras remain rather mysterious objects.
Further developing our understanding of their structure and representation theory is therefore of fundamental significance, and deserves continued attention.
\\

In this section we shall define the quantum toroidal algebras and some of their basic structures, before introducing the corresponding objects within the braid group setting -- the extended double affine braid groups $\Bdd$.
We will then deduce from our results in Section \ref{subsubsection:affinized braid group action} an action of $\Bdd$ on $\Utor$, as well as outline a Coxeter-style presentation for $\Bdd$ due to Ion-Sahi \cite{IS20}, each of which is essential for our work in later sections.

\begin{defn} \label{defn:quantum toroidal algebra}
    The quantum toroidal algebra $\Utor$ is the unital associative $\kk$-algebra with generators $\xpm_{i,m}$, $h_{i,r}$, $k_{i}^{\pm 1}$, $C^{\pm 1}$ ($i\in I$, $m\in\Zbb$, $r\in\Zbb^{*}$), subject to the following relations:
\begin{itemize}
    \item $C^{\pm 1}$ central,
    \item $\displaystyle C^{\pm 1}C^{\mp 1} = k_{i}^{\pm 1} k_{i}^{\mp 1} = 1$,
    \item $\displaystyle [k_{i},k_{j}] = [k_{i},h_{j,r}] = 0$,
    \item $\displaystyle [h_{i,r},h_{j,s}] = \delta_{r+s,0} \frac{[ra_{ij}]_{i}}{r} \frac{C^{r}-C^{-r}}{q_{j}-q_{j}^{-1}}$,
    \item $\displaystyle k_{i} \xpm_{j,m} k_{i}^{-1} = q_{i}^{\pm a_{ij}} \xpm_{j,m}$,
    \item $\displaystyle [h_{i,r},\xpm_{j,m}] = \pm \frac{[ra_{ij}]_{i}}{r} C^{\frac{r \mp \lvert r\rvert}{2}} \xpm_{j,r+m}$,
    \item $\displaystyle [\xp_{i,m},\xm_{j,l}] = \frac{\delta_{ij}}{q_{i}-q_{i}^{-1}} (C^{-l}\phi^{+}_{i,m+l} - C^{-m}\phi^{-}_{i,m+l})$,
    \item $[\xpm_{i,m+1},\xpm_{j,l}]_{q_{i}^{\pm a_{ij}}} + [\xpm_{j,l+1},\xpm_{i,m}]_{q_{i}^{\pm a_{ij}}} = 0$,
\end{itemize}
and whenever $i\not= j$, for any integers $m$ and $m_{1},\dots,m_{a'}$ where $a' = 1 - a_{ij}$,
\begin{itemize}
    \item $\displaystyle
    \sum_{\pi\in S_{a'}}
    \sum_{s=0}^{a'} (-1)^{s}
    {\begin{bmatrix}a'\\s\end{bmatrix}}_{i}
    \xpm_{i,m_{\pi(1)}}\dots\xpm_{i,m_{\pi(s)}}
    \xpm_{j,m}
    \xpm_{i,m_{\pi(s+1)}}\dots\xpm_{i,m_{\pi(a')}}
    = 0$.
\end{itemize}
Here, the $\phi^{\pm}_{i,\pm s}$ are given by the formula
$$ \sum_{s\geq 0} \phi^{\pm}_{i,\pm s} z^{\pm s} =
k_{i}^{\pm 1} \exp{\left( \pm (q_{i}-q_{i}^{-1})\sum_{s'>0}h_{i,\pm s'} z^{\pm s'} \right)}
$$
when $s\geq 0$, and are zero otherwise.
\end{defn}

By construction $\Utor$ possesses many of the structures introduced in Section \ref{subsection:quantum affinizations}, for example the gradings, scaling automorphisms, topological coproducts and $\ell$-highest weight theory.
Furthermore, our finite presentation and action of $\widehat{\B}$ exist in all types except $A_{2}^{(2)}$ for now, which fails condition (D).

\begin{rmk}
    \begin{itemize}
        \item Some sources -- for example \cites{Saito98,Miki00,Tsymbaliuk19} -- add horizontal or vertical degree-style generators to their definitions of $\Utor$.
        These correspond via Remark \ref{rmk:degree operators for quantum affinizations} to $\sfrak^{(0)}_{q}$ and $\sfrakZ_{q}$ respectively, with the former moreover equal to $q^{d}$.
        \item In type $A_{n}^{(1)}$ there is a two-parameter deformation $U_{q,\kappa}(\mathfrak{sl}_{n+1,\mathrm{tor}})$ where some of the relations in Definition \ref{defn:quantum toroidal algebra} are modified to involve additional central generators $\kappa^{\pm 1}$.
        The extra parameter $\kappa$ relates to the rotational symmetry of the Dynkin diagram, and specialising to $\kappa = 1$ recovers the above presentation.
        However, such a deformation is not known to exist in other types and thus will not be treated in this paper.
    \end{itemize}
\end{rmk}

So we see that the quantum toroidal algebra $\Utor$ of type $X_{n}^{(r)}$ can be obtained from the corresponding finite quantum group $\Uq$ by affinizing twice within the quantum setting.
In fact, $\Utor$ contains two natural quantum affine subalgebras.
There is a horizontal subalgebra $\Uh$ of type $X_{n}^{(r)}$, defined as the image of the homomorphism $h : U'_{q}(X_{n}^{(r)}) \rightarrow \Utor$ sending
\begin{align*}
    \xipm \mapsto \xpm_{i,0}, \qquad k_{i} \mapsto k_{i},
\end{align*}
for all $i\in I$.
Additionally, there is a vertical subalgebra $\Uv$ of untwisted type $Z_{n}^{(1)}$, where $Z_{n}$ is the finite Cartan type of the simple Lie algebra $\g$.
It is the image of the homomorphism $v : U'_{q}(Z_{n}^{(1)}) \rightarrow \Utor$ given by
\begin{align*}
    \xpm_{i,m} \mapsto \xpm_{i,m}, \qquad h_{i,r} \mapsto h_{i,r}, \qquad k_{i} \mapsto k_{i}, \qquad C \mapsto C,
\end{align*}
for all $i\in I_{0}$, $m\in\Zbb$ and $r\in\Zbb^{*}$.
Furthermore, we are able to deduce from the next proposition that $\Uh$ and $\Uv$ together generate the entire quantum toroidal algebra.
Figure \ref{Utor illustration} provides a simple illustration of $\Utor$ which highlights its generators and their $\degZ$ grading, as well as the horizontal and vertical subalgebras.
\\
\begin{figure}[H]
    \centering
\begin{tikzpicture}[scale=1]
    \node at (-0.9,0.35) {$\xpm_{0,0} ~~ k^{\pm 1}_{0}$};
    \node at (-0.9,1.5) {$\xpm_{0,1} ~~ h_{0,1}$};
    \node at (-0.9,-0.8) {$\xpm_{0,-1} ~~ h_{0,-1}$};
    \node at (-0.9,2.4) {$\vdots$};
    \node at (-0.9,-1.5) {$\vdots$};
    \node at (1.5,0.35) {$\xpm_{1,0} ~~ k^{\pm 1}_{1}$};
    \node at (1.5,1.5) {$\xpm_{1,1} ~~ h_{1,1}$};
    \node at (1.5,-0.8) {$\xpm_{1,-1} ~~ h_{1,-1}$};
    \node at (1.5,2.4) {$\vdots$};
    \node at (1.5,-1.5) {$\vdots$};
    \node at (3,-1.5) {$C^{\pm 1}$};
    \node at (3,0.35) {$\cdots$};
    \node at (3,1.5) {$\cdots$};
    \node at (3,-0.8) {$\cdots$};
    \node at (4.5,0.35) {$\xpm_{n,0} ~~ k^{\pm 1}_{n}$};
    \node at (4.5,1.5) {$\xpm_{n,1} ~~ h_{n,1}$};
    \node at (4.5,-0.8) {$\xpm_{n,-1} ~~ h_{n,-1}$};
    \node at (4.5,2.4) {$\vdots$};
    \node at (4.5,-1.5) {$\vdots$};
    \node at (-1.8,0.7) {$\color{blue} \Uh$};
    \node at (0.6,2.55) {$\color{red} \Uv$};
    \draw[draw=blue] (-2.1,-0.25) rectangle ++(7.8,1.2);
    \draw[draw=red] (0.3,-2.1) rectangle ++(5.5,4.9);
    \draw[draw=black] (-2.2,-2.2) rectangle ++(8.1,5.1);
\end{tikzpicture}
\caption[Illustration of $\Utor$]{\hspace{.5em}An illustration of $\Utor$ and its quantum affine subalgebras $\Uh$ and $\Uv$}
\label{Utor illustration}
\end{figure}

The following is obtained by applying $\eta$ to \cite{Laurie24a}*{Prop. 4.3}.

\begin{prop}
    For each $i\in I$, the quantum toroidal algebra is generated by $\Uh$, $\xpm_{i,\pm 1}$ and $C^{\pm 1}$.
\end{prop}

\begin{cor} \label{toroidal generated by horizontal and vertical}
    The quantum toroidal algebra is generated by its horizontal and vertical subalgebras.
\end{cor}

Recall from Section \ref{subsection:quantum affinizations} the following standard automorphisms and anti-automorphisms of $\Utor$.
\begin{itemize}
    \item Every outer automorphism $\pi\in\Omega$ of the affine Dynkin diagram gives rise to an automorphism $\Scal_{\pi}$ which restricts to $S_{\pi}$ on $\Uh$.
    \item The anti-involution $\eta$ restricts to $\eta'$ on $\Uv$ and $\sigma$ on $\Uh$.
    \item For each $i\in I$ there exists an automorphism $\X_{i}$ defined using some affine sign function $o:I\rightarrow\lbrace\pm 1\rbrace$, which restricts to $\Xb_{i}$ on $\Uv$ if $i\in I_{0}$ and to the identity if $i=0$.
\end{itemize}

\subsection{Extended double affine braid groups} \label{subsection:Extended double affine braid groups}

Just as the quantum toroidal algebras $\Utor$ are in some sense formed by fusing together their horizontal and vertical quantum affine subalgebras in an appropriate way, we can similarly define the extended double affine braid groups $\Bdd$ by combining the Coxeter and Bernstein presentations for $\Bd$.
\\

Recall from Section \ref{subsection:basic notations} that $\Omega$ acts naturally on the affine braid group $\B = \langle T_{i}~|~i\in I\rangle$.
There is also a linear action of $\Omega$ on $P^{\vee}$ given by $\pi(\Lambda_{i}^{\vee}) = \Lambda_{\pi(i)}^{\vee}$, which preserves $\Pov \subset P^{\vee}$ and thus defines an action on $\lbrace X_{\beta}~|~\beta \in \Pov \rbrace$.
These actions are compatible with relations (\ref{first extended affine Bernstein}) and (\ref{second extended affine Bernstein}), extended to all $\beta \in \Pov$ and $i\in I$, hence the following is well-defined.

\begin{defn} \label{defn:Bdd}
    The extended double affine braid group $\Bdd$ is generated by the affine braid group $\B = \langle T_{i}~|~i\in I\rangle$, the lattice $\lbrace X_{\beta}~|~\beta \in \Pov \rbrace$ and the group $\Omega$, subject to the relations
    \begin{itemize}
        \item $T_{i}X_{\beta} = X_{\beta}T_{i} \mathrm{~if~} (\beta,\alpha_{i}) = 0$,
        \item $T_{i}^{-1}X_{\beta}T_{i}^{-1} = X_{s_{i}(\beta)} \mathrm{~if~} (\beta,\alpha_{i}) = 1$,
        \item $\pi T_{i} \pi^{-1} = T_{\pi(i)}$,
        \item $\pi X_{\beta} \pi^{-1} = X_{\pi(\beta)}$.
    \end{itemize}
\end{defn}

\begin{rmk}
    \begin{itemize}
        \item The action of $W$ on $\Pov$ in the definition above is with respect to the embedding $\Pov \hookrightarrow P^{\vee}$ of type $X_{n}^{(r)}$ rather than $Z_{n}^{(1)}$.
        \item Our group $\Bdd$ is the quotient of the $X,Y$-extended double affine Artin group of Ion and Sahi \cite{IS20}*{Ch. 9} by the subgroup generated by its central element $X_{\frac{1}{m}\delta}$.
    \end{itemize}
\end{rmk}

It is clear that $\Bdd$ contains two extended affine braid subgroups which together generate the entire group: a horizontal subgroup $\Bh$ of type $X_{n}^{(r)}$ generated by $\B$ and $\Omega$, and a vertical subgroup $\Bv$ of type $Z_{n}^{(1)}$ generated by $T_{1},\dots,T_{n}$ and $\lbrace X_{\beta}~|~\beta \in \Pov \rbrace$.
Figure \ref{Bdd illustration} illustrates how these subgroups fit together inside $\Bdd$, as well as indicating a natural vertical $\Zbb$--grading.
We remark that there only exists an isomorphism between $\Bh$ and $\Bv$ which acts by the identity on $\B_{0} \cong \Bh \cap \Bv$ in the untwisted case.
\\
\begin{figure}[H]
    \centering
\begin{tikzpicture}[scale=1]
    \node at (-0.8,0.35) {$\Omega ~~ T_{0}^{\pm 1}$};
    \node at (1.5,0.35) {$T_{1}^{\pm 1}$};
    \node at (1.5,1.5) {$X_{\omega_{1}^{\vee}}$};
    \node at (1.5,-0.8) {$X_{-\omega_{1}^{\vee}}$};
    \node at (1.5,2.4) {$\vdots$};
    \node at (1.5,-1.5) {$\vdots$};
    \node at (3,0.35) {$\cdots$};
    \node at (3,1.5) {$\cdots$};
    \node at (3,-0.8) {$\cdots$};
    \node at (4.5,0.35) {$T_{n}^{\pm 1}$};
    \node at (4.5,1.5) {$X_{\omega_{n}^{\vee}}$};
    \node at (4.5,-0.8) {$X_{-\omega_{n}^{\vee}}$};
    \node at (4.5,2.4) {$\vdots$};
    \node at (4.5,-1.5) {$\vdots$};
    \node at (-1.8,0.7) {$\color{blue} \Bh$};
    \node at (0.6,2.55) {$\color{red} \Bv$};
    \draw[draw=blue] (-2.1,-0.25) rectangle ++(7.8,1.2);
    \draw[draw=red] (0.3,-2.1) rectangle ++(5.5,4.9);
    \draw[draw=black] (-2.2,-2.2) rectangle ++(8.1,5.1);
\end{tikzpicture}
\caption[An illustration of $\Bdd$]{\hspace{.5em}An illustration of $\Bdd$ and its extended affine braid subgroups $\Bh$ and $\Bv$}
\label{Bdd illustration}
\end{figure}

From Section \ref{subsection:basic notations} we know that $\Bh$ and $\Bv$ each have both Coxeter and Bernstein presentations -- Table \ref{extended double affine braid group table} summarises our choice of notation.
In particular, for $\Bh$ we use the alternative Bernstein presentation of Remark \ref{rmk:alternative Bernstein presentation} so that while the $X_{\beta}$ satisfy relations (\ref{first extended affine Bernstein}) and (\ref{second extended affine Bernstein}) with $T_{0},\dots,T_{n}$, the $Y_{\mu}$ satisfy relations (\ref{first alternative extended affine Bernstein}) and (\ref{second alternative extended affine Bernstein}) with $T^{v}_{0},T_{1},\dots,T_{n}$.
Note that in all untwisted types, each $\pi_{i}$ and $\rho_{i}$ correspond to the same outer automorphism of the affine Dynkin diagram.
\\
\begin{table}[H]
\centering
\begin{tabular}{ |c||c|c| }
 \hline
  & Coxeter generators & Bernstein generators \\
 \hline
 & & \\[-12.5pt]
 \hline
 $\Bh$
 &
 \begin{tabular}{@{}c@{}}$T_{1},\dots,T_{n}$ \\[1pt] $T_{0} = \Theta^{-1} Y_{-\beta_{\theta}}$ \\[1pt] $\Omega = \lbrace \pi_{i} = Y_{\beta_{i}}T_{v_{i}^{-1}} : i\in\Imin \rbrace$\end{tabular}
 &
 \begin{tabular}{@{}c@{}}$T_{1},\dots,T_{n}$ \\[1pt] $\lbrace Y_{\mu} : \mu \in N \rbrace$\end{tabular}
 \\[18pt]
 \hline
 $\Bv$
 &
 \begin{tabular}{@{}c@{}}$T_{1},\dots,T_{n}$ \\[1pt] $T^{v}_{0} = X_{\theta^{\vee}} \Theta^{-1}$ \\[1pt] $\Omega^{v} = \lbrace \rho_{i} = X_{\omega_{i}^{\vee}}T_{v_{i}}^{-1} : i\in\Imin \rbrace$\end{tabular}
 &
 \begin{tabular}{@{}c@{}}$T_{1},\dots,T_{n}$ \\[1pt] $\lbrace X_{\beta} : \beta \in \Pov \rbrace$\end{tabular}
 \\[17pt]
 \hline
\end{tabular}
\caption[Coxeter and Bernstein generators for $\Bh$ and $\Bv$]{\hspace{.5em}Coxeter and Bernstein generators for $\Bh$ and $\Bv$}\label{extended double affine braid group table}
\end{table}

We conclude with several automorphisms of $\Bdd$ which will be important in Section \ref{section:horizontal-vertical symmetries}.
For ease of notation, we restrict to the untwisted case since this is all we shall require.

\begin{itemize}
    \item There is an involution $\te$ which inverts $T_{1},\dots,T_{n}$ and interchanges $X_{\beta}$ and $Y_{\beta}$ for all $\beta\in\Pov$.
    It follows that $\te$ exchanges each $\pi_{i}$ and $\rho_{i}$, as well as $T_{0}$ and $(T^{v}_{0})^{-1}$.
    It is equal to the composition of the anti-involution $\e$ of Ion and Sahi \cite{IS20}*{Ch. 9} with the anti-automorphism that inverts every element.
    When restricted to the natural copy of the (non-extended) double affine braid group inside $\Bdd$, which is generated by $\B = \langle T_{0},\dots,T_{n} \rangle$ and $\lbrace X_{\beta}~|~\beta \in \Qov \rbrace$, this is the involution of Ion \cite{Ion03}*{Thm. 2.2}.
    \item There exists an involution $\gv$ inverting $T_{0},\dots,T_{n}$ and all $X_{\beta}$, while fixing each element of $\Omega$.
    Similarly, there is an involution $\gh = \te\circ\gv\circ\te$ which inverts $T^{v}_{0},T_{1},\dots,T_{n}$ and all $Y_{\mu}$ but fixes each element of $\Omega^{v}$.
\end{itemize}

\subsubsection{Action on quantum toroidal algebras} \label{subsubsection:Bdd action}

In this subsection we consider all affine types except $A_{2}^{(2)}$, since it does not satisfy condition (D).

\begin{prop}
    The automorphisms $\T_{i}$ of $\Utor$ defined in Section \ref{subsubsection:affinized braid group action} satisfy
    \begin{itemize}
        \item $\T_{i} h = h \emph{\Tb}_{i}$ for all $i\in I$,
        \item $\T_{i} v = v \emph{\Tb}_{i}$ for all $i\in I_{0}$.
    \end{itemize}
\end{prop}

Similar to Section \ref{subsubsection:affinized braid group action}, in type $A_{2n}^{(1)}$ we must consider a slightly modified version of $\Bdd$ acting on $\Utor$.
In particular, $\zeta_{0}$ acts by $\s^{(0)}_{-1}$ and there is a minor change to Lemma \ref{lem:lemma for psi theorem}.
However, the involutions $\te$, $\gv$ and $\gh$ extend naturally to this case and our results are not otherwise impacted.
\\

It is clear that the extended double affine braid group $\Bdd$ embeds inside the corresponding $\widehat{\B}$ by sending $T_{i} \mapsto T_{i}$, $X_{\omega_{i}^{\vee}} \mapsto X_{i}X_{0}^{-a_{i}}$ and $\pi \mapsto \pi$ for each $i\in I$ and $\pi\in\Omega$, as well as $\zeta_{0} \mapsto \zeta_{0}$ in type $A_{2n}^{(1)}$.
The following result is then an immediate consequence of Theorem \ref{thm:affinized braid group action}.

\begin{thm} \label{thm:Bdd action on Utor}
    The extended double affine braid group $\Bdd$ acts on the quantum toroidal algebra $\Utor$ via
    $T_{i} \mapsto \T_{i}$ and
    $X_{\omega_{i}^{\vee}} \mapsto \Zcal_{\omega_{i}^{\vee}} = \X_{i}\X_{0}^{-a_{i}}$
    for all $i\in I$,
    $\pi \mapsto \Scal_{\pi}$ for all $\pi\in \Omega$,
    and $\zeta_{0} \mapsto \s^{(0)}_{-1}$ in type $A_{2n}^{(1)}$.
\end{thm}

\begin{rmk} \label{rmk:horizontal and vertical restricted actions}
    Our extended double affine braid group action restricts to both an action of $\Bh$ on $\Uh$ and an action of $\Bv$ on $\Uv$, each of which coincides with Lusztig and Beck's action of the extended affine braid group on the quantum affine algebra from Corollary \ref{cor:Bd action} and Theorem \ref{thm:loop Bd action} respectively.
\end{rmk}

\subsubsection{Coxeter-style presentation} \label{subsubsection:Coxeter-style presentation}

It has been shown by Ion-Sahi \cite{IS20} that while the double affine braid groups are not Coxeter braid groups themselves, they can be realized as \emph{quotients} of the braid groups associated to so-called `double affine Coxeter diagrams'.
This realization can be extended to $\Bdd$, and provides a finer understanding of its structure that is essential for extending our proof of Theorem \ref{thm:psi} from the simply laced case \cite{Laurie24a} to all untwisted types.
\\

We present the Coxeter-style presentation for $\Bdd$ in the untwisted case only, since this is all we shall require for this paper.
Here, the double affine Coxeter diagram $D(\Xtrip)$ of type $\Xtrip$ is formed as follows.
First take the affine Dynkin diagram of type $\Xun$, and consider the underlying, undirected Coxeter graph.
Then replace the $0$ vertex with three affine nodes, connected to one another by four edges and to each finite node $i \in I_{0}$ by $a_{0i}a_{i0}$ edges.
We illustrate this process with two examples in Figure \ref{fig:double_affine_Coxeter_diagrams} below.
\begin{figure}[H]
\centering
\begin{tabular}{r c}
$\overset{\ldots}{C}_{n}$ ($n\geq 2$):
&
\begin{tikzpicture}[baseline={([yshift=-0.1cm]current bounding box.east)}, double distance = 1.5pt, line width=0.65pt, every node/.style={outer sep=-2pt}, scale=1.6]
    \node (01) at (-2.866,-1) {\large $\circ$};
    \node (02) at (-2,-0.5) {\large $\circ$};
    \node (03) at (-2.866,0) {\large $\circ$};
    \node (1) at (-1,-0.5) {\large $\bullet$};
    \node (dots) at (0,-0.5) {\large $\cdots$};
    \node (n-1) at (1,-0.5) {\large $\bullet$};
    \node (n) at (2,-0.5) {\large $\bullet$};
    \draw (1) edge (dots) (dots) edge (n-1);
    \draw[double] (01.340) .. controls +(0.5,0) and +(-0.5,-0.5) .. (1);
    \draw[double] (02) -- (1);
    \draw[double] (03.20) .. controls +(0.5,0) and +(-0.5,0.5) .. (1);
    \draw[double] (n-1) -- (n);
    \draw[transform canvas={xshift=1pt}] (01) edge (03);
    \draw[transform canvas={xshift=-1pt}] (01) edge (03);
    \draw[transform canvas={xshift=3pt}] (01) edge (03);
    \draw[transform canvas={xshift=-3pt}] (01) edge (03);
    \draw[transform canvas={yshift=0.866pt,xshift=0.5pt}] (02) edge (03);
    \draw[transform canvas={yshift=2.598pt,xshift=1.5pt}] (02) edge (03);
    \draw[transform canvas={yshift=-0.866pt,xshift=-0.5pt}] (02) edge (03);
    \draw[transform canvas={yshift=-2.598pt,xshift=-1.5pt}] (02) edge (03);
    \draw[transform canvas={yshift=0.866pt,xshift=-0.5pt}] (02) edge (01);
    \draw[transform canvas={yshift=2.598pt,xshift=-1.5pt}] (02) edge (01);
    \draw[transform canvas={yshift=-0.866pt,xshift=0.5pt}] (02) edge (01);
    \draw[transform canvas={yshift=-2.598pt,xshift=1.5pt}] (02) edge (01);
\end{tikzpicture}
\\[1.5cm]
$\overset{\ldots}{D}_{n}$ ($n\geq 4$):
&
\begin{tikzpicture}[baseline={([yshift=0.0cm]current bounding box.east)}, double distance = 1.5pt, line width=0.65pt, every node/.style={outer sep=-2pt}, scale=1.6]
    \node (01) at (-3,0) {\large $\circ$};
    \node (02) at (-2,0) {\large $\circ$};
    \node (03) at (-2.5,0.866) {\large $\circ$};
    \node (1) at (-2,-1) {\large $\bullet$};
    \node (2) at (-1,-0.5) {\large $\bullet$};
    \node (dots) at (0,-0.5) {\large $\cdots$};
    \node (n-2) at (1,-0.5) {\large $\bullet$};
    \node (n-1) at (2,0) {\large $\bullet$};
    \node (n) at (2,-1) {\large $\bullet$};
    \draw (1) edge (2) (2) edge (dots) (dots) edge (n-2) (n-2) edge (n-1) (n-2) edge (n);
    \draw (01.310) edge[bend right=10] (2);
    \draw (02) -- (2);
    \draw (03.350) edge[bend left=10] (2);
    \draw[transform canvas={yshift=1pt}] (01) edge (02);
    \draw[transform canvas={yshift=-1pt}] (01) edge (02);
    \draw[transform canvas={yshift=3pt}] (01) edge (02);
    \draw[transform canvas={yshift=-3pt}] (01) edge (02);
    \draw[transform canvas={xshift=0.866pt,yshift=-0.5pt}] (01) edge (03);
    \draw[transform canvas={xshift=2.598pt,yshift=-1.5pt}] (01) edge (03);
    \draw[transform canvas={xshift=-0.866pt,yshift=0.5pt}] (01) edge (03);
    \draw[transform canvas={xshift=-2.598pt,yshift=1.5pt}] (01) edge (03);
    \draw[transform canvas={xshift=0.866pt,yshift=0.5pt}] (02) edge (03);
    \draw[transform canvas={xshift=2.598pt,yshift=1.5pt}] (02) edge (03);
    \draw[transform canvas={xshift=-0.866pt,yshift=-0.5pt}] (02) edge (03);
    \draw[transform canvas={xshift=-2.598pt,yshift=-1.5pt}] (02) edge (03);
\end{tikzpicture}
\\
\end{tabular}
\caption[Examples of double affine Coxeter diagrams]{\hspace{.5em}Examples of double affine Coxeter diagrams}
\label{fig:double_affine_Coxeter_diagrams}
\end{figure}
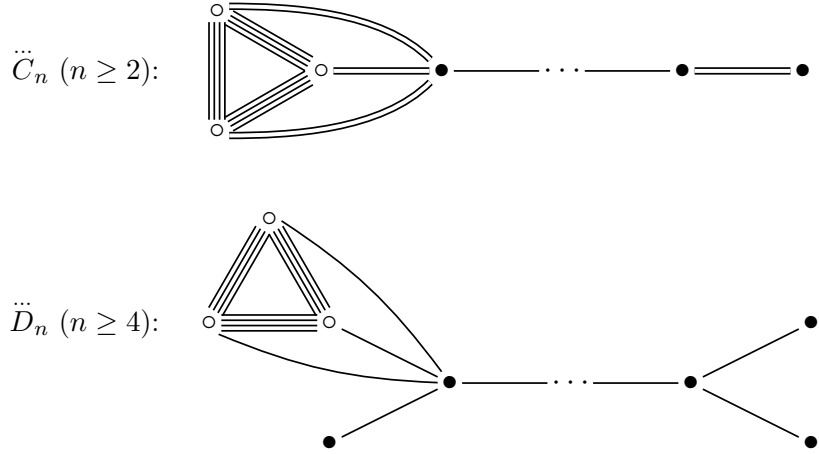

The braid group $B(\Xtrip)$ associated to this diagram has affine generators $\Th{1},\Th{2},\Th{3}$ and finite generators $T_{1},\dots,T_{n}$, with braid relations of type $\Xun$ on each
$\lbrace \Th{i},T_{1},\dots,T_{n} \rbrace$.
Letting $\Bb(\Xtrip)$ be its quotient by the relation
$\Th{1}\Th{2}\Th{3}\Theta = 1$,
as well as
$\Th{i}T_{1}^{-1}\Th{j}T_{1} = T_{1}^{-1}\Th{j}T_{1}\Th{i}$
for all $i<j$ if $X=C$,
the following comes from \cite{IS20}*{Thm. 5.19}.

\begin{thm}
    There is an isomorphism between $\Bb(\Xtrip)$ and the (non-extended) double affine braid group of type $\Xun$ sending $T_{i} \mapsto T_{i}$ for all $i\in I_{0}$ and
    \begin{align*}
        \Th{1} \mapsto T_{0}, \qquad
        \Th{2} \mapsto T_{0}^{-1} X_{-\theta^{\vee}}, \qquad
        \Th{3} \mapsto X_{\theta^{\vee}}\Theta^{-1}.
    \end{align*}
\end{thm}

In order to upgrade this to a Coxeter presentation for the \emph{extended} double affine braid group $\Bdd$, we must take the semidirect product of $\Bb(\Xtrip)$ with two copies of the outer automorphism group $\Omega$ of the affine Dynkin diagram.
The first, which we shall denote by
$\Omega_{1} = \lbrace \pi_{i} ~|~ i\in \Imin \rbrace$,
acts naturally by permuting
$\Th{1},T_{1},\dots,T_{n}$ and by
\begin{align*}
    \pi_{i}(\Th{2}) =
    T_{u_{i}} \Theta^{-1} \Th{3} \Theta T_{u_{i}}^{-1},
    \qquad
    \pi_{i}(\Th{3}) =
    T_{u_{i}} \Th{1} \Th{2} \Th{1}^{-1} T_{u_{i}}^{-1},
\end{align*}
for all $i\not= 0$, where $u_{i}$ is the minimal length element in the finite Weyl group such that
$\Theta = T_{u_{i}^{-1}} T_{i} T_{u_{i}}$.
In particular, $\pi_{i}(T_{u_{i^{*}}^{-1}}) = T_{u_{i}}$ where $i^{*}$ is defined by $\pi_{i^{*}} = \pi_{i}^{-1}$ and therefore
$\pi_{i}(\Theta) = T_{u_{i}} \Th{1} T_{u_{i}^{-1}}$.
The second copy
$\Omega_{3} = \lbrace \rho_{i} ~|~ i\in \Imin \rbrace$
permutes $\Th{3},T_{1},\dots,T_{n}$ instead, with
\begin{align*}
    \rho_{i}(\Th{1}) =
    T_{u_{i}^{-1}}^{-1} \Th{3}^{-1} \Th{2} \Th{3} T_{u_{i}^{-1}},
    \qquad
    \rho_{i}(\Th{2}) =
    T_{u_{i}^{-1}}^{-1} \Theta \Th{1} \Theta^{-1} T_{u_{i}^{-1}},
\end{align*}
for all $i\not= 0$, as well as
$\rho_{i}(T_{u_{i^{*}}^{-1}}) = T_{u_{i}}$
and hence
$\rho_{i}(\Theta) = T_{u_{i}} \Th{3} T_{u_{i}^{-1}}$.

\begin{thm} \label{thm:Coxeter Bdd}
    The previous theorem extends to an isomorphism between
    $\Omega_{1} \ltimes (\Omega_{3} \ltimes \Bb(\Xtrip))$
    and $\Bdd$ by sending $\pi_{i} \mapsto \pi_{i}$ and
    $\rho_{i} \mapsto X_{\beta_{i}} T_{v_{i}}^{-1}$
    for all $i\in\Imin$, such that
    \begin{itemize}
        \item $\Omega_{1} \ltimes \langle \Th{1},T_{1},\dots,T_{n} \rangle$ is identified with the horizontal subgroup $\Bh$,
        \item $\Omega_{3} \ltimes \langle \Th{3},T_{1},\dots,T_{n} \rangle$ is identified with the vertical subgroup $\Bv$.
    \end{itemize}
\end{thm}

\subsubsection{Diagonal subgroup} \label{subsubsection:diagonal subgroup}

The Coxeter presentation for $\Bdd$ as an (extended) quotient of $B(\Xtrip)$ suggests that we define a third extended affine braid subgroup, first introduced by the author in \cite{Laurie24b}, which will play an important role in our proof of Theorem \ref{thm:psi}.
For ease of notation, as in Section \ref{subsubsection:Coxeter-style presentation}, we restrict to the untwisted case since this is all we shall require.

\begin{defn}
    The diagonal subgroup $\Bdiag$ is the copy of $\Bd$ inside $\Bdd$ generated by $\Th{2},T_{1},\dots,T_{n}$ and
    $\Omega_{2} = \langle X_{\omega_{i}^{\vee}}\pi_{i} ~|~ i\in\Imin \rangle$.
\end{defn}

So $\Bh$, $\Bdiag$ and $\Bv$ come from the first, second and third affine nodes of $D(\Xtrip)$ respectively, together with vertices $1,\dots,n$.
The next result then says that $\te$ corresponds to the graph involution that swaps the first and third affine nodes.
Let $\mathfrak{j}$ be the involution of $\Bd$ which inverts $T_{0},\dots,T_{n}$ and fixes every $\pi\in\Omega$.

\begin{prop} \label{prop:t on Bh Bd Bv}
    The involution $\te$ of $\Bdd$ exchanges $\Bh$ and $\Bv$ via $\mathfrak{j}$, and moreover restricts to $\mathfrak{j}$ on $\Bdiag$.
\end{prop}
\begin{proof}
    First note that $\te$ fixes each $X_{\omega_{i}^{\vee}}\pi_{i}$ since
    \begin{align*}
        Y_{\omega_{i}^{\vee}} \rho_{i}
        = (\rho_{i^{*}} Y_{\omega_{i}^{\vee}}^{-1})^{-1}
        = (Y_{\omega_{i^{*}}^{\vee}} \rho_{i^{*}})^{-1}
        = \rho_{i} Y_{\omega_{i^{*}}^{\vee}}^{-1}
        = \rho_{i} (\pi_{i^{*}} T_{v_{i^{*}}^{-1}}^{-1})^{-1}
        = \rho_{i} T_{v_{i^{*}}^{-1}} \pi_{i}
        = \rho_{i} T_{v_{i}} \pi_{i}
        = X_{\omega_{i}^{\vee}}\pi_{i},
    \end{align*}
    where the penultimate equality holds provided that $v_{i^{*}}^{-1} = v_{i}$.
    Indeed, conjugating by the longest element $w_{0}$ of a finite Weyl group permutes the simple reflections according to the unique automorphism of the finite Dynkin diagram that maps $i \mapsto i^{*}$ for each $i\in\Imin$.
    (Extra care is required regarding the parity of $n$ in type $D_{n}^{(1)}$.)
    It follows that $w_{0i}$ is sent to $w_{0i^{*}}$, and hence
    $v_{i} = w_{0}w_{0i} = w_{0i^{*}}w_{0} = v_{i^{*}}^{-1}$
    since the longest element of any finite Weyl group is self-inverse \cite{Bourbaki68}*{p.171}.
    Furthermore, we have
    \begin{align*}
        \Th{2}
        = T_{0}^{-1} X_{-\theta^{\vee}}
        \xmapsto{~\te~}
        T_{0}^{v} Y_{-\theta^{\vee}}
        = X_{\theta^{\vee}} \Theta^{-1} \Theta T_{0}
        = X_{\theta^{\vee}} T_{0}
        = \Th{2}^{-1},
    \end{align*}
    and the rest of the proposition is easily checked.
\end{proof}

In Section \ref{subsection:Discussion and direct consequences of psi theorem} we will see that $\Bdiag$ corresponds to a diagonal quantum affine subalgebra $\Udiag$ of $\Utor$, first defined by the author in \cite{Laurie24b}.

\subsection{Quantum toroidal \texorpdfstring{$\glone$}{gl1}} \label{subsection:Quantum toroidal gl1}

Let us now introduce the related object quantum toroidal $\glone$.
This algebra has several alternative names, due to its appearance within different mathematical contexts.
For example, it is often called the...
\begin{itemize}
    \item Ding-Iohara-Miki (DIM) algebra \cites{DI97,Miki07},
    \item deformed $W_{1+\infty}$ algebra \cite{Miki07},
    \item elliptic Hall algebra \cites{BS12,Schiffmann12,SV13},
    \item spherical double affine Hecke algebra of $GL_{\infty}$ \cite{SV13},
    \item quantum continuous $\mathfrak{gl}_{\infty}$ algebra \cite{FFJMM11}.
\end{itemize}
Its representation theory is rich, with many wide-ranging connections across mathematics and physics, and is at this stage further developed than that of general quantum toroidal algebras $\Utor$.
\\

Loosely speaking, quantum toroidal $\glone$ may be viewed as the quantum affinization of the deformed Heisenberg algebra $U_{q}(\widehat{\mathfrak{gl}}_{1})$.
Alternatively, one can think of it as the quantum affinization associated to the Cartan matrix $(0)$.
However, it is important to note that neither interpretation is strictly speaking well-defined.
\\

Fix complex numbers $q_{1}$, $q_{2}$, $q_{3}$ such that $q_{1} q_{2} q_{3} = 1$, each not a root of unity, and consider all quantum integers $[r]$ with respect to $q_{1}$.

\begin{defn} \label{defn:quantum toroidal gl1}
    The quantum toroidal algebra $\Utorglone$ of type $\glone$ is the unital associative $\Qbb(q_{1},q_{3})$-algebra with generators $\xpm_{m}$, $h_{r}$, $k^{\pm 1}$, $C^{\pm 1}$ ($m\in\Zbb$, $r\in\Zbb^{*}$), subject to the following relations:
\begin{itemize}
    \item $C^{\pm 1}$, $k^{\pm 1}$ central,
    \item $\displaystyle C^{\pm 1}C^{\mp 1} = k^{\pm 1} k^{\mp 1} = 1$,
    \item $\displaystyle [h_{r},h_{s}] = \delta_{r+s,0} \frac{[r]}{r} \frac{q_{2}^{-r} - q_{2}^{r}}{q_{3}^{r} - q_{3}^{-r}} \frac{C^{r}-C^{-r}}{q_{1} - q_{1}^{-1}}$,
    \item $\displaystyle [h_{r},\xpm_{m}] = \pm \frac{[r]}{r} (q_{2}^{r} - q_{2}^{-r}) C^{\frac{r \mp \lvert r\rvert}{2}} \xpm_{r+m}$,
    \item $\displaystyle [\xp_{m},\xm_{l}] = \frac{q_{2}^{-1} - q_{2}}{(q_{1} - q_{1}^{-1})(q_{3} - q_{3}^{-1})} (C^{-l}\phi^{+}_{m+l} - C^{-m}\phi^{-}_{m+l})$,
    \item $[\xpm_{m},[\xpm_{m-1},\xpm_{m+1}]] = 0$,
\end{itemize}
where
$\sum_{s\in\Zbb} \phi^{\pm}_{\pm s} z^{\pm s} =
k^{\pm 1} \exp{( (q_{1} - q_{1}^{-1}) \sum_{s'>0} (q_{3}^{\pm s'} - q_{3}^{\mp s'}) h_{\pm s'} z^{\pm s'} )}$.
\end{defn}

\begin{rmk}
The above presentation resembles Definition \ref{defn:quantum toroidal algebra} for $\Utor$, with an extra deformation parameter $q_{3}$.
By scaling the generators
\begin{align*}
    \xpm_{m} \mapsto (q_{2}^{-1} - q_{2}) \xpm_{m}, \qquad
    h_{r} \mapsto (q_{1} - q_{1}^{-1})^{-1} (q_{3}^{r} - q_{3}^{-r})^{-1} h_{r},
\end{align*}
one obtains an alternative set of relations for $\Utorglone$ which highlights a symmetry with respect to permuting $q_{1}$, $q_{2}$ and $q_{3}$:
\begin{itemize}
    \item $C^{\pm 1}$, $k^{\pm 1}$ central,
    \item $\displaystyle C^{\pm 1}C^{\mp 1} = k^{\pm 1} k^{\mp 1} = 1$,
    \item $\displaystyle [h_{r},h_{s}] = \delta_{r+s,0} \frac{\kappa_{r}}{r} (C^{r}-C^{-r})$,
    \item $\displaystyle [h_{r},\xpm_{m}] = \pm \frac{\kappa_{r}}{r} C^{\frac{r \mp \lvert r\rvert}{2}} \xpm_{r+m}$,
    \item $\displaystyle [\xp_{m},\xm_{l}] = \frac{1}{\kappa_{1}} (C^{-l}\phi^{+}_{m+l} - C^{-m}\phi^{-}_{m+l})$,
    \item $[\xpm_{m},[\xpm_{m-1},\xpm_{m+1}]] = 0$,
\end{itemize}
where
$\kappa_{r} = (q_{1}^{r}-q_{1}^{-r})(q_{2}^{r}-q_{2}^{-r})(q_{3}^{r}-q_{3}^{-r})$
and
$\sum_{s\in\Zbb} \phi^{\pm}_{\pm s} z^{\pm s} =
k^{\pm 1} \exp{(\sum_{s'>0} h_{\pm s'} z^{\pm s'})}$.
\end{rmk}

Quantum toroidal $\glone$ possesses analogues of various properties already mentioned for $\Uqaffs$ or $\Utor$.
For example, there exists...
\begin{itemize}
    \item a $\Zbb^{2}$--grading given by
    $\deg(\xpm_{m}) = (\pm 1, m)$,
    $\deg(h_{r}) = (0,r)$ and
    $\deg(C^{\pm 1}) = \deg(k^{\pm 1}) = (0,0)$,
    \item a finite generating set
    $\lbrace \xpm_{0},\, h_{\pm 1},\, k^{\pm 1},\, C^{\pm 1} \rbrace$,
    \item a finite presentation -- see \cite{Miki07}*{Lem. 9.2},
    \item a topological coproduct $\Delta_{u}$ defined as in Theorem \ref{thm:Damiani topological coproduct}, without the $i$ indices,
    \item an automorphism $\X$ given by
    \begin{align*}
        &\X(\xpm_{m}) = \xpm_{m \mp 1}, \qquad
        \X(h_{r}) = h_{r}, \qquad
        \X(k) = C^{-1} k, \qquad
        \X(C) = C,
    \end{align*}
    \item an anti-involution $\eta$ given by
    \begin{align*}
        &\eta(\xpm_{m}) = \xpm_{-m}, \qquad
        \eta(h_{r}) = -C^{r} h_{-r}, \qquad
        \eta(k) = k^{-1}, \qquad
        \eta(C) = C,
    \end{align*}
    \item a $\Qbb$-algebra involution $\Wcal$ sending each $q_{i} \mapsto q_{i}^{-1}$ such that
    \begin{align*}
        &\Wcal(\xpm_{m}) = C^{m} x^{\mp}_{m}, \qquad
        \Wcal(h_{r}) = h_{r}, \qquad
        \Wcal(k) = k, \qquad
        \Wcal(C) = C^{-1}.
    \end{align*}
\end{itemize}

One can also develop an $\ell$-highest weight theory, similarly to Section \ref{subsubsection:l-highest weight theory}, since the algebra possesses a natural loop triangular decomposition
\begin{align*}
    \Utorglone \cong
    \langle \xm_{m} ~|~ m\in\Zbb \rangle
    \otimes
    \langle C^{\pm 1},\, k^{\pm 1},\, h_{r} ~|~ r\in\Zbb^{*} \rangle
    \otimes
    \langle \xp_{m} ~|~ m\in\Zbb \rangle.
\end{align*}
Here, $\ell$-weights $(\lambda,\Psi,c)$ must have $\lambda = (\Psi^{\pm}_{0})^{\pm 1}$ and we may without loss of generality assume that $c = 1$.
For our purposes, $\ell$-weights therefore correspond to pairs $(\Psi^{+}(z),\Psi^{-}(z))$ of power series in $\Cbb\llbracket z\rrbracket$.
\\

All representations $V = \bigoplus_{n\in\Zbb} V_{n}$ are $\Zbb$--graded with $\U_{a,b}\cdot V_{n} \subset V_{a+n}$, and said to be integrable if every $V_{n}$ is finite dimensional.
For $\ell$-highest weight modules, $\Psi^{\pm}(z)$ must be the expansions at $z^{\mp 1} = 0$ of some rational function $\Pcal(z)$ for which $\Pcal(0)\Pcal(\infty) = 1$.
Furthermore, the irreducible $\ell$-highest weight module $V(\Psi^{\pm}(z))$ is integrable precisely when this condition is satisfied, and may alternatively be denoted by $V(\Pcal(z))$.
The category $\Oaff$ consists of integrable modules with $V_{n} = 0$ for $n\gg 0$, and in particular contains all such representations.
See \cite{Miki07} and \cite{FJMM17} for more details.
\\

Burban-Schiffmann \cite{BS12}, working in the elliptic Hall algebra realization, showed that the natural action $\SL\curvearrowright\Zbb^{2}$ lifts to an action on $\Utorglone$.
In particular, the following order $4$ automorphism corresponds to \emph{clockwise rotation by 90 degrees}, and was later proven by Miki \cite{Miki07} via purely algebraic methods.

\begin{thm} \cites{BS12,Miki07}
    There is an automorphism $\Phi$ of quantum toroidal $\glone$ given by:
    \[\begin{tikzcd}[ampersand replacement=\&,column sep=small,sep=small]
	\& {h_{1}} \&\&\&\&\&\& k \\
	{\xm_{0}} \&\& {\xp_{0}} \&\& {\mathrm{and}} \&\& {C^{-1}} \&\& C \\
	\& {h_{-1}} \&\&\&\&\&\& {k^{-1}}
	\arrow[from=1-2, to=2-3]
	\arrow[from=1-8, to=2-9]
	\arrow[from=2-1, to=1-2]
	\arrow[from=2-3, to=3-2]
	\arrow[from=2-7, to=1-8]
	\arrow[from=2-9, to=3-8]
	\arrow[from=3-2, to=2-1]
	\arrow[from=3-8, to=2-7]
    \end{tikzcd}\]
\end{thm}

In Sections \ref{subsubsection:Miki automorphism} and \ref{subsubsection:Congruence group actions on quantum toroidal algebras} we obtain analogues of these results for the quantum toroidal algebras in all untwisted types.
In particular, we prove an action of the universal cover $\widetilde{\SL}$ on $\Utor$, as well as the existence of automorphisms which generalise $\Phi$.
\\

For various reasons, when investigating the representation theoretic applications of these symmetries, we prefer to work with a related anti-involution $\psi$ proved in Theorem \ref{thm:psi}.
The corresponding result for quantum toroidal $\glone$ comes by combining the Miki automorphism $\Phi$, the anti-automorphism $\eta$, and a scaling automorphism $\sfrak_{C}$ which maps
$\xpm_{m} \mapsto C^{m} \xpm_{m}$ and
$h_{r} \mapsto C^{r} h_{r}$
while fixing $k^{\pm 1}$ and $C^{\pm 1}$.

\begin{cor} \label{cor:glone psi}
    There is an anti-involution $\psi = \sfrak_{C}^{-1} \eta \Phi$ of quantum toroidal $\glone$ given by:
    
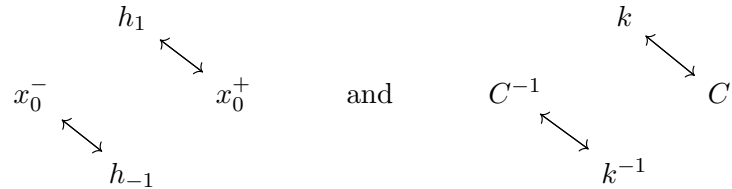
\[\begin{tikzcd}[ampersand replacement=\&,column sep=small,sep=small]
	\& {h_{1}} \&\&\&\&\&\& k \\
	{\xm_{0}} \&\& {\xp_{0}} \&\& {\mathrm{and}} \&\& {C^{-1}} \&\& C \\
	\& {h_{-1}} \&\&\&\&\&\& {k^{-1}}
	\arrow[from=1-2, to=2-3]
	\arrow[from=1-8, to=2-9]
	\arrow[from=3-2, to=2-1]
	\arrow[from=3-8, to=2-7]
	\arrow[to=1-2, from=2-3]
	\arrow[to=1-8, from=2-9]
	\arrow[to=3-2, from=2-1]
	\arrow[to=3-8, from=2-7]
    \end{tikzcd}\]
\end{cor}

In the case of untwisted $\Utor$, the anti-involution $\psi$ enlarges our $\widetilde{\SL}$ action to one of $\widetilde{\GL}$, and in particular corresponds to reflection
$\begin{bsmallmatrix} 0 & 1 \\ 1 & 0 \end{bsmallmatrix}$
in the line $x=y$.
Similarly, using Corollary \ref{cor:glone psi} we can extend the famous result of \cite{BS12} to a $\GL$ symmetry for $\Utorglone$, lifted from the lattice $\Zbb^{2}$.
\\

One may wonder whether $\Utorglone$ carries an action of some appropriate extended double affine braid group $\Bdd$, analogous to our work in Theorem \ref{thm:Bdd action on Utor}.
However, since the underlying Dynkin diagram is (morally) just a single affine node, we have $I_{0} = \emptyset$ and $\Pov = 0$ and hence $\Bdd$ should be trivial.
While by no means interesting in its own right, this does provide some intuition for the following.
\\

Our action $\widetilde{\SL} \curvearrowright \Utor$ from Theorem \ref{thm:congruence group actions on Utor} does not seem to factor through $\SL$, in contrast to quantum toroidal $\glone$.
On the braid group side, this corresponds to the fact that while $\widetilde{\SL}$ acts on $\Bdd$ by automorphisms, $\SL$ only acts by \emph{outer automorphisms}.
Indeed, the kernel of the natural projection
$\widetilde{\SL} \twoheadrightarrow \SL$
is generated by a single element, which acts as conjugation by $T_{w_{0}}^{2}$ in all types \cite{IS20}*{Thm. 6.4}.
The descent to $\SL$ in the $\glone$ case is then explained by the triviality of $\Bdd$, which removes this obstacle.

\section{Horizontal--vertical symmetries} \label{section:horizontal-vertical symmetries}

We now look to construct certain automorphisms and anti-involutions of $\Utor$ which exchange the horizontal and vertical subalgebras.
For classical toroidal Lie algebras
$\g[s^{\pm 1},t^{\pm 1}] \oplus \mathbb{K}$,
such symmetries are useful but trivial -- simply swap the loop parameters $s$ and $t$ up to inverse, perhaps inverting the Cartan elements of $\g$.
But within the quantum setting their existence is remarkable, in part due to the asymmetry of the definition for $\Utor$.
Namely, while horizontal affinization is in the Drinfeld-Jimbo style, vertical affinization occurs via the loop-style quantum affinization procedure.
\\

Our horizontal--vertical symmetries possess a range of applications in studying the structure and representation theory of $\Utor$.
Indeed, the celebrated Miki automorphisms of $\UtorA$ and $\Utorglone$ have already been used extensively in works by many other authors -- see Section \ref{subsubsection:Miki automorphism} -- and the previous lack of such results outside type $A$ has been one of the major obstacles for studying quantum toroidal algebras in general.
Within this paper, our anti-involution $\psi$ from Theorem \ref{thm:psi} plays a fundamental role in the construction of tensor products, $R$-matrices and transfer matrices for $\ell$-highest weight modules of quantum toroidal algebras in Sections \ref{section:tensor products} and \ref{section:R matrices}.
\\

In \cite{Laurie24a} we dealt with the simply laced case, in particular generalising the Miki automorphism of $\UtorA$ from \cites{Miki99,Miki01} as a corollary.
Here we extend our treatment to all untwisted types by employing a finer consideration of the extended double affine braid groups involving the Coxeter presentation from Theorem \ref{thm:Coxeter Bdd}.

\begin{notation}
    For simplicity, we will henceforth identify elements of $\Bdd$ with the corresponding automorphisms of $\Utor$ from Theorem \ref{thm:Bdd action on Utor}.
\end{notation}

\begin{notation}
    We shall also write $X_{i}$ for $X_{\omega_{i}^{\vee}}$ and $Y_{i}$ for $Y_{\omega_{i}^{\vee}}$ for each $i\in I_{0}$.
\end{notation}

Our approach is roughly as follows.
We can in some sense build $\Utor$ out of the copy of the finite quantum group $\Uq$ lying inside $\Uh\cap\Uv$ and the braid group action from Theorem \ref{thm:Bdd action on Utor}.
Twisting the action by certain automorphisms of $\Bdd$ (which in particular swap $\Bh$ and $\Bv$) produces different `twisted' sets of generators for $\Utor$.
Then mapping the standard generators to their twisted counterparts gives our desired (anti-)automorphisms.
\\

More specifically, each generator of our simplified presentation for $\Utor$ from Theorem \ref{thm:finite Utor presentation} (other than $C^{\pm 1}$) can easily be written as $b(z)$ for some $b\in\Bdd$ and $z\in\Uq$.
For all $\xpm_{i,0}$ and $k_{i}^{\pm 1}$ with $i\in I_{0}$ we may set $b=1$, and of course $\xpm_{i,\pm 1} = o(i) X_{i}^{-1}(\xpm_{i,0})$ for each $i\in I_{0}$.
For the other generators we have
\begin{itemize}
    \item $\xpm_{0,0} = T_{\ell}T_{0}(\xpm_{\ell,0}) = T_{\ell}^{-1}T_{0}^{-1}(\xpm_{\ell,0})$ for any $\ell\in\Tilde{I}$,
    \item $\xpm_{0,\pm 1} = o(0) T_{\ell}\Th{2}(\xpm_{\ell,0}) = o(0) T_{\ell}^{-1}\Th{2}^{-1}(\xpm_{\ell,0})$ for any $\ell\in\Tilde{I}, \hfill \refstepcounter{equation}(\theequation)\label{eqn:xpm0pm1 identities}$
    \item $k_{0}^{\pm 1} = T_{\ell}T_{0}(k_{\ell}^{\pm 1}) = T_{\ell}^{-1}T_{0}^{-1}(k_{\ell}^{\pm 1})$ for any $\ell\in\Tilde{I}$,
\end{itemize}
where $\Tilde{I}$ is the set of vertices adjacent to $0$ in the affine Dynkin diagram, except in types $A_{n=1}^{(1)}$ and $C_{n}^{(1)}$ where we instead have
\begin{itemize}
    \item $\xpm_{0,0} = \pi_{n} (\xpm_{n,0})$,
    \item $\xpm_{0,\pm 1}
    = o(0) \pi_{n} X_{n}^{-1} (\xpm_{n,0})
    = o(0) X_{n} \pi_{n} (\xpm_{n,0})$,
    \item $k_{0}^{\pm 1} = \pi_{n} (k_{n}^{\pm 1})$.
\end{itemize}
Finally, for $A_{1}^{(1)}$ we also require
\begin{itemize}
    \item $\xpm_{1,\mp 1}
    = o(1) X_{1} (\xpm_{1,0})$,
    \item $\xpm_{0,\mp 1}
    = o(0) \pi_{1} X_{1} (\xpm_{1,0})
    = o(0) X_{1}^{-1} \pi_{1} (\xpm_{1,0})$.
\end{itemize}

\begin{rmk}
    Types $A_{1}^{(1)}$ and $C_{n}^{(1)}$ are treated separately since $a_{0\ell}a_{\ell 0} \not= 1$ for all $\ell\in\Tilde{I}$, and so unlike in other types we cannot `drag' generators at vertex $\ell$ to vertex $0$ by applying $T_{\ell}^{\pm 1}T_{0}^{\pm 1}$ and $T_{\ell}^{\pm 1}\Th{2}^{\pm 1}$.
\end{rmk}

Recall the involution $\te$ of $\Bdd$ from Section \ref{subsection:Extended double affine braid groups}.
For each $\xpm_{i,m} = b(z)$ above define $\xbpm_{i,m} = \te(b)(z)$,
and for each $k_{i}^{\pm 1} = b(z)$ let $\kb_{i}^{\pm 1} = \te(b)(z^{-1})$.
In particular,
\begin{align*}
    \kb_{i}^{\pm 1} = k_{i}^{\mp 1}, \qquad
    \xbpm_{i,0} = \xpm_{i,0}, \qquad
    \xbpm_{i,\pm 1} = o(i) Y_{i}^{-1}(\xpm_{i,0}),
\end{align*}
for all $i\in I_{0}$, and outside types $A_{1}^{(1)}$ and $C_{n}^{(1)}$ we have
\begin{align*}
    &\kb_{0}^{\pm 1} = T_{\ell}^{-1}(T_{0}^{v})^{-1}(k_{\ell}^{\mp 1}) = T_{\ell}T_{0}^{v}(k_{\ell}^{\mp 1}),
    \\
    &\xbpm_{0,0} = T_{\ell}^{-1}(T_{0}^{v})^{-1}(\xpm_{\ell,0}) = T_{\ell}T_{0}^{v}(\xpm_{\ell,0}),
    \\
    &\xbpm_{0,\pm 1} = o(0) T_{\ell}^{-1}\Th{2}^{-1}(\xpm_{\ell,0}) = o(0) T_{\ell}\Th{2}(\xpm_{\ell,0}),
\end{align*}
for any $\ell\in\Tilde{I}$, from which we see that $\xbpm_{0,\pm 1} = \xpm_{0,\pm 1}$.
For $C_{n}^{(1)}$ these are replaced by
$\kb_{0}^{\pm 1} = \rho_{n} (k_{n}^{\mp 1})$,
$\xbpm_{0,0} = \rho_{n} (\xpm_{n,0})$
and
\begin{align*}
    \xbpm_{0,\pm 1}
    = o(0) \rho_{n} Y_{n}^{-1} (\xpm_{n,0})
    = o(0) X_{n} T_{v_{n}}^{-1} T_{v_{n}^{-1}} \pi_{n} (\xpm_{n,0})
    = o(0) X_{n} (\xpm_{0,0})
    = \xpm_{0,\pm 1},
\end{align*}
where for the penultimate equality we use the identity $v_{i^{*}}^{-1} = v_{i}$ from our proof of Proposition \ref{prop:t on Bh Bd Bv}.
In type $A_{1}^{(1)}$, since $\Bdd$ has a particularly simple structure, we may in fact easily compute the images of all simplified generators under $\psi$ explicitly in terms of the standard generators:
\begin{align} \label{eqn:A1(1) psi expressions}
\begin{split}
    &\kb_{1}^{\pm 1}
    = k_{1}^{\mp 1}
    \\
    &\xbpm_{1,0}
    = \xpm_{1,0}
    \\
    &\xbp_{1,1}
    = o(1) Y_{1}^{-1} (\xp_{1,0})
    = o(1) T_{1} \pi_{1} (\xp_{1,0})
    = o(1) [2]^{-1} [\xp_{1,0}, [\xp_{1,0}, \xp_{0,0}]_{q^{-2}} ]
    \\
    &\xbm_{1,-1}
    = o(1) Y_{1}^{-1} (\xm_{1,0})
    = o(1) T_{1} \pi_{1} (\xm_{1,0})
    = o(1) [2]^{-1} [ [\xm_{0,0}, \xm_{1,0}]_{q^{2}}, \xm_{1,0}]
    \\
    &\xbp_{1,-1}
    = o(1) Y_{1} (\xp_{1,0})
    = o(1) \pi_{1} T_{1}^{-1} (\xp_{1,0})
    = o(0) k_{0}^{-1} \xm_{0,0}
    \\
    &\xbm_{1,1}
    = o(1) Y_{1} (\xm_{1,0})
    = o(1) \pi_{1} T_{1}^{-1} (\xm_{1,0})
    = o(0) \xp_{0,0} k_{0}
    \\
    &\kb_{0}^{\pm 1}
    = \rho_{1} (k_{1}^{\mp 1})
    = X_{1} T_{1}^{-1} (k_{1}^{\mp 1})
    = C^{\mp 1} k_{1}^{\pm 1}
    \\
    &\xbp_{0,0}
    = \rho_{1} (\xp_{1,0})
    = X_{1} T_{1}^{-1} (\xp_{1,0})
    = o(0) C k_{1}^{-1} \xm_{1,1}
    \\
    &\xbm_{0,0}
    = \rho_{1} (\xm_{1,0})
    = X_{1} T_{1}^{-1} (\xm_{1,0})
    = o(0) \xp_{1,-1} C^{-1} k_{1}
    \\
    &\xbpm_{0,\pm 1}
    = o(0) \rho_{1} Y_{1}^{-1} (\xpm_{1,0})
    = o(0) X_{1} \pi_{1} (\xpm_{1,0})
    = \xpm_{0,\pm 1}
    \\
    &\xbp_{0,-1}
    = o(0) Y_{1}^{-1} \rho_{1} (\xp_{1,0})
    = o(0) T_{1} \pi_{1} X_{1} T_{1}^{-1} (\xp_{1,0})
    = o(0) [2]^{-1} C k_{0}^{-1} k_{1}^{-2}
    [ [\xm_{0,1}, \xm_{1,0}]_{q^{2}}, \xm_{1,0}]
    \\
    &\xbm_{0,1}
    = o(0) Y_{1}^{-1} \rho_{1} (\xm_{1,0})
    = o(1) T_{1} \pi_{1} X_{1} T_{1}^{-1} (\xm_{1,0})
    = o(1) [2]^{-1}
    [\xp_{1,0}, [\xp_{1,0}, \xp_{0,-1}]_{q^{-2}} ]
    C^{-1} k_{0} k_{1}^{2}
\end{split}
\end{align}
It is immediate that $\kb_{0}^{\pm 1} = C^{\mp 1}k_{\theta}^{\pm 1}$ in all types.
If we moreover define $\Cb^{\pm 1} = k_{\delta}^{\mp 1}$, then the following theorem shows that mapping each generator to its bold counterpart extends to an anti-involution of $\Utor$ which exchanges $\Uh$ and $\Uv$ (via a twist by $\sigma$).

\begin{thm} \label{thm:psi}
    There exists a unique anti-involution $\psi$ of $\Utor$ sending
    \begin{align*}
        \xpm_{i,m} \mapsto \xbpm_{i,m}, \qquad
        k_{i}^{\pm 1} \mapsto \kb_{i}^{\pm 1}, \qquad
        C^{\pm 1} \mapsto \Cb^{\pm 1},
    \end{align*}
    for all generators (\ref{eqn:simplified generators 1})--(\ref{eqn:simplified generators 2}), determined by the conditions $\psi v = h \sigma$ and $\psi h = v \sigma$.
\end{thm}

We postpone the proof to Section \ref{subsection:proof of psi theorem}, and first focus on some immediate consequences of this result.
\\

Figure \ref{illustrations} provides simple illustrations of the quantum toroidal algebra containing the two finite generating sets
$\lbrace \xpm_{i,0}, \xpm_{i,\pm 1}, k_{i}^{\pm 1}, C^{\pm 1}~|~i\in I \rbrace$ and
$\lbrace \xbpm_{i,0}, \xbpm_{i,\pm 1}, \kb_{i}^{\pm 1}, \Cb^{\pm 1}~|~i\in I \rbrace$.
In particular, in each case they highlight where the generators lie inside $\Utor$ with respect to the horizontal and vertical subalgebras, as well as their $\degZ$ grading (except for $C^{\pm 1}$ and $\kb^{\pm 1}_{0}$).
\\
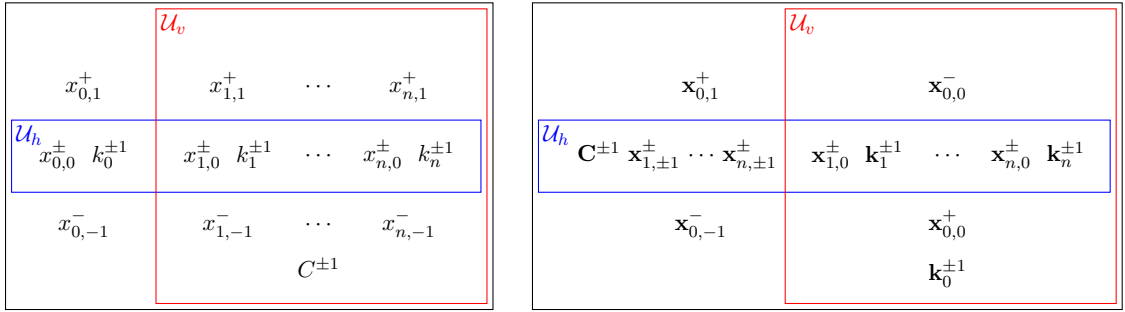
\begin{figure}[H]
    \centering
\begin{tikzpicture}[scale=0.8, transform shape]
    \node at (-0.9,0.35) {$\xpm_{0,0} ~~ k^{\pm 1}_{0}$};
    \node at (-0.9,1.5) {$\xp_{0,1}$};
    \node at (-0.9,-0.8) {$\xm_{0,-1}$};
    \node at (1.5,0.35) {$\xpm_{1,0} ~~ k^{\pm 1}_{1}$};
    \node at (1.5,1.5) {$\xp_{1,1}$};
    \node at (1.5,-0.8) {$\xm_{1,-1}$};
    \node at (3,-1.5) {$C^{\pm 1}$};
    \node at (3,0.35) {$\cdots$};
    \node at (3,1.5) {$\cdots$};
    \node at (3,-0.8) {$\cdots$};
    \node at (4.5,0.35) {$\xpm_{n,0} ~~ k^{\pm 1}_{n}$};
    \node at (4.5,1.5) {$\xp_{n,1}$};
    \node at (4.5,-0.8) {$\xm_{n,-1}$};
    \node at (-1.8,0.7) {$\color{blue} \Uh$};
    \node at (0.6,2.55) {$\color{red} \Uv$};
    \draw[draw=blue] (-2.1,-0.25) rectangle ++(7.8,1.2);
    \draw[draw=red] (0.3,-2.1) rectangle ++(5.5,4.9);
    \draw[draw=black] (-2.2,-2.2) rectangle ++(8.1,5.1);
\end{tikzpicture}
    ~~
\begin{tikzpicture}[scale=0.8, transform shape]
    \node at (-1.5,0.35) {$\Cb^{\pm 1} ~ \xbpm_{1,\pm 1} \,\cdots\, \xbpm_{n,\pm 1}$};
    \node at (-1.1,1.5) {$\xbp_{0,1}$};
    \node at (-1.1,-0.8) {$\xbm_{0,-1}$};
    \node at (1.5,0.35) {$\xbpm_{1,0} ~~ \kb^{\pm 1}_{1}$};
    \node at (3,1.5) {$\xbm_{0,0}$};
    \node at (3,-0.8) {$\xbp_{0,0}$};
    \node at (3,-1.6) {$\kb^{\pm 1}_{0}$};
    \node at (3,0.35) {$\cdots$};
    \node at (4.5,0.35) {$\xbpm_{n,0} ~~ \kb^{\pm 1}_{n}$};
    \node at (-3.5,0.7) {$\color{blue} \Uh$};
    \node at (0.6,2.55) {$\color{red} \Uv$};
    \draw[draw=blue] (-3.8,-0.25) rectangle ++(9.5,1.2);
    \draw[draw=red] (0.3,-2.1) rectangle ++(5.5,4.9);
    \draw[draw=black] (-3.9,-2.2) rectangle ++(9.8,5.1);
\end{tikzpicture}
    \caption[Illustrations of $\Utor$ displaying two generating sets]{\hspace{.5em}Illustrations of $\Utor$ displaying the two finite generating sets}\label{illustrations}
\end{figure}
We remark that the bold generators in some sense give $\Utor$ as a quantum affinization of its vertical rather than horizontal subalgebra, with $\Uv$ in a Drinfeld-Jimbo presentation and $\Uh$ in a Drinfeld new presentation (although the multiplication is of course reversed).
\\

Expressing $\psi(z)$ in terms of the standard generators of $\Utor$ -- and thus understanding in precise detail how $\psi$ acts -- is a difficult task in general.
However, passing to the classical setting provides a useful perspective.
In the limit $q \rightarrow 1$, $\psi$ becomes the anti-involution of $\g[s^{\pm 1},t^{\pm 1}] \oplus \mathbb{K}$ (the universal central extension of the toroidal Lie algebra) which sends
\begin{align*}
    h_{i} \mapsto - h_{i}, \qquad
    e_{i} \mapsto e_{i}, \qquad
    f_{i} \mapsto f_{i},
\end{align*}
for each $i\in I_{0}$, swaps the loop parameters $s$ and $t$, and acts on
$\mathbb{K} = \Omega_{1}\Cbb[s^{\pm 1},t^{\pm 1}] / d\Cbb[s^{\pm 1},t^{\pm 1}]$
accordingly.

\subsection{Discussion and direct consequences of Theorem \ref{thm:psi}} \label{subsection:Discussion and direct consequences of psi theorem}

\subsubsection{Miki automorphism} \label{subsubsection:Miki automorphism}

By composing $\psi$ with the standard anti-involution $\eta$, we obtain an automorphism of $\Utor$ which in type $A_{n}^{(1)}$ is precisely the Miki automorphism from \cites{Miki99,Miki01} (with the extra deformation parameter $\kappa$ set to $1$).

\begin{cor}\label{Phi corollary}
    There exists a unique automorphism $\Phi = \eta \psi$ of $\Utor$ with inverse $\Phi^{-1} = \eta \Phi \eta = \psi \eta$, determined by the conditions $\Phi v = h$ and $\Phi h = v \eta' \sigma$.
\end{cor}

The importance of the Miki automorphisms for $\UtorA$ and $\Utorglone$ cannot be overstated.
They have been fundamental not only for studying the structure and representation theory of the algebras themselves
(eg. \cites{FJMM13,Miki00,Miki01,Miki07,Tsymbaliuk19}),
but also their connections to other fields such as symmetric function and Macdonald theory
(eg. \cites{OS24,OSW22,Wen19})
and mathematical physics
(see \cites{FJMM15,FJM19,MNNZ24} and references therein).
One therefore hopes that our results inspire the extension of such directions to more general settings.
\\

Within the context of our action of the universal cover of $\SL$ on $\Utor$ from Theorem \ref{thm:congruence group actions on Utor} below, the automorphism $\Phi$ coincides with the action of
$S = \begin{bsmallmatrix} 0 & 1 \\ -1 & 0 \end{bsmallmatrix}$.
In the case of quantum toroidal $\glone$ this correspondence is known (cf. Section \ref{subsection:Quantum toroidal gl1}), and moreover
\begin{itemize}
    \item relates to $S$-dualities in physics, which provide equivalences between different quantum field theories or string theories,
    \item exists as the limit of Cherednik's Fourier transform on the (spherical) double affine Hecke algebras from \cite{Cherednik05}.
\end{itemize}

In terms of central elements, $\psi$ exchanges $C$ and $(k_{0}^{a_{0}}\dots k_{n}^{a_{n}})^{-1}$ while $\Phi$ maps $C \mapsto k_{0}^{a_{0}}\dots k_{n}^{a_{n}}$
and $k_{0}^{a_{0}}\dots k_{n}^{a_{n}} \mapsto C^{-1}$.
Twisting level $(a,b)$ representations of $\Utor$ by $\Phi$ therefore produces level $(b,-a)$ representations, and in this way we can obtain many new modules for quantum toroidal algebras.

\begin{eg}
    \begin{itemize}
        \item In symmetric types, this should relate certain $\ell$-highest weight and (future) Fock space representations with vertex representations, since level $(0,b)$ modules become level $(b,0)$.
        \item To the author's knowledge, outside the symmetric case there do not yet exist representations of $\Utor$ with level $(a,0)$ for $a \not= 0$, such as vertex representations.
        The first examples then come from twisting modules with $\ell$-highest weight $(\lambda,\Psi)$ and thus level $(0,\langle\lambda,c\rangle)$ by $\Phi$.
    \end{itemize}
\end{eg}

Since $\psi$ fixes $\xpm_{0,\pm 1}$ by construction, it follows that $\Phi(\xpm_{0,\pm 1}) = \xpm_{0,\mp 1}$.
This was originally shown for $\UtorA$ in \cite{Tsymbaliuk19}*{Prop. 2.6(d)} using a type $A_{n}^{(1)}$ specific argument.

\begin{rmk}
    Computing the images under $\psi$ or $\Phi$ for arbitrary elements of $\Utor$ is a difficult problem in general.
    A useful tool in type $A_{n}^{(1)}$ has been the situation of $\UtorA$ within the framework of combinatorially defined \textit{shuffle algebras} through works of Negu\c{t} \cites{Negut20,Negut24} and Tsymbaliuk \cites{Tsymbaliuk19,Tsymbaliuk23}.
    We expect these directions to extend to all untwisted types and perhaps even beyond, providing new methods for approaching quantum toroidal algebras.
\end{rmk}

\subsubsection{Compatibility relations}

Our (anti-)automorphisms $\psi$ and $\Phi^{\pm 1}$ enjoy the following compatibilities with our braid group action $\Bdd \curvearrowright \Utor$, and may therefore be considered as quantum toroidal analogues of the corresponding automorphisms of $\Bdd$ from Section \ref{subsection:Extended double affine braid groups}.

\begin{prop}\label{prop:compatibilities}
    \begin{itemize}
        \item For all $b\in\Bdd$ we have $\psi \circ b = \te(b) \circ \psi$ as anti-automorphisms of $\Utor$.
        \item For all $b\in\Bdd$ we have $\Phi^{\pm 1} \circ b = (\gamma_{v}\te)^{\pm 1}(b) \circ \Phi^{\pm 1} = (\gamma_{h}\te)^{\mp 1}(b) \circ \Phi^{\pm 1}$ as automorphisms of $\Utor$.
    \end{itemize}
\end{prop}
\begin{proof}
    See the author's thesis \cite{Laurie24b}*{§3.3}.
\end{proof}

Identities such as these often prove to be useful tools, for example allowing us to transfer computations for $\Utor$ over to $\Bdd$.
Indeed, working within the braid group setting is usually far easier than performing calculations inside quantum algebras.

\subsubsection{Congruence group actions on quantum toroidal algebras} \label{subsubsection:Congruence group actions on quantum toroidal algebras}

The Coxeter presentation for $\Bdd$ from Section \ref{subsubsection:Coxeter-style presentation} has numerous applications, including in all affine types $X_{n}^{(r)}$ an action of the corresponding congruence group $\Gam{r} \leq \SL$ on $\Bdd$ by outer automorphisms.
This moreover descends from an action by automorphisms of its universal cover $\TGam{r}$, which is isomorphic to the braid group of type $A_{2}$, $B_{2}$ or $G_{2}$ when $r=1,2$ or $3$ respectively.
For $r=1$ these results are originally due to Cherednik \cite{Cherednik95}, while the general case was proven by Ion-Sahi \cites{IS06,IS20}.
\\

In the author's thesis \cite{Laurie24b} we obtained for all untwisted types a quantum analogue of these results, in particular a congruence group action
$\TGam{1} \curvearrowright \Utor$.
The proof relies on the existence of our anti-involution $\psi$, together with compatibility relations such as those in Proposition \ref{prop:compatibilities}.
Since the congruence groups $\Gam{r}$ are defined by
\begin{align*}
    \Gam{r} =
    \left\lbrace
    \begin{bmatrix}
        a & b \\ c & d
    \end{bmatrix}
    \in \SL ~\middle\vert~
    \begin{bmatrix}
        a & b \\ c & d
    \end{bmatrix}
    =
    \begin{bmatrix}
        1 & * \\ 0 & 1
    \end{bmatrix}
    \mathrm{~mod~} r
    \right\rbrace
\end{align*}
for $r \in \lbrace 1,2,3 \rbrace$, in the untwisted case we are simply dealing with $\Gam{1} = \SL$ and its universal cover $\TGam{1} = \widetilde{\SL}$.

\begin{thm} \label{thm:congruence group actions on Utor}
\begin{itemize}
    \item There exists an action
    $\widetilde{\SL} \curvearrowright \Utor$
    given by
    $\begin{bsmallmatrix} 1 & -1 \\ 0 & 1 \end{bsmallmatrix}
    \mapsto \X_{0}^{-1}$
    and
    $\begin{bsmallmatrix} 1 & 0 \\ 1 & 1 \end{bsmallmatrix}
    \mapsto \psi\X_{0}\psi$,
    which fixes $\Uh\cap\Uv \cong \Uq$ pointwise.
    \item This is compatible with
    $\widetilde{\SL} \curvearrowright \Bdd$
    and our braid group action, namely
    $m\cdot(b\cdot z) = (m\cdot b)\cdot z$
    for all $m\in\widetilde{\SL}$, $b\in\Bdd$ and $z\in\Utor$.
    \item We can therefore combine our congruence and braid group actions to obtain
    $\widetilde{\SL} \ltimes \Bdd \curvearrowright \Utor$.
\end{itemize}
\end{thm}

As mentioned in Section \ref{subsection:Quantum toroidal gl1}, in the specific case of quantum toroidal $\glone$, an action of $\SL$ was realized geometrically by Burban-Schiffmann \cites{BS12,Schiffmann12} as Fourier-Mukai transforms of coherent sheaves on an elliptic curve over a finite field.
Our results therefore motivate the extension of such work to more general settings.

\begin{rmk}
    Our theorem can be extended to the universal cover of $\GL$ by letting its additional generator
    $\begin{bsmallmatrix} 0 & 1 \\ 1 & 0 \end{bsmallmatrix}$
    act on $\Utor$ via our anti-involution $\psi$.
\end{rmk}

See the author's thesis \cite{Laurie24b}*{§3.3} for further discussion and the proofs of these results.

\subsubsection{Diagonal subalgebras of quantum toroidal algebras}

Our anti-involution $\psi$ indicates the importance of a third quantum affine subalgebra $\Udiag$ which we shall call the \emph{diagonal subalgebra}, first introduced by the author in \cite{Laurie24b}.
This is defined as the image of the homomorphism
$U'_{q}(\Xun) \rightarrow \Utor$ sending
\begin{align*}
    \xipm \mapsto \xpm_{i,0}, \qquad
    k_{i}^{\pm 1} \mapsto k_{i}^{\pm 1}, \qquad
    \xpm_{0} \mapsto \xpm_{0,\pm 1}, \qquad
    k_{0}^{\pm 1} \mapsto (C k_{0})^{\pm 1},
\end{align*}
for each $i\in I_{0}$, with $C k_{\delta}$ as its canonical central element.
We immediately see that $\psi$ restricts to the anti-involution $\sigma$ on $\Udiag = \X_{0}^{-1}(\Uh)$, which therefore also equals
$\psi\X_{0}^{-1}\psi(\Uv)$.
\\

The diagonal subalgebra $\Udiag$ corresponds on the braid group side to the diagonal subgroup $\Bdiag$ of $\Bdd$ from Section \ref{subsubsection:diagonal subgroup}, just as $\Uh$ and $\Uv$ correspond to $\Bh$ and $\Bv$.
Indeed, $\Bdiag$ preserves $\Udiag$ under our braid group action from Proposition \ref{thm:Bdd action on Utor}, in particular acting via Lusztig and Beck's affine action (cf. Remark \ref{rmk:horizontal and vertical restricted actions}).

\begin{rmk}
    Consideration of $\Udiag$ is crucial to our proof of Theorem \ref{thm:psi} outside the simply laced case.
\end{rmk}

\subsubsection{Embeddings of quantum affine algebras}

While it is clear that $v$ is an embedding \cite{Hernandez05}*{Cor. 3} and hence $\Uv$ is a copy of the quantum affine algebra of type $Z_{n}^{(1)}$, the analogous horizontal statement is non-obvious.
Namely, it could be the case that relations of $\Utor$ involving generators not contained in $\Uh$ might have `shadows' inside the horizontal subalgebra.
However, using Theorem \ref{thm:psi} we may in fact deduce the injectivity of $h$ from that of $v$.

\begin{cor}
    The homomorphism $h : U'_{q}(X_{n}^{(1)}) \rightarrow \Utor$ is an embedding, and hence $\Uh$ is isomorphic to the quantum affine algebra of type $X_{n}^{(1)}$.
\end{cor}

Moreover, a corresponding diagonal result follows immediately by composing with $\X_{0}^{-1}$.

\begin{rmk}
    In the case of $\UtorA$, Tsymbaliuk \cite{Tsymbaliuk19}*{Rmk. 2.3} verified the injectivity of both $v$ and $h$ using Hopf pairings.
    These arguments should extend naturally to the general case.
\end{rmk}

\subsection{Proof of Theorem \ref{thm:psi}} \label{subsection:proof of psi theorem}

First we must verify the $\xpm_{0,\pm 1} = b(z)$ expressions given in (\ref{eqn:xpm0pm1 identities}) outside types $A_{1}^{(1)}$ and $C_{n}^{(1)}$, which imply that $\xbpm_{0,\pm 1} = \xpm_{0,\pm 1}$ since $\te$ inverts both $T_{\ell}$ and $\Th{2}$.
\begin{align*}
    T_{\ell}\Th{2}(\xpm_{\ell,0})
    &=
    T_{\ell}T_{0}^{-1}
    {\textstyle \prod_{i\in\Tilde{I}}X_{i}^{-1}}
    (\xpm_{\ell,0})
    =
    T_{\ell}T_{0}^{-1}
    {\textstyle \left(
    \prod_{i\in\Tilde{I}}\X_{i}^{-1}
    \right)}
    \X_{0}^{2}(\xpm_{\ell,0})
    =
    T_{\ell}\X_{\ell}^{-1}T_{0}^{-1}(\xpm_{\ell,0})
    \\
    &=
    \X_{\ell}^{-1}
    {\textstyle \left(
    \prod_{i\in I}\X_{i}^{a_{\ell i}}
    \right)}
    T_{\ell}^{-1}T_{0}^{-1}(\xpm_{\ell,0})
    =
    \X_{\ell}^{-1}
    {\textstyle \left(
    \prod_{i\in I}\X_{i}^{a_{\ell i}}
    \right)}
    (\xpm_{0,0})
    \\
    &= o(0)\xpm_{0,\pm 1}
    \\
    T_{\ell}^{-1}\Th{2}^{-1}(\xpm_{\ell,0})
    &=
    T_{\ell}^{-1}
    {\textstyle \left( \prod_{i\in\Tilde{I}}X_{i} \right)}
    T_{0}(\xpm_{\ell,0})
    =
    T_{\ell}^{-1}
    {\textstyle \left( \prod_{i\in\Tilde{I}}\X_{i} \right)}\X_{0}^{-2}
    T_{0}(\xpm_{\ell,0})
    \\
    &=
    {\textstyle \left( \prod_{i\in\Tilde{I}}\X_{i} \right)}
    {\textstyle \left( \prod_{j\in I}\X_{j}^{-a_{\ell j}} \right)}
    \X_{0}^{-2}
    T_{\ell}T_{0}(\xpm_{\ell,0})
    \\
    &=
    {\textstyle \left( \prod_{i\in\Tilde{I}}\X_{i} \right)}
    {\textstyle \left( \prod_{j\in I}\X_{j}^{-a_{\ell j}} \right)}
    \X_{0}^{-2}
    (\xpm_{0,0})
    \\
    &=
    o(0)\xpm_{0,\pm 1}
\end{align*}
In addition, the following alternative expressions for $\xbpm_{0,\pm 1}$ shall be useful in calculations.
\begin{align*}
    \xbpm_{0,\pm 1}
    &= o(0) \te(T_{\ell}\Th{2}) (\xpm_{\ell,0})
    = o(0) \te(T_{\ell}T_{0}^{-1}X_{\ell}^{-1}) (\xpm_{\ell,0})
    = o(0) T_{\ell}^{-1}T_{0}^{v}Y_{\ell}^{-1} (\xpm_{\ell,0})
    \\
    \xbpm_{0,\pm 1}
    &= o(0) \te(T_{\ell}^{-1}\Th{2}^{-1}) (\xpm_{\ell,0})
    = o(0) T_{\ell}Y_{\ell}(T_{0}^{v})^{-1} (\xpm_{\ell,0})
    \\
    &= o(0) Y_{s_{\ell}(\varpi_{\ell}^{\vee})} T_{\ell}^{-1} (T_{0}^{v})^{-1} (\xpm_{\ell,0})
    = o(0) Y_{s_{\ell}(\varpi_{\ell}^{\vee})} T_{\ell} T_{0}^{v} (\xpm_{\ell,0})
\end{align*}

A brief technical lemma provides an assortment of identities required for the proof of Theorem \ref{thm:psi}.
Note that in type $A_{2n}^{(1)}$ we restrict to $\rho = \rho_{1}$ for (\ref{lemma (4)}), while in type $A_{1}^{(1)}$ we can extend (\ref{lemma (2)}) and (\ref{lemma (4)}) to include $m = \mp 1$.

\begin{lem} \label{lem:lemma for psi theorem}
\begin{itemize}
        \item $Y_{i}(\xbpm_{j,0}) = \xbpm_{j,0}$ and $Y_{i}(\kb_{j}^{\pm 1}) = \kb_{j}^{\pm 1}$ for all distinct $i,j \in I_{0}, \hfill \refstepcounter{equation}(\theequation)\label{lemma (1)}$
        \item $\xbpm_{i,m} = h\sigma(\xpm_{i,m})$, $\kb_{i}^{\pm 1} = h\sigma(k_{i}^{\pm 1})$ and $\Cb^{\pm 1} = h\sigma(C^{\pm 1})$ for all $i \in I_{0}$ and $m=0,\pm 1, \hfill \refstepcounter{equation}(\theequation)\label{lemma (2)}$
        \item $\xbpm_{i,0} = v\sigma(\xpm_{i})$ and $\kb_{i}^{\pm 1} = v\sigma(k_{i}^{\pm 1})$ for all $i\in I, \hfill \refstepcounter{equation}(\theequation)\label{lemma (3)}$
        \item $\rho(\xbpm_{i,m}) = o_{i,\rho(i)}^{m}\xbpm_{\rho(i),m}$ and $\rho(\kb_{i}^{\pm 1}) = \kb_{\rho(i)}^{\pm 1}$ for all $i\in I$, $m=0,\pm 1$ and $\rho\in\Omega^{v}. \hfill \refstepcounter{equation}(\theequation)\label{lemma (4)}$
    \end{itemize}
\end{lem}
\begin{proof}
We know from Proposition \ref{prop:defining Ti} that $T_{i}h = h \Tb_{i} = h \sigma \Tb_{i}^{-1} \sigma$ for all $i\in I$, and it is immediate from the definitions that $\pi h = h S_{\pi} = h \sigma S_{\pi} \sigma$ for each $\pi\in\Omega$.
Each $Y_{\beta}$ can be written as $\pi T_{i_{1}}^{\pm 1}\dots T_{i_{s}}^{\pm 1}$ and so as $\sigma^{2}$ is the identity,
    \begin{align}\label{Y on Uh}
        Y_{\beta} h = h \sigma S_{\pi}\Tb_{i_{1}}^{\mp 1}\dots\Tb_{i_{s}}^{\mp 1} \sigma = h \sigma \Xb_{\beta} \sigma.
    \end{align}
Note that (\ref{lemma (2)}) is trivial for $\xbpm_{i,0}$, $\kb_{i}^{\pm 1}$ and $\Cb^{\pm 1}$, and using (\ref{Y on Uh}) we get
\begin{align*}
    \xbpm_{i,\pm 1}
    = o(i) Y_{i}^{-1} (\xpm_{i,0})
    = o(i) Y_{i}^{-1} h(\xpm_{i,0})
    = o(i) h\sigma\Xb_{i}^{-1}(\xpm_{i,0})
    = h\sigma(\xpm_{i,\pm 1}),
\end{align*}
and so our proof of (\ref{lemma (2)}) is complete.
Fixing distinct $i,j\in I_{0}$ we have from (\ref{Y on Uh}) that
\begin{align*}
    &Y_{i}(\xbpm_{j,0})
    = Y_{i}(\xpm_{j,0})
    = Y_{i} h(\xpm_{j,0})
    = h\sigma\Xb_{i}\sigma(\xpm_{j,0})
    = h(\xpm_{j,0})
    = \xpm_{j,0}
    = \xbpm_{j,0},
    \\
    &Y_{i}(\kb_{j}^{\pm 1})
    = Y_{i}(k_{j}^{\mp 1})
    = Y_{i} h(k_{j}^{\mp 1})
    = h\sigma\Xb_{i}\sigma(k_{j}^{\mp 1})
    = h(k_{j}^{\mp 1})
    = k_{j}^{\mp 1}
    = \kb_{j}^{\pm 1},
\end{align*}
which verifies (\ref{lemma (1)}).
Note that (\ref{lemma (3)}) is trivial when $i\in I_{0}$, and moreover since $\Bv$ acts on $\Uv$ via Lusztig and Beck's affine action, outside types $A_{1}^{(1)}$ and $C_{n}^{(1)}$ we have
\begin{align*}
    &\xbpm_{0,0}
    = T_{\ell}T_{0}(\xpm_{\ell,0})
    = T_{\ell}T_{0} v(\xpm_{\ell})
    = v\Tb_{\ell}\Tb_{0}(\xpm_{\ell})
    = v(\xpm_{0})
    = v\sigma(\xpm_{0}),
    \\
    &\kb_{0}^{\pm 1}
    = T_{\ell}T_{0}(k_{\ell}^{\mp 1})
    = T_{\ell}T_{0} v(k_{\ell}^{\mp 1})
    = v\Tb_{\ell}\Tb_{0}(k_{\ell}^{\mp 1})
    = v(k_{0}^{\mp 1})
    = v\sigma(k_{0}^{\pm 1}).
\end{align*}
In types $A_{n=1}^{(1)}$ and $C_{n}^{(1)}$ this is replaced with
\begin{align*}
    &\xbpm_{0,0}
    = \rho_{n}(\xpm_{n,0})
    = \rho_{n} v(\xpm_{n})
    = v S_{\rho_{n}}(\xpm_{n})
    = v(\xpm_{0})
    = v\sigma(\xpm_{0}),
    \\
    &\kb_{0}^{\pm 1}
    = \rho_{n}(k_{n}^{\mp 1})
    = \rho_{n} v(k_{n}^{\mp 1})
    = v S_{\rho_{n}}(k_{n}^{\mp 1})
    = v(k_{0}^{\mp 1})
    = v\sigma(k_{0}^{\pm 1}),
\end{align*}
completing the proof of (\ref{lemma (3)}).
For all $\rho \in \Omega^{v}$ we then have that
\begin{align*}
    & \rho(\xbpm_{i,0}) = \rho v(\xpm_{i}) = v S_{\rho} (\xpm_{i}) = v(\xpm_{\rho(i)}) = \xbpm_{\rho(i),0}, \\
    & \rho(\kb_{i}^{\pm 1}) = \rho v(k_{i}^{\mp 1}) = v S_{\rho} (k_{i}^{\mp 1}) = v(k_{\rho(i)}^{\mp 1}) = \kb_{\rho(i)}^{\pm 1},
\end{align*}
using (\ref{lemma (3)}).
The equality $\rho(\xbpm_{i,\pm 1}) = o_{i,\rho(i)}\xbpm_{\rho(i),\pm 1}$ is trivial if either $\rho = \mathrm{id}$ or we are in type $A_{1}^{(1)}$ or $C_{n}^{(1)}$, so we shall henceforth assume otherwise.
If $i,\rho(i)\not= 0$ then $\rho Y_{i}^{-1} \rho^{-1} = Y_{\rho(i)}^{-1} Y_{\rho(0)}^{a_{i}}$ and therefore
\begin{align*}
    \rho(\xbpm_{i,\pm 1})
    &= o(i) \rho Y_{i}^{-1}(\xpm_{i,0})
    = o(i) Y_{\rho(i)}^{-1} Y_{\rho(0)}^{a_{i}} \rho(\xpm_{i,0})
    = o(i) Y_{\rho(i)}^{-1} Y_{\rho(0)}^{a_{i}}(\xpm_{\rho(i),0})
    = o_{i,\rho(i)} \xbpm_{\rho(i),\pm 1}
\end{align*}
by (\ref{lemma (1)}) since $\rho(i),\rho(0)\in I_{0}$ are distinct.
If $i = 0$ then
$(\rho(s_{\ell}(\varpi_{\ell}^{\vee})),\alpha_{\rho(0)})
= (s_{\ell}(\varpi_{\ell}^{\vee}),\alpha_{0}) = -1$
and we have
\begin{align*}
    \rho(\xbpm_{0,\pm 1})
    &= o(0) \rho Y_{s_{\ell}(\varpi_{\ell}^{\vee})} T_{\ell} T_{0}^{v} (\xpm_{\ell,0})
    = o(0) Y_{\rho(s_{\ell}(\varpi_{\ell}^{\vee}))} T_{\rho(\ell)} T_{\rho(0)} \rho (\xpm_{\ell,0})
    \\
    &= o(0) Y_{\rho(s_{\ell}(\varpi_{\ell}^{\vee}))} T_{\rho(\ell)} T_{\rho(0)} (\xpm_{\rho(\ell),0})
    = o_{0,\rho(0)} o(\rho(0)) Y_{\rho(s_{\ell}(\varpi_{\ell}^{\vee}))} (\xpm_{\rho(0),0})
    \\
    &= o_{0,\rho(0)} \xbpm_{\rho(0),\pm 1}
\end{align*}
where we again make use of (\ref{lemma (1)}).
Outside type $A_{2n}^{(1)}$, the case $\rho(i) = 0$ then follows immediately since
\begin{align*}
    \rho(\xbpm_{i,\pm 1})
    = \rho(\xbpm_{\rho^{-1}(0),\pm 1})
    = \rho\left( o_{0,\rho^{-1}(0)}^{-1} \rho^{-1}(\xbpm_{0,\pm 1})\right)
    = o_{\rho^{-1}(0),0} \xbpm_{0,\pm 1}
    = o_{i,\rho(i)} \xbpm_{\rho(i),\pm 1}.
\end{align*}
Type $A_{2n}^{(1)}$ requires more care, and for space reasons we refer the reader to \cite{Laurie24a}*{Lem. 5.2}.
This completes our proof of (\ref{lemma (4)}).
\end{proof}

A second technical lemma gives information about how certain $Y_{\beta}\in\Bdd$ act on the twisted generators $\xbpm_{0,0}$ and $\xbpm_{0,\pm 1}$.

\begin{lem} \label{lem:Y on zero generators}
Our action of $\Bdd$ on $\Utor$ satisfies the following relations.
\\
\renewcommand{\arraystretch}{1.2}
\begin{table}[H]
    \centering
    \begin{tabular}{|c|c||c|c|}
        \hline
        $(\beta,\alpha_{0})$ & $(\beta,\alpha_{\ell})$ & $Y_{\beta}(\xbpm_{0,0})$ & $Y_{\beta}(\xbpm_{0,\pm 1})$ \\
        \hline
        & & & \\[-15pt]
        \hline
        $-1$ & $-2$ & $o(0)\xbpm_{0,\pm 1}$ & \\
        \hline
        $-1$ & $-1$ & $o(0)\xbpm_{0,\pm 1}$ & \\
        \hline
        $-1$ & $0$ & $o(0)\xbpm_{0,\pm 1}$ & \\
        \hline
        $-1$ & $1$ & $o(0)\xbpm_{0,\pm 1}$ & \\
        \hline
        $0$ & $-1$ & $\xbpm_{0,0}$ & $\xbpm_{0,\pm 1}$ \\
        \hline
        $0$ & $0$ & $\xbpm_{0,0}$ & $\xbpm_{0,\pm 1}$ \\
        \hline
        $0$ & $1$ & $\xbpm_{0,0}$ & $\xbpm_{0,\pm 1}$ \\
        \hline
        $1$ & $-1$ & & $o(0)\xbpm_{0,0}$ \\
        \hline
        $1$ & $0$ & & $o(0)\xbpm_{0,0}$ \\
        \hline
        $1$ & $1$ & & $o(0)\xbpm_{0,0}$ \\
        \hline
        $1$ & $2$ & & $o(0)\xbpm_{0,0}$ \\
        \hline
    \end{tabular}
    \caption[Actions of $Y_{\beta}$ on $\xbpm_{0,m}$]{\hspace{.5em}Actions of $Y_{\beta}$ on $\xbpm_{0,m}$}
    \label{table:Y on zero generators}
\end{table}
\renewcommand{\arraystretch}{1}
\end{lem}
\begin{proof}
We start by noting that the first five rows of the table follow immediately from the last five.
Moreover the proofs in types $A_{n=1}^{(1)}$ and $C_{n}^{(1)}$ are easily deduced from
\begin{align*}
    Y_{\beta}(\xbpm_{0,0})
    &= Y_{\beta} \rho_{n} (\xbpm_{n,0})
    = \rho_{n} Y_{\rho_{n}(\beta)} (\xbpm_{n,0}),
    \\
    Y_{\beta}(\xbpm_{0,\pm 1})
    &= o(0) Y_{\beta} \rho_{n} Y_{n}^{-1} (\xbpm_{n,0})
    = o(0) \rho_{n} Y_{\rho_{n}(\beta)} Y_{n}^{-1} (\xbpm_{n,0}),
\end{align*}
together with (\ref{lemma (1)}) and (\ref{lemma (4)}), and so we may restrict to the other types from now on.
In the following, we shall freely use without mention the various expressions for $\xbpm_{0,\pm 1}$ already presented, equation (\ref{lemma (1)}), and the relations of $\Bdd$.
\\

If $(\beta,\alpha_{0}) = 0$ and $(\beta,\alpha_{\ell}) = 0$ then
    \begin{align*}
    Y_{\beta}(\xbpm_{0,0})
    &= Y_{\beta}T_{\ell} T_{0} (\xpm_{\ell,0})
    = T_{\ell} T_{0} Y_{\beta} (\xpm_{\ell,0})
    = T_{\ell} T_{0} (\xpm_{\ell,0})
    \\
    &= \xbpm_{0,0},
    \\
    o(0) Y_{\beta}(\xbpm_{0,\pm 1})
    &= Y_{\beta} T_{\ell}^{-1}T_{0}^{v}Y_{\ell}^{-1} (\xpm_{\ell,0})
    = T_{\ell}^{-1}T_{0}^{v}Y_{\ell}^{-1} Y_{\beta} (\xpm_{\ell,0})
    = T_{\ell}^{-1}T_{0}^{v}Y_{\ell}^{-1} (\xpm_{\ell,0})
    \\
    &= o(0) \xbpm_{0,\pm 1}.
    \end{align*}
If $(\beta,\alpha_{0}) = 0$ and $(\beta,\alpha_{\ell}) = 1$ then
    \begin{align*}
    Y_{\beta}(\xbpm_{0,0})
    &= Y_{\beta}T_{\ell} T_{0}^{v} (\xpm_{\ell,0})
    = T_{\ell}^{-1} Y_{s_{\ell}(\beta)} T_{0}^{v} (\xpm_{\ell,0})
    \\
    &= T_{\ell}^{-1} (T_{0}^{v})^{-1} Y_{s_{0}s_{\ell}(\beta)} (\xpm_{\ell,0})
    = T_{\ell}^{-1} (T_{0}^{v})^{-1} (\xpm_{\ell,0})
    \\
    &= \xbpm_{0,0},
    \\
    o(0) Y_{\beta}(\xbpm_{0,\pm 1})
    &= Y_{\beta}Y_{s_{\ell}(\varpi_{\ell}^{\vee})} T_{\ell} T_{0}^{v} (\xpm_{\ell,0})
    = Y_{s_{\ell}(\varpi_{\ell}^{\vee})} T_{\ell}^{-1} Y_{s_{\ell}(\beta)} T_{0}^{v} (\xpm_{\ell,0})
    \\
    &= Y_{s_{\ell}(\varpi_{\ell}^{\vee})} T_{\ell}^{-1} (T_{0}^{v})^{-1} Y_{s_{0}s_{\ell}(\beta)} (\xpm_{\ell,0})
    = Y_{s_{\ell}(\varpi_{\ell}^{\vee})} T_{\ell}^{-1} (T_{0}^{v})^{-1} (\xpm_{\ell,0})
    \\
    &= o(0) \xbpm_{0,\pm 1}.
    \end{align*}
If $(\beta,\alpha_{0}) = 1$ and $(\beta,\alpha_{\ell}) = -1$ then
    \begin{align*}
    o(0) Y_{\beta}(\xbpm_{0,\pm 1})
    &= Y_{\beta}T_{\ell}^{-1}T_{0}^{v}Y_{\ell}^{-1} (\xpm_{\ell,0})
    = T_{\ell}
    Y_{s_{\ell}(\beta)}
    T_{0}^{v}Y_{\ell}^{-1} (\xpm_{\ell,0})
    = T_{\ell} T_{0}^{v}
    Y_{s_{\ell}(\beta) - \varpi_{\ell}^{\vee}} (\xpm_{\ell,0})
    \\
    &= \xbpm_{0,0}.
    \end{align*}
If $(\beta,\alpha_{0}) = 1$ and $(\beta,\alpha_{\ell}) = 0$ then
    \begin{align*}
    o(0) Y_{\beta}(\xbpm_{0,\pm 1})
    &= Y_{\beta}T_{\ell}^{-1}T_{0}^{v}Y_{\ell}^{-1} (\xpm_{\ell,0})
    = T_{\ell}^{-1}Y_{\beta}T_{0}^{v}Y_{\ell}^{-1} (\xpm_{\ell,0})
    = T_{\ell}^{-1}(T_{0}^{v})^{-1}
    Y_{s_{0}(\beta) - \varpi_{\ell}^{\vee}} (\xpm_{\ell,0})
    \\
    &= \xbpm_{0,0}.
    \end{align*}
If $(\beta,\alpha_{0}) = 1$ and $(\beta,\alpha_{\ell}) = 1$ then
    \begin{align*}
    o(0) Y_{\beta}(\xbpm_{0,\pm 1})
    &= Y_{\beta + s_{\ell}(\varpi_{\ell}^{\vee})} T_{\ell} T_{0}^{v} (\xpm_{\ell,0})
    = T_{\ell} T_{0}^{v}
    Y_{\beta + s_{\ell}(\varpi_{\ell}^{\vee})} (\xpm_{\ell,0})
    = T_{\ell} T_{0}^{v} (\xpm_{\ell,0})
    \\
    &= \xbpm_{0,0}.
    \end{align*}
If $(\beta,\alpha_{0}) = 1$ and $(\beta,\alpha_{\ell}) = 2$ then
    \begin{align*}
    o(0) Y_{\beta}(\xbpm_{0,\pm 1})
    &= Y_{\beta + s_{\ell}(\varpi_{\ell}^{\vee})} T_{\ell} T_{0}^{v} (\xpm_{\ell,0})
    = T_{\ell}^{-1}
    Y_{s_{\ell}(\beta + s_{\ell}(\varpi_{\ell}^{\vee}))}
    T_{0}^{v} (\xpm_{\ell,0})
    \\
    &= T_{\ell}^{-1} (T_{0}^{v})^{-1}
    Y_{s_{0}s_{\ell}(\beta + s_{\ell}(\varpi_{\ell}^{\vee}))}
    (\xpm_{\ell,0})
    = T_{\ell}^{-1} (T_{0}^{v})^{-1}
    (\xpm_{\ell,0})
    \\
    &= \xbpm_{0,0}. \qedhere
    \end{align*}
\end{proof}

We are now ready to prove Theorem \ref{thm:psi} in all untwisted types other than $G_{2}^{(1)}$, which shall require some additional consideration -- see Lemmas \ref{lem:hb2r on xbpm0m} and \ref{lem:hb1r on xbpm0m}.
This stems from $\Pov$ being `too small' within $P^{\vee}$ due to the particular $a_{i}$ labels, and so $\Bdd$ does not quite reach every relation of $\Utor$ so directly from those lying inside $\Uh$, $\Uv$ or $\Udiag$.
We shall therefore make clear precisely which relations are not covered by our initial methods, and then deal with these separately afterwards.

\begin{proof}[Proof of Theorem~{\upshape\ref{thm:psi}}]
To show that $\psi$ is an anti-homomorphism, we must check that the relations of Theorem \ref{thm:finite Utor presentation} still hold if we reverse the order of multiplication and replace each generator with its image under $\psi$.
Denote these modified relations by \textbf{(\ref{eqn:simplified relations 1})}--\textbf{(\ref{eqn:simplified relations 13})}.
\\

Every relation with indices in $I_{0}$ follows immediately from the Drinfeld new presentation of $\Uh$ using (\ref{lemma (2)}).
Moreover, relations involving only $\xbpm_{i,0}$ and $\kb_{i}^{\pm 1}$ terms follow from the Drinfeld-Jimbo presentation for $\Uv$ by (\ref{lemma (3)}).
Furthermore, all of the relations containing only $\xbpm_{0,\pm 1}$, $\xbpm_{i,0}$ and $\kb_{i}^{\pm 1}$ with $i\in I_{0}$ are verified with the Drinfeld-Jimbo presentation for $\Udiag$ since $\psi$ acts by $\sigma$ on these generators.
We shall now address the remaining relations not already covered by these arguments.
\\

\textbf{(\ref{eqn:simplified relations 4})}
For $A_{1}^{(1)}$ everything is easily checked using (\ref{eqn:A1(1) psi expressions}).
In other types, only the $i = 0$, $m = \pm 1$ cases remain, which are verified as follows with $j\not= 0$.
\begin{align*}
    \kb_{0} \xbpm_{0,\pm 1} \kb_{0}^{-1}
    &= C k_{\theta}^{-1} \xpm_{0,\pm 1} k_{\theta} C^{-1}
    = k_{\delta} k_{\theta}^{-1} \xpm_{0,\pm 1} k_{\theta} k_{\delta}^{-1}
    = k_{0} \xpm_{0,\pm 1} k_{0}^{-1}
    = q_{0}^{\pm a_{00}} \xpm_{0,\pm 1}
    \\
    &= q_{0}^{\pm a_{00}} \xbpm_{0,\pm 1}
    \\
    \kb_{0} \xbpm_{j,\pm 1} \kb_{0}^{-1}
    &= C k_{\theta}^{-1} \xbpm_{j,\pm 1} k_{\theta} C^{-1}
    = k_{\delta} k_{\theta}^{-1} \xbpm_{j,\pm 1} k_{\theta} k_{\delta}^{-1}
    = k_{0} \xbpm_{j,\pm 1} k_{0}^{-1}
    \\
    &= h\sigma(k_{0}\xpm_{j,\pm 1}k_{0}^{-1})
    = h\sigma(C k_{\theta}^{-1} \xpm_{j,\pm 1} k_{\theta} C^{-1})
    = h\sigma\left(\textstyle \prod_{i\in I_{0}} (q_{i}^{\mp a_{ij}})^{a_{i}} \xpm_{j,\pm 1}\right)
    \\
    &= q^{\mp \sum_{i\in I_{0}} a_{i}d_{i}a_{ij}} h\sigma(\xpm_{j,\pm 1})
    = q^{\pm a_{0}d_{0}a_{0j}} h\sigma(\xpm_{j,\pm 1})
    \\
    &= q_{0}^{\pm a_{0j}} \xbpm_{j,\pm 1}
\end{align*}

\textbf{(\ref{eqn:simplified relations 5})}
The only case left to check is $i = 0$, $m = -1$ in type $A_{1}^{(1)}$, which by (\ref{lemma (4)}) comes from applying $\rho_{1}$ to the $i = 1$, $m = -1$ relation.
\\

\textbf{(\ref{eqn:simplified relations 6})}
These are only present in type $A_{1}^{(1)}$, where applying $\rho_{1}$ to the $i = 1$ relation gives the $i = 0$ one.
\\

\textbf{(\ref{eqn:simplified relations 7})}
In type $A_{1}^{(1)}$ we can check everything directly using (\ref{eqn:A1(1) psi expressions}), so assume otherwise.
By Lemma \ref{lem:Y on zero generators}, all
$[\xbm_{j,-1},\xbp_{0,1}] = 0$ and
$[\xbm_{0,-1},\xbp_{j,1}] = 0$
with $j\in I_{0}$ are obtained by applying some $Y_{\beta}$ with
$(\beta,\alpha_{0}) = (\beta,\alpha_{j}) = -1$ and $-2 \leq (\beta,\alpha_{\ell}) \leq 1$
to the corresponding relations
$[\xbm_{j,0},\xbp_{0,0}] = 0$ and
$[\xbm_{0,0},\xbp_{j,0}] = 0$.
\textit{In type $G_{2}^{(1)}$ this argument fails for $j = 1$.}
\\

Using (\ref{lemma (1)}) and Lemma \ref{lem:Y on zero generators}, every
$[\xbm_{0,0},\xbp_{j,1}] = 0$ and
$[\xbm_{j,-1},\xbp_{0,0}] = 0$
with $j\in I_{0}$ can be reached via one of the following:
\begin{itemize}
    \item Apply $Y_{\beta}$ with
    $(\beta,\alpha_{0}) = 1$,
    $(\beta,\alpha_{j}) = -1$ and
    $-1 \leq (\beta,\alpha_{\ell}) \leq 2$
    to $[\xbm_{0,-1},\xbp_{j,0}] = 0$ and
    $[\xbm_{j,0},\xbp_{0,1}] = 0$
    respectively.
    \item Apply $Y_{\beta}$ with
    $(\beta,\alpha_{0}) = 0$,
    $(\beta,\alpha_{j}) = -1$ and
    $-1 \leq (\beta,\alpha_{\ell}) \leq 1$
    to $[\xbm_{0,0},\xbp_{j,0}] = 0$ and
    $[\xbm_{j,0},\xbp_{0,0}] = 0$
    respectively.
\end{itemize}
\textit{In type $G_{2}^{(1)}$ this argument fails for $j = 2$.}
\\

\textbf{(\ref{eqn:simplified relations 8})}
Again, the $A_{1}^{(1)}$ case may be checked with (\ref{eqn:A1(1) psi expressions}).
In all other types, combining (\ref{lemma (3)}) with Jing's isomorphism between the presentations of $\Udash$ gives
\begin{align*}
    \xbp_{0,0}
    = v(\xp_{0})
    = [\xm_{i_{h-1},0},\dots,\xm_{i_{2},0},\xm_{i_{1},1}]_{q^{\epsilon_{1}}\dots q^{\epsilon_{h-2}}} C k_{\theta}^{-1},
\end{align*}
so by centrality of $k_{\delta}$ and relation $7$ of Definition \ref{defn:quantum toroidal algebra} we have
\begin{align*}
    \xbp_{0,0}\xbp_{0,1}
    &= [\xm_{i_{h-1},0},\dots,\xm_{i_{2},0},\xm_{i_{1},1}]_{q^{\epsilon_{1}}\dots q^{\epsilon_{h-2}}} C k_{\theta}^{-1}
    \xp_{0,1}
    \\
    &= [\xm_{i_{h-1},0},\dots,\xm_{i_{2},0},\xm_{i_{1},1}]_{q^{\epsilon_{1}}\dots q^{\epsilon_{h-2}}} C k_{0}
    \xp_{0,1} k_{0}^{-1} k_{\theta}^{-1}
    \\
    &= [\xm_{i_{h-1},0},\dots,\xm_{i_{2},0},\xm_{i_{1},1}]_{q^{\epsilon_{1}}\dots q^{\epsilon_{h-2}}}
    q_{0}^{2} \xp_{0,1} C k_{\delta}^{-1}
    \\
    &= q_{0}^{2} \xp_{0,1} [\xm_{i_{h-1},0},\dots,\xm_{i_{2},0},\xm_{i_{1},1}]_{q^{\epsilon_{1}}\dots q^{\epsilon_{h-2}}} C k_{\delta}^{-1}
    \\
    &= q_{0}^{2} \xbp_{0,1}\xbp_{0,0}
\end{align*}
and thus $[\xbp_{0,0},\xbp_{0,1}]_{q_{0}^{2}} = 0$.
The relation $[\xbm_{0,-1},\xbm_{0,0}]_{q_{0}^{-2}} = 0$ is proved similarly.
\\

When $j\not\sim 0$ we obtain
$[\xbp_{j,0},\xbp_{0,1}]_{q_{0}^{a_{0j}}} + [\xbp_{0,0},\xbp_{j,1}]_{q_{0}^{a_{0j}}} = 0$
and
$[\xbm_{j,-1},\xbm_{0,0}]_{q_{0}^{-a_{0j}}} + [\xbm_{0,-1},\xbm_{j,0}]_{q_{0}^{-a_{0j}}} = 0$
as an immediate consequence of \textbf{(\ref{eqn:simplified relations 9})}, so assume otherwise.
Outside type $C_{n}^{(1)}$ we can apply both sides of
$\T_{2}^{-1} T_{2} \Th{2} = T_{0}^{-1} X_{-\theta^{\vee}}$
to $\xbp_{j,0}$ as follows, noting that
$\langle \theta^{\vee}, \alpha_{j} \rangle = 1$ and $o(0) = -o(j)$.
\begin{gather*}
    \xbp_{j,0}
    \xmapsto{T_{j}\Th{2}}
    o(0) \xbp_{0,1}
    \xmapsto{T_{j}^{-1}}
    o(0) [\xbp_{j,0},\xbp_{0,1}]_{q_{0}^{-1}}
    \\
    \xbp_{j,0}
    \xmapsto{X_{-\theta^{\vee}}}
    o(j) \xbp_{j,1}
    \xmapsto{T_{0}^{-1}}
    o(j) [\xbp_{0,0},\xbp_{j,1}]_{q_{0}^{-1}}
\end{gather*}
Furthermore, we prove $[\xbm_{j,-1},\xbm_{0,0}]_{q_{0}} + [\xbm_{0,-1},\xbm_{j,0}]_{q_{0}} = 0$
in the same manner, except with $\xbp_{j,0}$ replaced by $\xbm_{j,0}$.
For $C_{n}^{(1)}$ we instead apply $\rho_{n}$ to the corresponding relations with indices $n-1$ and $n$.
\\

\textbf{(\ref{eqn:simplified relations 9})}
Only the affine $q$-Serre relations with
$(y_{i},y_{j}) = (\xbpm_{0,0},\xbpm_{r,\pm 1}),(\xbpm_{r,\pm 1},\xbpm_{0,0})$ for each $r\in I_{0}$ remain,
which by (\ref{lemma (1)}) and Lemma \ref{lem:Y on zero generators} can be verified via one of the following.
\begin{itemize}
    \item Apply $Y_{\beta}$ with
    $(\beta,\alpha_{0}) = 1$,
    $(\beta,\alpha_{r}) = -1$ and
    $-1 \leq (\beta,\alpha_{\ell}) \leq 2$
    to the affine $q$-Serre relations with
    $(y_{i},y_{j}) = (\xbpm_{0,\pm 1},\xbpm_{r,0}),(\xbpm_{r,0},\xbpm_{0,\pm 1})$.
    \item Apply $Y_{\beta}$ with
    $(\beta,\alpha_{0}) = 0$,
    $(\beta,\alpha_{r}) = -1$ and
    $-1 \leq (\beta,\alpha_{\ell}) \leq 1$
    to the affine $q$-Serre relations with
    $(y_{i},y_{j}) = (\xbpm_{0,0},\xbpm_{r,0}), (\xbpm_{r,0},\xbpm_{0,0})$.
\end{itemize}
\textit{In type $G_{2}^{(1)}$ this argument fails for $j = 2$.}
\\

\textbf{(\ref{eqn:simplified relations 10})}
The $i = 0$ relations follow by applying $\rho_{1}$ to those with $i = 1$.
\\

\textbf{(\ref{eqn:simplified relations 11})}--\textbf{(\ref{eqn:simplified relations 12})}
These are checked directly using (\ref{eqn:A1(1) psi expressions}).
\\

We have therefore verified that $\psi$ is an anti-homomorphism.
The conditions $\psi v = h\sigma$ and $\psi h = v\sigma$ are then immediate from (\ref{lemma (2)}) and (\ref{lemma (3)}), and moreover determine $\psi$ uniquely since $\Uh$ and $\Uv$ generate $\Utor$.
Furthermore, it also follows that $\psi^{2} = \mathrm{id}$ on both $\Uh$ and $\Uv$ and so $\psi$ is in fact an anti-involution.
\end{proof}

\begin{notation}
    We shall write $\Rh$, $\Rv$ and $\Rdiag$ for the sets of relations in $\Utor$ involving only elements contained in $\Uh$, $\Uv$ and $\Udiag$ respectively.
\end{notation}

We are left to deduce the remaining relations
\textbf{(\ref{eqn:simplified relations 7})} and \textbf{(\ref{eqn:simplified relations 9})}
in type $G_{2}^{(1)}$ from those we already have.
To this end, define elements
$\hb_{i,1} = [\xbm_{i,0},\xbp_{i,1}] \kb_{i}^{-1}$
and
$\hb_{i,-1} = [\xbm_{i,-1},\xbp_{i,0}] \kb_{i}$
of $\Utor$ for each $i\in I$.
It follows from $\Rh$ that
\begin{align} \label{eqn:hbir on xbpmjm}
    \begin{split}
    [\xbp_{j,0},\hb_{i,1}] &= [a_{ij}]_{i} \, \xbp_{j,1},
    \\
    [\xbm_{j,-1},\hb_{i,1}] &= - [a_{ij}]_{i} \Cb \xbm_{j,0},
    \end{split}
    \begin{split}
    [\xbp_{j,1},\hb_{i,-1}] &= [a_{ij}]_{i} \Cb^{-1} \xbp_{j,0},
    \\
    [\xbm_{j,0},\hb_{i,-1}] &= - [a_{ij}]_{i} \xbm_{j,-1},
    \end{split}
\end{align}
whenever $i,j\in I_{0}$, as well as $[\hb_{1,r_{1}},\hb_{2,r_{2}}] = 0$ for all $r_{1},r_{2}\in\lbrace\pm 1\rbrace$.
The next two lemmas extend some of these identities to the $j=0$ case.

\begin{lem} \label{lem:hb2r on xbpm0m}
    In type $G_{2}^{(1)}$ we have $[\xbpm_{0,0},\hb_{2,\pm 1}] = \mp \xbpm_{0,\pm 1}$.
\end{lem}
\begin{proof}
    Both of these relations may be checked directly as follows.
    \begin{align*}
        [\xbp_{0,0},\hb_{2,1}]
        &= [\xbp_{0,0},[\xbm_{2,0},\xbp_{2,1}] \kb_{2}^{-1}]
        && \\
        &= [\xbp_{0,0},[\xbm_{2,0},\xbp_{2,1}]]_{q_{0}^{-1}} \kb_{2}^{-1}
        &&\text{by $\Rv$}
        \\
        &= [\xbm_{2,0},[\xbp_{0,0},\xbp_{2,1}]_{q_{0}^{-1}}] \kb_{2}^{-1}
        &&\text{by $\Rv$}
        \\
        &= - [\xbm_{2,0},[\xbp_{2,0},\xbp_{0,1}]_{q_{0}^{-1}}] \kb_{2}^{-1}
        &&\text{by \textbf{(\ref{eqn:simplified relations 8})}}
        \\
        &= - [[\xbm_{2,0},\xbp_{2,0}],\xbp_{0,1}]_{q_{0}^{-1}} \kb_{2}^{-1}
        &&\text{by $\Rdiag$}
        \\
        &= - (q_{2} - q_{2}^{-1})^{-1} [\kb_{2} - \kb_{2}^{-1},\xbp_{0,1}]_{q_{0}^{-1}} \kb_{2}^{-1}
        &&\text{by $\Rdiag$}
        \\
        &= - \xbp_{0,1}
        &&\text{by $\Rdiag$}
    \end{align*}
    \begin{align*}
        [\xbm_{0,0},\hb_{2,-1}]
        &= [\xbm_{0,0},[\xbm_{2,-1},\xbp_{2,0}] k_{2}]
        && \\
        &= [\xbm_{0,0},[\xbm_{2,-1},\xbp_{2,0}]]_{q_{0}^{-1}} \kb_{2}
        &&\text{by $\Rv$}
        \\
        &= [[\xbm_{0,0},\xbm_{2,-1}]_{q_{0}^{-1}},\xbp_{2,0}] \kb_{2}
        &&\text{by $\Rv$}
        \\
        &= - q_{0}^{-1} [[\xbm_{2,-1},\xbm_{0,0}]_{q_{0}},\xbp_{2,0}] \kb_{2}
        && \\
        &= q_{0}^{-1} [[\xbm_{0,-1},\xbm_{2,0}]_{q_{0}},\xbp_{2,0}] \kb_{2}
        &&\text{by \textbf{(\ref{eqn:simplified relations 8})}}
        \\
        &= q_{0}^{-1} [\xbm_{0,-1}, [\xbm_{2,0},\xbp_{2,0}] ]_{q_{0}} \kb_{2}
        &&\text{by $\Rdiag$}
        \\
        &= q_{0}^{-1} (q_{2} - q_{2}^{-1})^{-1} [\xbm_{0,-1},\kb_{2} - \kb_{2}^{-1}] ]_{q_{0}} \kb_{2}
        &&\text{by $\Rdiag$}
        \\
        &= \xbm_{0,-1}
        &&\text{by $\Rdiag$}
        \qedhere
    \end{align*}
\end{proof}

\begin{lem} \label{lem:hb1r on xbpm0m}
    In type $G_{2}^{(1)}$ we have $[\xbpm_{0,m},\hb_{1,r}] = 0$ for all $m = 0,\pm 1$ and $r\in\lbrace \pm 1\rbrace$.
\end{lem}
\begin{proof}
    First note that $\xbpm_{0,0}$ commutes with $\kb_{1}^{\pm 1}$ and $\xbpm_{1,0}$ by $\Rv$, and with $\xb^{\mp}_{1,\mp 1}$ due to already known relations from \textbf{(\ref{eqn:simplified relations 7})}, and therefore
    \begin{align*}
        [\xbp_{0,0},\hb_{1,-1}]
        = [[\xbm_{1,-1},\xbp_{1,0}] \kb_{1},\xbp_{0,0}]
        = 0,
        \qquad
        [\xbm_{0,0},\hb_{1,1}]
        = [\xbm_{0,0},[\xbm_{1,0},\xbp_{1,1}] \kb_{1}^{-1}]
        = 0.
    \end{align*}
    Furthermore, we have:
    \begin{align*}
        [\xbp_{0,0},\hb_{1,1}]
        &= [\xbp_{0,0},[\xbm_{1,0},\xbp_{1,1}] \kb_{1}^{-1}]
        && \\
        &= [\xbp_{0,0},[\xbm_{1,0},\xbp_{1,1} \kb_{1}^{-1}]_{q_{1}^{2}}]
        &&\text{by $\Rv$}
        \\
        &= [\xbm_{1,0},[\xbp_{0,0},\xbp_{1,1} \kb_{1}^{-1}] ]_{q_{1}^{2}}
        &&\text{by $\Rv$}
        \\
        &= -\Cb [\xbm_{1,0}, Y_{1}^{-1}Y_{2}T_{1} ([\xbp_{0,1},\xbm_{1,0}]) ]_{q_{1}^{2}}
        &&\text{by (\ref{lemma (1)}) and Lemma \ref{lem:Y on zero generators}}
        \\
        &= -\Cb [\xbm_{1,0}, Y_{1}^{-1}Y_{2}T_{1}(0) ]_{q_{1}^{2}}
        &&\text{by $\Rdiag$}
        \\
        &= 0
        &&
    \end{align*}
    \begin{align*}
        [\xbm_{0,0},\hb_{1,-1}]
        &= [\xbm_{0,0},[\xbm_{1,-1},\xbp_{1,0}] \kb_{1}]
        && \\
        &= [\xbm_{0,0},[\kb_{1} \xbm_{1,-1},\xbp_{1,0}]_{q_{1}^{-2}}]
        &&\text{by $\Rh$}
        \\
        &= [[\xbm_{0,0},\kb_{1} \xbm_{1,-1}],\xbp_{1,0} ]_{q_{1}^{-2}}
        &&\text{by $\Rv$}
        \\
        &= -\Cb^{-1} [Y_{1}^{-1}Y_{2}T_{1} ([\xbm_{0,-1},\xbp_{1,0}]),\xbp_{1,0}]_{q_{1}^{-2}}
        &&\text{by (\ref{lemma (1)}) and Lemma \ref{lem:Y on zero generators}}
        \\
        &= -\Cb^{-1} [Y_{1}^{-1}Y_{2}T_{1} (0),\xbp_{1,0}]_{q_{1}^{-2}}
        &&\text{by $\Rdiag$}
        \\
        &= 0
        &&
    \end{align*}
    We then swiftly deduce that
    $[\xbpm_{0,\pm 1},\hb_{1,\pm 1}]
    = \mp [[\xbpm_{0,0},\hb_{2,\pm 1}],\hb_{1,\pm 1}]
    = \mp [[\xbpm_{0,0},\hb_{1,\pm 1}],\hb_{2,\pm 1}]
    = 0$
    using the commutativity of each $\hb_{1,r_{1}}$ with $\hb_{2,r_{2}}$.
    The remaining identities require our results from the previous lemma:
    \begin{align*}
        [\xbp_{0,1},\hb_{1,-1}]
        &= [\xbp_{0,1},[\xbm_{1,-1},\xbp_{1,0}] \kb_{1}]
        && \\
        &= [\xbp_{0,1},[\xbm_{1,-1},\xbp_{1,0}]] \kb_{1}
        &&\text{by $\Rdiag$}
        \\
        &= [[\xbp_{0,1},\xbm_{1,-1}],\xbp_{1,0}] \kb_{1}
        &&\text{by $\Rdiag$}
        \\
        &= [[[\hb_{2,1},\xbp_{0,0}],\xbm_{1,-1}],\xbp_{1,0}] \kb_{1}
        &&\text{by Lemma \ref{lem:hb2r on xbpm0m}}
        \\
        &= [[\xbp_{0,0},[\xbm_{1,-1},\hb_{2,1}]],\xbp_{1,0}] \kb_{1}
        &&\text{by known relations in \textbf{(\ref{eqn:simplified relations 7})}}
        \\
        &= [\xbp_{0,0},[[\xbm_{1,-1},\hb_{2,1}],\xbp_{1,0}]] \kb_{1}
        &&\text{by $\Rv$}
        \\
        &= \Cb [\xbp_{0,0},[\xbm_{1,0},\xbp_{1,0}]] \kb_{1}
        &&\text{by (\ref{eqn:hbir on xbpmjm})}
        \\
        &= \Cb [\xbp_{0,0},0] \kb_{1}
        &&\text{by $\Rv$}
        \\
        &= 0
        &&
    \end{align*}
    \begin{align*}
        [\xbm_{0,-1},\hb_{1,1}]
        &= [\xbm_{0,-1},[\xbm_{1,0},\xbp_{1,1}] \kb_{1}^{-1}]
        && \\
        &= [\xbm_{0,-1},[\xbm_{1,0},\xbp_{1,1}\kb_{1}^{-1}]_{q_{1}^{2}} ]
        &&\text{by $\Rh$}
        \\
        &= [\xbm_{1,0},[\xbm_{0,-1},\xbp_{1,1}\kb_{1}^{-1}] ]_{q_{1}^{2}}
        &&\text{by $\Rdiag$}
        \\
        &= [\xbm_{1,0},[\xbm_{0,-1},\xbp_{1,1}]] \kb_{1}^{-1}
        &&\text{by $\Rdiag$}
        \\
        &= [\xbm_{1,0},[[\xbm_{0,0},\hb_{2,-1}],\xbp_{1,1}]] \kb_{1}^{-1}
        &&\text{by Lemma \ref{lem:hb2r on xbpm0m}}
        \\
        &= [\xbm_{1,0},[\xbm_{0,0},[\hb_{2,-1},\xbp_{1,1}]]] \kb_{1}^{-1}
        &&\text{by known relations in \textbf{(\ref{eqn:simplified relations 7})}}
        \\
        &= [\xbm_{0,0},[\xbm_{1,0},[\hb_{2,-1},\xbp_{1,1}]]] \kb_{1}^{-1}
        &&\text{by $\Rv$}
        \\
        &= \Cb^{-1} [\xbm_{0,0},[\xbm_{1,0},\xbp_{1,0}]] \kb_{1}^{-1}
        &&\text{by (\ref{eqn:hbir on xbpmjm})}
        \\
        &= \Cb^{-1} [\xbm_{0,0},0] \kb_{1}^{-1}
        &&\text{by $\Rv$}
        \\
        &= 0
        && \qedhere
    \end{align*}
\end{proof}

At long last, completing the proof of Theorem \ref{thm:psi} in type $G_{2}^{(1)}$ is now a manageable task.
In particular, the rest of \textbf{(\ref{eqn:simplified relations 7})} is obtained by applying
\begin{itemize}
    \item $\mathrm{ad}(h_{1,1})$ to
    $[\xbm_{0,-1},\xbp_{1,0}] = 0$ and
    $[\xbm_{0,0},\xbp_{2,0}] = 0$,
    \item $\mathrm{ad}(h_{1,-1})$ to
    $[\xbm_{1,0},\xbp_{0,1}] = 0$ and
    $[\xbm_{2,0},\xbp_{0,0}] = 0$,
\end{itemize}
using Lemma \ref{lem:hb1r on xbpm0m} and the identities (\ref{eqn:hbir on xbpmjm}).
Furthermore, the remaining affine $q$-Serre relations \textbf{(\ref{eqn:simplified relations 9})} come from applying $\mathrm{ad}(h_{1,\pm 1})$ to those with
$(y_{i},y_{j}) = (\xbpm_{0,0},\xbpm_{2,0}),(\xbpm_{2,0},\xbpm_{0,0})$, and we are done. \hfill $\qed$

\begin{rmk}
    In many types, our proof can be streamlined using (\ref{lemma (4)}).
    In particular, when $\lvert\Omega^{v}\rvert > 2$ all relations are obtained applying non-trivial $\rho_{i}$ to those with indices in $I_{0}$.
    Moreover if $\lvert\Omega^{v}\rvert = 2$ then applying these elements to relations either lying inside $\Udiag$ or with indices in $I_{0}$ reaches almost all other relations.
    Nevertheless, we have opted to detail the arguments above since they are effective in a more general situation.
\end{rmk}

\section{Tensor product representations} \label{section:tensor products}

Recall from Section \ref{subsection:quantum affinizations} the topological coproduct $\Delta_{u}$ and $\ell$-highest weight theory for quantum affinizations $\Uqaffs$.
It is easy to see that in general, $\Delta_{u}$ fails to produce a well-defined tensor product on modules in $\Oaff$.
Roughly speaking, this is because both $\Delta_{u}$ and the loop triangular decomposition for $\Uqaffs$ are \emph{infinite with respect to the vertical direction}.
As a consequence, $\mathrm{im}(\Delta_{u})$ contains infinite sums whose actions on various elements of a tensor product may not converge after specialising $u$.
\\

Let us provide some more details.
Suppose that $V$ is a $\Uqaffs$-module on which
$\langle q^{h} ~|~ h\in P^{\vee} \rangle$
acts semisimply, with finite dimensional weight spaces.
Then it is known -- see \cite{Hernandez07}*{Prop. 3.8} and \cite{GTL16}*{Prop. 3.6(ii)} -- that for all $i\in I$,
\begin{align*}
    \xp_{i}(z)^{\pm} = \pm \sum_{\pm m\geq 0} \xp_{i,m} z^{-m}, \qquad
    \xm_{i}(z)^{\pm} = \pm \sum_{\pm m\geq 0} \xm_{i,m} z^{-m}, \qquad
    \phi^{\pm}_{i}(z) = \pm \sum_{\pm r\geq 0} \phi^{\pm}_{i,r} z^{-r},
\end{align*}
each act on any $V_{\mu}$ by the expansions at $z^{\mp 1} = 0$ of certain rational functions.
Defining currents
\begin{align*}
    \xp_{i}(z) = \xp_{i}(z)^{+} - \xp_{i}(z)^{-}, \qquad
    \xm_{i}(z) = \xm_{i}(z)^{+} - \xm_{i}(z)^{-}, \qquad
    \phi_{i}(z) = \phi^{+}_{i}(z) - \phi^{-}_{i}(z),
\end{align*}
it is clear that $\Delta_{u}$ can be written as
\begin{align*}
    \xp_{i}(z) &\mapsto
    \xp_{i}(z) \otimes 1 +
    \phi^{+}_{i}(z) \otimes \xp_{i}(uz), \\
    \xm_{i}(z) &\mapsto
    1 \otimes \xm_{i}(uz) +
    \xm_{i}(z) \otimes \phi^{-}_{i}(uz), \\
    \phi^{\pm}_{i}(z) &\mapsto
    \phi^{\pm}_{i}(z) \otimes \phi^{\pm}_{i}(uz),
\end{align*}
working modulo $C^{\pm 1}$ for ease of notation.
Issues therefore arise when either $u$ or $1$ is a pole for one of the rational functions.
In particular, whereas for fixed representations $V\one$ and $V\two$ in $\Oaff$ we may pick some $u$ such that $\Delta_{u}$ defines a $\Uqaffs$-module structure on $V\one \otimes V\two$, it is not possible to produce in this way a well-defined tensor product on the category as a whole.
\\

However, in the special case of untwisted quantum toroidal algebras, we can overcome this problem by exploiting the horizontal--vertical symmetry afforded by our anti-involution $\psi$ from Theorem \ref{thm:psi}.
In particular, conjugating $\Delta_{u}$ by $\psi$ produces a topological coproduct which is instead \emph{infinite in the horizontal direction}, and gives rise to a well-defined tensor product on $\Oaff$.
In this way, we are able to endow the module category with a monoidal structure, and its Grothendieck group with the structure of a ring.
\\

Our tensor product is shown to satisfy a series of results that may be viewed as toroidal analogues of the highly influential works by Chari-Pressley for quantum affine algebras.
For example, there exists a compatibility with Drinfeld polynomials, the tensor product of irreducibles is generically irreducible, and all irreducibles are in some sense generated by a finite number of \emph{fundamental} modules.
\\

Furthermore, in Section \ref{section:R matrices} we prove the existence of $R$-matrices -- solutions to the Yang-Baxter equation in physics -- that act as intertwiners, exchanging the factors in tensor products of modules.
These $R$-matrices depend on a spectral parameter and are generically isomorphisms, thus equipping such products with a meromorphic braiding.

\begin{rmk}
    Let us briefly mention some of the existing works related to these directions.
    \begin{itemize}
        \item Hernandez \cites{Hernandez05,Hernandez07} takes a very different approach in order to define his \emph{fusion product}, constructing a much larger category in which the Drinfeld coproduct $\Delta_{u}$ does produce a tensor structure and then specializing back to $\Oaff$.
        \item Some work has been done for the particular case of $\UtorA$ by Miki \cites{Miki00,Miki01}, but conjugating with $\X_{0}^{-1} \Phi$ instead.
        We have chosen to use $\psi$ here since it acts more symmetrically with respect to the fine grading $\deg$ of $\Utor$ from Section \ref{subsubsection:Gradings and scaling automorphisms}.
    \end{itemize}
\end{rmk}

\begin{rmk}
    Our results extend naturally to quantum toroidal $\glone$, where they are in fact equivalent to \cite{Miki07}.
    We mention the connection here simply to frame this situation as a particular case of our more general programme.
\end{rmk}

Recall the $(Q\oplus\Zbb\delta')$--grading $\deg$ and associated decomposition (\ref{eqn:decomposition of Uqaffs}) of $\Utor$ from Section \ref{subsubsection:Gradings and scaling automorphisms}.
Just as $\delta \in Q$ is associated to the horizontal subalgebra $\Uh$, one can think of $\delta'$ as an imaginary root $\sum_{i\in I} a_{i} \alpha'_{i}$ for $\Uv$ where we identify $\alpha'_{i} = \alpha_{i}$ for each $i\in I_{0}$.
Then by considering the generating set
$\lbrace \xpm_{0,\pm 1},\, \xpm_{i,0},\, k_{i}^{\pm 1},\, C^{\pm 1} ~|~ i\in I \rbrace$
for $\Utor$, it is clear that
\begin{align} \label{eqn:psi on graded pieces}
    \psi : \U_{\beta + k\delta, \ell\delta'} \rightarrow \U_{\beta + \ell\delta, k\delta'}
\end{align}
for any $\beta\in\mathring{Q}$ and $k,\ell\in\Zbb$.
By conjugating $\Delta_{u}$ with $\psi$, we obtain a new (horizontally infinite) topological coproduct
\begin{align*}
    \Dpsi = (\psi \otimes \psi) \circ \Delta_{u} \circ \psi
\end{align*}
for $\Utor$.
Where does $\Dpsi$ send each graded piece $\U_{\beta + k\delta, \ell\delta'}$?
From (\ref{eqn:psi on graded pieces}) we have that $\psi$ sends elements of $\U_{\beta + k\delta, \ell\delta'}$ to elements of $\U_{\beta + \ell\delta, k\delta'}$, which can of course be expressed as polynomials in the $\xpm_{i,m}$, $h_{i,r}$, $k_{i}^{\pm 1}$ and $C^{\pm 1}$ generators.
Then using the formulae in Theorem \ref{thm:Damiani topological coproduct}, any such expression is mapped by $\Delta_{u}$ into
\begin{align*}
    \sum_{\substack{\mu\in\mathring{Q} \\ n\in\Zbb}}
    \sum_{r\in\Zbb}
    \left(
    \U_{\beta - \mu + (\ell-n)\delta,(k-r)\delta'}
    \otimes
    \U_{\mu + n\delta,r\delta'}
    \right)
    u^{-r}
\end{align*}
where the sum over $\mu$ and $n$ is finite, but the sum over $r$ may be infinite.
Finally, applying $\psi\otimes\psi$ gives
\begin{align} \label{eqn:image of new topological coproduct}
    \Dpsi : \U_{\beta + k\delta, \ell\delta'} \rightarrow
    \sum_{\substack{\mu\in\mathring{Q} \\ n\in\Zbb}}
    \sum_{r\in\Zbb}
    \left(
    \U_{\beta - \mu + (k-r)\delta,(\ell-n)\delta'}
    \otimes
    \U_{\mu + r\delta,n\delta'}
    \right)
    u^{-r}.
\end{align}
In particular, a quick check verifies that
\begin{align} \label{eqn:Dpsi on C and k}
    \Dpsi(C^{\pm 1}) = C^{\pm 1}\otimes C^{\pm 1},
    \qquad
    \Dpsi(k_{i}^{\pm 1}) = k_{i}^{\pm 1}\otimes k_{i}^{\pm 1}
    \quad(i\in I).
\end{align}

\subsection{Main results}

Let us now specialise the coproduct parameter $u$ to any non-zero complex number.
Our first result then shows that $\Dpsi$ gives rise to a well-defined tensor product on the category $\Oaff$.
Throughout this section, we shall therefore assume that $V\one$ and $V\two$ are representations of $\Utor$ lying inside $\Oaff$.

\begin{thm} \label{thm:integrability of tensor products}
    Our topological coproduct $\Dpsi$ endows the tensor product $V\one \otimes V\two$ with a well-defined, integrable $\Utor$-module structure such that $V\one \otimes V\two \in \Oaff$.
\end{thm}
\begin{proof}
Each $V\alphapower$ decomposes as a direct sum
$\bigoplus_{j=1}^{N\alphapower} \bigoplus_{\gamma\leq\lambda\alphapower_{j}} V\alphapower_{\gamma}$
of finite dimensional weight spaces for some $N\alphapower\in\Nbb$ and $\lambda\alphapower_{j} \in P$, so it follows from (\ref{eqn:action on weight spaces}) that
\begin{align*}
    \Big(
    \U_{\beta - \mu + (k-r)\delta,(\ell-n)\delta'}
    \otimes
    \U_{\mu + r\delta,n\delta'}
    \Big)
    \cdot
    \Big(
    V\one_{\gamma} \otimes V\two_{\tau}
    \Big)
    &\subset
    V\one_{\gamma + \beta - \mu + (k-r)\delta}
    \otimes
    V\two_{\tau + \mu + r\delta}
    \\[6pt]
    &=
    \begin{cases}
        V\one_{\gamma + \beta - \mu + (k-r)\delta}
        \otimes
        \lbrace 0 \rbrace
        &\mathrm{~for~} r \gg 0, \\[6pt]
        \lbrace 0 \rbrace
        \otimes
        V\two_{\tau + \mu + r\delta}
        &\mathrm{~for~} r \ll 0,
    \end{cases}
\end{align*}
is zero for $\lvert r\rvert \gg 0$.
Hence by (\ref{eqn:image of new topological coproduct}) every element of $\mathrm{im}(\Dpsi)$ has a well-defined action on $V\one\otimes V\two$.
Furthermore, as
$\Dpsi(k_{i}^{\pm 1}) = k_{i}^{\pm 1}\otimes k_{i}^{\pm 1}$
for all $i\in I$, each weight space
\begin{align*}
    (V\one\otimes V\two)_{\mu}
    =
    \sum_{\substack{1\leq j\leq N\one \\ 1\leq \ell\leq N\two}}
    \sum_{\substack{\gamma + \tau = \mu \\ \gamma\leq\lambda\one_{j} \\ \tau\leq\lambda\two_{\ell}}}
    V\one_{\gamma} \otimes V\two_{\tau}
\end{align*}
has only finitely many non-zero summands and is thus finite dimensional.
In particular, $(V\one\otimes V\two)_{\mu}$ is non-zero only if $\mu$ lies in
$\bigcup_{j=1}^{N\one} \bigcup_{\ell=1}^{N\two} (\lambda\one_{j} + \lambda\two_{\ell} - Q^{+})$
and our proof is complete.
\end{proof}

\begin{rmk}
    If $V\one$ and $V\two$ are moreover type $1$ representations, then so is $V\one\otimes V\two$ by (\ref{eqn:Dpsi on C and k}).
\end{rmk}

The following lemma shows how to factorise certain vector subspaces of these tensor modules, and is fundamental to later proofs.
As in the proof above, suppose that the weights of each $V\alphapower$ are contained in some
$\bigcup_{j=1}^{N\alphapower} (\lambda\alphapower_{j} - Q^{+})$.

\begin{lem} \label{lem:decomposing submodules of tensor products}
    As vector spaces, $(V\one \otimes V\two)(J) = V\one(J) \otimes V\two(J)$ for any $J\subset I$.
\end{lem}
\begin{proof}
    For each $\mu \in Q(J)^{+}$ we have that
    \begin{align*}
        \bigoplus_{\substack{1\leq j\leq N\one \\ 1\leq \ell\leq N\two}}
        (V\one \otimes V\two)_{\lambda\one_{j} + \lambda\two_{\ell} - \mu}
        &=
        \bigoplus_{\substack{1\leq j\leq N\one \\ 1\leq \ell\leq N\two}}
        \bigoplus_{\substack{\mu\one,\,\mu\two\in Q^{+} \\ \mu\one + \mu\two = \mu}}
        V\one_{\lambda\one_{j} - \mu\one} \otimes V\two_{\lambda\two_{\ell} - \mu\two}
        \\
        &=
        \bigoplus_{\substack{1\leq j\leq N\one \\ 1\leq \ell\leq N\two}}
        \bigoplus_{\substack{\mu\one,\,\mu\two\in Q(J)^{+} \\ \mu\one + \mu\two = \mu}}
        V\one_{\lambda\one_{j} - \mu\one} \otimes V\two_{\lambda\two_{\ell} - \mu\two}
    \end{align*}
    where the first equality comes from (\ref{eqn:Dpsi on C and k}).
    Then by summing over all $\mu$ we are done.
\end{proof}

Our next result demonstrates that the tensor product of $\ell$-highest weight vectors is again an $\ell$-highest weight vector, with Drinfeld polynomials equal to the product of those for its factors.

\begin{thm} \label{thm:l-highest weight vector in tensor products}
    Suppose that $v\one\in V\one$ and $v\two\in V\two$ are $\ell$-highest weight vectors with Drinfeld polynomials $\Pcal\one(z)$ and $\Pcal\two(z)$ respectively.
    Then $v\one \otimes v\two$ is $\ell$-highest weight inside $V\one \otimes V\two$ with Drinfeld polynomials $\Pcal\one(z)\Pcal\two(z)$.
\end{thm}
\begin{proof}
Our strategy is as follows:
\begin{enumerate}
    \item Consider the action
    $\Udash \cong \Uv \curvearrowright (V\one \otimes V\two)(I_{0})$
    obtained by restricting
    $\Utor \curvearrowright V\one \otimes V\two$.
    \item Show that this coincides with the action
    $\Udash \cong \Uv \curvearrowright V\one(I_{0}) \otimes V\two(I_{0})$
    defined using the coproduct $\Dbarplus$ (Proposition \ref{prop:actions of Udash coincide}).
    \item Deduce from results of Chari-Pressley for quantum affine algebras that $v\one \otimes v\two$ is an $\ell$-highest weight vector inside this module, with Drinfeld polynomials
    $\big(P\one_{i}(z) P\two_{i}(z)\big)_{i\in I_{0}}$.
    \item Prove via direct computations that $v\one \otimes v\two$ is an $\ell$-highest weight vector of the representation
    $\U(0) \curvearrowright (V\one \otimes V\two)(0)$,
    with Drinfeld polynomials $P\one_{0}(z)P\two_{0}(z)$
    (Corollaries \ref{cor:xp0m annihilate v1xv2} and \ref{cor:0 Drinfeld polynomials}).
    \item Combine these results to complete the proof.
    \qedhere
\end{enumerate}
\end{proof}

\begin{cor} \label{cor:irreducible switch factors}
    If $V\one \otimes V\two$ is irreducible, then it is isomorphic to $V\two \otimes V\one$.
\end{cor}
\begin{proof}
    The irreducibility assumption ensures that
    $\big(\!\dim_{\Cbb}(V\one \otimes V\two)_{\nu} \, | \, \nu \in P\big)$
    is \emph{strictly minimal} over all $\Utor$-modules $V$ containing an $\ell$-highest weight vector with Drinfeld polynomials $\Pcal\one(z)\Pcal\two(z)$.
    Namely, such $V$ have
    $\dim_{\Cbb}V_{\nu} \geq \dim_{\Cbb}(V\one \otimes V\two)_{\nu}$
    for each $\nu \in P$, and at least one inequality is strict whenever $V$ is reducible -- this is because $V$ must contain a subquotient isomorphic to $V\one \otimes V\two$.
    (Note that Theorem \ref{thm:integrability of tensor products} implies that all
    $\dim_{\Cbb}(V\one \otimes V\two)_{\nu}$
    are finite.)
    But
    \begin{align*}
        (V\alphapower \otimes V\betapower)_{\nu}
        = \bigoplus_{\nu\alphapower + \nu\betapower = \nu}
        V\alphapower_{\nu\alphapower} \otimes V\betapower_{\nu\betapower}
    \end{align*}
    by (\ref{eqn:Dpsi on C and k}), so every
    $\dim_{\Cbb}(V\one \otimes V\two)_{\nu}
    = \dim_{\Cbb}(V\two \otimes V\one)_{\nu}$.
    Since $V\two \otimes V\one$ moreover contains an $\ell$-highest weight vector with Drinfeld polynomials $\Pcal\one(z)\Pcal\two(z)$, it must also be irreducible and thus isomorphic to $V\one \otimes V\two$.
\end{proof}

The next theorem demonstrates that \emph{generically}, a tensor product of irreducible representations is itself irreducible.

\begin{notation}
    For any $a\in\Cbb^{\times}$ and $\Utor$-module $V$, we shall write $V_{a}$ for the twist of $V$ by the scaling automorphism $\sfrakv_{a}$ from Section \ref{subsubsection:Gradings and scaling automorphisms}.
\end{notation}

Recall from Lemma \ref{lem:twisting by scaling automorphisms} that twisting with $\sfrakv_{a}$ acts on Drinfeld polynomials via $z \mapsto a^{\hslash} z$.

\begin{thm} \label{thm:irreducible except countable}
    If $V\one$ and $V\two$ are irreducible, then the tensor product $V\one_{a} \otimes V\two_{b}$ is irreducible for all but countably many $\frac{b}{a}\in\Cbb^{\times}$.
\end{thm}

Since the proof of this result is rather technical, we defer it to Section \ref{subsection:Proof of generic irreducibility}.
\\

For each $j\in I$ and $a\in \Cbb^{\times}$, define the associated \emph{fundamental representation} $V(\lambda_{j},a)$ of $\Utor$ to be the irreducible integrable $\ell$-highest weight module with Drinfeld polynomials $( (1-u/a)^{\delta_{ij}})_{i\in I}$.

\begin{rmk}
    In type $A$ these fundamental modules are precisely the Fock space representations, as constructed in \cites{FJMM13,STU98,Tsymbaliuk19,VV98}.
\end{rmk}

\begin{cor} \label{cor:subquotient of tensor product}
    Every irreducible integrable $\ell$-highest weight representation is isomorphic to a subquotient of a tensor product of fundamental representations.
\end{cor}
\begin{proof}
    Take such a module $V(\Pcal(z))$ where $\Pcal(z) = (P_{j}(z))_{j\in I}$, and denote by 
    $a_{j,1},\dots,a_{j,\deg(P_{j})}$
    the roots of each $P_{j}(z)$ including multiplicities.
    Consider the tensor product
    $\bigotimes_{j\in I} \bigotimes_{k=1}^{\deg(P_{j})} V(\lambda_{j},a_{j,k})$ with respect to our coproduct $\Dpsi$.
    By Theorem \ref{thm:l-highest weight vector in tensor products}, this contains an $\ell$-highest weight vector with Drinfeld polynomials $\Pcal(z)$.
    Then $V(\Pcal(z))$ is isomorphic to a quotient of the submodule generated by this vector.
\end{proof}

See Corollaries 12.1.13 and 12.2.8 of \cite{CP94} for the corresponding Yangian and quantum affine results.
Let us also remark that while Hernandez' fusion product is constructed in an entirely different way to our tensor product on $\Oaff$, it nevertheless enjoys a similar property \cite{Hernandez07}*{Prop. 6.1}.

\begin{notation}
    Throughout the rest of this paper we may assume without loss of generality that the coproduct parameter $u$ is specialised to $1$.
    We shall write $\Delta^{\psi}$ as shorthand for $\Delta^{\psi}_{1}$ in this case.
\end{notation}
\begin{proof}
First note that since $\sfrak^{(0)}_{u}$ fixes $q^{h}$ and $h_{i,r}$, and moreover scales every $\xp_{i,m}$, twisting a $\Utor$-module $V$ by $\sfrak^{(0)}_{u}$ preserves
\begin{itemize}
    \item a vector $v\in V$ being $\ell$-highest weight,
    \item the $\ell$-weight and Drinfeld polynomials of such $v$,
    \item the irreducibility of $V$,
\end{itemize}
and thus the assumptions of each result in this section.
Then as
$\Delta^{\psi}_{1} = (1\otimes \sfrak^{(0)}_{u}) \circ \Dpsi$
by equation (\ref{eqn:image of new topological coproduct}), we are done.
\end{proof}

Perhaps it is worth indicating why our results in Sections \ref{section:tensor products} and \ref{section:R matrices} relate $\Dpsi$ with $\Dbarplus$, even though Remark \ref{rmk:affinizations of coproducts} presents $\Delta_{u}$ as the `affinization' of $\Delta_{+}$ instead.
This is explained by the commutativity of the following diagram for $\Uv$, and similarly for the other subalgebras of $\Utor$ considered in Sections \ref{subsection:action of Uv} and \ref{subsection:action of U0},
together with the fact that
$\Dbarplus = (\sigma\otimes\sigma) \circ \Delta_{+} \circ \sigma$.
\[\begin{tikzcd}
	{\U_{\beta,\ell\delta'}} & {\U_{\beta+\ell\delta,0}} & {\displaystyle \sum_{\mu,n,r} \U_{\beta-\mu+(\ell-n)\delta,-r\delta'} \otimes \U_{\mu+n\delta,r\delta'}} && {\displaystyle \sum_{\mu,n,r} \U_{\beta-\mu-r\delta,(\ell-n)\delta'} \otimes \U_{\mu+r\delta,n\delta'}} \\
	\\
	{\U_{\beta,\ell\delta'}} & {\U_{\beta+\ell\delta,0}} & {\displaystyle \sum_{\mu,n} \U_{\beta-\mu+(\ell-n)\delta,0} \otimes \U_{\mu+n\delta,0}} && {\displaystyle \sum_{\mu,n} \U_{\beta-\mu,(\ell-n)\delta'} \otimes \U_{\mu,n\delta'}}
	\arrow["\psi", from=1-1, to=1-2]
	\arrow["{\rotatebox{90}{$=$}}", from=1-1, to=3-1]
	\arrow["{\Delta_{1}}", from=1-2, to=1-3]
	\arrow["{\rotatebox{90}{$=$}}", from=1-2, to=3-2]
	\arrow["{\psi\otimes\psi}", from=1-3, to=1-5]
	\arrow["\begin{array}{c} \substack{\text{project to} \\ \Uh\otimes\Uh} \end{array}", from=1-3, to=3-3]
	\arrow["\begin{array}{c} \substack{\text{project to} \\ \Uv\otimes\Uv} \end{array}", from=1-5, to=3-5]
	\arrow["{h \sigma v^{-1}}", from=3-1, to=3-2]
	\arrow["{h \Delta_{+} h^{-1}}", from=3-2, to=3-3]
	\arrow["{(v \sigma h^{-1}) \otimes (v \sigma h^{-1})}", from=3-3, to=3-5]
\end{tikzcd}\]

\begin{rmk}
    \begin{itemize}
        \item Of course, analogous results involving the other topological coproducts for $\Utor$ mentioned in Remark \ref{rmk:affinizations of coproducts} are obtained by conjugating with $\Wcal$, $\eta$ and $\Wcal\eta$.
        \item Furthermore, since vertex representations \cite{Jing98(2)} can be obtained from elements of $\Oaff$ by twisting with a horizontal--vertical symmetry such as $\Phi$, our work implies that $\Delta_{u}$ leads to a well-defined tensor product on these modules.
    \end{itemize}
\end{rmk}

\subsection{Action of vertical subalgebras on tensor products} \label{subsection:action of Uv}

Here we consider the action on $V\one \otimes V\two$ of each `vertical' quantum affine subalgebra $\U(I_{i}) \cong \Udash$ with $i\in\Imin$, noting in particular that $\Uv$ occurs as a special case.
To this end, let us fix some $i\in\Imin$ and define $f : \Udash \rightarrow \Utor$ to be the composition $(X_{i}\pi_{i}) \circ v$.
\\

On the one hand, we can pull back the action of $\Utor$ along $f$ to define an action of $\Udash$ on each $V\alphapower$, restrict to the submodules $V\alphapower(I_{i})$, and then take the tensor product with respect to $\Dbarplus$.
On the other hand, the pullback of
$\Utor \curvearrowright V\one \otimes V\two$
to $\Udash$ along $f$ contains $(V\one \otimes V\two)(\Ii)$ as a submodule.

\begin{prop} \label{prop:actions of Udash coincide}
    The representations of $\Udash$ on $(V\one \otimes V\two)(\Ii) = V\one(\Ii) \otimes V\two(\Ii)$ defined by $\Delta^{\psi} \circ f$ and $(f \otimes f) \circ \Dbarplus$ are isomorphic, via the identity map from Lemma \ref{lem:decomposing submodules of tensor products}.
\end{prop}
\begin{proof}
It suffices to show that $\Delta^{\psi} \circ f(z)$ and $(f \otimes f) \circ \Dbarplus(z)$ act on $V\one(\Ii) \otimes V\two(\Ii)$ in the same way whenever $z \in \lbrace \xjpm,\, k_{j}^{\pm 1} \, | \, j\in I\rbrace$.
To this end, we first calculate the images of $(f \otimes f) \circ \Dbarplus$ on each generator of $\Udash$.
For any $i\in \Imin\setminus\lbrace 0\rbrace$ we have
\begin{alignat*}{5}
    \xp_{j} &\mapsto \xp_{\pi_{i}(j),0} \otimes 1 + k_{\pi_{i}(j)}^{-1} \otimes \xp_{\pi_{i}(j),0}
    \quad
    &&\xm_{j} &&\mapsto 1 \otimes \xm_{\pi_{i}(j),0} + \xm_{\pi_{i}(j),0} \otimes k_{\pi_{i}(j)}
    \quad
    &&k_{j}^{\pm 1} &&\mapsto k_{\pi_{i}(j)}^{\pm 1} \otimes k_{\pi_{i}(j)}^{\pm 1}
    \\
    \xp_{0} &\mapsto o(i) (\xbp_{i,-1} \otimes 1 + k_{\delta}k_{i}^{-1} \otimes \xbp_{i,-1})
    \quad
    &&\xm_{0} &&\mapsto o(i) (1 \otimes \xbm_{i,1} + \xbm_{i,1} \otimes k_{\delta}^{-1}k_{i})
    \quad
    &&k_{0}^{\pm 1} &&\mapsto (k_{\delta}^{-1}k_{i})^{\pm 1} \otimes (k_{\delta}^{-1}k_{i})^{\pm 1}
    \\
    \xp_{i^{*}} &\mapsto o(0) (\xp_{0,1} \otimes 1 + (C k_{0})^{-1} \otimes \xp_{0,1})
    \quad
    &&\xm_{i^{*}} &&\mapsto o(0) (1 \otimes \xm_{0,-1} + \xm_{0,-1} \otimes C k_{0})
    \quad
    &&k_{i^{*}}^{\pm 1} &&\mapsto (C k_{0})^{\pm 1} \otimes (C k_{0})^{\pm 1}
\end{alignat*}
where $j \not\in \lbrace i^{*},0\rbrace$, whereas if $i = 0$ these are replaced by
\begin{alignat*}{5}
    \xp_{j} &\mapsto \xp_{j,0} \otimes 1 + k_{j}^{-1} \otimes \xp_{j,0}
    \quad
    &&\xm_{j} &&\mapsto 1 \otimes \xm_{j,0} + \xm_{j,0} \otimes k_{j}
    \quad
    &&k_{j}^{\pm 1} &&\mapsto k_{j}^{\pm 1} \otimes k_{j}^{\pm 1}
    \\
    \xp_{0} &\mapsto \xbp_{0,0} \otimes 1 + C^{-1}k_{\theta} \otimes \xbp_{0,0}
    \quad
    &&\xm_{0} &&\mapsto 1 \otimes \xbm_{0,0} + \xbm_{0,0} \otimes C k_{\theta}^{-1}
    \quad
    &&k_{0}^{\pm 1} &&\mapsto (C k_{\theta}^{-1})^{\pm 1} \otimes (C k_{\theta}^{-1})^{\pm 1}
\end{alignat*}
for each $j\not= 0$, using the fact that
$f = \psi \circ (X_{i}\pi_{i}) \circ h\sigma$
by Proposition \ref{prop:compatibilities} since $\te$ fixes all $X_{i}\pi_{i}$.
On the other hand, it is clear that $\Delta^{\psi} \circ f$ maps
$\lbrace k_{j}^{\pm 1} \, | \, j\in I \rbrace$
exactly as above in each case.
Moreover if $i\not= 0$ then
\begin{alignat*}{1}
    \xp_{j} &\mapsto \xp_{\pi_{i}(j),0} \otimes 1 + k_{\pi_{i}(j)}^{-1} \otimes \xp_{\pi_{i}(j),0} + \sum_{\ell>0} k_{\delta}^{\ell} \boldsymbol{\phi}^{+}_{\pi_{i}(j),\ell} \otimes \xbp_{\pi_{i}(j),-\ell}
    \\
    \xm_{j} &\mapsto 1 \otimes \xm_{\pi_{i}(j),0} + \xm_{\pi_{i}(j),0} \otimes k_{\pi_{i}(j)} + \sum_{\ell<0} \xbm_{\pi_{i}(j),-\ell} \otimes k_{\delta}^{\ell} \boldsymbol{\phi}^{-}_{\pi_{i}(j),\ell}
    \\
    \xp_{0} &\mapsto o(i) \left(\xbp_{i,-1} \otimes 1 + k_{\delta}k_{i}^{-1} \otimes \xbp_{i,-1} + \sum_{\ell>0} k_{\delta}^{\ell+1} \boldsymbol{\phi}^{+}_{i,\ell} \otimes \xbp_{i,-\ell-1} \right)
    \\
    \xm_{0} &\mapsto o(i) \left( 1 \otimes \xbm_{i,1} + \xbm_{i,1} \otimes k_{\delta}^{-1}k_{i} + \sum_{\ell<0} \xbm_{i,1-\ell} \otimes k_{\delta}^{\ell-1} \boldsymbol{\phi}^{-}_{i,\ell} \right)
    \\
    \xp_{i^{*}} &\mapsto o(0) \left(\xp_{0,1} \otimes 1 + (C k_{0})^{-1} \otimes \xp_{0,1} + \sum_{\ell>0} k_{\delta}^{\ell-1} \boldsymbol{\phi}^{+}_{0,\ell} \otimes \xbp_{0,1-\ell} \right)
    \\
    \xm_{i^{*}} &\mapsto o(0) \left( 1 \otimes \xm_{0,-1} + \xm_{0,-1} \otimes C k_{0} + \sum_{\ell<0} \xbm_{0,-\ell-1} \otimes k_{\delta}^{\ell+1} \boldsymbol{\phi}^{-}_{0,\ell} \right)
\end{alignat*}
where $j \not\in \lbrace i^{*},0\rbrace$, while for $i = 0$ we instead have
\begin{alignat*}{1}
    \xp_{j} &\mapsto \xp_{j,0} \otimes 1 + k_{j}^{-1} \otimes \xp_{j,0} + \sum_{\ell>0} k_{\delta}^{\ell} \boldsymbol{\phi}^{+}_{j,\ell} \otimes \xbp_{j,-\ell}
    \\
    \xm_{j} &\mapsto 1 \otimes \xm_{j,0} + \xm_{j,0} \otimes k_{j} + \sum_{\ell<0} \xbm_{j,-\ell} \otimes k_{\delta}^{\ell} \boldsymbol{\phi}^{-}_{j,\ell}
    \\
    \xp_{0} &\mapsto \xbp_{0,0} \otimes 1 + C^{-1}k_{\theta} \otimes \xbp_{0,0} + \sum_{\ell>0} k_{\delta}^{\ell} \boldsymbol{\phi}^{+}_{0,\ell} \otimes \xbp_{0,-\ell}
    \\
    \xm_{0} &\mapsto 1 \otimes \xbm_{0,0} + \xbm_{0,0} \otimes C k_{\theta}^{-1} + \sum_{\ell<0} \xbm_{0,-\ell} \otimes k_{\delta}^{\ell} \boldsymbol{\phi}^{-}_{0,\ell}
\end{alignat*}
where $j\not= 0$, again using the identity
$f = \psi \circ (X_{i}\pi_{i}) \circ h\sigma$
for $\xpm_{0}$.
Since $\boldsymbol{\phi}^{\pm}_{j,\ell} \in \U_{\ell\delta,0}$
and $\xbpm_{j,\ell} \in \U_{\pm \alpha_{j} + (\ell \mp \delta_{j0}) \delta, \pm \delta_{j0} \delta'}$
by (\ref{eqn:psi on graded pieces}), we can deduce from
$\U_{\beta+k\delta,\ell\delta'} \cdot V\alphapower_{\mu} \subset V\alphapower_{\mu+\beta+k\delta}$
that each of the sums above must act by zero on $V\one(\Ii) \otimes V\two(\Ii)$ for any $i\in\Imin$, whereby our proof is complete.
\end{proof}

What can we say about the action of $\Udash$ on each $V\alphapower(\Ii)$?
First note that Table \ref{tab:sign functions and outer automorphisms} contains the values of $o_{j,\pi_{i}(j)}$ for all $j\in I$, and since these are independent of $j$ we may denote the common value by $o(\pi_{i})$.
\renewcommand{\arraystretch}{1.4}
\begin{table}[H]
\centering
\begin{tabular}{|c||c|c|c|c|c|c|c|c|c|}
    \hline
    Type & $A_{n}^{(1)}$ & $B_{n}^{(1)}$ & $C_{n}^{(1)}$ & $D_{n}^{(1)}$ & $E_{6}^{(1)}$ & $E_{7}^{(1)}$ & $E_{8}^{(1)}$ & $F_{4}^{(1)}$ & $G_{2}^{(1)}$ \\
    \hline
     & & & & & & & & & \\[-17pt]
    \hline
    $o_{j,\pi_{i}(j)}$ & $(-1)^{i}$ & $1$ & $(-1)^{ni}$ & $(-1)^{n \cdot \mathds{1}_{i>1}}$ & $1$ & $1$ & $1$ & $1$ & $1$ \\
    \hline
\end{tabular}
\caption{\hspace{.5em}Values of $o_{j,\pi_{i}(j)}$ for each $j\in I$ and $i\in\Imin$ in untwisted types}
\label{tab:sign functions and outer automorphisms}
\end{table}
\renewcommand{\arraystretch}{1}
The images under $f = (X_{i}\pi_{i}) \circ v$ of the Drinfeld new generators for $\Udash$ are therefore as follows:
\begin{align} \label{eqn:image of f}
    \xpm_{j,m} \mapsto o(0)^{\delta_{\pi_{i}(j),0}} o(\pi_{i})^{m} \xpm_{\pi_{i}(j),m \pm \delta_{\pi_{i}(j),0}}
    \qquad
    h_{j,r} \mapsto o(\pi_{i})^{r} h_{\pi_{i}(j),r}
    \qquad
    k_{j} \mapsto C^{\delta_{\pi_{i}(j),0}} k_{\pi_{i}(j)}
\end{align}
So from the definition of $V\alphapower$ we see that $\Udash \curvearrowright V\alphapower(\Ii)$ contains $v\alphapower$ as an $\ell$-highest weight vector with Drinfeld polynomials
$\big(P\alphapower_{\pi_{i}(j)}(o(\pi_{i})z)\big)_{j\in I_{0}}$.
\\

It then follows from results of Chari-Pressley \cite{CP94}*{Thm. 12.2.6} on the affine level\footnote{Technically, Chari-Pressley \cite{CP94} consider the alternative coproduct $\overline{\Delta}_{-}$.
However one obtains a corresponding result for $\Dbarplus$ via essentially the same proof.}
that $v\one \otimes v\two$ is an $\ell$-highest weight vector for the representation
$\Udash \curvearrowright V\one(\Ii) \otimes V\two(\Ii)$,
with Drinfeld polynomials
$\big(P\one_{\pi_{i}(j)}(o(\pi_{i})z)P\two_{\pi_{i}(j)}(o(\pi_{i})z)\big)_{j\in I_{0}}$.
Hence by Proposition \ref{prop:actions of Udash coincide} the same is true for the action of $\Udash$ on $(V\one \otimes V\two)(\Ii)$ defined via $\Delta^{\psi} \circ f$.
\\

Using (\ref{eqn:image of f}) we can deduce that
$\U(\Ii) \curvearrowright (V\one \otimes V\two)(\Ii)$
also contains $v\one \otimes v\two$ as an $\ell$-highest weight vector, but with Drinfeld polynomials
$\big(P\one_{j}(z)P\two_{j}(z)\big)_{j\in \Ii}$
instead.

\begin{rmk}
    In all cases with $\lvert\Imin\rvert>1$, by tying these results together for different $i\in\Imin$ it immediately follows that $v\one \otimes v\two$ is an $\ell$-highest weight vector of
    $\Utor \curvearrowright V\one \otimes V\two$
    with Drinfeld polynomials
    $\Pcal\one(z)\Pcal\two(z)$.
    This completes our proof of Theorem \ref{thm:l-highest weight vector in tensor products} in types $A_{n}^{(1)}$, $B_{n}^{(1)}$, $C_{n}^{(1)}$, $D_{n}^{(1)}$, $E_{6}^{(1)}$ and $E_{7}^{(1)}$.
    Extending to $E_{8}^{(1)}$ and $F_{4}^{(1)}$ -- and indeed, providing a uniform proof -- requires a more detailed consideration of the action of $\U(0)$ as in Section \ref{subsection:action of U0}.
\end{rmk}

\subsection{Action of remaining generators on tensor products} \label{subsection:action of U0}

\begin{notation}
    Throughout this subsection we shall write $\xi_{\beta+k\delta,\ell\delta'}$ for an arbitrary element of $\U_{\beta+k\delta,\ell\delta'}$.
\end{notation}

\begin{prop} \label{prop:action of xpm00}
    $\xpm_{0,0}$ acts on $(V\one \otimes V\two)(0)$ by
    $\xp_{0,0}\otimes k_{0}^{-2} + k_{0}^{-1}\otimes\xp_{0,0}$
    and
    $\xm_{0,0}\otimes k_{0} + k_{0}^{2}\otimes\xm_{0,0}$
    respectively.
\end{prop}
\begin{proof}
From Jing's isomorphism we have
\begin{align*}
    \psi(\xp_{0,0})
    = v\sigma(\xp_{0})
    = v(\xp_{0})
    = \big[\xm_{i_{h-1},0},\dots,\xm_{i_{2},0},\xm_{i_{1},1}\big]_{q^{\epsilon_{1}}\dots q^{\epsilon_{h-2}}} C k_{\theta}^{-1},
\end{align*}
which is in turn sent by
$(\psi\otimes\psi)\circ\Delta_{1}$
to
\begin{align*}
    \textstyle
    (k_{0}^{-1} \otimes k_{0}^{-1})
    \Big[ \,
    1\otimes\xbm_{i_{1},1} + \xbm_{i_{1},1}\otimes k_{\delta}^{-1} k_{i_{1}}
    + \sum_{\ell<0} \xi_{-\alpha_{i_{1}}+(1-\ell)\delta,0}\otimes \xi_{\ell\delta,0}& \, ,
    \\
    \textstyle
    1\otimes\xm_{i_{2},0} + \xm_{i_{2},0}\otimes k_{i_{2}}
    + \sum_{\ell<0} \xi_{-\alpha_{i_{2}}-\ell\delta,0}\otimes \xi_{\ell\delta,0}& \, ,
    \\
    \textstyle
    \dots,
    1\otimes\xm_{i_{h-1},0} + \xm_{i_{h-1},0}\otimes k_{i_{h-1}}
    + \sum_{\ell<0} \xi_{-\alpha_{i_{h-1}}-\ell\delta,0}\otimes \xi_{\ell\delta,0}&
    \, \Big]'_{q^{\epsilon_{h-2}}\dots q^{\epsilon_{1}}}.
\end{align*}
Expanding out all sums and brackets, each summand lies inside
\begin{align*}
    \U_{-\sum_{j\in J} \alpha_{i_{j}}
    + (\mathds{1}_{1\in J} - \sum_{j\in J} \ell_{j}) \delta,
    0}
    \otimes
    \U_{-\sum_{j\not\in J} \alpha_{i_{j}}
    + (\mathds{1}_{1\not\in J} + \sum_{j\in J} \ell_{j}) \delta,
    0}
\end{align*}
where $J\subset [h-1]$ is the set of $j$ for which $1\otimes \xbm_{i_{j},\delta_{j1}}$ is \emph{not} a factor, and in this case $l_{j}\leq 0$ is the index of the factor chosen instead.
Since $\sum_{j=1}^{h-1} \alpha_{i_{j}} = \theta$, all summands except those with
\begin{itemize}
    \item $J = [h-1]$ and all $l_{j} = 0$, which lie in
    $\U_{\alpha_{0},0} \otimes \U_{0,0}$,
    \item $J = \emptyset$, which lie in
    $\U_{0,0} \otimes \U_{\alpha_{0},0}$,
\end{itemize}
map non-zero vectors in $(V\one \otimes V\two)(0)$ outside $(V\one \otimes V\two)(0)$, and hence their actions on $(V\one \otimes V\two)(0)$ must cancel.
Moreover, the summands in these two cases add up respectively to
\begin{itemize}
    \item 
    $(k_{0}^{-1} \otimes k_{0}^{-1}) \cdot
        \big(\big[\xbm_{i_{1},1}, \xm_{i_{2},0}, \dots, \xm_{i_{h-1},0}\big]'_{q^{\epsilon_{h-2}}\dots q^{\epsilon_{1}}}
        \otimes k_{\delta}^{-1} k_{\theta} \big)
        \\
        =
        h\sigma \big(\big[\xm_{i_{h-1},0},\dots,\xm_{i_{2},0},\xm_{i_{1},1}\big]_{q^{\epsilon_{1}}\dots q^{\epsilon_{h-2}}} C k_{\theta}^{-1} \big)
        \otimes k_{0}^{-2}
        \\
        = h\sigma(\xp_{0}) \otimes k_{0}^{-2}
        \\
        = \xp_{0,0} \otimes k_{0}^{-2}$,
    \item 
    $(k_{0}^{-1} \otimes k_{0}^{-1}) \cdot
        \big( 1 \otimes \big[\xbm_{i_{1},1}, \xm_{i_{2},0}, \dots, \xm_{i_{h-1},0}\big]'_{q^{\epsilon_{h-2}}\dots q^{\epsilon_{1}}} \big)
        \\
        =
        k_{0}^{-1} \otimes
        h\sigma \big(\big[\xm_{i_{h-1},0},\dots,\xm_{i_{2},0},\xm_{i_{1},1}\big]_{q^{\epsilon_{1}}\dots q^{\epsilon_{h-2}}} C k_{\theta}^{-1} \big)
        \\
        = k_{0}^{-1} \otimes h\sigma(\xp_{0})
        \\
        = k_{0}^{-1} \otimes \xp_{0,0}$,
\end{itemize}
and hence $\xp_{0,0}$ acts on $(V\one \otimes V\two)(0)$ by
$\xp_{0,0}\otimes k_{0}^{-2} + k_{0}^{-1}\otimes\xp_{0,0}$.
Similarly, we that have
\begin{align*}
    \psi(\xm_{0,0})
    = v\sigma(\xm_{0})
    = v(\xm_{0})
    = a (-q)^{-\epsilon} C^{-1} k_{\theta}
    \big[\xp_{i_{h-1},0},\dots,\xp_{i_{2},0},\xp_{i_{1},-1}\big]_{q^{\epsilon_{1}}\dots q^{\epsilon_{h-2}}}
\end{align*}
is mapped by
$(\psi\otimes\psi)\circ\Delta_{1}$
to
\begin{align*}
    \textstyle
    a (-q)^{-\epsilon}
    \Big[ \,
    \xbp_{i_{1},-1}\otimes 1 + k_{\delta} k_{i_{1}}^{-1} \otimes \xbp_{i_{1},-1}
    + \sum_{\ell>0} \xi_{\ell\delta,0} \otimes \xi_{\alpha_{i_{1}}-(1+\ell)\delta,0}& \, ,
    \\
    \textstyle
    \xp_{i_{2},0}\otimes 1 + k_{i_{2}}^{-1} \otimes \xp_{i_{2},0}
    + \sum_{\ell>0} \xi_{\ell\delta,0} \otimes \xi_{\alpha_{i_{2}}-\ell\delta,0}& \, ,
    \\
    \textstyle
    \dots,
    \xp_{i_{h-1},0}\otimes 1 + k_{i_{h-1}}^{-1} \otimes \xp_{i_{h-1},0}
    + \sum_{\ell>0} \xi_{\ell\delta,0} \otimes \xi_{\alpha_{i_{h-1}}-\ell\delta,0}&
    \, \Big]'_{q^{\epsilon_{h-2}}\dots q^{\epsilon_{1}}}
    (k_{0} \otimes k_{0}).
\end{align*}
Summands of the above lie inside
\begin{align*}
    \U_{\sum_{j\not\in J} \alpha_{i_{j}}
    - (\mathds{1}_{1\not\in J} - \sum_{j\in J} \ell_{j}) \delta,
    0}
    \otimes
    \U_{\sum_{j\in J} \alpha_{i_{j}}
    - (\mathds{1}_{1\in J} + \sum_{j\in J} \ell_{j}) \delta,
    0}
\end{align*}
where $J\subset [h-1]$ is the set of $j$ for which
$\xbp_{i_{j},-\delta_{j1}} \otimes 1$
is \emph{not} a factor, in which case $l_{j}\geq 0$ is the index of the factor chosen instead.
Again, the actions of all summands except those with
\begin{itemize}
    \item $J = [h-1]$ and all $l_{j} = 0$, which lie in
    $\U_{0,0} \otimes \U_{-\alpha_{0},0}$,
    \item $J = \emptyset$, which lie in
    $\U_{-\alpha_{0},0} \otimes \U_{0,0}$,
\end{itemize}
cancel on $(V\one \otimes V\two)(0)$, while the summands in these two cases add up to
\begin{itemize}
    \item 
    $\big(k_{\delta} k_{\theta}^{-1} \otimes
    a (-q)^{-\epsilon}
    \big[\xbp_{i_{1},-1}, \xp_{i_{2},0}, \dots, \xp_{i_{h-1},0}\big]'_{q^{\epsilon_{h-2}}\dots q^{\epsilon_{1}}} \big)
        \cdot (k_{0} \otimes k_{0})
        \\
        =
        k_{0}^{2} \otimes
        h\sigma \big(
        a (-q)^{-\epsilon} C^{-1} k_{\theta}
        \big[\xp_{i_{h-1},0},\dots,\xp_{i_{2},0},\xp_{i_{1},-1}\big]_{q^{\epsilon_{1}}\dots q^{\epsilon_{h-2}}} \big)
        \\
        = k_{0}^{2} \otimes h\sigma(\xm_{0})
        \\
        = k_{0}^{2} \otimes \xm_{0,0}$,
    \item 
    $\big( a (-q)^{-\epsilon}
    \big[\xbp_{i_{1},-1}, \xp_{i_{2},0}, \dots, \xp_{i_{h-1},0}\big]'_{q^{\epsilon_{h-2}}\dots q^{\epsilon_{1}}} \otimes 1 \big)
    \cdot (k_{0} \otimes k_{0})
        \\
        =
        h\sigma \big(
        a (-q)^{-\epsilon} C^{-1} k_{\theta}
        \big[\xp_{i_{h-1},0},\dots,\xp_{i_{2},0},\xp_{i_{1},-1}\big]_{q^{\epsilon_{1}}\dots q^{\epsilon_{h-2}}} \big)
        \otimes k_{0}
        \\
        = h\sigma(\xm_{0}) \otimes k_{0}
        \\
        = \xm_{0,0} \otimes k_{0}$,
\end{itemize}
and therefore $\xm_{0,0}$ acts on $(V\one \otimes V\two)(0)$ by
$\xm_{0,0}\otimes k_{0} + k_{0}^{2}\otimes\xm_{0,0}$.
\end{proof}

\begin{prop} \label{prop:action of xpm0pm1}
    $\xpm_{0,\pm 1}$ acts on $(V\one \otimes V\two)(0)$ by
    $\xp_{0,1}\otimes 1 + (C k_{0})^{-1}\otimes\xp_{0,1}$
    and
    $\xm_{0,-1}\otimes (C k_{0}) + 1\otimes\xm_{0,-1}$
    respectively.
\end{prop}
\begin{proof}
Using the identity
$\psi(\xpm_{0,\pm 1}) = \xpm_{0,\pm 1}$
one quickly verifies that
\begin{align*}
    \Delta^{\psi}(\xp_{0,1})
    &=
    \xp_{0,1} \otimes 1 + (C k_{0})^{-1} \otimes \xp_{0,1}
    + \sum_{\ell>0} \xi_{\ell\delta,0} \otimes \xi_{-\theta + (1-\ell)\delta, \delta'} \, ,
    \\
    \Delta^{\psi}(\xm_{0,-1})
    &=
    1 \otimes \xm_{0,-1} + \xm_{0,-1} \otimes C k_{0}
    + \sum_{\ell<0} \xi_{\theta - (\ell+1)\delta, -\delta'} \otimes \xi_{\ell\delta,0} \, ,
\end{align*}
where each sum must act by zero on $(V\one \otimes V\two)(0)$ by (\ref{eqn:action on weight spaces}).
\end{proof}

\begin{prop} \label{prop:action of hpm0pm1}
    $h_{0,\pm 1}$ acts on $(V\one \otimes V\two)(0)$ by
    $h_{0,1}\otimes 1 + C^{-1}\otimes h_{0,1}
    + (q_{0}^{-4} - 1)(k_{0}\xp_{0,1}\otimes k_{0}^{-1}\xm_{0,0})$
    and
    $h_{0,-1}\otimes C + 1\otimes h_{0,-1}
    - (q_{0}^{-4} - 1)(k_{0}\xp_{0,0}\otimes k_{0}^{-1}\xm_{0,-1})$
    respectively.
\end{prop}
\begin{proof}
From the relations of $\Utor$ we have that
$k_{0} h_{0,1} = [\xp_{0,1},\xm_{0,0}]$,
which acts on $(V\one \otimes V\two)(0)$ via
\begin{align*}
    &[ \xp_{0,1} \otimes 1 + (C k_{0})^{-1} \otimes \xp_{0,1} ,
    \xm_{0,0} \otimes k_{0} + k_{0}^{2} \otimes \xm_{0,0} ]
    \\
    &=
    [ \xp_{0,1} \otimes 1 , \xm_{0,0} \otimes k_{0} ]
    +
    [ (C k_{0})^{-1} \otimes \xp_{0,1} , \xm_{0,0} \otimes k_{0} ]
    +
    [ \xp_{0,1} \otimes 1 , k_{0}^{2} \otimes \xm_{0,0} ]
    +
    [ (C k_{0})^{-1} \otimes \xp_{0,1} , k_{0}^{2} \otimes \xm_{0,0} ]
    \\
    &=
    [ \xp_{0,1} , \xm_{0,0} ] \otimes k_{0}
    +
    (C k_{0})^{-1} \xm_{0,0} \otimes \xp_{0,1} k_{0}
    -
    \xm_{0,0} (C k_{0})^{-1} \otimes k_{0} \xp_{0,1}
    +
    [ \xp_{0,1} , k_{0}^{2} ] \otimes \xm_{0,0}
    +
    C^{-1} k_{0} \otimes [ \xp_{0,1} , \xm_{0,0} ]
    \\
    &=
    k_{0} h_{0,1} \otimes k_{0}
    +
    (q_{0}^{-2} - q_{0}^{-2})
    \big( (C k_{0})^{-1} \xm_{0,0} \otimes k_{0} \xp_{0,1} \big)
    +
    (q_{0}^{-4} - 1)
    (k_{0}^{2} \xp_{0,1} \otimes \xm_{0,0})
    +
    C^{-1} k_{0} \otimes k_{0} h_{0,1}
    \\
    &=
    k_{0} h_{0,1} \otimes k_{0}
    +
    C^{-1} k_{0} \otimes k_{0} h_{0,1}
    +
    (q_{0}^{-4} - 1)
    (k_{0}^{2} \xp_{0,1} \otimes \xm_{0,0})
\end{align*}
by Propositions \ref{prop:action of xpm00} and \ref{prop:action of xpm0pm1}.
Similarly, $k_{0}^{-1} h_{0,-1} = [\xp_{0,0},\xm_{0,-1}]$ acts on $(V\one \otimes V\two)(0)$ by
\begin{align*}
    &[ \xp_{0,0} \otimes k_{0}^{-2} + k_{0}^{-1} \otimes \xp_{0,0} ,
    \xm_{0,-1} \otimes C k_{0} + 1 \otimes \xm_{0,-1} ]
    \\
    &=
    [ \xp_{0,0} \otimes k_{0}^{-2} , \xm_{0,-1} \otimes C k_{0} ]
    +
    [ k_{0}^{-1} \otimes \xp_{0,0} , \xm_{0,-1} \otimes C k_{0} ]
    +
    [ \xp_{0,0} \otimes k_{0}^{-2} , 1 \otimes \xm_{0,-1} ]
    +
    [ k_{0}^{-1} \otimes \xp_{0,0} , 1 \otimes \xm_{0,-1} ]
    \\
    &=
    [ \xp_{0,0} , \xm_{0,-1} ] \otimes C k_{0}^{-1}
    +
    k_{0}^{-1} \xm_{0,-1} \otimes \xp_{0,0} C k_{0}
    -
    \xm_{0,-1} k_{0}^{-1} \otimes C k_{0} \xp_{0,0}
    +
    \xp_{0,0} \otimes [ k_{0}^{-2} , \xm_{0,-1} ]
    +
    k_{0}^{-1} \otimes [ \xp_{0,0} , \xm_{0,-1} ]
    \\
    &=
    k_{0}^{-1} h_{0,-1} \otimes C k_{0}^{-1}
    +
    (q_{0}^{-2} - q_{0}^{-2})
    \big( k_{0}^{-1} \xm_{0,-1} \otimes C k_{0} \xp_{0,0} \big)
    -
    (q_{0}^{-4} - 1)
    (\xp_{0,0} \otimes k_{0}^{-2} \xm_{0,-1})
    +
    k_{0}^{-1} \otimes k_{0}^{-1} h_{0,-1}
    \\
    &=
    k_{0}^{-1} h_{0,-1} \otimes C k_{0}^{-1}
    +
    k_{0}^{-1} \otimes k_{0}^{-1} h_{0,-1}
    -
    (q_{0}^{-4} - 1)
    (\xp_{0,0} \otimes k_{0}^{-2} \xm_{0,-1})
\end{align*}
and our proof is complete.
\end{proof}

\begin{cor} \label{cor:xp0m annihilate v1xv2}
    Every $\xp_{0,m}$ annihilates $v\one \otimes v\two$.
\end{cor}
\begin{proof}
From Proposition \ref{prop:action of hpm0pm1} we see that $h_{0,\pm 1}$ acts by $h_{0,\pm 1} \otimes 1 + 1 \otimes h_{0,\pm 1}$ on, and thus scales, $v\one \otimes v\two$.
Hence if some $\xp_{0,m}$ annihilates $v\one \otimes v\two$ then so does
$\xp_{0,m\pm 1} = [2]_{0}^{-1} C^{\frac{1\mp 1}{2}} [ h_{0,\pm 1} , \xp_{0,m} ]$.
By Proposition \ref{prop:action of xpm00} or \ref{prop:action of xpm0pm1} we are done.
\end{proof}

\begin{prop} \label{prop:action of phi+0m}
    Each $\phi^{+}_{0,m}$ acts on $v\one \otimes v\two$ by
    $\sum_{k+\ell=m} \phi^{+}_{i,k} \otimes \phi^{+}_{i,\ell}$.
\end{prop}

In order to prove this result we first require a brief technical lemma, for which we employ the following shorthand notations.
\begin{itemize}
    \item $\alpha = h_{0,1} \otimes 1 + 1 \otimes h_{0,1}$
    \item $\beta = (q_{0}^{-4} - 1)(k_{0} \xp_{0,1} \otimes k_{0}^{-1} \xm_{0,0})$
    \item $\gamma_{\ell} = \xp_{0,\ell} \otimes 1 + k_{0}^{-1} \otimes \xp_{0,\ell}$ for all $\ell\in\Zbb$
    \item $\eta^{(k,\ell)} = \xp_{0,k} \otimes k_{0}^{-1} \phi^{+}_{0,\ell}$ for all $k\in\Zbb_{>0}$ and $\ell\in\Zbb$
\end{itemize}
We shall use without comment that the actions of all $k_{0}^{\pm 1}$, $h_{0,r}$ and $\phi^{\pm}_{0,\ell}$ commute since $C^{\pm 1}$ acts trivially.

\begin{lem} \label{lem:actions on v1xv2}
    \begin{itemize}
        \item $[\alpha,\gamma_{\ell}]$ acts on $v\one \otimes v\two$ by $[2]_{0} \gamma_{\ell+1}$.
        \item $[\alpha,\eta^{(k,\ell)}]$ acts on $v\one \otimes v\two$ by $[2]_{0} \eta^{(k+1,\ell)}$.
        \item $\gamma_{\ell} \beta$ acts on $v\one \otimes v\two$ by $-[2]_{0}\eta^{(1,\ell)}$.
    \end{itemize}
\end{lem}
\begin{proof}
    The first two parts are trivially checked using the relation
    $[h_{0,1},\xp_{0,\ell}] = [2]_{0} \xp_{0,\ell+1}$, while from Corollary \ref{cor:xp0m annihilate v1xv2} we see that $\gamma_{\ell} \beta$ acts as
    \begin{align*}
        (k_{0}^{-1} \otimes \xp_{0,\ell}) \beta
        =
        (q_{0}^{-4} - 1) (\xp_{0,1} \otimes \xp_{0,\ell} k_{0}^{-1} \xm_{0,0})
        =
        (q_{0}^{-2} - q_{0}^{2}) (\xp_{0,1} \otimes k_{0}^{-1} \xp_{0,\ell} \xm_{0,0}),
    \end{align*}
    which in turn acts by
    $(q_{0}^{-2} - q_{0}^{2}) (\xp_{0,1} \otimes k_{0}^{-1} [\xp_{0,\ell},\xm_{0,0}])
    =
    -[2]_{0} (\xp_{0,1} \otimes k_{0}^{-1} \phi^{+}_{0,\ell})$.
\end{proof}

\begin{proof}[Proof of Proposition \ref{prop:action of phi+0m}]
From the relations
$\xp_{0,m+1} = [2]_{0}^{-1} [h_{0,1},\xp_{0,m}]$
and
$\phi^{+}_{0,m} = (q_{0} - q_{0}^{-1}) [\xp_{0,m},\xm_{0,0}]$
up to their actions on $v\one \otimes v\two$,
together with Propositions \ref{prop:action of xpm00}, \ref{prop:action of xpm0pm1} and \ref{prop:action of hpm0pm1},
we have that $\phi^{+}_{0,m}$ acts on $v\one \otimes v\two$ via
\begin{align*}
    (q_{0} - q_{0}^{-1}) [2]_{0}^{1-m}
    \big[ \big[\underbrace{\alpha+\beta,\dots,\alpha+\beta}_{m-1},
    \xp_{0,1}\otimes 1 + k_{0}^{-1}\otimes\xp_{0,1} \big],
    \xm_{0,0}\otimes k_{0} + k_{0}^{2}\otimes\xm_{0,0} \big].
\end{align*}
Expand out all pluses, and note that every $\alpha$ factor in a summand must act by a scalar no matter its position.
Each summand moreover contains one of the following pairs of factors.
\begin{enumerate}
    \item $\xp_{0,1}\otimes 1$ and $\xm_{0,0}\otimes k_{0}$
    \item $\xp_{0,1}\otimes 1$ and $k_{0}^{2}\otimes\xm_{0,0}$
    \item $k_{0}^{-1}\otimes\xp_{0,1}$ and $\xm_{0,0}\otimes k_{0}$
    \item $k_{0}^{-1}\otimes\xp_{0,1}$ and $k_{0}^{2}\otimes\xm_{0,0}$
\end{enumerate}
It is clear that summands with more than one $\beta$ factor annihilate (the first entry of) $v\one \otimes v\two$ by (\ref{eqn:action on weight spaces}), as do those in cases 1, 2 and 4 above that contain a single $\beta$ factor.
Furthermore, a summand in case 3 with exactly one $\beta$ factor, which in addition occurs either before $k_{0}^{-1}\otimes\xp_{0,1}$ or after $\xm_{0,0}\otimes k_{0}$, must also annihilate $v\one \otimes v\two$.
Therefore only the following may contribute to the action on $v\one \otimes v\two$:
\begin{enumerate}
    \item[1'.] Summands without any $\beta$ factors.
    \item[2'.] Summands in case 3 with a single $\beta$ factor, ordered as
    $\dots k_{0}^{-1}\otimes\xp_{0,1} \dots \beta \dots \xm_{0,0}\otimes k_{0}$.
\end{enumerate}
The first set add up to
\begin{align*}
    &(q_{0} - q_{0}^{-1}) [2]_{0}^{1-m}
    \big[ \big[\underbrace{\alpha,\dots,\alpha}_{m-1},
    \gamma_{1} \big],
    \xm_{0,0}\otimes k_{0} + k_{0}^{2}\otimes\xm_{0,0} \big]
    \\
    &=
    (q_{0} - q_{0}^{-1})
    \big[ \gamma_{m}, \xm_{0,0}\otimes k_{0} + k_{0}^{2}\otimes\xm_{0,0} \big]
    \\
    &=
    (q_{0} - q_{0}^{-1})
    \big( [\xp_{0,m},\xm_{0,0}] \otimes k_{0}
    +
    [\xp_{0,m},k_{0}^{2}] \otimes \xm_{0,0}
    +
    k_{0}^{-1} \xm_{0,0} \otimes \xp_{0,m} k_{0}
    -
    \xm_{0,0} k_{0}^{-1} \otimes k_{0} \xp_{0,m}
    +
    k_{0} \otimes [\xp_{0,m},\xm_{0,0}] \big)
    \\
    &=
    \phi^{+}_{0,m} \otimes k_{0}
    +
    k_{0} \otimes \phi^{+}_{0,m}
    +
    (q_{0} - q_{0}^{-1})
    \big( [\xp_{0,m},k_{0}^{2}] \otimes \xm_{0,0}
    +
    k_{0}^{-1} \xm_{0,0} \otimes \xp_{0,m} k_{0}
    -
    \xm_{0,0} k_{0}^{-1} \otimes k_{0} \xp_{0,m} \big),
\end{align*}
which simply acts by
$\phi^{+}_{0,m} \otimes k_{0} + k_{0} \otimes \phi^{+}_{0,m}$.
The second set sum to
\begin{align*}
    &- (q_{0} - q_{0}^{-1}) [2]_{0}^{1-m}
    \sum_{\ell = 1}^{m-1}
    \big[\underbrace{\alpha,\dots,\alpha}_{m-1-\ell},
    \big[\underbrace{\alpha,\dots,\alpha}_{\ell-1},\gamma_{1}\big]
    \beta \big]
    \cdot
    (\xm_{0,0} \otimes k_{0})
    \\
    &=
    - (q_{0} - q_{0}^{-1}) [2]_{0}^{1-m}
    \sum_{\ell = 1}^{m-1} [2]_{0}^{\ell-1}
    \big[\underbrace{\alpha,\dots,\alpha}_{m-1-\ell},
    \gamma_{\ell}
    \beta \big]
    \cdot
    (\xm_{0,0} \otimes k_{0})
    \\
    &=
    (q_{0} - q_{0}^{-1}) [2]_{0}^{1-m}
    \sum_{\ell = 1}^{m-1} [2]_{0}^{\ell}
    \big[\underbrace{\alpha,\dots,\alpha}_{m-1-\ell},
    \eta^{(1,\ell)} \big]
    \cdot
    (\xm_{0,0} \otimes k_{0})
    \\
    &=
    (q_{0} - q_{0}^{-1}) [2]_{0}^{1-m}
    \sum_{\ell = 1}^{m-1} [2]_{0}^{m-1}
    \eta^{(m-\ell,\ell)}
    (\xm_{0,0} \otimes k_{0})
    \\
    &=
    (q_{0} - q_{0}^{-1}) \sum_{\ell = 1}^{m-1}
    \xp_{0,m-\ell} \xm_{0,0} \otimes k_{0}^{-1} \phi^{+}_{0,\ell} k_{0}
    \\
    &=
    (q_{0} - q_{0}^{-1}) \sum_{\ell = 1}^{m-1}
    [\xp_{0,m-\ell},\xm_{0,0}] \otimes \phi^{+}_{0,\ell}
    \\
    &=
    \sum_{\ell = 1}^{m-1}
    \phi^{+}_{0,m-\ell} \otimes \phi^{+}_{0,\ell}
\end{align*}
by Lemma \ref{lem:actions on v1xv2}, where each equality is up to the action on $v\one \otimes v\two$.
This completes our proof.
\end{proof}

\begin{cor} \label{cor:0 Drinfeld polynomials}
    $\U(0)$ acts on $v\one \otimes v\two$ with Drinfeld polynomials $P\one_{0}(z)P\two_{0}(z)$.
\end{cor}
\begin{proof}
    This follows immediately from Proposition \ref{prop:action of phi+0m}.
\end{proof}

\subsection{Proof of Theorem \ref{thm:irreducible except countable}} \label{subsection:Proof of generic irreducibility}

The overall structure of our proof is as follows.
\begin{enumerate}
    \item Without loss of generality we can take $a = 1$.
    \item If conditions (\ref{eqn:first irreducibility condition}) and (\ref{eqn:second irreducibility condition}) hold on all $\mu \lneq \lambda\one + \lambda\two$ weight spaces, then $V\one \otimes V\two_{b}$ is irreducible.
    \item Since $Q$ is countable, it is therefore enough to show that for any such $\mu$, conditions (\ref{eqn:first irreducibility condition}) and (\ref{eqn:second irreducibility condition}) each fail for finitely many $b\in\Cbb^{\times}$.
    \item The elements of $\Utor$ that are involved in conditions (\ref{eqn:first irreducibility condition}) and (\ref{eqn:second irreducibility condition}) all lie inside $\psi(\Utor^{\pm})$ (Lemma \ref{lem:images of Utor pm}).
    \item This allows us to write their images under $\Delta^{\psi}$ as polynomials in $b$ for which the constant term is an elementary tensor.
    \item So their actions on $V\one \otimes V\two_{b}$, and thus conditions (\ref{eqn:first irreducibility condition}) and (\ref{eqn:second irreducibility condition}) themselves, may also be expressed in terms of polynomials in $b$ with simple constant terms.
    \item It then suffices to consider conditions (\ref{eqn:first irreducibility condition}) and (\ref{eqn:second irreducibility condition}) only in the limit $b\rightarrow 0$ (Lemma \ref{lem:kernels and images generically}).
    \item Lemma \ref{lem:annihilating and spanning weight spaces} completes the proof in this case.
\end{enumerate}

\begin{lem} \label{lem:images of Utor pm}
    The subalgebras
    $\mathcal{A}^{\pm}
    = \langle \xpm_{i,m},\, x^{\mp}_{i,k} C^{k} k_{i}^{\mp 1},\, h_{i,r} ~|~ {i \in I},\, \pm {m \geq \delta_{i0}},\, {\pm k > -\delta_{i0}},\, {\pm r > 0} \rangle$
    are contained in $\psi(\Utor^{\pm})$ respectively.
\end{lem}
\begin{proof}
In the following we shall work only up to multiplication by non-zero scalars, since this is all we require.
For each $i\in I_{0}$ we have
\begin{align*}
    &\xm_{i,1} C k_{i}^{-1}
    = \X_{i}(\xm_{i,0} k_{i}^{-1})
    = \X_{i}(k_{i}^{-1} \xm_{i,0})
    = \X_{i} \T_{i}^{-1} (\xp_{i,0})
    = v(\Xb_{i} \Tb_{i}^{-1} (\xp_{i}) ), \\
    &\xp_{i,-1} C^{-1} k_{i}
    = \X_{i}(\xp_{i,0} k_{i})
    = \X_{i} \T_{i}^{-1} (\xm_{i,0})
    = v(\Xb_{i} \Tb_{i}^{-1} (\xm_{i}) ).
\end{align*}
Then by \cite{Beck94}*{Defn. 3.1},
$\Xb_{i} \Tb_{i}^{-1} (\xipm) \in \Udash^{\pm}$
and thus
\begin{align*}
    x^{\mp}_{i,\pm 1} C^{\pm 1} k_{i}^{\mp 1}
    \in v(\Udash^{\pm})
    = \psi h (\Udash^{\pm})
    \subset \psi(\Utor^{\pm}).
\end{align*}
Furthermore, it is clear that $\psi(\Utor^{\pm})$ contains
$\xpm_{i,0} = \psi(\xpm_{i,0})$
and so by relation 7 of our definition for $\Utor$ we see that $h_{i,\pm 1} \in \psi(\Utor^{\pm})$ as well.
From relation 6 we then obtain
$\xpm_{i,m},\, x^{\mp}_{i,k} C^{k} k_{i}^{\mp 1} \in \psi(\Utor^{\pm})$
for all $\pm m \geq 0$ and $\pm k > 0$,
whereby relation 7 gives $k_{i}^{\mp 1} \phi^{\pm}_{i,r} \in \psi(\Utor^{\pm})$ for each $\pm r > 0$.
Using the identities
\begin{align*}
    h_{i,\pm r}
    =
    \frac{\pm 1}{q_{i} - q_{i}^{-1}} k_{i}^{\mp 1} \phi^{\pm}_{i,\pm r}
    -
    \sum_{\ell = 1}^{r-1} \frac{\ell}{r} k_{i}^{\mp 1} \phi^{\pm}_{i,\pm r \mp \ell} h_{i,\pm \ell}
\end{align*}
for all $r>0$ -- for example from \cite{Beck94}*{p.10--11} -- we are done by induction.
The case $i = 0$ is similar.
Combining Jing's isomorphism with $h = \psi v\sigma$ immediately gives
$x^{\mp}_{0,0} k_{0}^{\mp 1} \in \psi(\Utor^{\pm})$.
In addition, $\psi(\Utor^{\pm})$ clearly contains
$\xpm_{0,\pm 1} = \psi(\xpm_{0,\pm 1})$,
and the remaining identities are then obtained exactly as for $i\in I_{0}$.
\end{proof}

\begin{lem} \label{lem:annihilating and spanning weight spaces}
    Let $V = V(\lambda,\Psi)$ be an irreducible integrable $\Utor$-module with $\ell$-highest vector $v_{\lambda}$, and fix some weight $\mu < \lambda$.
    Then for every $m\in\Zbb$ and $\epsilon = \pm 1$,
    \begin{enumerate}
        \item $\lbrace v\in V_{\mu} ~|~ \xp_{i,k}\cdot v = 0 \mathrm{~for~all~} i\in I \mathrm{~and~} k\in\Zbb \mathrm{~with~} \epsilon k > m \rbrace = 0$,
        \item $V_{\mu} = \mathrm{Sp}_{\Cbb}
        \lbrace \xm_{i_{1},k_{1}}\cdots\xm_{i_{s},k_{s}}\cdot v_{\lambda} ~|~ s\in\Nbb,\, \mathrm{all~} \epsilon k_{j} > m,\, \sum_{j=1}^{s} \alpha_{i_{j}} = \lambda - \mu \rbrace$.
    \end{enumerate}
\end{lem}

Our proof requires the following brief result.

\begin{sublem}
    For each $i\in I$ there exists some $f_{i} \in \mathrm{Aut}_{\Cbb} V$ such that $f_{i}(v_{\lambda}) = v_{\lambda}$, and $f_{i}(z\cdot v) = \X_{i}(z)\cdot v$ for all $z\in\Utor$ and $v\in V$.
\end{sublem}
\begin{proof}
    The representation $V^{\X_{i}}$ is irreducible, as a twist of the irreducible module $V$, with $v_{\lambda}$ still an $\ell$-highest weight vector since $\X_{i}(\Utor^{+}) = \Utor^{+}$.
    Moreover, the action of each $z\in\Utor^{0}$ on $v_{\lambda}$ is the same as in $V$ because $C^{\pm 1}$ acts by $1$.
    Hence by Theorem \ref{thm:l-highest weight classification} we have an isomorphism $V^{\X_{i}} \cong V$ which fixes $v_{\lambda}$, which defines an automorphism $f_{i} \in \mathrm{Aut}_{\Cbb} V$ with the desired properties.
\end{proof}

\begin{proof}[Proof of Lemma \ref{lem:annihilating and spanning weight spaces}]
For each $\ell\in\Zbb$ let
\begin{align*}
    V[\ell] = \lbrace v\in V_{\mu} ~|~ \xp_{i,k}\cdot v = 0 \mathrm{~for~all~} i\in I \mathrm{~and~} k\in\Zbb \mathrm{~with~} \epsilon k > \ell \rbrace \leq V_{\mu},
\end{align*}
which is finite dimensional since $V$ is integrable.
Clearly every $V[\ell-\epsilon] \leq V[\ell]$, but also
$(f_{0}\cdots f_{n})^{\epsilon} \in \mathrm{Aut}_{\Cbb} V$
sends $V[\ell]$ inside $V[\ell-\epsilon]$ and therefore
$\dim_{\Cbb} V[\ell] \leq \dim_{\Cbb} V[\ell-\epsilon]$,
forcing $V[\ell-\epsilon] = V[\ell]$.
\\

It follows that
$V[m] = \bigcap_{\ell\in\Zbb} V[\ell]
= \lbrace v\in V_{\mu} ~|~ \xp_{i,k}\cdot v = 0 ~\forall\, i\in I,\, k\in\Zbb \rbrace$.
Any non-zero $v\in V[m]$ is then an $\ell$-highest weight vector of weight $\mu<\lambda$ inside $V$ by relation 7 of our definition for $\Utor$.
But this contradicts the irreducibility of $V$, and thus $V[m] = 0$ as desired.
\\

Similarly, define
$W[\ell] = \mathrm{Sp}_{\Cbb}
\lbrace \xm_{i_{1},k_{1}}\cdots\xm_{i_{s},k_{s}}\cdot v_{\lambda} \, | \, s\in\Nbb,\, \mathrm{all~} \epsilon k_{j} > \ell,\, \sum_{j=1}^{s} \alpha_{i_{j}} = \lambda - \mu \rbrace$
for each $\ell\in\Zbb$, which are finite dimensional subspaces of $V_{\mu}$.
Every $W[\ell] \leq W[\ell-\epsilon]$ by construction, while
$\dim_{\Cbb} W[\ell-\epsilon] \leq \dim_{\Cbb} W[\ell]$
since
$(f_{0}\cdots f_{n})^{\epsilon}$
maps $W[\ell-\epsilon]$ into $W[\ell]$,
hence we have $W[\ell] = W[\ell-\epsilon]$.
\\

Therefore
$W[m] = \bigcap_{\ell\in\Zbb} W[\ell] = \mathrm{Sp}_{\Cbb}
\lbrace \xm_{i_{1},k_{1}}\cdots\xm_{i_{s},k_{s}}\cdot v_{\lambda} \, | \, s\in\Nbb,\, \sum_{j=1}^{s} \alpha_{i_{j}} = \lambda - \mu \rbrace$,
and this must in turn equal $V_{\mu}$ since $V$ is spanned by vectors of the form
$\xm_{i_{1},k_{1}}\cdots\xm_{i_{s},k_{s}} \cdot v_{\lambda}$
which have weight
$\lambda - \sum_{j=1}^{s} \alpha_{i_{j}}$.
\end{proof}

\begin{lem} \label{lem:kernels and images generically}
    Let
    $\big\lbrace f^{(k)}_{b} : A \rightarrow B \big\rbrace_{k\in\Nbb}$
    be a collection of morphisms between free $\Cbb[b]$-modules of countable rank.
    \begin{enumerate}
        \item If $\bigcap_{k\in\Nbb} \ker f^{(k)}_{0} = 0$ and $\dim_{\Cbb[b]} A < \infty$, then $\bigcap_{k\in\Nbb} \ker f^{(k)}_{\beta} = 0$ for all but finitely many $\beta \in \Cbb$.
        \item If $\sum_{k\in\Nbb} \mathrm{im} f^{(k)}_{0} = B$ and $\dim_{\Cbb[b]} B < \infty$, then $\sum_{k\in\Nbb} \mathrm{im} f^{(k)}_{\beta} = B$ for all but finitely many $\beta \in \Cbb$.
    \end{enumerate}
\end{lem}
\begin{proof}
Write $d_{A}$ and $d_{B}$ as shorthand for the ranks of $A$ and $B$ as $\Cbb[b]$-modules.
We shall start with the first implication.
Since $A$ is finite dimensional, we must have
$\bigcap_{k=0}^{N} \ker f^{(k)}_{0} = 0$
for some $N\in\Nbb$.
With respect to fixed bases for $A$ and $B$, the linear map
$\bigoplus_{k=0}^{N} f^{(k)}_{b} : A \rightarrow \bigoplus_{k=0}^{N} B$
corresponds to some matrix
$M \in \mathrm{Mat}_{N d_{B}\times d_{A}}(\Cbb[b])$.
Since $\Cbb[b]$ is principal, the ideal generated by all $d_{A}\times d_{A}$ minors of $M$ is equal to some $\langle f(b) \rangle$.
\\

The rank of $M$ is the size of its largest non-zero minor, and moreover equals $d_{A} - \dim(\ker M)$ by rank-nullity, so it must be the case that $f(0) \not= 0$.
As a non-zero polynomial, $f$ therefore has finitely many roots.
For all other $\beta\in\Cbb$ we then have
$I = \langle f(\beta) \rangle \not= 0$
and hence
$\bigcap_{k=0}^{N} \ker f^{(k)}_{\beta} = 0$.
\\

Let us now move to the second implication, where the finite-dimensionality of $B$ forces
$\sum_{k=0}^{N} \mathrm{im} f^{(k)}_{0} = B$
for some $N\in\Nbb$.
After fixing bases for $A$ and $B$, we can view
$\bigoplus_{k=0}^{N} f^{(k)}_{b} : \bigoplus_{k=0}^{N} A \rightarrow B$
as a matrix
$M \in \mathrm{Mat}_{d_{B}\times N d_{A}}(\Cbb[b])$.
The ideal of $\Cbb[b]$ generated by its $d_{B}\times d_{B}$ minors is some $\langle f(b) \rangle$, in particular with $f(0) \not= 0$ since $M$ is surjective at $b = 0$.
Hence $f(\beta)$ is non-zero and thus $\mathrm{rk}(M) = d_{B}$ for all but finitely many $\beta\in\Cbb$.
\end{proof}

\begin{proof}[Proof of Theorem \ref{thm:irreducible except countable}]
Irreducibility is preserved under twisting by automorphisms of $\Utor$, so as
(\ref{eqn:image of new topological coproduct})
implies that
$(\sfrakv_{a} \otimes \sfrakv_{b}) \circ \Delta^{\psi} = (\sfrakv_{1} \otimes \sfrakv_{b/a}) \circ \Delta^{\psi} \circ \sfrakv_{a}$
and thus
$V\one_{a}\otimes V\two_{b}
\cong
(V\one\otimes V\two_{b/a})_{a}$,
we may without loss of generality take $a = 1$.
If $V\one \otimes V\two_{b}$ is reducible, then at least one of the following holds.
\begin{itemize}
    \item $V\one \otimes V\two_{b}$ is not generated by $v\one \otimes v\two$
    \item $V\one \otimes V\two_{b}$ contains an $\ell$-highest weight vector of weight $\mu \lneq \lambda\one + \lambda\two$
\end{itemize}
Neither of these occurs -- and hence $V\one \otimes V\two_{b}$ is irreducible -- provided that both of the following hold for all $\mu \lneq \lambda\one + \lambda\two$.
\begin{itemize}
    \item $(V\one \otimes V\two_{b})_{\mu}
    = \mathrm{Sp}_{\Cbb}
    \lbrace \xm_{i_{1},k_{1}}\cdots\xm_{i_{s},k_{s}}\cdot (v\one \otimes v\two) ~|~ \sum_{j=1}^{s} \alpha_{i_{j}} = \lambda - \mu \rbrace
    \hfill \refstepcounter{equation}(\theequation)\label{eqn:first irreducibility condition}$
    \item $\big\lbrace v\in (V\one \otimes V\two_{b})_{\mu} ~|~ b^{-(1+h(m-\delta_{i0})) \cdot \mathds{1}_{k<\delta_{i0}}} \xp_{i,k}\cdot v = 0 ~\forall\, {i\in I}, {k\in\Zbb} \big\rbrace = 0
    \hfill \refstepcounter{equation}(\theequation)\label{eqn:second irreducibility condition}$
\end{itemize}
As $Q$ is countable, it is enough to prove that for every $\mu \lneq \lambda\one + \lambda\two$ these conditions each hold for all but finitely many $b\in\Cbb^{\times}$.
From (\ref{eqn:image of new topological coproduct}) we see that
\begin{align*}
    \psi(\xp_{i,m}) \xmapsto{\Delta^{\psi}}
    \psi(\xp_{i,m}) \otimes 1
    + \sum_{\ell\geq 0} x_{\ell} \otimes y_{\ell}
    \xmapsto{1\otimes\sfrakv_{b}}
    \psi(\xp_{i,m}) \otimes 1
    + \sum_{\ell\geq 0} (x_{\ell} \otimes y_{\ell}) b
\end{align*}
and thus by Lemma \ref{lem:images of Utor pm},
$(1\otimes\sfrakv_{b}) \circ \Delta^{\psi}$
sends
$\xp_{i,k} \mapsto \xp_{i,k} \otimes 1 + O(b)$
whenever $i\in I$ and $k \geq \delta_{i0}$, while
$\xm_{i,k} C^{k} k_{i}^{-1} \mapsto \xm_{i,k} C^{k} k_{i}^{-1} \otimes 1 + O(b)$
for all $i\in I$ and $k > -\delta_{i0}$.
Similarly,
\begin{align*}
    b^{-\degv(\psi(\xm_{i,m}))} \psi(\xm_{i,m})
    \xmapsto{(1\otimes\sfrakv_{b})\circ\Delta^{\psi}}
    \psi(\xm_{i,m}) \otimes 1
    + \sum_{\ell\leq 0} (x_{\ell} \otimes y_{\ell})
    b^{-\degv(x_{\ell})}
\end{align*}
where we note that all
$\degv(x_{\ell}) = \degv(\psi(\xm_{i,m})) = 1$.
It follows from Lemma \ref{lem:images of Utor pm} that
$(1\otimes\sfrakv_{b}) \circ \Delta^{\psi}$
sends
$b^{1-h(k+\delta_{i0})} \xm_{i,k} \mapsto
1 \otimes \xm_{i,k} + O(b)$
whenever $i\in I$ and $k \leq -\delta_{i0}$.
\\

In particular, the action of $\xp_{i,k}$ then defines a morphism
$f^{(k)}_{b} : (V\one \otimes V\two)_{\mu} \rightarrow (V\one \otimes V\two)_{\mu + \alpha_{i}}$
of free $\Cbb[b]$-modules for each $k\geq\delta_{i0}$.
Here we use the well-definedness afforded by Theorem \ref{thm:integrability of tensor products}, and the fact that $(V\one \otimes V\two_{b})_{\mu}$ is independent of $b\in\Cbb^{\times}$ as a vector space.
Due to Lemma \ref{lem:annihilating and spanning weight spaces}~(1), $\big\lbrace f^{(k)}_{b} \big\rbrace$ satisfies the assumptions of Lemma \ref{lem:kernels and images generically}~(1), whereby
\begin{align*}
    &\big\lbrace w\in (V\one \otimes V\two_{b})_{\mu} ~|~ \xp_{i,k}\cdot w = 0 ~\forall\, {i\in I}, {k\in\Zbb} \big\rbrace
    \\
    &\subset
    \big\lbrace w\in (V\one \otimes V\two_{b})_{\mu} ~|~ \xp_{i,k}\cdot w = 0 ~\forall\, {i\in I}, {k\geq\delta_{i0}} \big\rbrace
    \\
    &= \bigcap_{k\geq\delta_{i0}} \ker f^{(k)}_{b}
    \\
    &= 0
\end{align*}
for all but finitely many $b\in\Cbb$, verifying condition (\ref{eqn:first irreducibility condition}).
\\

In order to prove condition (\ref{eqn:second irreducibility condition}), define
$\Tilde{x}^{-}_{i,k} = b^{1-h(k+\delta_{i0})} \xm_{i,k}$
whenever $k \leq -\delta_{i0}$, and
$\Tilde{x}^{-}_{i,k} = \xm_{i,k} C^{k} k_{i}^{-1}$
otherwise.
Let $\lbrace f^{(k)}_{b} \rbrace$ be the morphisms of free $\Cbb[b]$-modules
$(V\one \otimes V\two)_{\lambda\one + \lambda\two} \rightarrow (V\one \otimes V\two)_{\mu}$
given by the actions of all
$\Tilde{x}^{-}_{i_{1},k_{1}}\cdots\Tilde{x}^{-}_{i_{s},k_{s}}$
with
$\sum \alpha_{i_{j}} = \lambda\one + \lambda\two - \mu$,
in some order.
Again, we use Theorem \ref{thm:integrability of tensor products} and the independence of $(V\one \otimes V\two_{b})_{\mu}$ from $b$ to define these.
In this case, we have
\begin{align*}
    \sum_{k\geq 0} \mathrm{im} f^{(k)}_{b}
    &=
    \mathrm{Sp}_{\Cbb}
    \lbrace
    \Tilde{x}^{-}_{i_{1},k_{1}} \cdots \Tilde{x}^{-}_{i_{s},k_{s}}
    \cdot (v\one \otimes v\two)
    ~|~
    {\textstyle \sum} \, \alpha_{i_{j}} = \lambda\one + \lambda\two - \mu \rbrace
    \\
    &=
    \mathrm{Sp}_{\Cbb}
    \lbrace
    (\xm_{i_{1},k_{1}}\cdots\xm_{i_{s},k_{s}} \cdot v\one)
    \otimes
    (\xm_{i'_{1},k'_{1}}\cdots\xm_{i'_{r},k'_{r}} \cdot v\two)
    + O(b)
    \\
    &\qquad\qquad
    ~|~
    k_{j} \leq -\delta_{i0},\,
    k'_{j} > -\delta_{i0},\,
    {\textstyle \sum} \, \alpha_{i_{j}} + {\textstyle \sum} \, \alpha_{i'_{j}} = \lambda\one + \lambda\two - \mu \rbrace
    \\
    &=
    \bigoplus_{\mu\one+\mu\two=\mu}
    \mathrm{Sp}_{\Cbb}
    \lbrace
    \xm_{i_{1},k_{1}}\cdots\xm_{i_{s},k_{s}} \cdot v\one
    + O(b)
    ~|~
    k_{j} \leq -\delta_{i0},\,
    {\textstyle \sum} \, \alpha_{i_{j}} = \lambda\one + \lambda\two - \mu\one \rbrace
    \\[-9pt]
    &\qquad\qquad\qquad\quad\otimes
    \mathrm{Sp}_{\Cbb}
    \lbrace
    \xm_{i'_{1},k'_{1}}\cdots\xm_{i'_{r},k'_{r}} \cdot v\two
    + O(b)
    ~|~
    k'_{j} > -\delta_{i0},\,
    {\textstyle \sum} \, \alpha_{i'_{j}} = \lambda\one + \lambda\two - \mu\two \rbrace
\end{align*}
where the second equality holds because
\begin{itemize}
    \item $\xm_{i,k} C^{k} k_{i}^{-1} \mapsto \xm_{i,k} C^{k} k_{i}^{-1} \otimes 1 + O(b)$ for $k > -\delta_{i0}$,
    $\hfill \refstepcounter{equation}(\theequation)\label{eqn:technical fact 1}$
    \item $b^{1-h(k+\delta_{i0})} \xm_{i,k} \mapsto 1 \otimes \xm_{i,k} + O(b)$ for $k \leq -\delta_{i0}$,
    $\hfill \refstepcounter{equation}(\theequation)\label{eqn:technical fact 2}$
    \item $C$ and all $k_{i}^{-1}$ commute with every $\xm_{j,k}$ up to non-zero scalar factors,
    $\hfill \refstepcounter{equation}(\theequation)\label{eqn:technical fact 3}$
    \item $C$ and all $k_{i}^{-1}$ act by non-zero scalars on both $v\one$ and $v\two$.
    $\hfill \refstepcounter{equation}(\theequation)\label{eqn:technical fact 4}$
\end{itemize}
By Lemma \ref{lem:annihilating and spanning weight spaces}~(2) we therefore have
$\sum_{k\geq 0} \mathrm{im} f^{(k)}_{0}
= \bigoplus_{\mu\one+\mu\two=\mu}
V\one_{\mu\one} \otimes V\two_{\mu\two}
= (V\one \otimes V\two)_{\mu}$,
whereby $\lbrace f^{(k)}_{b} \rbrace$ satisfies the assumptions of Lemma \ref{lem:kernels and images generically}~(2) and thus
\begin{align*}
    &\mathrm{Sp}_{\Cbb}
    \lbrace \xm_{i_{1},k_{1}}\cdots\xm_{i_{s},k_{s}}\cdot (v\one \otimes v\two) ~|~
    {\textstyle \sum} \, \alpha_{i_{j}} = \lambda\one + \lambda\two - \mu \rbrace
    \\
    &=
    \mathrm{Sp}_{\Cbb}
    \lbrace \Tilde{x}^{-}_{i_{1},k_{1}}\cdots\Tilde{x}^{-}_{i_{s},k_{s}}\cdot (v\one \otimes v\two) ~|~
    {\textstyle \sum} \, \alpha_{i_{j}} = \lambda\one + \lambda\two - \mu \rbrace
    \\
    &= \sum \mathrm{im} f^{(k)}_{b}
    \\
    &= (V\one \otimes V\two_{b})_{\mu}
\end{align*}
for all but finitely many $b\in\Cbb$.
Note that the first equality here follows from (\ref{eqn:technical fact 1})--(\ref{eqn:technical fact 4}) above.
\end{proof}

\section{\texorpdfstring{$q$}{q}-characters} \label{section:q-characters}

The character morphism
$\ch : V \mapsto \sum_{\lambda\in\mathfrak{h}^{*}} \dim(V_{\lambda}) e_{\lambda}$
is a fundamental mechanism for approaching the representation theory of both Kac-Moody Lie algebras and their Drinfeld-Jimbo quantum groups, where for example it takes different values on each simple module inside $\Ocal_{\mathrm{int}}$.
\\

For quantum affine algebras, a finer $q$-character morphism -- introduced by Frenkel and Reshetikhin \cite{FR99} -- is required to distinguish the finite dimensional modules.
Furthermore, explicit formulas for the $q$-characters of various classes of representations can be computed via the iterative Frenkel-Mukhin algorithm \cite{FM01}.
These constructions provide a powerful combinatorial tool for studying the category $\Oaff$, are related to the cluster algebra structure on its Grothendieck ring \cite{HL16}, and may even be used in the computation of $R$-matrices \cites{DM25a,DM25b}.
\\

Hernandez \cite{Hernandez05} later generalised the $q$-character morphism to all quantum affinizations, in particular as a group homomorphism
$\chi_{q} : K(\Oaff) \rightarrow \Ycal$
to some commutative ring (see Section \ref{subsubsection:q-characters}).
An extension of Frenkel-Mukhin's algorithm -- first introduced in \cite{Hernandez04} -- is proved to be well-defined whenever $a_{ij}a_{ji} \leq 3$ for all $i\not= j$, as well as for the remaining quantum toroidal algebras of types $A_{1}^{(1)}$ (with $d_{1} = d_{2} = 2$) and $A_{2}^{(2)}$.
\\

Of course, unlike the case of quantum affine algebras, $\Oaff$ does not in general come naturally equipped with a tensor product, and thus $K(\Oaff)$ the structure of a ring.
However, $\chi_{q}$ is proved to be injective, with image equal to the intersection of kernels of certain screening operators.
Moreover, $\mathrm{im}(\chi_{q})$ is shown to be a \textit{subring} of $\Ycal$, and hence we may pull back its natural multiplication to a \textit{fusion product} on $K(\Oaff)$ -- namely, the product of module classes is again a module class.
In a later work, Hernandez \cite{Hernandez07} proved that this may in fact be induced from a fusion product $\ast_{f}$ of representations, defined via a deformation renormalisation process using a large category of modules for $\Uqaffs \otimes \Cbb(u)$.
\\

Our primary goal in this section is to prove that our tensor product on $\Oaff$ is compatible with $q$-characters, in particular
$\chi_{q}(V\one \otimes V\two) = \chi_{q}(V\one) \cdot \chi_{q}(V\two)$
for all modules $V\one,V\two\in\Oaff$.
Since $\chi_{q}$ is injective and
$\chi_{q}(V\one \ast_{f} V\two) = \chi_{q}(V\one) \cdot \chi_{q}(V\two)$,
we may deduce that $\otimes$ and $\ast_{f}$ give rise to the same product on the level of the Grothendieck ring $K(\Oaff)$.
\\

This is surprising and fascinating, since while Hernandez' work uses the \textit{vertically infinite} Drinfeld topological coproduct $\Delta_{u}$, our approach goes via the \textit{horizontally infinite} topological coproduct $\Dpsi$.
Our results therefore indicate that there perhaps exists a \textit{true} coproduct (and even Hopf algebra structure) for quantum toroidal algebras underlying everything, as has been found in the base cases $A_{1}^{(1)}$ and $A_{2}^{(1)}$ \cite{JZ22}. The author will investigate these directions in future work.
\\

Consider representations $V\one,V\two\in\Oaff$ whose weights are contained in finite unions of cones $D\one$ and $D\two$ respectively.
Recall from Section \ref{subsubsection:q-characters} that the $\ell$-weights of each $V\alphapower$ lie inside $QP_{\ell}^{+}$, and associated to every such $(\lambda\alphapower,\Psi\alphapower)$ are polynomials
$Q\alphapower_{i}(z) = \prod_{a\in\Cbb^{\times}} (1-az)^{\beta\alphapower_{i,a}}$
and
$R\alphapower_{i}(z) = \prod_{a\in\Cbb^{\times}} (1-az)^{\gamma\alphapower_{i,a}}$
with
\begin{align*}
    \sum_{s\geq 0} \Psi^{(\alpha),\pm}_{i,\pm s} z^{\pm s}
    =
    q_{i}^{\deg(Q\alphapower_{i})-\deg(R\alphapower_{i})} \frac{Q\alphapower_{i}(zq_{i}^{-1}) R\alphapower_{i}(zq_{i})}{Q\alphapower_{i}(zq_{i}) R\alphapower_{i}(zq_{i}^{-1})}
\end{align*}
and $Q_{i}(0) = R_{i}(0) = 1$ for all $i\in I$.
If we moreover define
$(\Psi\one\cdot\Psi\two)^{\pm}_{i,\pm s}
=
\sum_{r=0}^{s} \Psi^{(1),\pm}_{i,\pm r} \cdot \Psi^{(2),\pm}_{i,\pm (s-r)}$ so that
\begin{align*}
    Y_{\lambda\one,\Psi\one} Y_{\lambda\two,\Psi\two}
    =
    k_{\nu(\lambda\one + \lambda\two)}
    \prod_{\substack{i\in I \\ a\in\Cbb^{\times}}}
    Y_{i,a}^{(\beta\one_{i,a} + \beta\two_{i,a}) - (\gamma\one_{i,a} + \gamma\two_{i,a})}
    =
    Y_{\lambda\one + \lambda\two,\Psi\one\cdot\Psi\two},
\end{align*}
then the product of $q$-characters can be written as
\begin{align*}
    \chi_{q}(V\one) \cdot \chi_{q}(V\two)
    =
    \sum_{(\lambda\alphapower,\Psi\alphapower) \in QP_{\ell}^{+}}
    \dim\Big(V\one_{\lambda\one,\Psi\one} \otimes V\two_{\lambda\two,\Psi\two}\Big)
    Y_{\lambda\one + \lambda\two,\Psi\one\cdot\Psi\two}.
\end{align*}
On the other hand, using equation (\ref{eqn:Dpsi on C and k}) we may decompose any weight space into finite direct sums
\begin{align} \label{eqn:tensor l-weight space decomposition of weight space}
    (V\one \otimes V\two)_{\lambda}
    =
    \bigoplus V\one_{\lambda\one} \otimes V\two_{\lambda\two}
    =
    \bigoplus V\one_{\lambda\one,\Psi\one} \otimes V\two_{\lambda\two,\Psi\two}
\end{align}
over
$(\lambda\alphapower,\Psi\alphapower) \in QP_{\ell}^{+}$
with $\lambda\alphapower \in D\alphapower$ and $\lambda\one + \lambda\two = \lambda$.
Therefore, in order to prove the desired compatibility
$\chi_{q}(V\one \otimes V\two) = \chi_{q}(V\one) \cdot \chi_{q}(V\two)$
between our tensor product and the $q$-character morphism,
it suffices to verify the following.

\begin{prop} \label{prop:eigenvalues of phipmipms}
    For all $i\in I$ and $s\geq 0$, action of $\phi^{\pm}_{i,\pm s}$ on $(V\one \otimes V\two)_{\lambda}$ has eigenvalue
    $(\Psi\one\cdot\Psi\two)^{\pm}_{i,\pm s}$
    with multiplicity
    $\dim\Big(V\one_{\lambda\one,\Psi\one} \otimes V\two_{\lambda\two,\Psi\two}\Big)$.
\end{prop}

Our proof strategy is to first decompose $(V\one \otimes V\two)_{\lambda}$ into blocks
\begin{align} \label{eqn:block decomposition}
    \bigoplus_{\langle \lambda\two,\Lambda^{\vee}_{0}\rangle = N_{0}}
    \bigoplus_{(\lambda\two,\theta) = N_{\theta}}
    V\one_{\lambda\one,\Psi\one} \otimes V\two_{\lambda\two,\Psi\two}
\end{align}
ordered increasingly in $N_{0}$ and $N_{\theta}$.
If we are able to show that every $\phi^{\pm}_{i,\pm s}$ acts via a block upper triangular matrix, with the diagonal blocks moreover describing the action of
$\sum_{r=0}^{s} \phi^{\pm}_{i,\pm r} \otimes \phi^{\pm}_{i,\pm (s-r)}$,
then Proposition \ref{prop:eigenvalues of phipmipms} follows.
The proof of this result is rather involved, and therefore deferred to Section \ref{subsection:q-characters proof}.
Nevertheless, it has the following consequences as outlined above.

\begin{thm}
    Our tensor product on $\Oaff$ is compatible with the $q$-character morphism, in particular
    $\chi_{q}(V\one \otimes V\two) = \chi_{q}(V\one) \cdot \chi_{q}(V\two)$
    for all representations $V\one,V\two\in\Oaff$.
\end{thm}

Next, consider the Grothendieck group $K(\Oaff)$ as a ring, with multiplication extended linearly from our tensor product.

\begin{cor}
    \begin{itemize}
        \item The $q$-character morphism
        $\chi_{q} : K(\Oaff) \rightarrow \Ycal$
        is a ring homomorphism.
        \item Our tensor product $\otimes$ and Hernandez' fusion product $\ast_{f}$ give rise to the same product on $K(\Oaff)$.
        \item Our tensor product is commutative on the level of $K(\Oaff)$.
    \end{itemize}
\end{cor}

\begin{rmk}
    In fact, the commutativity of our tensor product on direct sums of tensor products of simple modules in $\Oaff$ may alternatively be deduced from the existence of meromorphic $R$-matrices -- see Corollary \ref{cor:quantum toroidal R-matrices}.
\end{rmk}

\subsection{Proof of Proposition \ref{prop:eigenvalues of phipmipms}} \label{subsection:q-characters proof}

Our work here has a similar progression to that of Sections \ref{subsection:action of Uv} and \ref{subsection:action of U0}, however certain steps are slightly more delicate since the weight spaces treated are arbitrary rather than simply highest weight.
Recall that we consider each $(V\one \otimes V\two)_{\lambda}$ with respect to the block decomposition (\ref{eqn:block decomposition}), and let us begin with the case $i\in I_{0}$.

\begin{proof}[Proof of Proposition \ref{prop:eigenvalues of phipmipms} for $\phi^{+}_{i,s}$ with $i\in I_{0}$]
From our proof of Lemma \ref{lem:images of Utor pm} we know that any $\psi(\phi^{+}_{i,s}) \in \U_{s\delta,0}$ with $i\in I_{0}$ equals a linear combination
$\sum \xi \xp_{i_{1},0} \dots \xp_{i_{k},0} k_{i}^{-1}$
with every $\sum_{j=1}^{k} \alpha_{i_{j}} = s\delta$, which is then sent by $\Delta_{1}$ to
\begin{align*}
    \sum \xi
    \left(
    \xp_{i_{1},0} \otimes 1 + \sum_{\ell_{1}\geq 0} C^{-\ell_{1}} \phi^{+}_{i_{1},\ell_{1}} \otimes \xp_{i_{1},-\ell_{1}}
    \right)
    \dots
    \left(
    \xp_{i_{k},0} \otimes 1 + \sum_{\ell_{k}\geq 0} C^{-\ell_{k}} \phi^{+}_{i_{k},\ell_{k}} \otimes \xp_{i_{k},-\ell_{k}}
    \right)
    (k_{i}^{-1} \otimes k_{i}^{-1}).
\end{align*}
Expanding out the brackets, each summand is equal to a product of terms
$\xp_{i_{j},0} \otimes 1$ for $j\in J_{1}$
and
$C^{-\ell_{j}} \phi^{+}_{i_{j},\ell_{j}} \otimes \xp_{i_{j},-\ell_{j}}$ for $j\in J_{2}$,
together with $\xi\in\Cbb^{\times}$ and $k_{i}^{-1} \otimes k_{i}^{-1}$.
If we furthermore set
$J_{0} = \lbrace j\in [k] ~|~ i_{j} = 0 \rbrace$
and $\ell = \sum_{j=1}^{k} \ell_{j}$, then applying $\psi\otimes\psi$ maps this into
\begin{align*}
    \U_{\sum_{j\in J_{1}\setminus J_{0}} \alpha_{i_{j}}
    - \lvert J_{0}\cap J_{1} \rvert \theta + \ell\delta,
    \lvert J_{0}\cap J_{1} \rvert \delta'}
    \otimes
    \U_{\sum_{j\in J_{2}\setminus J_{0}} \alpha_{i_{j}}
    - \lvert J_{0}\cap J_{2} \rvert \theta - \ell\delta,
    \lvert J_{0}\cap J_{2} \rvert \delta'}
\end{align*}
by equation (\ref{eqn:psi on graded pieces}).
An element of the above sends any
$V\one_{\lambda\one,\Psi\one} \otimes V\two_{\lambda\two,\Psi\two}$
to
\begin{align} \label{eqn:l-weight space image of phi+is}
    V\one_{\lambda\one
    + \sum_{j\in J_{1}\setminus J_{0}} \alpha_{i_{j}}
    - \lvert J_{0}\cap J_{1} \rvert \theta + \ell\delta}
    \otimes
    V\two_{\lambda\two
    + \sum_{j\in J_{2}\setminus J_{0}} \alpha_{i_{j}}
    - \lvert J_{0}\cap J_{2} \rvert \theta - \ell\delta}
\end{align}
using (\ref{eqn:action on weight spaces}), and thus acts via a block upper triangular matrix on $(V\one \otimes V\two)_{\lambda}$.
This is moreover strictly block upper triangular if $\ell > 0$ since
$\langle \delta,\Lambda^{\vee}_{0}\rangle = 1$.
Therefore, the associated block diagonal matrix describes the action of
\begin{align*}
    &(\psi\otimes\psi)
    \left( \sum \xi
    (\xp_{i_{1},0} \otimes 1 + k_{i_{1}} \otimes \xp_{i_{1},0})
    \dots
    (\xp_{i_{k},0} \otimes 1 + k_{i_{k}} \otimes \xp_{i_{k},0})
    (k_{i}^{-1} \otimes k_{i}^{-1})
    \right)
    \\
    &=
    (\psi\otimes\psi) \circ (h \otimes h) \circ \Delta_{+}
    \left( \sum \xi \xp_{i_{1}} \dots \xp_{i_{k}} k_{i}^{-1} \right)
    \\
    &=
    (v \otimes v) \circ (\sigma\otimes\sigma) \circ \Delta_{+} \circ \sigma (\phi^{+}_{i,s})
    \\
    &=
    (v \otimes v) \circ \Delta (\phi^{+}_{i,s})
\end{align*}
where the penultimate equality is due to Theorem \ref{thm:psi}.
The corresponding result for quantum affine algebras -- see for example \cite{FR99}*{Rmk. 2.6} -- then implies that $\phi^{+}_{i,s}$ acts on
$V\one_{\lambda\one,\Psi\one} \otimes V\two_{\lambda\two,\Psi\two}$
with eigenvalue $(\Psi\one\cdot\Psi\two)^{+}_{i,s}$ for all $i\in I_{0}$ and $s\geq 0$ as desired.
\end{proof}

\begin{proof}[Proof of Proposition \ref{prop:eigenvalues of phipmipms} for $\phi^{-}_{i,-s}$ with $i\in I_{0}$]
Similarly, we have
$\psi(\phi^{-}_{i,-s}) = \sum \xi \xm_{i_{1},0} \dots \xm_{i_{k},0} k_{i}$
with every $\sum_{j=1}^{k} \alpha_{i_{j}} = s\delta$ by Lemma \ref{lem:images of Utor pm}.
It follows that summands of $\Delta^{\psi}(\phi^{-}_{i,-s})$ are ordered products of factors
$1 \otimes \xbm_{i_{j},0}$ for $j\in J_{1}$
and
$\xbm_{i_{j},-\ell_{j}} \otimes \Cb^{-\ell_{j}} \boldsymbol{\phi}^{+}_{i_{j},\ell_{j}}$ ($\ell_{j}\leq 0$) for $j\in J_{2}$,
together with $\xi\in\Cbb^{\times}$ and $k_{i}^{-1} \otimes k_{i}^{-1}$.
These lie inside
\begin{align*}
    \U_{- \sum_{j\in J_{2}\setminus J_{0}} \alpha_{i_{j}}
    + \lvert J_{0}\cap J_{2} \rvert \theta - \ell\delta,
    - \lvert J_{0}\cap J_{2} \rvert \delta'}
    \otimes
    \U_{- \sum_{j\in J_{1}\setminus J_{0}} \alpha_{i_{j}}
    + \lvert J_{0}\cap J_{1} \rvert \theta + \ell\delta,
    - \lvert J_{0}\cap J_{1} \rvert \delta'}
\end{align*}
where again
$J_{0} = \lbrace j\in [k] ~|~ i_{j} = 0 \rbrace$
and $\ell = \sum_{j=1}^{k} \ell_{j}$.
The rest of the proof is then essentially as above.
\end{proof}

Dealing with $\phi^{\pm}_{0,\pm s}$ requires extra care.
Recall from the proofs of Propositions \ref{prop:action of xpm00} and \ref{prop:action of xpm0pm1} that:
\begin{align}
    \Delta^{\psi}(\xp_{0,1})
    &=
    \xp_{0,1} \otimes 1 + (C k_{0})^{-1} \otimes \xp_{0,1}
    + \sum_{\ell>0} \xi_{\ell\delta,0} \otimes \xi_{-\theta + (1-\ell)\delta, \delta'}
    \label{eqn:Dpsi(xp01)}
    \\
    \Delta^{\psi}(\xm_{0,-1})
    &=
    1 \otimes \xm_{0,-1} + \xm_{0,-1} \otimes C k_{0}
    + \sum_{\ell>0} \xi_{\theta - (1-\ell)\delta, -\delta'} \otimes \xi_{-\ell\delta,0}
    \label{eqn:Dpsi(xm0-1)}
    \\
    \begin{split}
    \Delta^{\psi}(\xp_{0,0})
    &=
    \textstyle
    (k_{0}^{-1} \otimes k_{0}^{-1})
    \Big[ \,
    1\otimes\xbm_{i_{1},1} + \xbm_{i_{1},1}\otimes k_{\delta}^{-1} k_{i_{1}}
    + \sum_{\ell>0} \xi_{-\alpha_{i_{1}}+(\ell+1)\delta,0}\otimes \xi_{-\ell\delta,0} \, ,
    \\
    &
    \hspace{\widthof{$)=(k_{0}^{-1} \otimes k_{0}^{-1})
    \Big[ \,
    1\otimes\xbm_{i_{1},1} + \xbm_{i_{1},1}\otimes k_{\delta}^{-1} k_{i_{1}}
    + \sum_{\ell>0} \xi_{-\alpha_{i_{1}}+(\ell+1)\delta,0}\otimes \xi_{-\ell\delta,0}$}
    -
    \widthof{$1\otimes\xm_{i_{2},0} + \xm_{i_{2},0}\otimes k_{i_{2}}
    + \sum_{\ell>0} \xi_{-\alpha_{i_{2}}+\ell\delta,0}\otimes \xi_{-\ell\delta,0}$}}
    \textstyle
    1\otimes\xm_{i_{2},0} + \xm_{i_{2},0}\otimes k_{i_{2}}
    + \sum_{\ell>0} \xi_{-\alpha_{i_{2}}+\ell\delta,0}\otimes \xi_{-\ell\delta,0} \, ,
    \\
    &
    \hspace{\widthof{$)=(k_{0}^{-1} \otimes k_{0}^{-1})
    \Big[ \,
    1\otimes\xbm_{i_{1},1} + \xbm_{i_{1},1}\otimes k_{\delta}^{-1} k_{i_{1}}
    + \sum_{\ell>0} \xi_{-\alpha_{i_{1}}+(\ell+1)\delta,0}\otimes \xi_{-\ell\delta,0}$}
    -
    \widthof{$\dots,
    1\otimes\xm_{i_{h-1},0} + \xm_{i_{h-1},0}\otimes k_{i_{h-1}}
    + \sum_{\ell>0} \xi_{-\alpha_{i_{h-1}}+\ell\delta,0}\otimes \xi_{-\ell\delta,0}$}}
    \textstyle
    \dots,
    1\otimes\xm_{i_{h-1},0} + \xm_{i_{h-1},0}\otimes k_{i_{h-1}}
    + \sum_{\ell>0} \xi_{-\alpha_{i_{h-1}}+\ell\delta,0}\otimes \xi_{-\ell\delta,0}
    \, \Big]'_{q^{\epsilon_{h-2}}\dots q^{\epsilon_{1}}}
    \label{eqn:Dpsi(xp00)}
    \end{split}
    \\
    \begin{split}
    \Delta^{\psi}(\xm_{0,0})
    &=
    \textstyle
    a (-q)^{-\epsilon}
    \Big[ \,
    \xbp_{i_{1},-1}\otimes 1 + k_{\delta} k_{i_{1}}^{-1} \otimes \xbp_{i_{1},-1}
    + \sum_{\ell>0} \xi_{\ell\delta,0} \otimes \xi_{\alpha_{i_{1}}-(1+\ell)\delta,0} \, ,
    \\
    &
    \hspace{\widthof{$) = \textstyle
    a (-q)^{-\epsilon}
    \Big[ \,
    \xbp_{i_{1},-1}\otimes 1 + k_{\delta} k_{i_{1}}^{-1} \otimes \xbp_{i_{1},-1}
    + \sum_{\ell>0} \xi_{\ell\delta,0} \otimes \xi_{\alpha_{i_{1}}-(1+\ell)\delta,0}$}-\widthof{$) \textstyle
    \xp_{i_{2},0}\otimes 1 + k_{i_{2}}^{-1} \otimes \xp_{i_{2},0}
    + \sum_{\ell>0} \xi_{\ell\delta,0} \otimes \xi_{\alpha_{i_{2}}-\ell\delta,0}$}}
    \textstyle
    \xp_{i_{2},0}\otimes 1 + k_{i_{2}}^{-1} \otimes \xp_{i_{2},0}
    + \sum_{\ell>0} \xi_{\ell\delta,0} \otimes \xi_{\alpha_{i_{2}}-\ell\delta,0} \, ,
    \\
    &
    \hspace{\widthof{$) = \textstyle
    a (-q)^{-\epsilon}
    \Big[ \,
    \xbp_{i_{1},-1}\otimes 1 + k_{\delta} k_{i_{1}}^{-1} \otimes \xbp_{i_{1},-1}
    + \sum_{\ell>0} \xi_{\ell\delta,0} \otimes \xi_{\alpha_{i_{1}}-(1+\ell)\delta,0}$}-\widthof{$) \textstyle
    \dots,
    \xp_{i_{\hslash-1},0}\otimes 1 + k_{i_{\hslash-1}}^{-1} \otimes \xp_{i_{\hslash-1},0}
    + \sum_{\ell>0} \xi_{\ell\delta,0} \otimes \xi_{\alpha_{i_{\hslash-1}}-\ell\delta,0}$}}
    \textstyle
    \dots,
    \xp_{i_{\hslash-1},0}\otimes 1 + k_{i_{\hslash-1}}^{-1} \otimes \xp_{i_{\hslash-1},0}
    + \sum_{\ell>0} \xi_{\ell\delta,0} \otimes \xi_{\alpha_{i_{\hslash-1}}-\ell\delta,0}
    \, \Big]'_{q^{\epsilon_{\hslash-2}}\dots q^{\epsilon_{1}}}
    (k_{0} \otimes k_{0})
    \label{eqn:Dpsi(xm00)}
    \end{split}
\end{align}
where we employ the $\xi_{\beta+k\delta,\ell\delta'}$ notation introduced at the start of Section \ref{subsection:action of U0}.
Inputting (\ref{eqn:Dpsi(xp01)}) and (\ref{eqn:Dpsi(xm00)}) into
$\Delta^{\psi}(h_{0,1}) = (k_{0}^{-1} \otimes k_{0}^{-1})
[\Delta^{\psi}(\xp_{0,1}),\Delta^{\psi}(\xm_{0,0})]$ and expanding everything out, let us consider the action of each summand containing an $\ell>0$ factor on some arbitrary
$V\one_{\lambda\one,\Psi\one} \otimes V\two_{\lambda\two,\Psi\two}$
in
$(V\one \otimes V\two)_{\lambda}$.
From (\ref{eqn:action on weight spaces}), either
\begin{itemize}
    \item it strictly decreases
    $\langle \lambda\two,\Lambda^{\vee}_{0}\rangle$,
    \item its only $\ell>0$ factor is
    $\xi_{\delta,0} \otimes \xi_{-\theta, \delta'}$
    from $\Delta^{\psi}(\xp_{0,1})$, and it moreover contains $\xbp_{i_{1},-1}\otimes 1$, hence it fixes
    $\langle \lambda\two,\Lambda^{\vee}_{0}\rangle$
    but decreases $(\lambda\two,\theta)$,
    \item its only $\ell>0$ factor is some
    $\xi_{\delta,0} \otimes \xi_{\alpha_{i_{j}}-\delta,0}$
    from $\Delta^{\psi}(\xm_{0,0})$, and it moreover contains $\xbp_{i_{1},-1}\otimes 1$ and $(C k_{0})^{-1} \otimes \xp_{0,1}$, hence it fixes
    $\langle \lambda\two,\Lambda^{\vee}_{0}\rangle$
    but decreases $(\lambda\two,\theta)$,
\end{itemize}
and therefore the summand must act via a strictly block upper triangular matrix with respect to (\ref{eqn:block decomposition}).
The remainder of $\Delta^{\psi}(h_{0,1})$ is then given by
\begin{align*}
    \mathcal{D}
    =
    (k_{0}^{-1} \otimes k_{0}^{-1})
    [\xp_{0,1} \otimes 1 + (C k_{0})^{-1} \otimes \xp_{0,1},
    a (-q)^{-\epsilon} \mathcal{D}' (k_{0} \otimes k_{0})]
\end{align*}
where
\begin{align*}
    \mathcal{D}'
    &=
    [\xbp_{i_{1},-1}\otimes 1 + k_{\delta} k_{i_{1}}^{-1} \otimes \xbp_{i_{1},-1},
    \xp_{i_{2},0}\otimes 1 + k_{i_{2}}^{-1} \otimes \xp_{i_{2},0},
    \dots,
    \xp_{i_{\hslash-1},0}\otimes 1 + k_{i_{\hslash-1}}^{-1} \otimes \xp_{i_{\hslash-1},0}]'_{q^{\epsilon_{\hslash-2}}\dots q^{\epsilon_{1}}}
    \\
    &=
    (h\sigma)^{\otimes 2}
    \underbrace{[\xp_{i_{\hslash-1},0}\otimes 1 + k_{i_{\hslash-1}} \otimes \xp_{i_{\hslash-1},0},
    \dots,
    \xp_{i_{2},0}\otimes 1 + k_{i_{2}} \otimes \xp_{i_{2},0},
    \xp_{i_{1},-1}\otimes 1 + C^{-1} k_{i_{1}} \otimes \xp_{i_{1},-1}]_{q^{\epsilon_{1}}\dots q^{\epsilon_{\hslash-2}}}}_{\mathcal{D}''}.
\end{align*}
Since
$\epsilon_{j} = (\alpha_{i_{1}}+\dots +\alpha_{i_{j}}, \alpha_{i_{j+1}})$
and
$[a\otimes b,c\otimes d]_{u} = [a,b]_{u} \otimes cd$
whenever $[c,d] = 0$, we can recursively show that $\mathcal{D}''$ is equal to
\begin{align*}
    [\xp_{i_{\hslash -1},0},\dots,\xp_{i_{2},0},\xp_{i_{1},-1}]_{q^{\epsilon_{1}}\dots q^{\epsilon_{\hslash -2}}}
    \otimes 1
    +
    C^{-1} k_{\theta} \otimes
    [\xp_{i_{\hslash -1},0},\dots,\xp_{i_{2},0},\xp_{i_{1},-1}]_{q^{\epsilon_{1}}\dots q^{\epsilon_{\hslash -2}}}
\end{align*}
plus a sum of elementary tensors with $\xp_{i_{1},-1}$ in the right factor and some $\xp_{i_{j},0}$ in the left factor.
We include the first step to give the idea:
\begin{align*}
    &[\xp_{i_{2},0}\otimes 1 + k_{i_{2}} \otimes \xp_{i_{2},0},
    \xp_{i_{1},-1}\otimes 1 + C^{-1} k_{i_{1}} \otimes \xp_{i_{1},-1}]_{q^{\epsilon_{1}}}
    \\
    &=
    [\xp_{i_{2},0}, \xp_{i_{1},-1}]_{q^{\epsilon_{1}}} \otimes 1
    +
    C^{-1} k_{i_{1}} k_{i_{2}} \otimes
    [\xp_{i_{2},0}, \xp_{i_{1},-1}]_{q^{\epsilon_{1}}}
    +
    [k_{i_{2}}, \xp_{i_{1},-1}]_{q^{\epsilon_{1}}}
    \otimes \xp_{i_{2},0}
    +
    C^{-1} [\xp_{i_{2},0}, k_{i_{1}}]_{q^{\epsilon_{1}}}
    \otimes \xp_{i_{1},-1}
    \\
    &=
    [\xp_{i_{2},0}, \xp_{i_{1},-1}]_{q^{\epsilon_{1}}} \otimes 1
    +
    C^{-1} k_{i_{1}} k_{i_{2}} \otimes
    [\xp_{i_{2},0}, \xp_{i_{1},-1}]_{q^{\epsilon_{1}}}
    +
    (q^{-\epsilon_{1}} - q^{\epsilon_{1}}) C^{-1} k_{i_{1}} \xp_{i_{2},0}
    \otimes \xp_{i_{1},-1}
\end{align*}
It then follows that
$a (-q)^{-\epsilon} \mathcal{D}' (k_{0} \otimes k_{0})$
equals
$\xm_{0,0} \otimes k_{0} + k_{0}^{2} \otimes \xm_{0,0}$
plus a sum of elementary tensors with $\xbp_{i_{1},-1}$ in the right factor and some $\xp_{i_{j},0}$ in the left factor.
In turn, from our proof of Proposition \ref{prop:action of hpm0pm1}, we have that
\begin{align} \label{eqn:h01 block action}
    \mathcal{D}
    =
    h_{0,1}\otimes 1 + C^{-1}\otimes h_{0,1}
    + (q_{0}^{-4} - 1)(k_{0}\xp_{0,1}\otimes k_{0}^{-1}\xm_{0,0})
\end{align}
plus a sum of elementary tensors that act via strictly block upper triangular matrices with respect to our decomposition of
$(V\one \otimes V\two)_{\lambda}$.
Note that this result generalises Proposition \ref{prop:action of hpm0pm1}.

\begin{cor}
    The action of $h_{0,1}$ on any $(V\one \otimes V\two)_{\lambda}$ is block upper triangular with respect to the decomposition (\ref{eqn:block decomposition}), and its diagonal blocks describe the action of
    $h_{0,1} \otimes 1 + C^{-1} \otimes h_{0,1}$.
\end{cor}

\begin{proof}[Proof of Proposition \ref{prop:eigenvalues of phipmipms} for $\phi^{+}_{0,s}$]
The case $s = 0$ is trivial so assume otherwise.
Consider the action of
\begin{align*}
    \Delta^{\psi}(\phi^{+}_{0,s})
    =
    (q_{0} - q_{0}^{-1}) [2]_{0}^{-s} (C^{-1} \otimes C^{-1})
    [[\underbrace{\Delta^{\psi}(h_{0,1}),\dots,\Delta^{\psi}(h_{0,1})}_{s},\Delta^{\psi}(\xp_{0,1})],\Delta^{\psi}(\xm_{0,-1})]
\end{align*}
on some arbitrary
$V\one_{\lambda\one,\Psi\one} \otimes V\two_{\lambda\two,\Psi\two}$
in
$(V\one \otimes V\two)_{\lambda}$.
Inputting (\ref{eqn:Dpsi(xp01)}) and (\ref{eqn:Dpsi(xm0-1)}) into this expression and expanding everything out, by equation (\ref{eqn:action on weight spaces}) the following summands all act via strictly block upper triangular matrices with respect to our decomposition (\ref{eqn:block decomposition}):
\begin{itemize}
    \item those containing a factor
    $\xi_{\ell\delta,0} \otimes \xi_{-\theta + (1-\ell)\delta, \delta'}$
    with $\ell>0$ from (\ref{eqn:Dpsi(xp01)}),
    \item those containing a factor
    $\xi_{\theta - (1-\ell)\delta, -\delta'} \otimes \xi_{-\ell\delta,0}$
    with $\ell>0$ from (\ref{eqn:Dpsi(xm0-1)}),
    \item those containing factors $\xp_{0,1} \otimes 1$ from (\ref{eqn:Dpsi(xp01)}) and $1 \otimes \xm_{0,-1}$ from (\ref{eqn:Dpsi(xm0-1)}).
\end{itemize}
We are therefore left to consider the following three cases:
\begin{enumerate}
    \myitem[(LL)]\label{case ll} summands containing a factor $\xp_{0,1} \otimes 1$ from (\ref{eqn:Dpsi(xp01)}) and a factor $\xm_{0,-1} \otimes C k_{0}$ from (\ref{eqn:Dpsi(xm0-1)}),
    \myitem[(RR)]\label{case rr} summands containing a factor $(C k_{0})^{-1} \otimes \xp_{0,1}$ from (\ref{eqn:Dpsi(xp01)}) and a factor $1 \otimes \xm_{0,-1}$ from (\ref{eqn:Dpsi(xm0-1)}),
    \myitem[(RL)]\label{case rl} summands containing a factor $(C k_{0})^{-1} \otimes \xp_{0,1}$ from (\ref{eqn:Dpsi(xp01)}) and a factor $\xm_{0,-1} \otimes C k_{0}$ from (\ref{eqn:Dpsi(xm0-1)}).
\end{enumerate}
Inputting the identity (\ref{eqn:h01 block action}), the total of all summands in case \ref{case ll} acts via a block upper triangular matrix, with the diagonal blocks giving the action of
\begin{align*}
    &(q_{0} - q_{0}^{-1}) [2]_{0}^{-s}
    [[\underbrace{\alpha,\dots,\alpha}_{s},\xp_{0,1} \otimes 1],\xm_{0,-1} \otimes k_{0}]
    \\
    &=
    (q_{0} - q_{0}^{-1}) [2]_{0}^{-s}
    [[h_{0,1},\dots,h_{0,1},\xp_{0,1}] \otimes 1,\xm_{0,-1} \otimes k_{0}]
    \\
    &=
    (q_{0} - q_{0}^{-1}) [2]_{0}^{-s}
    [[h_{0,1},\dots,h_{0,1}, \xp_{0,1}], \xm_{0,-1}] \otimes k_{0}
    \\
    &=
    \phi^{+}_{0,s} \otimes \phi^{+}_{0,0}
\end{align*}
where we employ the shorthand notations
$\alpha = h_{0,1} \otimes 1 + 1 \otimes h_{0,1}$
and
$\beta = (q_{0}^{-4} - 1)(k_{0} \xp_{0,1} \otimes k_{0}^{-1} \xm_{0,0})$
as in Section \ref{subsection:action of U0}.
Similarly, the total of case \ref{case rr} acts by $k_{0}^{-1} \otimes \phi^{+}_{0,s}$ and so it remains to treat case \ref{case rl}, whose total acts by
\begin{align*}
    (q_{0} - q_{0}^{-1}) [2]_{0}^{-s}
    [[\underbrace{\alpha+\beta,\dots,\alpha+\beta}_{s},
    (C k_{0})^{-1} \otimes \xp_{0,1}, \xm_{0,-1} \otimes C k_{0}].
\end{align*}
Expanding out all pluses, summands with more than one $\beta$ factor are clearly strictly block upper triangular by (\ref{eqn:action on weight spaces}).
Summands with no $\beta$ factors contribute
\begin{align*}
    &(q_{0} - q_{0}^{-1}) [2]_{0}^{-s}
    [[\alpha,\dots,\alpha, (C k_{0})^{-1} \otimes \xp_{0,1}], \xm_{0,-1} \otimes C k_{0}]
    \\
    &=
    (q_{0} - q_{0}^{-1}) [2]_{0}^{-s}
    [(C k_{0})^{-1} \otimes [\alpha,\dots,\alpha,\xp_{0,1}], \xm_{0,-1} \otimes C k_{0}]
    \\
    &=
    (q_{0} - q_{0}^{-1})
    [(C k_{0})^{-1} \otimes \xp_{0,s}, \xm_{0,-1} \otimes C k_{0}]
    \\
    &=
    0,
\end{align*}
while summands with a single $\beta$ factor contribute
\begin{align*}
    &\sum_{r=1}^{s}
    (q_{0} - q_{0}^{-1}) [2]_{0}^{-s}
    [[\underbrace{\alpha,\dots,\alpha}_{s-r}, \beta, \underbrace{\alpha,\dots,\alpha}_{r-1}, k_{0}^{-1} \otimes \xp_{0,1}],\xm_{0,-1} \otimes k_{0}]
    \\
    &=
    \sum_{r=1}^{s}
    (q_{0} - q_{0}^{-1}) [2]_{0}^{r-s-1}
    [[\alpha,\dots,\alpha, \beta, k_{0}^{-1} \otimes \xp_{0,r}],\xm_{0,-1} \otimes k_{0}]
    \\
    &=
    \sum_{r=1}^{s}
    (q_{0}^{2} - q_{0}^{-2}) [2]_{0}^{r-s-1}
    [[\alpha,\dots,\alpha, \xp_{0,1} \otimes k_{0}^{-1} \phi^{+}_{0,r}], \xm_{0,-1} \otimes k_{0}]
    \\
    &=
    \sum_{r=1}^{s}
    (q_{0}^{2} - q_{0}^{-2}) [2]_{0}^{-1}
    [\xp_{0,s-r+1} \otimes k_{0}^{-1} \phi^{+}_{0,r}, \xm_{0,-1} \otimes k_{0}]
    \\
    &=
    \sum_{r=1}^{s}
    (q_{0}^{2} - q_{0}^{-2}) [2]_{0}^{-1}
    [\xp_{0,s-r+1},\xm_{0,-1}] \otimes \phi^{+}_{0,r}
    \\
    &=
    (k_{0} - k_{0}^{-1}) \otimes \phi^{+}_{0,s}
    +
    \sum_{r=1}^{s}
    \phi^{+}_{0,s-r} \otimes \phi^{+}_{0,r}.
\end{align*}
We have therefore verified that $\phi^{+}_{0,s}$ acts on $(V\one \otimes V\two)_{\lambda}$ via a block upper triangular matrix with respect to our decomposition (\ref{eqn:block decomposition}), and moreover the diagonal blocks give the action of
$\sum_{r=0}^{s} \phi^{+}_{i,r} \otimes \phi^{+}_{i,s-r}$
as desired.
\end{proof}

Treating the $\phi^{-}_{0,-s}$ case is similar, so we include slightly fewer details.
Inputting (\ref{eqn:Dpsi(xm0-1)}) and (\ref{eqn:Dpsi(xp00)}) into
$\Delta^{\psi}(h_{0,-1}) = (k_{0} \otimes k_{0})
[\Delta^{\psi}(\xp_{0,0}),\Delta^{\psi}(\xm_{0,-1})]$ and expanding everything out, the action of any summand containing some $\ell>0$ factor on
$(V\one \otimes V\two)_{\lambda}$ is easily shown to be strictly block upper triangular.
The rest of $\Delta^{\psi}(h_{0,-1})$ is given by:
\begin{align*}
    \mathcal{D}
    &=
    (k_{0} \otimes k_{0})
    [(k_{0}^{-1} \otimes k_{0}^{-1}) (h\sigma)^{\otimes 2}(\mathcal{D}''),
    1 \otimes \xm_{0,-1} + \xm_{0,-1} \otimes C k_{0}]
    \\
    \mathcal{D}''
    &=
    [1 \otimes \xm_{i_{\hslash-1},0} + \xm_{i_{\hslash-1},0} \otimes k_{i_{\hslash-1}}^{-1},
    \dots,
    1 \otimes \xm_{i_{2},0} + \xm_{i_{2},0} \otimes k_{i_{2}}^{-1},
    1 \otimes \xm_{i_{1},1} + \xm_{i_{1},1} \otimes C k_{i_{1}}^{-1} ]_{q^{\epsilon_{1}}\dots q^{\epsilon_{\hslash-2}}}
\end{align*}
Moreover $\mathcal{D}''$ is recursively shown to equal
\begin{align*}
    1 \otimes
    [\xm_{i_{\hslash -1},0},\dots,\xm_{i_{2},0}, \xm_{i_{1},1}]_{q^{\epsilon_{1}}\dots q^{\epsilon_{\hslash -2}}}
    +
    [\xm_{i_{\hslash -1},0},\dots,\xm_{i_{2},0}, \xm_{i_{1},1}]_{q^{\epsilon_{1}}\dots q^{\epsilon_{\hslash -2}}}
    \otimes C k_{\theta}^{-1}
\end{align*}
plus a sum of elementary tensors with $\xm_{i_{1},1}$ in the left factor and some $\xm_{i_{j},0}$ in the right factor.
Therefore
$(k_{0}^{-1} \otimes k_{0}^{-1}) (h\sigma)^{\otimes 2}(\mathcal{D}'')$
equals
$k_{0}^{-1} \otimes \xp_{0,0} + \xp_{0,0} \otimes k_{0}^{-2}$
plus a sum of elementary tensors with $\xbm_{i_{1},1}$ in the left factor and some $\xm_{i_{j},0}$ in the right factor.
Our proof of Proposition \ref{prop:action of hpm0pm1} then gives
\begin{align} \label{eqn:h0-1 block action}
    \mathcal{D}
    =
    h_{0,-1} \otimes C
    +
    1 \otimes h_{0,-1}
    -
    (q_{0}^{-4} - 1)
    (k_{0} \xp_{0,0} \otimes k_{0}^{-1} \xm_{0,-1})
\end{align}
plus a sum of elementary tensors which act on
$(V\one \otimes V\two)_{\lambda}$
via strictly block upper triangular matrices.

\begin{cor}
    The action of $h_{0,-1}$ on any $(V\one \otimes V\two)_{\lambda}$ is block upper triangular with respect to the decomposition (\ref{eqn:block decomposition}), and its diagonal blocks describe the action of
    $h_{0,-1} \otimes C + 1 \otimes h_{0,-1}$.
\end{cor}

\begin{proof}[Proof of Proposition \ref{prop:eigenvalues of phipmipms} for $\phi^{-}_{0,-s}$]
Again, the case $s = 0$ is trivial so assume otherwise.
Inputting (\ref{eqn:Dpsi(xp01)}) and (\ref{eqn:Dpsi(xm0-1)}) into
\begin{align*}
    \Delta^{\psi}(\phi^{-}_{0,-s})
    =
    (-1)^{s+1} (q_{0} - q_{0}^{-1}) [2]_{0}^{-s}
    (C^{-1} \otimes C^{-1})
    [\Delta^{\psi}(\xp_{0,1}),
    \underbrace{\Delta^{\psi}(h_{0,-1}),\dots,\Delta^{\psi}(h_{0,-1})}_{s},
    \Delta^{\psi}(\xm_{0,-1})]
\end{align*}
and expanding everything out, all summands except those in classes \ref{case ll}, \ref{case rr} and \ref{case rl} clearly act by strictly block upper triangular matrices on
$(V\one \otimes V\two)_{\lambda}$.
Inputting (\ref{eqn:h0-1 block action}), the totals of all \ref{case ll} and \ref{case rr} summands are block upper triangular, with the diagonal blocks giving the actions of
$\phi^{-}_{0,-s} \otimes k_{0}$
and
$\phi^{-}_{0,0} \otimes \phi^{-}_{0,-s}$
respectively.
The rest of the proof, in particular dealing with the \ref{case rl} case, is similar to before.
\end{proof}

\section{\texorpdfstring{$R$}{R}-matrices and transfer matrices} \label{section:R matrices}

A fundamental result in the representation theory of quantum groups $\Uq$ is the existence of $R$-matrices, intertwining morphisms which exchange the factors in a tensor product of modules.
Crucially, these $R$-matrices satisfy the quantum Yang-Baxter equation
\begin{align} \label{eqn:qYBE}
    \Rcal_{12} \Rcal_{13} \Rcal_{23} = \Rcal_{23} \Rcal_{13} \Rcal_{12}
\end{align}
and thus endow the category of finite dimensional modules with a natural braiding structure.
In this way, quantum groups are connected to low-dimensional topology as various knot, link and $3$-manifold invariants (such as the celebrated Jones polynomial) may be constructed using $R$-matrices on certain modules \cites{RT90,RT91}.
\\

For quantum affine algebras, such $R$-matrices were obtained by Chari-Pressley \cite{CP94} and are directly linked with quantum mechanics and integrable systems.
The quantum Yang-Baxter equation ensures that various operations (such as the scattering of particles) are consistent, and the existence of $R$-matrix solutions dictates the solvability of the model.
In particular, they are used to construct large families of commuting \emph{transfer matrices}, which can be diagonalised via Bethe ansatz techniques to study the integrable system \cites{FR99,FH15,FH18}.
\\

In another direction, quantum affine $R$-matrices are an essential tool in the monoidal categorification of cluster algebras \cites{HL10,HL13,Nakajima11,Qin17}, which can in turn be used to aid with the calculation of $q$-characters \cites{HL16,Nakajima11}.
In particular, (normalised) $R$-matrices give rise to exact sequences in certain categories of finite dimensional representations that categorify the mutation relations.
Related works provide connections to KLR algebras \cite{KKKO18} and establish generalised Schur-Weyl dualities between module categories \cites{KKK15,KKK18,Fujita20,Fujita22}, with $R$-matrices playing a fundamental role.
\\

Needless to say, the importance of $R$-matrices within mathematics and (quantum) physics cannot be overstated.
Our aim in this section is to lay the foundation for such directions on the quantum toroidal level.
In particular, we consider direct sums of tensor products of irreducible integrable $\ell$-highest weight $\Utor$-modules with respect to our topological coproduct $\Dpsi$.
Our results then prove the existence and uniqueness of $R$-matrices which satisfy the quantum Yang-Baxter equation, are generically isomorphisms, and thus equip $\Oaff$ with a meromorphic braiding on these objects.
Moreover, we are able to relate our toroidal $R$-matrices to those which already exist on the affine level, as well as define families of transfer matrices and show that they commute.
Our expectation is that these constructions should extend to the entire category, and in this way equip $\Oaff$ with a meromorphic braiding in the sense of \cites{GTL16,Soibelman99,FR92}.
We plan to address this in future work.
\\

It is worth mentioning that in the finite and affine cases, the $R$-matrices mentioned thus far can be realised as the images inside $\mathrm{End}(V_{1} \otimes V_{2})$ of a \emph{universal $R$-matrix} -- a solution $\widetilde{\Rcal}$ of (\ref{eqn:qYBE}) lying inside a completion of the tensor square of the quantum group.
Formulae for these universal $R$-matrices have moreover been obtained in \cites{KR90,KT91,LS91} and \cites{KT92,KT93}.
However, even with such explicit expressions, it is very difficult to compute the action of $\widetilde{\Rcal}$ on tensor representations in all but the simplest cases.
\\

A series of works by Negu\c{t} considers the universal $R$-matrices of quantum toroidal $\glone$ \cite{Negut23}, $\mathfrak{gl}_{n}$ \cite{Negut20} and $\mathfrak{sl}_{n}$ \cite{Negut15}.
Each is shown to factor as an infinite tensor product of $R$-matrices associated to certain quantum affine subalgebras, using shuffle techniques such as slope subalgebras.
It would be interesting to explore the connections between those results and ours -- namely the anti-involution $\psi$, $R$-matrices, and extended double affine braid group action -- in and beyond type $A$.
We leave this for future work.

\subsection{Main results}

For each $\alpha\in\lbrace 1,2,3\rbrace$, consider an irreducible integrable $\Utor$-module
$V\alphapower = V(\lambda\alphapower,\Psi\alphapower)$
with $\ell$-highest weight vector $v\alphapower$.
As in Section \ref{section:tensor products}, we may without loss of generality specialise the coproduct parameter $u$ to $1$.
For vector spaces $V$ and $W$, call $f(x)$ a $\Hom_{\Cbb}(V,W)$-valued rational function if
\begin{itemize}
    \item $f(a) \in \Hom_{\Cbb}(V,W)$ for all $a\in\Cbb$,
    \item $\langle v,f(x)w \rangle$ is a rational function in $\Cbb(x)$ for each $v\in V$ and $w\in W$.
\end{itemize}

Tensor products of any such modules with respect to our topological coproduct $\Dpsi$ possess unique $R$-matrices that depend on a spectral parameter, and satisfy the Yang-Baxter equation as desired.

\begin{thm} \label{thm:quantum toroidal R-matrices}
    There exist unique
    $\Hom_{\Cbb}(V\alphapower \otimes V\betapower,V\betapower \otimes V\alphapower)$-valued rational functions $\Rcal\alphabetapower(x)$ such that
    \begin{itemize}
        \item $\Rcal\alphabetapower(b/a)$
        is a $\Utor$-module homomorphism
        $V\alphapower_{a} \otimes V\betapower_{b} \rightarrow
        V\betapower_{b} \otimes V\alphapower_{a}$
        sending
        $v\alphapower \otimes v\betapower \mapsto
        v\betapower \otimes v\alphapower$
        whenever $\Rcal\alphabetapower(x)$ does not have a pole at $b/a$,
        \item $\Rcal\alphabetapower(b/a)$ is moreover an isomorphism if
        $V\alphapower_{a} \otimes V\betapower_{b}$
        is irreducible,
        \item the (trigonometric, quantum) Yang-Baxter equation
        \begin{align} \label{eqn:proper qYBE}
            \begin{split}
                &\big(\mathrm{Id}_{V\three} \otimes \Rcal\onetwo(b/a) \big) \circ
                \big(\Rcal\onethree(c/a) \otimes \mathrm{Id}_{V\two} \big) \circ
                \big(\mathrm{Id}_{V\one} \otimes \Rcal\twothree(c/b) \big)
                \\
                &=
                \big(\Rcal\twothree(c/b) \otimes \mathrm{Id}_{V\one} \big) \circ
                \big(\mathrm{Id}_{V\two} \otimes \Rcal\onethree(c/a) \big) \circ
                \big(\Rcal\onetwo(b/a) \otimes \mathrm{Id}_{V\three} \big)
            \end{split}
        \end{align}
        as $\Utor$-module homomorphisms
        $V\one_{a} \otimes V\two_{b} \otimes V\three_{c}
        \rightarrow
        V\three_{c} \otimes V\two_{b} \otimes V\one_{a}$
        is satisfied for all $a,b,c\in\Cbb^{\times}$ for which both maps are well-defined.
    \end{itemize}
\end{thm}

The proof is a little technical, so we postpone it to Section \ref{subsection:proof of R-matrices}.
Diagrammatically, we can view the Yang-Baxter equation (\ref{eqn:proper qYBE}) as the equality of braids in Figure \ref{fig:YBE}, where strands are coloured according to the spectral parameter and morphisms are applied from top to bottom.
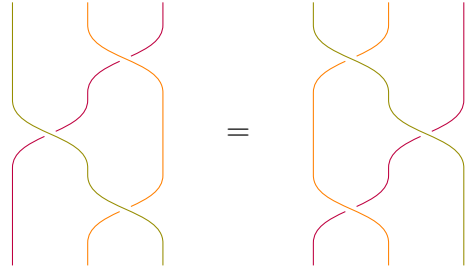
\begin{figure}[H]
    \centering
\begin{tikzpicture}
\pic[
braid/anchor=east,
braid/strand 1/.style={olive},
braid/strand 2/.style={orange},
braid/strand 3/.style={purple},
] at (1,1) {braid={s_2 s_1 s_2}};
\node[font=\large] at (2,1) {\(=\)};
\pic[
braid/anchor=west,
braid/strand 1/.style={olive},
braid/strand 2/.style={orange},
braid/strand 3/.style={purple},
] at (3,1) {braid={s_1 s_2 s_1}};
\end{tikzpicture}
\caption[An illustration of the Yang-Baxter equation]{\hspace{.5em}An illustration of the Yang-Baxter equation}
\label{fig:YBE}
\end{figure}

Let us now generalise the above to direct sums
$W\alphapower = \bigoplus_{k=1}^{K}
V^{(\alpha_{k1})} \otimes\dots\otimes V^{(\alpha_{kL})}$
of tensor products of irreducible representations $V^{(\alpha_{k\ell})} \in \Oaff$ with $\ell$-highest weight vectors $v^{(\alpha_{k\ell})}$.
In this case, we shall use the following notations:
\begin{itemize}
    \item $W^{(\alpha_{k})} = V^{(\alpha_{k1})} \otimes \dots \otimes V^{(\alpha_{kL})}$
    \item $w^{(\alpha_{k})} = v^{(\alpha_{k1})} \otimes \dots \otimes v^{(\alpha_{kL})}$
    \item $\lambda^{(\alpha_{k})} = \lambda^{(\alpha_{k1})} + \dots + \lambda^{(\alpha_{kL})}$
    \item $w\alphapower = \bigoplus_{k=1}^{K}
    v^{(\alpha_{k1})} \otimes\dots\otimes v^{(\alpha_{kL})}$
    \item $\lambda\betapower = \sum_{k,\ell=1}^{K,L} \lambda^{(\alpha_{k\ell})}$
\end{itemize}
Note that
$W\alphapower_{a} = \bigoplus_{k=1}^{K}
V^{(\alpha_{k1})}_{a} \otimes\dots\otimes V^{(\alpha_{kL})}_{a}$
for all $a\in\Cbb^{\times}$ by equation (\ref{eqn:image of new topological coproduct}).

\begin{cor} \label{cor:quantum toroidal R-matrices}
    There exist unique
    $\Hom_{\Cbb}(W\alphapower \otimes W\betapower,W\betapower \otimes W\alphapower)$-valued rational functions
    \begin{align*}
        \Rcal\alphabetapower(x)
        =
        \bigoplus_{k,r=1}^{K,R}
        \Big(\Rcal^{(\alpha_{k1},\beta_{rS})}(x)
        \dots
        \Rcal^{(\alpha_{k1},\beta_{r1})}(x)\Big)
        \dots
        \Big(\Rcal^{(\alpha_{kL},\beta_{rS})}(x)
        \dots
        \Rcal^{(\alpha_{kL},\beta_{r1})}(x)\Big)
    \end{align*}
    such that
    \begin{itemize}
        \item $\Rcal\alphabetapower(b/a)$
        is a $\Utor$-module homomorphism sending
        $w\alphapower \otimes w\betapower \mapsto
        w\betapower \otimes w\alphapower$
        whenever no
        $\Rcal^{(\alpha_{k\ell},\beta_{rs})}(x)$
        has a pole at $b/a$,
        \item $\Rcal\alphabetapower(b/a)$ is moreover an isomorphism if every
        $V^{(\alpha_{k\ell})}_{a} \otimes V^{(\beta_{rs})}_{b}$
        is irreducible,
        \item the Yang-Baxter equation (\ref{eqn:proper qYBE}) holds whenever both sides are well-defined.
    \end{itemize}
\end{cor}
\begin{proof}
    That $\Rcal\alphabetapower(b/a)$ satisfies the first two conditions is trivial, while the third follows easily from the Yang-Baxter equations for all
    $\Rcal^{(\alpha_{k\ell},\beta_{rs})}(x)$.
    We illustrate this through braid diagrams in the case
    \begin{align*}
        W\one = V\one\otimes V\two, \qquad
        W\two = V\three, \qquad
        W\three = V^{(4)}.
    \end{align*}
    The first and last equalities below come from swapping the order of morphisms which act on entirely different factors, while the second and third are due to (\ref{eqn:proper qYBE}) within the context of Theorem \ref{thm:quantum toroidal R-matrices}.
    \smallskip
    \begin{center}
\resizebox{\textwidth}{!}{
\begin{tikzpicture}
\pic[
braid/anchor=east,
braid/strand 1/.style={olive},
braid/strand 2/.style={olive},
braid/strand 3/.style={orange},
braid/strand 4/.style={purple},
] at (1,1) {braid={s_3 s_2 s_1 s_3 s_2}};
\node[font=\huge] at (2,1) {\(=\)};
\pic[
braid/anchor=west,
braid/strand 1/.style={olive},
braid/strand 2/.style={olive},
braid/strand 3/.style={orange},
braid/strand 4/.style={purple},
] at (3,1) {braid={s_3 s_2 s_3 s_1 s_2}};
\node[font=\huge] at (7,1) {\(=\)};
\pic[
braid/anchor=west,
braid/strand 1/.style={olive},
braid/strand 2/.style={olive},
braid/strand 3/.style={orange},
braid/strand 4/.style={purple},
] at (8,1) {braid={s_2 s_3 s_2 s_1 s_2}};
\node[font=\huge] at (12,1) {\(=\)};
\pic[
braid/anchor=west,
braid/strand 1/.style={olive},
braid/strand 2/.style={olive},
braid/strand 3/.style={orange},
braid/strand 4/.style={purple},
] at (13,1) {braid={s_2 s_3 s_1 s_2 s_1}};
\node[font=\huge] at (17,1) {\(=\)};
\pic[
braid/anchor=west,
braid/strand 1/.style={olive},
braid/strand 2/.style={olive},
braid/strand 3/.style={orange},
braid/strand 4/.style={purple},
] at (18,1) {braid={s_2 s_1 s_3 s_2 s_1}};
\end{tikzpicture}
}
\end{center}
    \smallskip
    The same argument shows that we may keep adding extra tensor factors to $W\one$.
    Proving this for $W\two$ and $W\three$ is similar, and compatibility with direct sums is clear.
    As in Section \ref{subsection:proof of R-matrices}, uniqueness follows from Theorem \ref{thm:irreducible except countable}, Corollary \ref{cor:irreducible switch factors}, Schur's lemma, and the rationality of $\Rcal\alphabetapower(b/a)$.
\end{proof}

The next result relates our quantum toroidal $R$-matrices with those obtained by Chari-Pressley on the affine level.
In particular, the toroidal $R$-matrices can in some sense be formed by \emph{gluing together} infinitely many affine $R$-matrices in an appropriate way.

\begin{defn}
    For any proper subset $J\subset I$, let $\Rcal\alphabetapower_{J}(x)$ be the restriction of $\Rcal\alphabetapower(x)$ to
    $(W\alphapower \otimes W\betapower)(J) = W\alphapower(J) \otimes W\betapower(J)$.
\end{defn}

\begin{prop}
    Whenever it is well-defined, $\Rcal\alphabetapower_{J}(b/a)$ is a morphism
    $(W\alphapower_{a})(J) \otimes (W\betapower_{b})(J) \rightarrow (W\betapower_{b})(J) \otimes (W\alphapower_{a})(J)$
    of $\U(J)$-modules which coincides with the quantum affine $R$-matrix obtained from \cite{CP94}*{Thm. 12.5.5}.
\end{prop}
\begin{proof}
    Without loss of generality take $a=1$, as in the proof of Theorem \ref{thm:irreducible except countable}.
    It is clear that $\Rcal\alphabetapower(b)$ preserves the weight of any vector and therefore sends
    $(W\alphapower \otimes W\betapower_{b})(J) \rightarrow (W\betapower_{b} \otimes W\alphapower)(J)$,
    since it intertwines the actions of all
    $\Delta^{\psi}(k_{i}^{\pm 1}) = k_{i}^{\pm 1} \otimes k_{i}^{\pm 1}$.
    Thus by Lemma \ref{lem:decomposing submodules of tensor products} it restricts to a morphism of $\U(J)$-modules
    $\Rcal\alphabetapower_{J}(b) :
    (W\alphapower)(J) \otimes (W\betapower_{b})(J)
    \rightarrow
    (W\betapower_{b})(J) \otimes (W\alphapower)(J)$
    which maps
    $w\alphapower \otimes w\betapower \mapsto w\betapower \otimes w\alphapower$.
    \\

    If $W\alphapower \otimes W\betapower_{b}$ is an irreducible representation of $\Utor$, then
    $(W\alphapower)(J) \otimes (W\betapower_{b})(J)
    =
    (W\alphapower \otimes W\betapower_{b})(J)$
    must be an irreducible $\U(J)$-module.
    In this case, $\Rcal\alphabetapower_{J}(b)$ must equal the $R$-matrix from \cite{CP94}*{Thm. 12.5.5} by Schur's lemma.
    It follows that when $W\alphapower \otimes W\betapower_{b}$ is a \emph{sum} of irreducibles, $\Rcal\alphabetapower_{J}(b)$ also coincides with the quantum affine $R$-matrix.
    With respect to fixed bases, each morphism has matrix coefficients which are rational functions in $b$.
    Then since they take the same values at all but countably many $b\in\Cbb^{\times}$, the functions themselves must be equal and so we are done.
\end{proof}

Using our quantum toroidal $R$-matrices, we can now define a family of \emph{transfer matrices} acting on each of the representations above.
Furthermore, the commutativity of these families comes as a direct consequence of the Yang-Baxter equation (\ref{eqn:proper qYBE}).
\\

On the affine level, such constructions have been used to establish the integrability of the corresponding quantum system via Bethe ansatz techniques.
Transfer matrices and their spectra are also important for understanding (Grothendieck rings of) the underlying module categories for $\Uaff$ \cites{FH15,FH18}.
We plan to explore these directions within the quantum toroidal setting in future work.
\\

For any $V\alphapower$ and $V\betapower$, define the associated transfer matrix $\Tcal\alphabetapower(x)$ to be the $\mathrm{End}_{\Cbb}(V\alphapower)$-valued rational function given by
\begin{align} \label{eqn:transfer matrix definition}
    \Rcal\alphabetapower(b/a) (u\otimes v\betapower)
    =
    v\betapower \otimes \Tcal\alphabetapower(b/a)(u)
    \mod{\sum_{\mu\lneq\lambda\betapower} V\betapower_{\mu} \otimes V\alphapower}
\end{align}
for all $a,b\in\Cbb^{\times}$ and $u\in V\alphapower$, whenever $\Rcal\alphabetapower(x)$ does not have a pole at $b/a$.
Note in particular that every $\Tcal\alphabetapower(x)$ fixes the $\ell$-highest weight vector $v\alphapower$.
The next theorem ensures that these form sets of commuting $\Cbb$-linear operators on each irreducible representation in $\Oaff$.

\begin{thm} \label{thm:transfer matrices commute}
    We have
    $[\Tcal\onetwo(b/a),\Tcal\onethree(c/a)] = 0$
    for all $V\one$, $V\two$, $V\three$ and $a,b,c\in\Cbb^{\times}$ such that both transfer matrices are well-defined.
\end{thm}
\begin{proof}
    This follows simply by applying the Yang-Baxter equation (\ref{eqn:proper qYBE}) to
    $u \otimes v\two \otimes v\three$
    for any $u\in V\one$.
    In particular, our $R$-matrices are $\Utor$-module homomorphisms and thus weight-preserving by (\ref{eqn:Dpsi on C and k}), whereby:
    \begin{align*}
        u \otimes v\two \otimes v\three
        &\xrightarrow{\mathrm{Id}_{V\one} \otimes \Rcal\twothree(c/b)}
        u \otimes v\three \otimes v\two
        \\
        &\xrightarrow{\Rcal\onethree(c/a) \otimes \mathrm{Id}_{V\two}}
        v\three \otimes \Tcal\onethree(c/a) (u) \otimes v\two
        \mod{L\three}
        \\
        &\xrightarrow{\mathrm{Id}_{V\three} \otimes \Rcal\onetwo(b/a)}
        v\three \otimes v\two \otimes \Tcal\onetwo(b/a) \Tcal\onethree(c/a) (u)
        \mod{L^{(3,2)}}
        \\
        u \otimes v\two \otimes v\three
        &\xrightarrow{\Rcal\onetwo(b/a) \otimes \mathrm{Id}_{V\three}}
        v\two \otimes \Tcal\onetwo(b/a) (u) \otimes v\three
        \mod{L\two}
        \\
        &\xrightarrow{\mathrm{Id}_{V\two} \otimes \Rcal\onethree(c/a)}
        v\two \otimes v\three \otimes \Tcal\onethree(c/a) \Tcal\onetwo(b/a) (u)
        \mod{L^{(2,3)}}
        \\
        &\xrightarrow{\Rcal\twothree(c/b) \otimes \mathrm{Id}_{V\one}}
        v\three \otimes v\two \otimes \Tcal\onethree(c/a) \Tcal\onetwo(b/a) (u)
        \mod{L^{(3,2)}}
    \end{align*}
    where we let
    \begin{align*}
        L\two
        &=
        \sum_{\mu\lneq\lambda\two}
        V\two_{\mu} \otimes V\one \otimes V\three,
        \\
        L\three
        &=
        \sum_{\mu\lneq\lambda\three}
        V\three_{\mu} \otimes V\one \otimes V\two,
        \\
        L^{(2,3)}
        &=
        \sum_{\mu\lneq\lambda\two}
        V\two_{\mu} \otimes V\three \otimes V\one
        +
        \sum_{\mu\lneq\lambda\three}
        V\two \otimes V\three_{\mu} \otimes V\one,
        \\
        L^{(3,2)}
        &=
        \sum_{\mu\lneq\lambda\three}
        V\three_{\mu} \otimes V\two \otimes V\one
        +
        \sum_{\mu\lneq\lambda\two}
        V\three \otimes V\two_{\mu} \otimes V\one.
        \qedhere
    \end{align*}
\end{proof}

As with our $R$-matrices above, we can extend the transfer matrix construction to all direct sums $W\alphapower$ and $W\betapower$ of tensor products of simple objects in $\Oaff$.
Indeed, if we introduce some further notations
\begin{itemize}
    \item $\Tcal^{(\alpha_{k\ell},\beta_{r})}(x) = \Tcal^{(\alpha_{k\ell},\beta_{rS})}(x) \dots \Tcal^{(\alpha_{k\ell},\beta_{r1})}(x)$
    \item $\Tcal^{(\alpha_{k},\beta_{r})}(x) =
    \Tcal^{(\alpha_{k1},\beta_{rS})}(x) \dots \Tcal^{(\alpha_{k1},\beta_{r1})}(x)
    \otimes \dots \otimes
    \Tcal^{(\alpha_{kL},\beta_{rS})}(x) \dots \Tcal^{(\alpha_{kL},\beta_{r1})}(x)$
    \item $\Tcal^{(\alpha,\beta_{r})}(x)
    = \bigoplus_{k=1}^{K}
    \Tcal^{(\alpha_{k1},\beta_{rS})}(x) \dots \Tcal^{(\alpha_{k1},\beta_{r1})}(x)
    \otimes \dots \otimes
    \Tcal^{(\alpha_{kL},\beta_{rS})}(x) \dots \Tcal^{(\alpha_{kL},\beta_{r1})}(x)$
\end{itemize}
then it is relatively easy to show that equation (\ref{eqn:transfer matrix definition}) generalises to
\begin{align*}
    \Rcal\alphabetapower(b/a) (u \otimes w\betapower)
    =
    \bigoplus_{r=1}^{R}
    w^{(\beta_{r})} \otimes
    \Tcal^{(\alpha,\beta_{r})}(b/a) (u)
    \mod{\bigoplus_{r=1}^{R}
    \left(\sum_{\mu\leq\lambda^{(\beta_{r})}} W^{(\beta_{r})}_{\mu} \right)
    \otimes W\alphapower}
\end{align*}
for any $a,b\in\Cbb^{\times}$ and
$u = \bigoplus_{k=1}^{K} u_{k1} \otimes \dots \otimes u_{kL}$
in $W\alphapower$,
whenever $\Rcal\alphabetapower(b/a)$ is well-defined.
With this in mind, we can define a transfer matrix
\begin{align*}
    \Tcal\alphabetapower(x)
    =
    \bigoplus_{k=1}^{K}
    \Big(\sum_{r=1}^{R} \Tcal^{(\alpha_{k1},\beta_{rS})}(x)
    \dots
    \Tcal^{(\alpha_{k1},\beta_{r1})}(x)\Big)
    \otimes \dots \otimes
    \Big(\sum_{r=1}^{R} \Tcal^{(\alpha_{k\ell},\beta_{rS})}(x)
    \dots
    \Tcal^{(\alpha_{k\ell},\beta_{r1})}(x)\Big)
\end{align*}
associated to $\Rcal\alphabetapower(x)$, which scales the direct sum $w\alphapower$ of highest weight vectors by $R$.
Furthermore, the commutativity of all such endomorphisms extends to this broader setting as desired.

\begin{cor} \label{cor:transfer matrices commute}
    We have
    $[\Tcal\onetwo(b/a),\Tcal\onethree(c/a)] = 0$
    for all $W\one$, $W\two$, $W\three$ and $a,b,c\in\Cbb^{\times}$ such that both transfer matrices are well-defined.
\end{cor}
\begin{proof}
    This follows immediately from Theorem \ref{thm:transfer matrices commute}.
\end{proof}

Let $\widehat{\mathcal{O}}_{\mathrm{irr}}^{\oplus,\otimes}$ be the full subcategory of $\Oaff$ on direct sums of tensor products of irreducible modules.
Then by construction, $W\betapower \mapsto \Tcal\alphabetapower(x)$ defines a ring homomorphism
\begin{align*}
    K(\widehat{\mathcal{O}}_{\mathrm{irr}}^{\oplus,\otimes})
    \rightarrow
    \mathrm{End}_{\Cbb}(W\alphapower)(x)
\end{align*}
from its Grothendieck group to the algebra of $\mathrm{End}_{\Cbb}(W\alphapower)$-valued rational functions, which should in fact extend to all of $K(\Oaff)$.
Of course, the image forms a commutative subring inside $\mathrm{End}_{\Cbb}(W\alphapower)(x)$ by Corollary \ref{cor:transfer matrices commute}.
It is worth noting that just as we expect our $R$-matrices to be the images in
$\mathrm{End}(W\alphapower\otimes W\betapower)$
of a universal $R$-matrix, the transfer matrices should similarly come from an element in some completion of $\Utor$.

\begin{rmk}
    As in Section \ref{section:tensor products}, our results here carry over to $\Utorglone$ in an appropriate way.
    There they match those of \cite{Miki07}*{§7}, after accounting for the difference between our coproduct $\Dpsi$ and the one used by Miki.
    Applications in this case to quantum integrable systems have moreover been considered in \cites{FJMM15,FJMM17,FJM19}.
\end{rmk}

\subsection{Proof of Theorem \ref{thm:quantum toroidal R-matrices}} \label{subsection:proof of R-matrices}

As in our proof of Theorem \ref{thm:irreducible except countable}, we may without loss of generality take $a=1$.
Note that once existence is verified, uniqueness follows easily by Theorem \ref{thm:irreducible except countable}, Corollary \ref{cor:irreducible switch factors}, and Schur's lemma.
Fix a basis $\lbrace v_{1},\dots,v_{m}\rbrace$ for each non-zero weight space $(V\alphapower\otimes V\betapower)_{\mu}$, and define elements
\begin{align*}
    v\alphabetapower_{\ivector,\kvector}(b)
    &=
    \xm_{i_{1},k_{1}}\cdots\xm_{i_{s},k_{s}}\cdot (v\alphapower \otimes v\betapower)
    \in (V\alphapower\otimes V\betapower_{b})_{\mu}
    \\
    v\betaalphapower_{\ivector,\kvector}(b)
    &=
    \xm_{i_{1},k_{1}}\cdots\xm_{i_{s},k_{s}}\cdot (v\betapower \otimes v\alphapower)
    \in (V\betapower_{b}\otimes V\alphapower)_{\mu}
\end{align*}
for each $\ivector = (i_{1},\dots,i_{s})$ and $\kvector = (k_{1},\dots,k_{s})$ such that $\sum \alpha_{i_{j}} = \lambda\alphapower + \lambda\betapower - \mu$.
Each $v\alphabetapower_{\ivector,\kvector}(b)$ can be written as a linear combination of $v_{1},\dots,v_{m}$ with coefficients in $\Cbb[b^{\pm 1}]$.
\\

Fixing some $b_{0}$ lying outside the countable subset $S \subset \Cbb^{\times}$ for which $V\alphapower\otimes V\betapower_{b_{0}}$ is reducible, condition (\ref{eqn:first irreducibility condition}) holds and thus the $v\alphabetapower_{\ivector,\kvector}(b_{0})$ span $(V\alphapower\otimes V\betapower_{b_{0}})_{\mu}$.
Then for all $1\leq i \leq m$ we can conversely write
$w_{i} = \sum_{j=1}^{m} r^{(b_{0})}_{ij}(b_{0}) v\alphabetapower_{\ivector,\kvector}(b_{0})$
for some sequences $\ivector_{j}$, $\kvector_{j}$ and rational functions $r^{(b_{0})}_{ij}(x) \in \Cbb(x)$ which are regular at $x=b_{0}$.
\\

Whenever
$b\not\in P_{b_{0},\mu} = \bigcup_{i,j} \lbrace \mathrm{poles~of~} r^{(b_{0})}_{ij}(x) \rbrace$
we still have
$w_{i} = \sum_{j=1}^{m} r^{(b_{0})}_{ij}(b) v\alphabetapower_{\ivector,\kvector}(b)$,
and can therefore define a
$\Hom_{\Cbb}((V\alphapower \otimes V\betapower)_{\mu},(V\betapower \otimes V\alphapower)_{\mu})$-valued rational function by
\begin{align*}
    w_{i} = \sum_{j=1}^{m} r^{(b_{0})}_{ij}(b) v\alphabetapower_{\ivector,\kvector}(b)
    \mapsto
    \sum_{j=1}^{m} r^{(b_{0})}_{ij}(b) v\betaalphapower_{\ivector,\kvector}(b)
\end{align*}
for all $1\leq i \leq m$ and $b\not\in P_{b_{0},\mu}$.
Then summing over all $\mu$ produces a
$\Hom_{\Cbb}(V\alphapower \otimes V\betapower,V\betapower \otimes V\alphapower)$-valued rational function $\Rcal\alphabetapower_{b_{0}}(x)$ whose poles are contained in $P_{b_{0}} = \bigcup_{\mu} P_{b_{0},\mu}$.
\\

In order to verify that $\Rcal\alphabetapower_{b_{0}}(x)$ is independent of $b_{0}$, fix some other $b_{1}\not\in S$ and take any $b$ outside the countable set $S\cup P_{b_{0}}\cup P_{b_{1}}$.
Then
$\Rcal\alphabetapower_{b_{0}}(x), \Rcal\alphabetapower_{b_{1}}(x) :
V\alphapower \otimes V\betapower_{b} \rightarrow V\betapower_{b} \otimes V\alphapower$
each map $v\alphapower \otimes v\betapower \mapsto v\betapower \otimes v\alphapower$ by definition, so are both non-zero isomorphisms and thus equal due to Corollary \ref{cor:irreducible switch factors} and Schur's lemma.
Hence $\Rcal\alphabetapower_{b_{0}}(x) = \Rcal\alphabetapower_{b_{1}}(x)$ and we can drop the subscript from now on.
Moreover, the poles of $\Rcal\alphabetapower(x)$ are contained in $S$ since each $b_{0}\not\in P_{b_{0}}$, and the second part of the statement is proved.
\\

When $b$ lies outside the countable set $S\cup P_{b_{0}}$ we know that
$\Rcal\alphabetapower(b) : V\alphapower \otimes V\betapower_{b} \rightarrow V\betapower_{b} \otimes V\alphapower$
intertwines the action of $\Utor$ on each side.
With respect to fixed bases, both actions have coefficients in $\Cbb[b^{\pm 1}]$ by (\ref{eqn:technical fact 1}), (\ref{eqn:technical fact 2}) and the surrounding discussion.
Since $\Rcal\alphabetapower(b)$ has matrix coefficients in $\Cbb(b)$, the intertwining property must extend to all $b\in\Cbb^{\times}$ which are not poles of $\Rcal\alphabetapower(b)$ and thus our proof of the first part of the Theorem \ref{thm:quantum toroidal R-matrices} is complete.
\\

In order to verify that our $R$-matrices do indeed satisfy the trigonometric quantum Yang-Baxter equation, first note that each side of (\ref{eqn:proper qYBE}) maps
$v\one \otimes v\two \otimes v\three \mapsto v\three \otimes v\two \otimes v\one$
and is therefore a non-zero homomorphism.
By Theorem \ref{thm:irreducible except countable} both
$V\one \otimes V\two_{b} \otimes V\three_{c}$ and
$V\three_{c} \otimes V\two_{b} \otimes V\one$
are irreducible for all but countably many pairs $(b,c)$, in which case equation (\ref{eqn:proper qYBE}) holds by Schur's lemma.
But as the complement of a countable set is Zariski dense in $\Cbb^{2}$, the matrix coefficients for each side of (\ref{eqn:proper qYBE}) -- which are rational functions in $b$ and $c$ -- must in fact be equal.
$\hfill \qed$

\pagebreak


\addcontentsline{toc}{section}{References}

\begin{bibsection}
\begin{biblist}*{labels={shortalphabetic}}

\bib{AGW23}{article}{
    label={AGW23},
    title={The $R$-matrix of the affine Yangian},
    author={A. Appel},
    author={S. Gautam},
    author={C. Wendlandt},
    journal={arXiv preprint},
    year={2023},
    note={\url{https://doi.org/10.48550/arXiv.2309.02377}},
}

\bib{Beck94}{article}{
    label={Be94},
    title={Braid group action and quantum affine algebras},
    author={J. Beck},
    journal={Commun. Math. Phys.},
    volume={165},
    number={3},
    pages={555--568},
    year={1994},
    note={\url{https://doi.org/10.1007/BF02099423}},
}

\bib{BKOP14}{article}{
    label={BKOP14},
    title={Construction of Irreducible Representations over Khovanov-Lauda-Rouquier Algebras of Finite Classical Type},
    author={G. Benkart},
    author={S.-J. Kang},
    author={S.-J. Oh},
    author={E. Park},
    journal={Int. Math. Res. Not.},
    volume={2014},
    number={5},
    pages={1312--1366},
    year={2014},
    note={\url{https://doi.org/10.1093/imrn/rns244}},
}

\bib{Bourbaki68}{book}{
    label={Bo68},
    title={El\'{e}ments de math\'{e}matique},
    subtitle={Groupes et alg\`{e}bres de Lie, Chapitres 4, 5 et 6},
    year={1968},
    author={N. Bourbaki},
    publisher={Hermann},
    address={Paris},
    note={\url{https://doi.org/10.1007/978-3-540-34491-9}},
}

\bib{BS12}{article}{
    label={BS12},
    title={On the Hall algebra of an elliptic curve, I},
    author={I. Burban},
    author={O. Schiffmann},
    journal={Duke Math. J.},
    volume={161},
    number={7},
    pages={1171--1231},
    year={2012},
    note={\url{https://doi.org/10.1215/00127094-1593263}},
}

\bib{CP91}{article}{
    label={CP91},
    title={Quantum affine algebras},
    author={V. Chari},
    author={A. Pressley},
    journal={Commun. Math. Phys.},
    volume={142},
    pages={261--283},
    year={1991},
    note={\url{https://doi.org/10.1007/BF02102063}},
}

\bib{CP94}{book}{
    label={CP94},
    title={A Guide to Quantum Groups},
    author={V. Chari},
    author={A. Pressley},
    year={1994},
    publisher={Cambridge University Press},
    series={Graduate Studies in Mathematics},
    volume={42},
    note={\url{https://doi-org.ezproxy-prd.bodleian.ox.ac.uk/10.1090/gsm/042}},
}

\bib{CP95}{article}{
    label={CP95},
    title={Quantum affine algebras and their representations},
    author={V. Chari},
    author={A. Pressley},
    conference={
    title={``Representations of Groups'',  CMS Conf. Proc.},
    date={1994},
    address={Banff, AB},
    },
    book={
    address={Providence, RI},
    volume={16},
    publisher={Amer. Math. Soc.},
    date={1995},
    },
    pages={59--78},
    note={\url{https://doi.org/10.48550/arXiv.hep-th/9411145}},
}

\bib{CP97}{article}{
    label={CP97},
    title={Quantum affine algebras at roots of unity},
    author={V. Chari},
    author={A. Pressley},
    journal={Represent. Theory},
    volume={1},
    pages={280--328},
    year={1997},
    note={\url{https://doi.org/10.1090/S1088-4165-97-00030-7}},
}

\bib{Cherednik95}{article}{
    label={C95},
    title={Macdonald's evaluation conjectures and difference Fourier transform},
    author={I. Cherednik},
    journal={Invent. Math.},
    volume={122},
    year={1995},
    pages={119--145},
    note={\url{https://doi.org/10.1007/BF01231441}},
}

\bib{Cherednik05}{book}{
    label={C05},
    title={Double Affine Hecke Algebras},
    author={I. Cherednik},
    year={2005},
    publisher={Cambridge University Press},
    volume={319},
    series={London Mathematical Society Lecture Note Series},
    note={\url{https://doi.org/10.1017/CBO9780511546501}},
}

\bib{CGGS21a}{article}{
    label={CGGS21a},
    title={Punctual Hilbert schemes for Kleinian singularities as quiver varieties},
    author={A. Craw},
    author={S. Gammelgaard},
    author={A. Gyenge},
    author={B. Szendr\H{o}i},
    journal={Algebraic Geometry},
    volume={8},
    number={6},
    pages={680--704},
    year={2021},
    note={\url{https://doi.org/10.14231/AG-2021-021}},
}

\bib{CGGS21b}{article}{
    label={CGGS21b},
    title={Quot schemes for Kleinian orbifolds},
    author={A. Craw},
    author={S. Gammelgaard},
    author={A. Gyenge},
    author={B. Szendr\H{o}i},
    journal={Symmetry, Integrability and Geometry: Methods and Applications (SIGMA)},
    volume={17},
    pages={099},
    year={2021},
    note={\url{https://doi.org/10.3842/SIGMA.2021.099}},
}

\bib{DM25a}{article}{
    label={DM25a},
    title={Intertwiners of representations of untwisted quantum affine algebras and Yangians revisited},
    author={K. Dahiya},
    author={E. Mukhin},
    journal={arXiv preprint},
    year={2025},
    note={\url{https://doi.org/10.48550/arXiv.2503.09845}},
}

\bib{DM25b}{article}{
    label={DM25b},
    title={Intertwiners of representations of twisted quantum affine algebras},
    author={K. Dahiya},
    author={E. Mukhin},
    journal={arXiv preprint},
    year={2025},
    note={\url{https://doi.org/10.48550/arXiv.2503.12716}},
}

\bib{Damiani12}{article}{
    label={Da12},
    title={Drinfeld Realization of Affine Quantum Algebras: The Relations},
    author={I. Damiani},
    journal={Publ. Res. Inst. Math. Sci.},
    volume={48},
    number={3},
    pages={661--733},
    year={2012},
    note={\url{https://doi.org/10.2977/prims/86}},
}

\bib{Damiani15}{article}{
    label={Da15},
    title={From the Drinfeld realization to the Drinfeld-Jimbo presentation of affine quantum algebras: injectivity},
    author={I. Damiani},
    journal={Publ. Res. Inst. Math. Sci.},
    volume={51},
    number={1},
    pages={131--171},
    year={2015},
    note={\url{https://doi.org/10.4171/prims/150}},
}

\bib{Damiani24}{article}{
    label={Da24},
    title={On the Drinfeld coproduct},
    author={I. Damiani},
    journal={Pure Appl. Math. Q.},
    volume={20},
    number={1},
    pages={171--232},
    year={2024},
    note={\url{https://dx.doi.org/10.4310/PAMQ.2024.v20.n1.a6}},
}

\bib{DF93}{article}{
    label={DF93},
    title={Isomorphism of two realizations of quantum affine algebra $U_{q}(\widehat{\mathfrak{gl}(n)})$},
    author={J. Ding},
    author={I. Frenkel},
    journal={Commun. Math. Phys.},
    volume={156},
    pages={277--300},
    year={1993},
    note={\url{https://doi.org/10.1007/BF02098484}},
}

\bib{DI97}{article}{
    label={DI97},
    title={Generalization of Drinfeld Quantum Affine Algebras},
    author={J. Ding},
    author={K. Iohara},
    journal={Lett. Math. Phys.},
    volume={41},
    pages={181--193},
    year={1997},
    note={\url{https://doi.org/10.1023/A:1007341410987}},
}

\bib{Drinfeld88}{article}{
    label={Dr88},
    title={A new realization of Yangians and quantized affine algebras},
    author={V. G. Drinfeld},
    journal={Sov. Math. Dokl.},
    volume={36},
    pages={212--216},
    year={1988},
    note={\url{https://doi.org/10.1016/0393-0440(93)90070-U}},
}

\bib{Enriquez00}{article}{
    label={E00},
    title={On correlation functions of drinfeld currents and shuffle algebras},
    author={B. Enriquez},
    journal={Transform. Groups},
    volume={5},
    number={2},
    pages={111--120},
    year={2000},
    note={\url{https://doi.org/10.1007/BF01236465}},
}

\bib{Enriquez03}{article}{
    label={E03},
    title={PBW and duality theorems for quantum groups and quantum current algebras},
    author={B. Enriquez},
    journal={J. Lie Theory},
    volume={13},
    number={1},
    pages={21--64},
    year={2003},
    note={\url{https://eudml.org/doc/127692}},
}

\bib{FHKS24}{article}{
    label={FHKS24},
    title={Young wall construction of level-1 highest weight crystals over $U_{q}(D_{4}^{(3)})$ and $U_{q}(G_{2}^{(1)})$},
    author={Z. Fan},
    author={S. Han},
    author={S.-J. Kang},
    author={Y.-S. Shin},
    journal={J. Algebra},
    volume={655},
    pages={376--404},
    year={2024},
    note={\url{https://doi.org/10.1016/j.jalgebra.2023.08.001}},
}

\bib{FFJMM11}{article}{
    label={FFJMM11},
    title={Quantum continuous $\mathfrak{gl}_{\infty}$: Semiinfinite construction of representations},
    author={B. Feigin},
    author={E. Feigin},
    author={M. Jimbo},
    author={T. Miwa},
    author={E. Mukhin},
    journal={Kyoto J. Math.},
    volume={51},
    number={2},
    pages={337--364},
    year={2011},
    note={\url{https://doi.org/10.1215/21562261-1214375}},
}

\bib{FJMM13}{article}{
    label={FJMM13},
    title={Representations of quantum toroidal $\mathfrak{gl}_{n}$},
    author={B. Feigin},
    author={M. Jimbo},
    author={T. Miwa},
    author={E. Mukhin},
    journal={J. Algebra},
    volume={380},
    pages={78--108},
    year={2013},
    note={\url{https://doi.org/10.1016/j.jalgebra.2012.12.029}},
}

\bib{FJMM15}{article}{
    label={FJMM15},
    title={Quantum toroidal $\mathfrak{gl}_{1}$ and Bethe ansatz},
    author={B. Feigin},
    author={M. Jimbo},
    author={T. Miwa},
    author={E. Mukhin},
    journal={J. Phys. A: Math. Theor.},
    volume={48},
    number={24},
    year={2015},
    note={\url{https://doi.org/10.1088/1751-8113/48/24/244001}},
}

\bib{FJMM17}{article}{
    label={FJMM17},
    title={Finite type modules and Bethe ansatz for quantum toroidal $\mathfrak{gl}_{1}$},
    author={B. Feigin},
    author={M. Jimbo},
    author={T. Miwa},
    author={E. Mukhin},
    journal={Commun. Math. Phys.},
    volume={356},
    pages={285--327},
    year={2017},
    note={\url{https://doi.org/10.1007/s00220-017-2984-9}},
}

\bib{FJM19}{article}{
    label={FJM19},
    title={The $(\mathfrak{gl}_{m},\mathfrak{gl}_{n})$ duality in the quantum toroidal setting},
    author={B. Feigin},
    author={M. Jimbo},
    author={E. Mukhin},
    journal={Commun. Math. Phys.},
    volume={367},
    pages={455--481},
    year={2019},
    note={\url{https://doi.org/10.1007/s00220-019-03405-8}},
}

\bib{FH15}{article}{
    label={FH15},
    title={Baxter's relations and spectra of quantum integrable models},
    author={E. Frenkel},
    author={D. Hernandez},
    journal={Duke Math. J.},
    volume={164},
    number={12},
    pages={2407--2460},
    year={2015},
    note={\url{https://doi.org/10.1215/00127094-3146282}},
}

\bib{FH18}{article}{
    label={FH18},
    title={Spectra of Quantum KdV Hamiltonians, Langlands Duality, and Affine Opers},
    author={E. Frenkel},
    author={D. Hernandez},
    journal={Commun. Math. Phys.},
    volume={362},
    pages={361--414},
    year={2018},
    note={\url{https://doi.org/10.1007/s00220-018-3194-9}},
}

\bib{FJ88}{article}{
    label={FJ88},
    title={Vertex representations of quantum affine algebras},
    author={I. Frenkel},
    author={N. Jing},
    journal={Proc. Natl. Acad. Sci. U.S.A.},
    volume={85},
    number={24},
    pages={9373--9377},
    year={1988},
    note={\url{https://doi.org/10.1073/pnas.85.24.9373}},
}

\bib{FM01}{article}{
    label={FM01},
    title={Combinatorics of $q$-characters of finite-dimensional representations of quantum affine algebras},
    author={E. Frenkel},
    author={E. Mukhin},
    journal={Commun. Math. Phys.},
    volume={216},
    pages={23--57},
    date={2001},
    note={\url{https://doi.org/10.1007/s002200000323}},
}

\bib{FR92}{article}{
    label={FR92},
    title={Quantum affine algebras and holonomic difference equations},
    author={I. Frenkel},
    author={N. Reshetikhin},
    journal={Commun. Math. Phys.},
    volume={146},
    pages={1--60},
    date={1992},
    note={\url{https://doi.org/10.1007/BF02099206}},
}

\bib{FR99}{article}{
    label={FR99},
    title={The $q$-characters of representations of quantum affine algebras and deformations of $\mathcal{W}$-algebras},
    author={E. Frenkel},
    author={N. Reshetikhin},
    pages={163--205},
    book={
    title={Recent Developments in Quantum Affine Algebras and Related Topics},
    series={Contemp. Math.},
    volume={248},
    publisher={Amer. Math. Soc.},
    address={Providence, RI},
    date={1999},
    },
    note={\url{https://doi.org/10.1090/conm/248} (updated version available at \url{https://doi.org/10.48550/arXiv.math/9810055})},
}

\bib{Fujita20}{article}{
    label={F20},
    title={Geometric realization of Dynkin quiver type quantum affine Schur-Weyl duality},
    author={R. Fujita},
    journal={Int. Math. Res. Not.},
    volume={2020},
    number={22},
    pages={8353--8386},
    year={2020},
    note={\url{https://doi.org/10.1093/imrn/rny226}},
}

\bib{Fujita22}{article}{
    label={F22},
    title={Affine highest weight categories and quantum affine Schur-Weyl duality of Dynkin quiver types},
    author={R. Fujita},
    journal={Represent. Theory},
    volume={26},
    pages={211--263},
    year={2022},
    note={\url{https://doi.org/10.1090/ert/601}},
}

\bib{GTL16}{article}{
    label={GTL16},
    title={Yangians, quantum loop algebras, and abelian difference equations},
    author={S. Gautam},
    author={V. Toledano Laredo},
    journal={J. Amer. Math. Soc.},
    volume={29},
    number={3},
    pages={775--824},
    year={2016},
    note={\url{https://doi.org/10.1090/jams/851}},
}

\bib{GTL17}{article}{
    label={GTL17},
    title={Meromorphic tensor equivalence for Yangians and quantum loop algebras},
    author={S. Gautam},
    author={V. Toledano Laredo},
    journal={Publ. math. IH\'{e}S},
    volume={125},
    pages={267--337},
    year={2017},
    note={\url{https://doi.org/10.1007/s10240-017-0089-9}},
}

\bib{GKV95}{article}{
    label={GKV95},
    title={Langlands reciprocity for algebraic surfaces},
    author={V. Ginzburg},
    author={M. Kapranov},
    author={E. Vasserot},
    journal={Math. Res. Lett.},
    volume={2},
    number={2},
    pages={147--160},
    year={1995},
    note={\url{https://doi.org/10.4310/MRL.1995.v2.n2.a4}},
}

\bib{Grosse07}{article}{
    label={G07},
    title={On quantum shuffle and quantum affine algebras},
    author={P. Gross\'{e}},
    journal={J. Algebra},
    volume={318},
    number={2},
    pages={495--519},
    year={2007},
    note={\url{https://doi.org/10.1016/S0021-8693(03)00307-7}},
}

\bib{GNW18}{article}{
    label={GNW18},
    title={Coproduct for Yangians of affine Kac-Moody algebras},
    author={N. Guay},
    author={H. Nakajima},
    author={C. Wendlandt},
    journal={Adv. Math.},
    volume={338},
    pages={865--991},
    year={2018},
    note={\url{https://doi.org/10.1016/j.aim.2018.09.013}},
}

\bib{HJKL24}{article}{
    label={HJKL24},
    title={Young wall models for the level $1$ highest weight and Fock space crystals of $U_q(E_6^{(2)})$ and $U_q(F_4^{(1)})$},
    author={S. Han},
    author={Y. Jin},
    author={S.-J. Kang},
    author={D. Laurie},
    journal={Lett. Math. Phys.},
    volume={114},
    number={117},
    year={2024},
    note={\url{https://doi.org/10.1007/s11005-024-01845-5}},
}

\bib{Hernandez04}{article}{
    label={H04},
    title={Algebraic approach to $q,t$-characters},
    author={D. Hernandez},
    journal={Adv. Math.},
    volume={187},
    number={1},
    pages={1--52},
    year={2004},
    note={\url{https://doi.org/10.1016/j.aim.2003.07.016}},
}

\bib{Hernandez05}{article}{
    label={H05},
    title={Representations of quantum affinizations and fusion product},
    author={D. Hernandez},
    journal={Transform. Groups},
    volume={10},
    number={2},
    pages={163--200},
    year={2005},
    note={\url{https://doi.org/10.1007/s00031-005-1005-9}},
}

\bib{Hernandez07}{article}{
    label={H07},
    title={Drinfeld coproduct, quantum fusion tensor category and applications},
    author={D. Hernandez},
    journal={Proc. Lond. Math. Soc.},
    volume={95},
    number={3},
    pages={567--608},
    year={2007},
    note={\url{https://doi.org/10.1112/plms/pdm017}},
}

\bib{Hernandez09}{article}{
    label={H09},
    title={Quantum toroidal algebras and their representations},
    author={D. Hernandez},
    journal={Sel. Math. New Ser.},
    volume={14},
    number={3},
    pages={701--725},
    year={2009},
    note={\url{https://doi.org/10.1007/s00029-009-0502-4}},
}

\bib{HL10}{article}{
    label={HL10},
    title={Cluster algebras and quantum affine algebras},
    author={D. Hernandez},
    author={B. Leclerc},
    journal={Duke Math. J.},
    volume={154},
    number={2},
    pages={265--341},
    year={2010},
    note={\url{https://doi.org/10.1215/00127094-2010-040}},
}

\bib{HL13}{article}{
    label={HL13},
    title={Monoidal categorifications of cluster algebras of type $A$ and $D$},
    author={D. Hernandez},
    author={B. Leclerc},
    pages={175--193},
    book={
    title={Symmetries, Integrable Systems and Representations},
    series={Springer Proc. Math. Stat.},
    volume={40},
    publisher={Springer},
    address={London},
    date={2013},
    },
    note={\url{https://doi.org/10.1007/978-1-4471-4863-0_8}},
}

\bib{HL16}{article}{
    label={HL16},
    title={A cluster algebra approach to $q$-characters of Kirillov-Reshetikhin modules},
    author={D. Hernandez},
    author={B. Leclerc},
    journal={J. Eur. Math. Soc.},
    volume={18},
    number={5},
    pages={1113--1159},
    year={2016},
    note={\url{https://doi.org/10.4171/jems/609}},
}

\bib{HK02}{book}{
    label={HK02},
    title={Introduction to Quantum Groups and Crystal Bases},
    author={J. Hong},
    author={S.-J. Kang},
    year={2002},
    publisher={Amer. Math. Soc.},
    address={Providence, RI},
    series={Graduate Studies in Mathematics},
    volume={42},
    note={\url{https://doi-org.ezproxy-prd.bodleian.ox.ac.uk/10.1090/gsm/042}},
}

\bib{Ion03}{article}{
    label={I03},
    title={Involutions of double affine Hecke algebras},
    author={B. Ion},
    journal={Compos. Math.},
    volume={139},
    number={1},
    pages={67--84},
    year={2003},
    note={\url{https://doi.org/10.1023/B:COMP.0000005078.39268.8d}},
}

\bib{IS06}{article}{
    label={IS06},
    title={Triple groups and Cherednik algebras},
    author={B. Ion},
    author={S. Sahi},
    pages={183--206},
    book={
    title={Jack, Hall-Littlewood and Macdonald Polynomials},
    series={Contemp. Math.},
    volume={417},
    publisher={Amer. Math. Soc.},
    address={Providence, RI},
    date={2006},
    },
    note={\url{https://doi.org/10.1090/conm/417}},
}

\bib{IS20}{book}{
    label={IS20},
    title={Double affine Hecke algebras and congruence groups},
	author={B. Ion},
        author={S. Sahi},
	year={2020},
	publisher={Amer. Math. Soc.},
        address={Providence, RI},
        volume={268},
        number={1305},
        series={Memoirs},
        note={\url{https://doi.org/10.1090/memo/1305}},
}

\bib{Jimbo86}{article}{
    label={Jim86},
    title={A $q$-analogue of $U(\mathfrak{gl}(N+1))$, Hecke algebra, and the Yang-Baxter equation},
    author={M. Jimbo},
    journal={Lett. Math. Phys.},
    volume={11},
    pages={247--252},
    year={1986},
    note={\url{https://doi.org/10.1007/BF00400222}},
}

\bib{JM24}{article}{
    label={JM24},
    title={Combinatorial bases in quantum toroidal $\mathfrak{gl}_{2}$ modules},
    author={M. Jimbo},
    author={E. Mukhin},
    journal={arXiv preprint},
    year={2024},
    note={\url{https://doi.org/10.48550/arXiv.2403.16705}},
}

\bib{Jing89}{thesis}{
    label={Jin89},
    title={Vertex operators, symmetric functions and their $q$-deformations},
    author={N. Jing},
    organization={Yale University},
    type={Ph.D. Thesis},
    year={1989},
}

\bib{Jing90}{article}{
    label={Jin90},
    title={Twisted vertex representations of quantum affine algebras},
    author={N. Jing},
    journal={Invent. Math.},
    volume={102},
    number={3},
    pages={663--690},
    year={1990},
    note={\url{https://doi.org/10.1007/BF01233443}},
}

\bib{Jing98}{article}{
    label={Jin98a},
    title={On Drinfeld realization of quantum affine algebras},
    author={N. Jing},
    conference={
    title={The Monster and Lie Algebras: Proceedings of a Special Research Quarter at the Ohio State University},
    date={1996},
    },
    book={
    address={Berlin, New York},
    volume={7},
    publisher={De Gruyter},
    date={1998},
    },
    pages={195--206},
    note={\url{https://doi.org/10.1515/9783110801897.195}},
}

\bib{Jing98(2)}{article}{
    label={Jin98b},
    title={Quantum Kac-Moody Algebras and Vertex Representations},
    author={N. Jing},
    journal={Lett. Math. Phys.},
    address={Berlin, New York: De Gruyter},
    volume={44},
    number={4},
    pages={261--271},
    year={1998},
    note={\url{https://doi.org/10.1023/A:1007493921464}},
}

\bib{JZ07}{article}{
    label={JZ07},
    title={Drinfeld Realization of Twisted Quantum Affine Algebras},
    author={N. Jing},
    author={H. Zhang},
    journal={Commun. Algebra},
    volume={35},
    number={11},
    pages={3683--3698},
    year={2007},
    note={\url{https://doi.org/10.1080/00927870701404713}},
}

\bib{JZ10}{article}{
    label={JZ10},
    title={Addendum to ``Drinfeld Realization of Twisted Quantum Affine Algebras''},
    author={N. Jing},
    author={H. Zhang},
    journal={Commun. Algebra},
    volume={38},
    number={9},
    pages={3484--3488},
    year={2010},
    note={\url{https://doi.org/10.1080/00927870902933213}},
}

\bib{JZ22}{article}{
    label={JZ22},
    title={On Hopf algebraic structures of quantum toroidal algebras},
    author={N. Jing},
    author={H. Zhang},
    journal={Commun. Algebra},
    volume={51},
    number={3},
    pages={1135--1157},
    year={2022},
    note={\url{https://doi.org/10.1080/00927872.2022.2127604}},
}

\bib{Kac90}{book}{
    label={Kac90},
    title={Infinite-dimensional Lie algebras},
    author={V. Kac},
    year={1990},
    publisher={Cambridge University Press},
    note={\url{https://doi.org/10.1017/CBO9780511626234}},
}

\bib{KKK15}{article}{
    label={KKK15},
    title={Symmetric quiver Hecke algebras and $R$-matrices of quantum affine algebras, II},
    author={S.-J. Kang},
    author={M. Kashiwara},
    author={M. Kim},
    journal={Duke Math. J.},
    volume={164},
    number={8},
    pages={1549--1602},
    year={2015},
    note={\url{https://doi.org/10.1215/00127094-3119632}},
}

\bib{KKK18}{article}{
    label={KKK18},
    title={Symmetric quiver Hecke algebras and $R$-matrices of quantum affine algebras},
    author={S.-J. Kang},
    author={M. Kashiwara},
    author={M. Kim},
    journal={Invent. Math.},
    volume={211},
    pages={591--685},
    year={2018},
    note={\url{https://doi.org/10.1007/s00222-017-0754-0}},
}

\bib{KKKO18}{article}{
    label={KKKO18},
    title={Monoidal categorification of cluster algebras},
    author={S.-J. Kang},
    author={M. Kashiwara},
    author={M. Kim},
    author={S.-j. Oh},
    journal={J. Amer. Math. Soc.},
    volume={31},
    number={2},
    pages={349--426},
    year={2018},
    note={\url{https://doi.org/10.1090/jams/895}},
}

\bib{KK08}{article}{
    label={KK08},
    title={Fock space representations of quantum affine algebras and generalized Lascoux-Leclerc-Thibon algorithm},
    author={S.-J. Kang},
    author={J.-H. Kwon},
    journal={J. Korean Math. Soc.},
    volume={45},
    number={4},
    pages={1135--1202},
    year={2008},
    note={\url{https://doi.org/10.4134/JKMS.2008.45.4.1135}},
}

\bib{Kashiwara94}{article}{
    label={Kas94},
    title={Crystal bases of modified quantized enveloping algebra},
    author={M. Kashiwara},
    journal={Duke Math. J.},
    volume={73},
    number={2},
    pages={383--413},
    year={1994},
    note={\url{https://doi.org/10.1215/S0012-7094-94-07317-1}},
}

\bib{KMPY96}{article}{
    label={KMPY96},
    title={Perfect crystals and $q$-deformed Fock spaces},
    author={M. Kashiwara},
    author={T. Miwa},
    author={J.-U. H. Petersen},
    author={C. M. Yung},
    journal={Sel. Math. New Ser.},
    volume={2},
    number={3},
    pages={415--499},
    year={1996},
    note={\url{https://doi.org/10.1007/BF01587950}},
}

\bib{KT91}{article}{
    label={KT91},
    title={Universal $R$-matrix for quantized (super)algebras},
    author={S. M. Khoroshkin},
    author={V. N. Tolstoy},
    journal={Commun. Math. Phys.},
    volume={141},
    pages={599--617},
    year={1991},
    note={\url{https://doi.org/10.1007/BF02102819}},
}

\bib{KT92}{article}{
    label={KT92},
    title={Universal $R$-matrix for quantum untwisted affine Lie algebras},
    author={S. M. Khoroshkin},
    author={V. N. Tolstoy},
    journal={Funct. Anal. Appl.},
    volume={26},
    pages={69--71},
    year={1992},
    note={\url{https://doi.org/10.1007/BF01077085}},
}

\bib{KT93}{article}{
    label={KT93},
    title={The Cartan-Weyl basis and the universal $R$-matrix for quantum Kac-Moody algebras},
    author={S. M. Khoroshkin},
    author={V. N. Tolstoy},
    pages={69--71},
    book={
    title={Quantum Symmetries},
    series={Proc. Int. Workshop Math. Phys.},
    publisher={World Sci. Publ. Co.},
    address={Singapore},
    date={1993},
    },
    note={\url{https://doi.org/10.1142/9789814535502}},
}

\bib{KR90}{article}{
    label={KR90},
    title={$q$-Weyl Group and a multiplicative formula for universal $R$-matrices},
    author={A. N. Kirillov},
    author={N. Reshetikhin},
    journal={Commun. Math. Phys.},
    volume={134},
    pages={421--431},
    year={1990},
    note={\url{https://doi.org/10.1007/BF02097710}},
}

\bib{Laurie24a}{article}{
    label={La24a},
    title={Automorphisms of quantum toroidal algebras from an action of the extended double affine braid group},
    author={D. Laurie},
    journal={Algebr. Represent. Theory},
    volume={27},
    pages={2067--2097},
    year={2024},
    note={\url{https://doi.org/10.1007/s10468-024-10291-9}},
}

\bib{Laurie24b}{thesis}{
    label={La24b},
    title={Quantum toroidal algebras, quantum affine algebras, and their representation theory},
    author={D. Laurie},
    organization={University of Oxford},
    type={Ph.D. Thesis},
    year={2024},
    note={\url{http://dx.doi.org/10.5287/ora-9ovpag5bo}},
}

\bib{Laurie25}{article}{
    label={La25},
    title={Young wall realizations of level $1$ irreducible highest weight and Fock space crystals of quantum affine algebras in type $E$},
    author={D. Laurie},
    journal={J. Algebra},
    volume={661},
    pages={430--478},
    year={2025},
    note={\url{https://doi.org/10.1016/j.jalgebra.2024.07.047}},
}

\bib{LS91}{article}{
    label={LS91},
    title={Quantum Weyl group and multiplicative formula for the $R$-matrix of a simple Lie algebra},
    author={S. Z. Levendorskii},
    author={Y. S. Soibel'man},
    journal={Funct. Anal. Appl.},
    volume={25},
    pages={143--145},
    year={1991},
    note={\url{https://doi.org/10.1007/BF01079599}},
}

\bib{Lusztig93}{book}{
    label={Lu93},
    title={Introduction to quantum groups},
    author={G. Lusztig},
    year={1993},
    publisher={Birkh\"{a}user Boston},
    series={Progress in Mathematics no. 110},
    note={\url{https://doi.org/10.1007/978-0-8176-4717-9}},
}

\bib{MNNZ24}{book}{
    label={MNNZ24},
    title={Quantum toroidal algebras and solvable structures in gauge/string theory},
    author={Y. Matsuo},
    author={S. Nawata},
    author={G. Noshita},
    author={R.-D. Zhu},
    year={2024},
    publisher={Elsevier},
    volume={1055},
    series={Physics Reports},
    note={\url{https://doi.org/10.1016/j.physrep.2023.12.003}},
}

\bib{Miki99}{article}{
    label={M99},
    title={Toroidal braid group action and an automorphism of toroidal algebra $U_{q}(\mathfrak{sl}_{n+1,tor})$ $(n\geq 2)$},
    author={K. Miki},
    journal={Lett. Math. Phys.},
    volume={47},
    number={4},
    pages={365--378},
    year={1999},
    note={\url{https://doi.org/10.1023/A:1007556926350}},
}

\bib{Miki00}{article}{
    label={M00},
    title={Representations of quantum toroidal algebra $U_{q}(\mathfrak{sl}_{n+1,tor})$ $(n\geq 2)$},
    author={K. Miki},
    journal={J. Math. Phys.},
    volume={41},
    number={10},
    pages={7079--7098},
    year={2000},
    note={\url{https://doi.org/10.1063/1.1287436}},
}

\bib{Miki01}{article}{
    label={M01},
    title={Quantum toroidal algebra $U_{q}(\mathfrak{sl}_{2,tor})$ and $R$-matrices},
    author={K. Miki},
    journal={J. Math. Phys.},
    volume={42},
    number={5},
    pages={2293--2308},
    year={2001},
    note={\url{https://doi.org/10.1063/1.1357198}},
}

\bib{Miki07}{article}{
    label={M07},
    title={A $(q,\gamma)$ analog of the $W_{1+\infty}$ algebra},
    author={K. Miki},
    journal={J. Math. Phys.},
    volume={48},
    number={12},
    pages={2293--2308},
    year={2007},
    note={\url{https://doi.org/10.1063/1.2823979}},
}

\bib{MRY90}{article}{
    label={MRY90},
    title={Toroidal Lie algebras and vertex representations},
    author={R. V. Moody},
    author={S. E. Rao},
    author={T. Yokonuma},
    journal={Geom. Dedicata.},
    volume={35},
    pages={283--307},
    year={1990},
    note={\url{https://doi.org/10.1007/BF00147350}},
}

\bib{Nakajima01}{article}{
    label={Na01},
    title={Quiver varieties and finite-dimensional representations of quantum affine algebras},
    author={H. Nakajima},
    journal={J. Am. Math. Soc.},
    volume={14},
    number={1},
    pages={145--238},
    year={2001},
    note={\url{https://doi.org/10.1090/S0894-0347-00-00353-2}},
}

\bib{Nakajima02}{article}{
    label={Na02},
    title={Geometric construction of representations of affine algebras},
    author={H. Nakajima},
    pages={423--438},
    conference={
    title={Proceedings of the International Congress of Mathematicians},
    date={2002},
    address={Beijing},
    },
    book={
    address={Beijing},
    volume={1},
    publisher={Higher Ed. Press},
    date={2002},
    },
    note={\url{https://doi.org/10.48550/arXiv.math/0212401}},
}

\bib{Nakajima11}{article}{
    label={Na11},
    title={Quiver varieties and cluster algebras},
    author={H. Nakajima},
    journal={Kyoto J. Math.},
    volume={51},
    number={1},
    pages={71--126},
    year={2011},
    note={\url{https://doi.org/10.1215/0023608X-2010-021}},
}

\bib{Negut15}{thesis}{
    label={Ne15},
    title={Quantum algebras and cyclic quiver varieties},
    author={A. Negu\c{t}},
    organization={Columbia University},
    type={Ph.D. Thesis},
    year={2015},
    note={\url{https://doi.org/10.7916/D8J38RGF}},
}

\bib{Negut20}{article}{
    label={Ne20},
    title={Quantum toroidal and shuffle algebras},
    author={A. Negu\c{t}},
    journal={Adv. Math.},
    volume={372},
    number={107288},
    year={2020},
    note={\url{https://doi.org/10.1016/j.aim.2020.107288}},
}

\bib{Negut23}{article}{
    label={Ne23},
    title={The $R$-matrix of the quantum toroidal algebra},
    author={A. Negu\c{t}},
    journal={Kyoto J. Math.},
    volume={63},
    number={1},
    pages={23--49},
    year={2023},
    note={\url{https://doi.org/10.1215/21562261-2022-0030}},
}

\bib{Negut24}{article}{
    label={Ne24},
    title={A tale of two shuffle algebras},
    author={A. Negu\c{t}},
    journal={Sel. Math. New Ser.},
    volume={30},
    number={62},
    year={2024},
    note={\url{https://doi.org/10.1007/s00029-024-00941-7}},
}

\bib{OS24}{article}{
    label={OS24},
    title={Wreath Macdonald polynomials, a survey},
    author={D. Orr},
    author={M. Shimozono},
    pages={123--170},
    book={
    title={A Glimpse into Geometric Representation Theory},
    series={Contemp. Math.},
    volume={804},
    publisher={Amer. Math. Soc.},
    address={Providence, RI},
    date={2024},
    },
    note={\url{https://doi.org/10.1090/conm/804}},
}

\bib{OSW22}{article}{
    label={OSW22},
    title={Wreath Macdonald operators},
    author={D. Orr},
    author={M. Shimozono},
    author={J. J. Wen},
    journal={arXiv preprint},
    year={2022},
    note={\url{https://doi.org/10.48550/arXiv.2211.03851}},
}

\bib{Premat04}{article}{
    label={P04},
    title={Fock space representations and crystal bases for  $C_n^{(1)}$},
    author={A. Premat},
    journal={J. Algebra},
    volume={278},
    number={1},
    pages={227--241},
    year={2004},
    note={\url{https://doi.org/10.1016/j.jalgebra.2004.01.028}},
}

\bib{Qin17}{article}{
    label={Q17},
    title={Triangular bases in quantum cluster algebras and monoidal categorification conjectures},
    author={F. Qin},
    journal={Duke Math. J.},
    volume={166},
    number={12},
    pages={2337--2442},
    year={2017},
    note={\url{https://doi.org/10.1215/00127094-2017-0006}},
}

\bib{RT90}{article}{
    label={RT90},
    title={Ribbon graphs and their invariants derived from quantum groups},
    author={N. Y. Reshetikhin},
    author={V. G. Turaev},
    journal={Commun. Math. Phys.},
    volume={127},
    pages={1--26},
    year={1990},
    note={\url{https://doi.org/10.1007/BF02096491}},
}

\bib{RT91}{article}{
    label={RT91},
    title={Invariants of 3-manifolds via link polynomials and quantum groups},
    author={N. Y. Reshetikhin},
    author={V. G. Turaev},
    journal={Invent. Math.},
    volume={103},
    pages={547--597},
    year={1991},
    note={\url{https://doi.org/10.1007/BF01239527}},
}

\bib{Saito98}{article}{
    label={Sa98},
    title={Quantum toroidal algebras and their vertex representations},
    author={Y. Saito},
    journal={Publ. Res. Inst. Math. Sci.},
    volume={34},
    number={2},
    pages={155--177},
    year={1998},
    note={\url{https://doi.org/10.2977/prims/1195144759}},
}

\bib{STU98}{article}{
    label={STU98},
    title={Toroidal actions on level 1 modules of $U_q(\hat{\mathfrak{sl}}_{n})$},
    author={Y. Saito},
    author={K. Takemura},
    author={D. Uglov},
    journal={Transform. Groups},
    volume={3},
    number={1},
    pages={75--102},
    year={1998},
    note={\url{https://doi.org/10.1007/BF01237841}},
}

\bib{Schiffmann12}{article}{
    label={Sc12},
    title={Drinfeld realization of the elliptic Hall algebra},
    author={O. Schiffmann},
    journal={J. Algebr. Comb.},
    volume={35},
    pages={237--262},
    year={2012},
    note={\url{https://doi.org/10.1007/s10801-011-0302-8}},
}

\bib{SV13}{article}{
    label={SV13},
    title={The elliptic Hall algebra and the $K$-theory of the Hilbert scheme of $\mathbb{A}^{2}$},
    author={O. Schiffmann},
    author={E. Vasserot},
    journal={Duke Math. J.},
    volume={162},
    number={2},
    pages={279--366},
    year={2013},
    note={\url{https://doi.org/10.1215/00127094-1961849}},
}

\bib{Soibelman99}{article}{
    label={So99},
    title={The meromorphic braided category arising in quantum affine algebras},
    author={Y. Soibelman},
    journal={Int. Math. Res. Not.},
    volume={1999},
    number={19},
    pages={1067--1079},
    year={1999},
    note={\url{https://doi.org/10.1155/S1073792899000574}},
}

\bib{Tsymbaliuk19}{article}{
    label={T19},
    title={Several realizations of Fock modules for toroidal $\Ddot{U}_{q,d}(\mathfrak{sl}_{n})$},
    author={A. Tsymbaliuk},
    journal={Algebr. Represent. Theory},
    volume={22},
    pages={177--209},
    year={2019},
    note={\url{https://doi.org/10.1007/s10468-017-9761-5}},
}

\bib{Tsymbaliuk23}{book}{
    label={T23},
    title={Shuffle Approach Towards Quantum Affine and Toroidal Algebras},
    author={A. Tsymbaliuk},
    year={2023},
    publisher={Springer Singapore},
    series={SpringerBriefs in Mathematical Physics},
    volume={49},
    note={\url{https://doi.org/10.1007/978-981-99-3150-7}},
}

\bib{Ueda20}{article}{
    label={U20},
    title={Coproduct for the Yangian of type $A_{2}^{(2)}$},
    author={M. Ueda},
    journal={RIMS Kokyuroku},
    volume={2161},
    pages={181--194},
    year={2020},
    note={\url{https://www.kurims.kyoto-u.ac.jp/~kyodo/kokyuroku/contents/pdf/2161-17.pdf}},
}

\bib{VV96}{article}{
    label={VV96},
    title={Schur duality in the toroidal setting},
    author={M. Varagnolo},
    author={E. Vasserot},
    journal={Commun. Math. Phys.},
    volume={182},
    number={2},
    pages={469--483},
    year={1996},
    note={\url{https://doi.org/10.1007/BF02517898}},
}

\bib{VV98}{article}{
    label={VV98},
    title={Double-loop Algebras and the Fock Space},
    author={M. Varagnolo},
    author={E. Vasserot},
    journal={Invent. Math.},
    volume={133},
    number={1},
    pages={133--159},
    year={1998},
    note={\url{https://doi.org/10.1007/s002220050242}},
}

\bib{VV02}{article}{
    label={VV02},
    title={Standard modules of quantum affine algebras},
    author={M. Varagnolo},
    author={E. Vasserot},
    journal={Duke Math. J.},
    volume={111},
    number={3},
    pages={509--533},
    year={2002},
    note={\url{https://doi.org/10.1215/S0012-7094-02-11135-1}},
}

\bib{VV23a}{article}{
    label={VV23a},
    title={Quantum loop groups and critical convolution algebras},
    author={M. Varagnolo},
    author={E. Vasserot},
    journal={arXiv preprint},
    year={2023},
    note={\url{https://doi.org/10.48550/arXiv.2302.01418}},
}

\bib{VV23b}{article}{
    label={VV23b},
    title={Non symmetric quantum loop groups and K-theory},
    author={M. Varagnolo},
    author={E. Vasserot},
    journal={arXiv preprint},
    year={2023},
    note={\url{https://doi.org/10.48550/arXiv.2308.01809}},
}

\bib{Wen19}{article}{
    label={W19},
    title={Wreath Macdonald polynomials as eigenstates},
    author={J. J. Wen},
    journal={arXiv preprint},
    year={2019},
    note={\url{https://doi.org/10.48550/arXiv.1904.05015}},
}

\end{biblist}
\end{bibsection}
\end{document}